\documentclass[11pt,a4paper]{article}

\usepackage[ngerman,english]{babel}
\usepackage[utf8]{inputenc}

\usepackage[DIV=12]{typearea}

\usepackage{stmaryrd}

\usepackage[affil-it]{authblk}

\usepackage{xparse}
\usepackage{amsmath, amsfonts, amssymb, amsthm}
\usepackage{mathtools}
\usepackage{enumitem}

\usepackage{microtype}

\usepackage[noadjust]{cite}

\usepackage[capitalise]{cleveref}

\crefname{equation}{}{}

\usepackage{subcaption}

\usepackage{tikz}
\usepackage{standalone}
\usetikzlibrary{decorations.markings}
\usetikzlibrary{decorations.pathreplacing}

\usepackage{calc}

\usepackage{pgf}

\newcommand{\R}{\mathbb{R}}

\newcommand{\N}{\mathbb{N}}
\newcommand{\C}{\mathbb{C}}
\newcommand{\Ccut}{\mathbb{C}_{-}} % complex plane with a cut
\newcommand{\HH}{\mathbb{H}} %upper and lower half-plane
\newcommand{\ee}{\mathrm{e}}

\DeclareDocumentCommand\dd{ o g d() }{
	\IfNoValueTF{#2}{
		\IfNoValueTF{#3}
			{\mathrm{d}\IfNoValueTF{#1}{}{^{#1}}}
			{\mathinner{\mathrm{d}\IfNoValueTF{#1}{}{^{#1}}\argopen(#3\argclose)}}
		}
		{\mathinner{\mathrm{d}\IfNoValueTF{#1}{}{^{#1}}#2} \IfNoValueTF{#3}{}{(#3)}}
	}

\newcommand{\dx}{\dd{x}}
\newcommand{\dy}{\dd{y}}
\newcommand{\dz}{\dd{z}}
\newcommand{\dr}{\dd{r}}
\newcommand{\deta}{\dd{\eta}}
\newcommand{\dsig}{\dd{\sigma}}
\newcommand{\dxi}{\dd{\xi}}
\newcommand{\ds}{\dd{s}}
\newcommand{\dt}{\dd{t}}

\newcommand{\dtau}{\dd{\tau}}
\newcommand{\del}{\partial}
\newcommand{\eps}{\varepsilon}
\newcommand{\im}{\mathrm{i}}
\newcommand{\sgn}{\operatorname{sgn}}

\newcommand{\blp}{\mathfrak{a}} %parameter for the boundary layer estimate more precisely the integral estimate on Q0

\newcommand{\W}{W} 

\newcommand{\Qo}{\mathfrak{H}}
\newcommand{\U}{U}

\newcommand{\mom}{\mathfrak{m}}

\newcommand{\Ker}{\mathfrak{W}} % inverse 'Laplace transform' of W
\newcommand{\hol}{\mathfrak{r}} %Hölder exponent for W

\newcommand{\cut}{\eta} %cut-off function for boundary layer estimate 
\newcommand{\coup}{\gamma} %cut-off function for boundary layer estimate 

\newcommand{\T}{\mathcal{T}} %Laplace transform
\newcommand{\M}{\mathcal{M}} %space of measures
\newcommand{\LL}{\mathcal{L}} %linearised coagulation operator

\newcommand{\Pert}{\mathcal{N}} %nonlinear operator in uniform convergence proof

\renewcommand{\Re}{\operatorname{Re}}
\renewcommand{\Im}{\operatorname{Im}}

\DeclarePairedDelimiterXPP\lsnorm[4]{}{\lbrack}{\rbrack}{_{#1,#2,#3}}{#4}
\DeclarePairedDelimiterXPP\lfnorm[4]{}{\lVert}{\rVert}{_{#1,#2,#3}}{#4}

\newcommand{\X}[3]{X_{#1,#2,#3}}

\newcommand{\weight}[2]{\Lambda_{#1,#2}} % to simplify the notation, encodes the growth properties of the weight

\newcommand{\vcc}{\vcentcolon}

\DeclarePairedDelimiter\abs{\lvert}{\rvert}

\theoremstyle{plain}
\newtheorem{theorem}{Theorem}[section]
\newtheorem{lemma}[theorem]{Lemma}
\newtheorem{proposition}[theorem]{Proposition}

\theoremstyle{definition}
\newtheorem{definition}[theorem]{Definition}

\theoremstyle{remark}
\newtheorem{remark}[theorem]{Remark}

\numberwithin{equation}{section}

\title{Uniqueness of fat-tailed self-similar profiles to Smoluchowski's coagulation equation for a perturbation of the constant kernel}
\author{Sebastian Throm
  \thanks{E-mail: \texttt{throm@ma.tum.de}}}
\affil{Faculty of Mathematics, Technical University of Munich, Boltzmannstraße~3, 85748 Garching bei München, Germany}
\date{}

\begin{document}
 
\maketitle

\begin{abstract}
 This article is concerned with the question of uniqueness of self-similar profiles for Smoluchowski's coagulation equation which exhibit algebraic decay (fat tails) at infinity. More precisely, we consider a rate kernel $K$ which can be written as $K=2+\eps W$. The perturbation is assumed to have homogeneity zero and might also be singular both at zero and at infinity. Under further regularity assumptions on $W$, we will show that for sufficiently small  $\eps$ there exists, up to normalisation of the tail behaviour at infinity, at most one self-similar profile.
\end{abstract}

\section{Introduction}
 
 In this article, we study Smoluchowski's coagulation equation which reads as
\begin{equation}\label{eq:Smol}
 \del_{t}\phi(x,t)=\frac{1}{2}\int_{0}^{x}K(x-y,x)\phi(x-y,t)\phi(y,t)\dy-\phi(x,t)\int_{0}^{\infty}K(x,y)\phi(y,t)\dy.
\end{equation}
This model is used to describe a system of aggregating particles and has been derived by Marian von Smoluchowski in~\cite{Smo17}. The common interpretation of~\eqref{eq:Smol} is the following. The quantity $\phi(x,t)$ represents the number density of particles of size $x$ at time $t$, while the evolution of $\phi$ is given by the two integral operators on the right-hand side. The first one of these, the gain term, accounts for the creation of particles of size $x$ due to the collision of two particles of sizes $y$ and $x-y$ while we have to divide by two because of the symmetry of this process. On the other hand, particles of size $x$ will disappear from the system because they aggregate with other clusters and this is considered by the second operator on the right-hand side of~\eqref{eq:Smol}. The specific example which Smoluchowski had in mind when deriving this equation is the coagulation of gold particles in a colloidal solution under the effect of Brownian motion. Assuming that all the clusters in the system can at least be approximated by spheres, Smoluchowski obtained the integral kernel
\begin{equation}\label{eq:Smol:kernel}
 K(x,y)=\bigl(x^{1/3}+y^{1/3}\bigr)\bigl(x^{-1/3}+y^{-1/3}\bigr).
\end{equation}
Here, $x$ denotes the mass/volume of the cluster, while $x^{1/3}$ then corresponds, up to a constant, to the radius. The first factor in this expression is thus proportional to the interaction length of two clusters $x$ and $y$ while the second factor accounts for their diffusion constants and scales as inverse radii.

Since the original derivation by Smoluchowski, this model has been used to describe a broad range of different phenomena in many different branches of the natural sciences and also on very different length scales ranging from microscopic phenomena like aerosol physics (e.g.\@ \cite{Fri00,PrK10}) up to the formation of stars and galaxies in astrophysics (\cite{Dra72}). On the other hand, this equation is also used in biology to describe growth processes (e.g.\@ \cite{AcF97}) and in chemistry to model polymerisation (\cite{Zif80}) for example. More details and several references to these and further applications may be found in the review article~\cite{LaM04} as well as in~\cite{Dra72}.

The well-posedness of~\eqref{eq:Smol} is by now established for a large class of coagulation kernels $K$ that cover most of the cases that one typically finds in applications and especially also~\eqref{eq:Smol:kernel}. Results on the existence and uniqueness of solutions for~\eqref{eq:Smol} may be found for example in~\cite{FoL06a,Nor99} while more references can be found in~\cite{LaM04}.

\subsection{Long-time behaviour and self-similarity}\label{Sec:profiles}

In most applications the coagulation kernel $K$ occurring in~\eqref{eq:Smol} is a homogeneous function of a certain degree $\lambda$, while~\eqref{eq:Smol:kernel} for example satisfies this condition with $\lambda=0$. For such kernels, one fundamental problem in the context of Smoluchowski's coagulation equation concerns the long-time behaviour of solutions to~\eqref{eq:Smol}. On the one hand, physical considerations and numerical simulations suggest that the solutions of~\eqref{eq:Smol} for large times converge towards a self-similar shape $f$ in the sense that 
\begin{equation}\label{eq:longtime:convergence}
 \bigl(s(t)\bigr)^{1+\rho}\phi(s(t)x,t)\longrightarrow f(x)\quad \text{for }t\longrightarrow \infty
\end{equation}
with a scaling function $s(t)\to \infty$ as $t\to\infty$ (\cite{LaM04}). Here, the self-similar profile $f$ is expected to be given as a special solution to~\eqref{eq:Smol} of the form
\begin{equation}\label{eq:ansatz}
 \phi(x,t)=\frac{1}{(s(t))^{1+\rho}}f\biggl(\frac{x}{s(t)}\biggr).
\end{equation}
This conjecture is commonly known as the scaling hypothesis while a proof is still lacking in almost all cases.  The so far only exceptions are the three so-called solvable kernels $K=2$, $K(x,y)=x+y$ and $K(x,y)=xy$ as well as the diagonal kernel $K(x,y)=x^{1+\lambda}\delta(x-y)$. In the case of the solvable kernels solution formulas in terms of the Laplace transform can be computed explicitly and as a consequence the self-similar profiles exist and they are in particular unique, while also~\eqref{eq:longtime:convergence} could be established in the sense of weak convergence of measures (see~\cite{MeP04}). We will consider the constant kernel $K=2$ in some more detail below since this case is of particular interest for this work.

On the other hand, for the diagonal kernel no explicit formulas are available but the structure of the kernel only allows for clusters of exactly the same size to aggregate which reduces the corresponding integral equation to some non-local ordinary differential equation which is much easier to treat (\cite{LNV16}).

\subsection{The equation for self-similar profiles}\label{Sec:equation:profiles}

In the following, we will only restrict to the case $\lambda<1$. The reason for this is that for kernels with $\lambda>1$ a phenomenon known as gelation occurs which we do not want to consider here (for more details see for example~\cite{LaM04}) and $\lambda=1$ represents the borderline which thus has also a special behaviour (see for example~\cite{HNV17}).

To obtain the equation that has to be satisfied by the self-similar profile $f$, one formally plugs the ansatz~\eqref{eq:ansatz} into~\eqref{eq:Smol}, which leads to
\begin{equation}\label{eq:formal:self:sim}
 (1+\rho) f(y)+yf'(y)+\frac{1}{2}\int_{0}^{y}K(y-z,z)f(y-z)f(z)\dz-f(y)\int_{0}^{\infty}K(y,z)f(z)\dz=0.
\end{equation}
Moreover, one also finds
\begin{equation}\label{eq:scaling}
 s(t)=\bigl((\rho-\lambda)t\bigr)^{1/(\rho-\lambda)},
\end{equation}
while we note that in order to obtain $s(t)\to\infty$ we have to require $\rho>\lambda$.

It turns out to be more convenient to work with a more regularised version of this equation. Precisely, this means that one multiplies~\eqref{eq:formal:self:sim} by $y$ and integrates over $(0,x)$ which, after some elementary manipulations, leads to
\begin{equation}\label{eq:self:sim}
 x^{2}f(x)=(1-\rho)\int_{0}^{x}yf(y)\dy+\int_{0}^{x}\int_{x-y}^{\infty}yK(y,z)f(y)f(z)\dz\dy.
\end{equation}
This is the equation that we will study in this work and more precisely we will show that for a particular class of kernels there exists up to a certain normalisation at most one solution.

\subsection{Finite mass, fat-tailed profiles and scale invariance}\label{Sec:scale:invariance}

One fundamental property of~\eqref{eq:Smol} is the formal conservation of the total mass. More precisely this means that the first moment $\int_{(0,\infty)}x\phi(x,t)\dx$, i.e.\@ the total mass, of a solution $\phi$ of~\eqref{eq:Smol} is at least on a formal level constant in time if the same quantity is finite initially, i.e.\@ at $t=0$. If this is the case, one should also expect, that the first moment of the ansatz~\eqref{eq:ansatz} equals this constant which immediately leads to $\rho=1$.

However, in the case of the solvable kernels it has been shown in~\cite{MeP04} that besides these profiles with finite mass, which decay exponentially at infinity, there exists a whole family of self-similar profiles which have in general algebraic decay at infinity. In particular, for $K=2$ which has homogeneity $\lambda=0$, these profiles are parametrised by the parameter $\rho\in(0,1)$ and it holds $f_{\rho}(x)\sim C_{f_{\rho}}x^{-1-\rho}$ for $x\to\infty$ such that the total mass of $f_{\rho}$, i.e. the first moment, is infinite (see Section~\ref{Sec:constant:kernel}).

At first sight, it seems unphysical to consider self-similar profiles with infinite mass since in applications, the mass density should always be finite. However, it has been shown in the case of the constant kernel in~\cite{MeP04} that a solution $\phi$ to~\eqref{eq:Smol} is attracted by $f_{\rho}$ if and only if it holds $\int_{0}^{x}y\phi(y,0)\dy\sim x^{1-\rho}L(x)$ for $x\to\infty$ where $L$ is a function which is slowly varying at infinity. This means that the long-time behaviour of a solution $\phi$ of~\eqref{eq:Smol} depends only on the tail-behaviour of the initial condition at infinity. On the other hand, it is also clear that~\eqref{eq:Smol} cannot be valid for extremely large cluster sizes $x\gg 1$ in real applications. Thus, fat-tailed profiles might still be important to describe the long-time behaviour of coagulation processes. To emphasise this point let us consider the two initial conditions $\phi_{0}$ and $\widehat{\phi}_{0}=\phi_{0}\chi_{[0,R]}$ with $R\gg 1$ and $\phi_{0}(x)\sim x^{-1-\rho}$ for $x\to\infty$. If we denote by $\phi$ and $\widehat{\phi}$ the corresponding solutions of~\eqref{eq:Smol} it holds for the constant kernel $K=2$ that $\phi$ is attracted by the profile $f_{\rho}$ while $\widehat{\phi}$ is attracted by the finite-mass profile $f_{1}(x)=C_{1}\ee^{-x}$. However, as mentioned before, \eqref{eq:Smol} cannot capture the behaviour of the system for arbitrarily large clusters and thus, there is a priori no reason why the long-time asymptotics of $\widehat{\phi}$ should be better described by $f_{1}$ than by $f_{\rho}$.

Another important point, especially in the context of uniqueness of self-similar profiles, is a certain scale invariance of~\eqref{eq:self:sim}. Precisely, this means that for each $f$ which solves~\eqref{eq:self:sim} also the rescaled function $f_{a}(x)=af(ax)$ is a solution of~\eqref{eq:self:sim}. Due to this, uniqueness of profiles can only hold up to normalisation and it turns out to be convenient to fix the free parameter $a$ by specifying the tail-behaviour of the self-similar solutions at infinity (see Theorem~\ref{Thm:asymptotics}).

\subsection{Existence and uniqueness of self-similar profiles}

The existence of self-similar profiles for non-solvable kernels with both, finite and infinite mass, is by now quite well understood for large classes of kernels with $\lambda<1$. More precisely, in the case of finite mass, existence of profiles has been shown for example in~\cite{FoL05,EMR05}. 

Moreover, in~\cite{NiV13a,NTV16a} the existence of self-similar profiles with fat tails could be established for large classes of non-solvable coagulation kernels with homogeneity $\lambda<1$ while~\cite{NTV16a} in particular covers the case of~\eqref{eq:Smol:kernel}.

Concerning the uniqueness of these profiles much less is known. As already mentioned, in the case of solvable kernels, one also obtains uniqueness of profiles due to the fact that they can be computed explicitly in terms of the Laplace transform. However, for non-solvable kernels there are very few and quite recent results available. One of these results concerns the diagonal kernel in the case of finite mass (\cite{HNV17}). However, this kernel is a quite specific example since it simplifies the problem significantly by localising the integral equation in a certain sense. 

Apart from this, the to our knowledge only coagulation kernel for which uniqueness of self-similar profiles has been established completely is the perturbation of the constant kernel which we consider in this work (see Section~\ref{Sec:Assumptions:W} for the precise assumptions). More precisely, for this model uniqueness of self-similar profiles could be established first in the case of finite mass in~\cite{NiV14a,NTV15,NTV16}. However, as shown in~\cite{Thr16}, this result can further be extended to cover also the case of fat-tailed profiles and it is exactly the aim of this work to present the corresponding proof as contained in~\cite{Thr16}. We also note that, to our knowledge, this is the only uniqueness result for fat-tailed profiles which is currently available.

Finally, let us mention that a partial uniqueness result in the case of finite mass has also been given in~\cite{CaM11} where kernels of the form $K(x,y)=x^{a}y^{b}+x^{b}y^{a}$ with $a,b\in(-1,1)$ and $a+b<1$ have been considered. However, the result requires that the moments of order $a$ and $b$ of the self-similar profiles are invariant, i.e.\@ they do not depend on the self-similar solution. Unfortunately, such a property is yet also not known for such kernels.

\subsection{The constant kernel $K=2$}\label{Sec:constant:kernel}

Since we will consider in this article a perturbation $K=2+\eps W$ of the solvable constant kernel, we will collect for the latter several formulas on the self-similar profile on which we will rely in this work. Except for a different normalisation factor these expressions can be found in~\cite{MeP04}.

For $K=2$ and a fixed parameter $\rho\in(0,1)$ there exists up to normalisation exactly one self-similar profile which solves~\eqref{eq:self:sim}. As already indicated in Section~\ref{Sec:scale:invariance}, we will normalise the self-similar solutions according to their tail behaviour. More precisely, for $K=2$ and fixed $\rho\in(0,1)$ we choose the unique solution $\bar{f}$ of~\eqref{eq:self:sim} which satisfies
\begin{equation*}
 \bar{f}(x)\sim \frac{\rho^2}{\Gamma(1-\rho)}x^{-1-\rho}\quad \text{as }x\longrightarrow\infty.
\end{equation*}
This choice of the constant turns out to be convenient in the later calculations. Additionally, it holds $\bar{f}(x)\sim \rho^2/\Gamma(2-\rho)x^{\rho-1}$ as $x\to 0$. Moreover, the normalisation of the tail at infinity directly translates into the asymptotic behaviour at zero in Laplace variables. Precisely, if we define for $\bar{f}$ the Laplace transform $\bar{F}(q)\vcc=\int_{0}^{\infty}\ee^{-qx}\bar{f}(x)\dx$ and the desingualarised Laplace transform $\bar{Q}(q)\vcc=\int_{0}^{\infty}(1-\ee^{-qx})\bar{f}(x)\dx$ one has the formulas
\begin{equation}\label{eq:explicit:bar:quantities}
 \bar{F}(q)=\frac{\rho}{1+q^{\rho}}\quad \text{and}\qquad \bar{Q}(q)=\frac{\rho q^{\rho}}{1+q^{\rho}}.
\end{equation}
Thus, one can directly check that it holds
\begin{equation*}
 -\bar{F}'(q)=\bar{Q}'(q)\sim \rho^{2}q^{\rho-1}\quad \text{as }q\longrightarrow 0.
\end{equation*}

\subsection{Assumptions on the kernel $K$}\label{Sec:Assumptions:W}

We now collect the assumptions that we will have to impose on the integral kernel $K$ in order to show uniqueness of self-similar profiles while the same conditions have already been used in~\cite{NTV15,Thr16}.

Throughout this work we assume that the integral kernel $K$ can be written as
\begin{equation}\label{eq:Ass:basic}
 K(x,y)=2+\eps W(x,y) \quad \text{with}\quad 0\leq W(x,y)\leq \biggl(\Bigl(\frac{x}{y}\Bigr)^{\alpha}+\Bigl(\frac{y}{x}\Bigr)^{\alpha}\biggr) \quad\text{for some }\alpha\in[0,1),
\end{equation}
where $W$ is assumed to be a continuous, symmetric function which is homogeneous of degree zero. For simplicity we also assume $\eps\leq 1$.

For parts of our proofs we need much stronger regularity assumptions on $W$ because we have to represent $W$ as a Laplace integral. Therefore, we have to require that $W$ has an analytic extension to $\Ccut$ where we define
\begin{equation*}
 \Ccut\vcc=\C\setminus(-\infty,0].
\end{equation*}
Precisely, we need that
\begin{equation}\label{eq:Ass:analytic}
 \begin{split}
  W(\cdot,1)\text{ can be extended analytically to }\Ccut\text{ such that }\Re\bigl(W(\xi,1)\bigr)\geq 0 \text{ for }\Re(\xi)\geq 0,\\
  W(\xi,1)=W(\xi^{-1},1)\text{ and } \abs*{W(\xi,1)}\leq C\bigl(\abs*{\xi}^{-\alpha}+\abs*{\xi}^{\alpha}\bigr)\quad \text{for } \xi\in\Ccut.
 \end{split}
\end{equation}
Moreover, in order to represent $W$ as the Laplace transform of a certain kernel $\Ker$, we further need some Hölder condition on $W$. Therefore, we denote the upper and lower half-plane in the complex plane as
\begin{equation*}
 \HH_{+}\vcc=\{\xi\in\C\;|\;\Im(\xi)>0\}\quad \text{and}\quad \HH_{-}\vcc=\{\xi\in\C\;|\;\Im(\xi)<0\}
\end{equation*}
and we define the restrictions $W_{\pm}(\xi)\vcc=W(\xi,1)$ for $\xi\in\HH_{\pm}$. We then assume that there exists a Hölder exponent $\hol\in(0,1)$ such that 
\begin{equation}\label{eq:Ass:hoelder:1}
 W_{\pm} \text{ can be extended as } C^{1,\hol}\text{-function to } \overline{\HH}_{\pm}\setminus\{0\}. 
\end{equation}
This means precisely that $W_{\pm}$ is differentiable and moreover it holds for $z_1,z_2\in\HH_{\pm}$ with $\abs*{z_{1}-z_{2}}\leq \frac{1}{2}\min\{\abs{z_{1}},\abs{z_{2}}\}$ that
\begin{equation}\label{eq:Ass:hoelder:2}
 \frac{\abs*{W_{\pm}'(z_{1})-W_{\pm}'(z_{2})}}{\abs*{z_{1}-z_{2}}^{\hol}}\leq C\begin{cases}
                                                                                \min\{\abs{z_{1}},\abs{z_{2}}\}^{-\alpha-1-\hol} & \text{if }\min\{\abs{z_{1}},\abs{z_2}\}\leq 1\\
                                                                                \min\{\abs{z_{1}},\abs{z_{2}}\}^{\alpha-1-\hol} & \text{if }\min\{\abs{z_{1}},\abs{z_2}\}\geq 1.
                                                                               \end{cases}
\end{equation}
Moreover, we have to control the growth of $W_{\pm}$ and its derivative at zero and at infinity in the sense that
\begin{equation}\label{eq:Ass:growth}
 \abs{W_{\pm}(z)}\leq C\bigl(\abs{z}^{-\alpha}+\abs{z}^{\alpha}\bigr)\qquad \text{and}\qquad \abs*{W_{\pm}'}\leq C\bigl(\abs{z}^{-\alpha-1}+\abs{z}^{\alpha-1}\bigr).
\end{equation}

These rather strong regularity assumptions are necessary in order to be able to represent the perturbation $W$ as Laplace integral as stated by the following Proposition which has also been used in~\cite[Proposition~2.2]{NTV15} and~\cite[Proposition~7.15]{Thr16}.

\begin{proposition}\label{Prop:W:representation}
 If $W$ satisfies~\cref{eq:Ass:basic,eq:Ass:analytic,eq:Ass:hoelder:1,eq:Ass:hoelder:2,eq:Ass:growth} for some $\alpha\in(0,1)$, there exists a symmetric measure $\Ker$ on $\R_{+}\times\R_{+}$ such that it holds
 \begin{equation*}
  \frac{W(x,y)}{x+y}=\int_{0}^{\infty}\int_{0}^{\infty}\Ker(\xi,\eta)\ee^{-\xi x-\eta y}\dxi\deta.
 \end{equation*}
Moreover, $\Ker$ is homogeneous of degree $-1$ and splits as $\Ker(\xi,\eta)=\widetilde{\Ker}(\xi,\eta)+W_{\pm}(-1)\delta(\xi-\eta)$, where $\widetilde{\Ker}\colon \R_{+}\times\R_{+}\to\R$ is continuous and satisfies
\begin{equation*}
 \abs*{\widetilde{\Ker}(\xi,\eta)}\leq \frac{C}{(\xi+\eta)^{1-\alpha}}\biggl(\frac{1}{\xi^{\alpha}}+\frac{1}{\eta^{\alpha}}\biggr),
\end{equation*}
while $\delta(\cdot)$ denotes the Dirac measure.
\end{proposition}

\begin{remark}\label{Rem:existence:Wminusone}
 Note that~\cref{eq:Ass:analytic,eq:Ass:hoelder:1} ensure that $W_{\pm}(-1)$ exists and it holds $W_{+}(-1)=W_{-}(-1)$.
\end{remark}
\begin{remark}
 For $W(x,y)=(x/y)^{\alpha}+(y/x)^{\alpha}$ one obtains the explicit formula $\Ker(\xi,\eta)=\frac{\sin(\pi \alpha)}{\pi}\frac{(\xi/\eta)^{\alpha}-(\eta/\xi)^{\alpha}}{\xi-\eta}+2\cos(\pi \alpha)\delta(\xi-\eta)$.
\end{remark}

For kernels $W$ which have a singular behaviour at the origin, i.e.\@ $\alpha\in(0,1)$ there is a further technicality arising in form of a boundary layer for self-similar profiles close to the origin (see Section~\ref{Sec:bd:layer:formal}). In order to deal with this, we have to impose additionally the precise behaviour of $W$ close to the origin. In more detail this means that we require that
\begin{equation}\label{eq:Ass:asymptotic}
 W(\xi,1)\sim C_{W}\xi^{-\alpha}\quad \text{for } \xi\longrightarrow 0.
\end{equation}
Note that the notation $\sim$ here means that an analytic function $\varphi\colon \Ccut\to \C$ satisfies $\varphi\sim Az^{-\alpha}$ if and only if there is a further function $\tau\colon \Ccut\to \R_{\geq 0}$ such that it holds
\begin{equation}\label{eq:def:asymptotics:complex:plane}
 \abs*{z^{\alpha}\varphi(z)-A}\leq \tau(z) \quad \text{with}\quad \lim_{r\to 0^{+}}\sup_{\stackrel{\abs*{\xi}=r}{\xi\in\Ccut}} \tau(\xi)=0.
\end{equation}

\begin{remark}
 Note that~\eqref{eq:Smol:kernel} can be rewritten as $K(x,y)=2+(x/y)^{1/3}+(y/x)^{1/3}$. Thus, for $\alpha=1/3$ and $\eps=1$ this kernel has exactly the form~\eqref{eq:Ass:basic} and satisfies also~\cref{eq:Ass:analytic,eq:Ass:growth,eq:Ass:asymptotic,eq:Ass:hoelder:1,eq:Ass:hoelder:2} which might be seen as a motivation for our assumptions. However, we also emphasise that we have to choose $\eps$ in general rather small such that we cannot expect to obtain uniqueness of profiles in the case~\eqref{eq:Smol:kernel}.
\end{remark}

\subsection{Preliminary work and main result}

Throughout this work we use the following definition of self-similar profiles which is an adaptation of the corresponding definition in~\cite[Definition~7.2]{Thr16}.

\begin{definition}\label{Def:profile}
 Let $K$ be an integral kernel satisfying~\eqref{eq:Ass:basic}. For $\rho\in(\alpha,1)$, a function $f\in L^{1}_{\text{loc}}\bigl((0,\infty)\bigr)$ is denoted a self-similar profile of~\eqref{eq:Smol} or equivalently a solution to~\eqref{eq:self:sim}, provided that $f$ is non-negative, $f\not\equiv 0$ and $f$ satisfies~\eqref{eq:self:sim} almost everywhere. Furthermore, we require that $xf(x)\in L^{1}_{\text{loc}}\bigl([0,\infty)\bigr)$ and that
 \begin{equation*}
  \int_{1}^{\infty}x^{\alpha}f(x)\dx\leq C_{f}
 \end{equation*}
 for some constant $C_{f}>0$.
\end{definition}

As already indicated in Section~\ref{Sec:scale:invariance} the scale invariance of~\eqref{eq:self:sim} requires a normalisation condition for self-similar profiles. The approach that we will take here is to specify the tail behaviour of self-similar solutions in such a way that the profiles for the perturbed kernel $K=2+\eps W$ exhibit the same tail behaviour as $\bar{f}$. The following result, which summarises~\cite[Proposition~1.9 \& Theorem~1.10]{Thr17}, shows that this approach is always possible under the assumption~\eqref{eq:Ass:basic}.

\begin{theorem}\label{Thm:asymptotics}
 Let $K$ satisfy~\eqref{eq:Ass:basic}. Then each self-similar profile $f$ can be rescaled as $\tilde{f}(x)\vcc=af(ax)$ such that it holds
 \begin{equation*}
  \int_{0}^{R}y\tilde{f}(y)\dy\leq \frac{\rho^2}{\Gamma(2-\rho)}R^{1-\rho} \quad \text{and}\quad \int_{0}^{R}y\tilde{f}(y)\dy\sim \frac{\rho^2}{\Gamma(2-\rho)}R^{1-\rho}\quad \text{for }R\longrightarrow \infty.
 \end{equation*} 
 Moreover, $\tilde{f}$ is continuous on $(0,\infty)$ and satisfies $\tilde{f}(x)\sim \bar{f}(x)$ for $x\to \infty$.
\end{theorem}

\begin{remark}
 Note that in~\cite{Thr17} a different normalisation has been used. In this work, we will always assume that all profiles are normalised according to Theorem~\ref{Thm:asymptotics} while we also recall from Section~\ref{Sec:scale:invariance} that $\tilde{f}$ is again a self-similar profile.
\end{remark}

Moreover, we will need the following result which provides uniform estimates for certain moments of self-similar profiles and which is also shown in~\cite[Lemma~3.3]{Thr17}.

\begin{lemma}\label{Lem:moment:est:1}
 For any $\beta\in(0,\rho)$ there exists a constant $C_{\beta}>0$ such that it holds
 \begin{equation*}
  \int_{0}^{\infty}x^{\beta}f(x)\dx\leq C_{\beta}
 \end{equation*}
 for all solutions $f$ of~\eqref{eq:self:sim} which are normalised according to Theorem~\ref{Thm:asymptotics}.
\end{lemma}

The main statement that we will show in this work is the following uniqueness result for self-similar profiles with fat tails (see also~\cite{Thr16}).

\begin{theorem}\label{Thm:main}
 Let the coagulation kernel $K$ satisfy assumptions~\cref{eq:Ass:analytic,eq:Ass:basic,eq:Ass:growth,eq:Ass:hoelder:1,eq:Ass:hoelder:2,eq:Ass:asymptotic} for some parameter $\alpha\in[0,1/2)$ and let $\rho\in(1/2,1)$. Then, for $\eps>0$ sufficiently small, there exists at most one self-similar profile which is normalised according to Theorem~\ref{Thm:asymptotics}.
\end{theorem}

\begin{remark}
 We note that $\alpha=0$ in the previous statement represents a particular case since the problem of the boundary layer, which will be discussed in Section~\ref{Sec:bd:layer:formal} does not occur in this special situation of bounded perturbations. However, the proof of Theorem~\ref{Thm:main} relies heavily on Proposition~\ref{Prop:W:representation} which in this form is only valid for $\alpha\in(0,1)$. In fact, one could recover this statement also for $\alpha=0$ but then logarithmic corrections would occur. Instead, we mention that a bounded perturbation $\widetilde{W}$ still satisfies the assumptions~\cref{eq:Ass:analytic,eq:Ass:basic,eq:Ass:growth,eq:Ass:hoelder:1,eq:Ass:hoelder:2} with any small value $\tilde{\alpha}\in(0,1)$. The only exception is~\eqref{eq:Ass:asymptotic} which cannot be recovered for $\tilde{\alpha}\in(0,1)$ if $\widetilde{W}$ is bounded. However, this assumption is only required in the proof of the boundary layer estimate (Proposition~\ref{Prop:boundary:layer}) and it turns out that one only needs this restriction for sufficiently large values of $\alpha$, while for $\alpha$ sufficiently close to zero this requirement can be skipped. Due to the fact that the corresponding proof of Proposition~\ref{Prop:boundary:layer} is already quite long and technical and covers at least half of this work, we will not go into further details on how to adapt the proofs in the case of bounded perturbations $\widetilde{W}$ but we will only consider the case of $\alpha\in(0,1)$ which is more delicate. Some comments on the necessary adaptations for $\alpha=0$ are contained in~\cite{Thr16}.
\end{remark}

\begin{remark}\label{Rem:restrictions:alpha:rho}
 The decay properties of self-similar profiles and assumption~\eqref{eq:Ass:basic} necessarily require that it holds $\rho>\alpha$ to obtain finite integrals in~\eqref{eq:self:sim}. In principle one would then expect that Theorem~\ref{Thm:main} is valid for $\rho\in(\alpha,1)$. However, we were not able to obtain this stronger result while it seems that this is a technical issue in the proof of Proposition~\ref{Prop:boundary:layer} which stems from the singular behaviour of profiles close to zero.
 
 Similarly, one would in general expect that Theorem~\ref{Thm:main} also holds for $\alpha\in[1/2,1)$ and in fact, except for one estimate in the proof of Proposition~\ref{Prop:boundary:layer}, all main results either already hold for $\alpha\in[1/2,1)$ or could at least be recovered with some additional effort. Again, it is expected that this problem is mainly technical.
\end{remark}

\begin{remark}
 We note that the existence result of self-similar profiles with fat tails in~\cite{NTV16a} does not directly apply in general under the present conditions, since there also a precise lower bound on the kernel $K$ has been used. However, in the special case $W(x,y)=(x/y)^{\alpha}+(y/x)^{\alpha}$, the statement in~\cite{NTV16a} directly guarantees existence of self-similar solutions. Moreover, due to the condition~\eqref{eq:Ass:asymptotic} it should also be possible to adapt the existence proof in~\cite{NTV16a} to the present conditions for $W$.
\end{remark}

\subsection{The boundary layer at zero}\label{Sec:bd:layer:formal}

The main difficulty in the proof of Theorem~\ref{Thm:main} is the occurrence of a boundary layer for small cluster sizes. This problem originates directly from the singular behaviour of the perturbation $W$ which leads to a completely different asymptotic behaviour at zero of the perturbed self-similar profiles compared to $\bar{f}$. To illustrate this phenomenon, we will give a quite formal derivation of the asymptotic behaviour of profiles $f$ for $K=2+\eps W$ with $W(x,y)=\bigl((x/y)^{\alpha}+(y/x)^{\alpha}\bigr)$. Therefore, we rewrite equation~\eqref{eq:self:sim} as
\begin{multline*}
 x^{2}f(x)=(1-\rho)\int_{0}^{x}yf(y)\dy+\int_{0}^{x}yf(y)\int_{0}^{\infty}K(y,z)f(z)\dz\dy\\*
 -\int_{0}^{x}\int_{0}^{x-y}K(y,z)yf(y)f(z)\dz\dy.
\end{multline*}
The last integral on the right-hand side is of lower order for small values of $x$, thus we neglect this term and obtain that the asymptotic behaviour to the leading order is given by the equation
\begin{equation*}
 x^{2}f(x)=(1-\rho)\int_{0}^{x}yf(y)\dy+\int_{0}^{x}yf(y)\int_{0}^{\infty}K(y,z)f(z)\dz\dy.
\end{equation*}
We can rearrange this by differentiating on both sides to get
\begin{equation*}
 \del_{x}\bigl(x^{2}f(x)\bigr)=\frac{1}{x}\biggl((1-\rho)+\int_{0}^{\infty}K(x,z)f(z)\dz\biggr)\bigl(x^2f(x)\bigr).
\end{equation*}
After integrating this equation we find
\begin{equation*}
 x^{2}f(x)=Cx^{1-\rho}\exp\biggl(-\int_{x}^{1}t^{-1}\int_{0}^{\infty}K(t,z)f(z)\dz\dt\biggr).
\end{equation*}
We next evaluate the integral on the right-hand side and for this, we use $\mom_{\gamma}$ to denote the moment of order $\gamma$ of $f$ as defined in~\eqref{eq:def:moment}. With this, it holds
\begin{equation*}
 \int_{x}^{1}t^{-1}\int_{0}^{\infty}K(t,z)f(z)\dz\dt=-2\mom_{0}\log(x)+\frac{\eps}{\alpha}\mom_{\alpha}x^{-\alpha}+B(x,f)
\end{equation*}
with a remainder $B(x,f)=-(\eps/\alpha)\mom_{\alpha}+(\eps/\alpha)\mom_{-\alpha}(1-x^{\alpha})$ which is bounded for small values of $x$. Summarising, we thus have to the leading order that
\begin{equation*}
 f(x)=Cx^{2\mom_{0}-1-\rho}\exp\Bigl(-\frac{\eps}{\alpha}\mom_{\alpha}x^{-\alpha}-B(x,f)\Bigr).
\end{equation*}
It will turn out (see Lemma~\ref{Lem:moments:convergence}) that $\mom_{0}\to \bar{\mom}_{0}=\rho$ as $\eps\to 0$ which means that for small values of $\eps>0$ the singular behaviour of $x^{2\mom_{0}-1-\rho}$ at $x=0$ is similar to that one of $\bar{f}$. However, there is an additional correction to the asymptotic behaviour of $f$ close to zero which leads to an exponential decay of $f$ in a neighbourhood of order $\eps^{1/\alpha}$ of the origin. As a consequence, for positive $\eps$ each self-similar profile is globally bounded while the unperturbed profile $\bar{f}$ has a singularity at zero.

The most technical and intricate part of this work is exactly the careful study of this boundary layer and to obtain precise estimates on the regularising effect of the exponential correction. Unfortunately, the method that we use here requires that we additionally restrict to values $\rho\in(1/2,1)$ which means that the singularity of $\bar{f}$ at zero is weaker than $x^{-1/2}$. We strongly believe that this is only a technical problem and that our result still remains true if $\rho\in(\alpha,1/2]$ but a proof is still lacking. Moreover, it is exactly in the consideration of the boundary layer where the assumption $\alpha<1/2$ turns out to be essential, at least for the method that we use here (see also Remark~\ref{Rem:restrictions:alpha:rho}).

\subsection{Outline of the main ideas and strategy of the proof}

In the remainder of this article, we will present the proof of Theorem~\ref{Thm:main} as given in~\cite{Thr16}. The general idea consists in showing a contraction inequality for the difference of self-similar profiles in Laplace variables. More precisely, we rewrite the self-similar solutions as perturbations of the explicitly known profile $\bar{f}$ and we apply the Laplace transform to~\eqref{eq:self:sim}. Then, we invert the linearised coagulation operator and invoke some implicit function theorem like method. The advantage of this approach is that the unperturbed equation behaves well under the Laplace transform and we can explicitly invert the transformed linearised operator. On the other hand, working with the Laplace transform requires to represent all functionals as Laplace integrals and thus, we have to derive the corresponding inverse Laplace transform and suitable estimates for it (e.g.\@ \cref{Prop:W:representation,Prop:rep:H0,Prop:Q0:int:est}).

Moreover, the singular behaviour of the perturbation $W$ requires to consider the regions far and close to zero separately, while away from zero, we readily obtain the desired estimates. On the other hand, close to zero one has to study carefully the effect of the boundary layer as discussed in Section~\ref{Sec:bd:layer:formal}. If one combines the estimates for both regions, we finally obtain a contraction property for the difference of two self-similar profiles and sufficiently small $\eps>0$. We also note, that this general abstract approach has already been used in the finite-mass-case in~\cite{NTV15,NTV16,NiV14a} and we follow the main general idea here, while in the present situation several adaptations have been necessary since the fat-tailed profiles behave more singular both at zero and at infinity. On the other hand, compared to~\cite{NTV15} several proofs could be simplified. 

The article is structured as follows. In Section~\ref{Sec:function:spaces} we define a suitable norm and we provide some notation that will be useful in the following.

In Section~\ref{Sec:proof:main}, we will collect the key results of this work and based on those, we will then give the proof of Theorem~\ref{Thm:main}.

The remainder of the article is then concerned with the proof of the key ingredients and we will first prove several continuity estimates for the coagulation operator in self-similar variables in Section~\ref{Sec:continuity:estimates}. 

Section~\ref{Sec:linearised:operator} is then concerned with the study of the linearised coagulation operator.

In Section~\ref{Sec:uniform:convergence}, we will collect several uniform bounds on self-similar profiles and we will show that the profiles are on the one hand uniformly bounded with respect to the norms that we defined before while moreover, we also get that two self-similar profiles are close together in the topology induced by these norms, provided that $\eps$ is sufficiently small.

The remaining sections are then concerned with the proof of the boundary layer estimate which is the longest and most technical part of this article. More precisely, we will prove the boundary layer estimate in Section~\ref{Sec:bd:layer} while the proofs of several auxiliary results are then postponed to \cref{Sec:H0:representation,Sec:Q0:est,Sec:asymptotics}. We also note that the pictures contained in these sections have been reused from~\cite{Thr16} and~\cite{NTV15} respectively.

In Appendix~\ref{Sec:elementary:results} we collect several elementary results connected to the norms that we use here, while in Appendix~\ref{Sec:Proof:Gamma} we sketch the proof of Proposition~\ref{Prop:W:representation} and provide several integral estimates on the representation kernel $\Ker$.

\section{Functional setup and preliminaries}\label{Sec:function:spaces}

\subsection{Function spaces and norms}

In what follows, we will typically use lower-case letters to denote usual functions and measures, while the corresponding quantities on the level of the Laplace variables will be denoted by capital letters. In particular, let $g\in\M^{\text{fin}}(0,\infty)$ be a finite measure, then we denote by
\begin{equation*}
 G(q)\vcc=\int_{0}^{\infty}g(x)\ee^{-qx}\dx
\end{equation*}
the Laplace transform of $g$ which is well defined for all $q\geq 0$. Moreover, we denote by $\T$ the operator which maps $g\in\M^{\text{fin}}(0,\infty)$ to $G$ and we note that it holds
\begin{equation}\label{eq:Def:T}
 \begin{split}
  \T\colon \M^{\text{fin}}(0,\infty)&\longrightarrow C^{\infty}(0,\infty)\\
  g&\longmapsto G.
 \end{split}
\end{equation}

We now introduce a norm on the level of Laplace-transformed finite measures, i.e.\@ we fix two parameters $\chi>0$ and $\mu\geq 0$ and a function $G\in C^{k}\bigl((0,\infty)\bigr)\cap C\bigl([0,\infty)\bigr)$ and we set
\begin{equation}\label{eq:def:seminorm}
 \lsnorm*{k}{\mu}{\chi}{G}\vcc=\sup_{q>0}\Bigl((1+q)^{\chi+\mu+\rho}q^{k-\rho-\mu}\abs[\big]{\del_{q}^{k}G(q)}\Bigr)\qquad \text{for }k=0,1,2.
\end{equation}
\begin{remark}
 For $k=0$ this in fact defines a norm. Thus, we also write $\lfnorm*{0}{\mu}{\chi}{\cdot}\vcc=\lsnorm*{0}{\mu}{\chi}{\cdot}$. 
\end{remark}
In order to get also a norm for $k=1$ and $k=2$, we further define
\begin{equation*}
 \lfnorm*{k}{\mu}{\chi}{G}\vcc=\lfnorm*{0}{-\rho}{\chi}{G}+\sum_{\ell=1}^{k}\lsnorm*{\ell}{\mu}{\chi}{G}\quad \text{for }k=1,2.
\end{equation*}

Associated to these norms, we can also define sup-(Banach) spaces of $C^{k}\bigl((0,\infty)\bigr)\cap C\bigl([0,\infty)\bigr)$ as
\begin{equation*}
 \X{k}{\mu}{\chi}\vcc=\Bigl\{G\in C^{k}\bigl((0,\infty)\bigr)\cap C\bigl([0,\infty)\bigr)\; \Big| \; \lfnorm*{k}{\mu}{\chi}{G}<\infty\Bigr\}\quad \text{for }k=1,2.
\end{equation*}

\begin{remark}\label{Rem:explicit:profiles}
 From the formulas in~\eqref{eq:explicit:bar:quantities} one immediately verifies that it holds $\lfnorm{2}{0}{\rho}{\T\bar{f}}<\infty$ and thus $(\T\bar{f})\in\X{2}{0}{\theta}$ for all $\theta\in(0,\rho]$. Moreover, we note that the definition of $\lfnorm{k}{0}{\chi}{\cdot}$ is explicitly motivated by the scaling of $\bar{F}=\T\bar{f}$ and its derivatives.
\end{remark}

We further note, that most of the time, we use these norms with a fixed parameter $\chi$ which we will then denote by $\theta$. Since most of our results here hold for $\rho\in(0,1)$ the minimal assumption we have to impose is
\begin{equation}\label{eq:choice:theta:standard}
 \theta\in(\alpha,\min\{\rho,1/2\}).
\end{equation}
However, the proof of the boundary layer estimate (Proposition~\ref{Prop:boundary:layer} below) requires the further restriction $\rho\in(1/2,1)$ which makes it necessary to also refine the choice of $\theta$ as
\begin{equation}\label{eq:choice:theta:bl}
 \theta\in\bigl(\max\{\alpha,1-\rho\},1/2\bigr).
\end{equation}
Note that for $\rho\in(1/2,1)$ the condition~\eqref{eq:choice:theta:bl} really implies~\eqref{eq:choice:theta:standard}.

Moreover, we have to restrict the parameter $\mu$ in some places and the general upper bound will be given by
\begin{equation*}
 \mu_{*}\vcc=\min\{\rho,1-\rho\}
\end{equation*}
while we note that it holds $\mu_{*}>0$.

\subsection{Transforming the equation to Laplace variables}

In this section we will apply the Laplace transform to~\eqref{eq:self:sim}. We note that at this stage this is purely formal, since we do not know yet if $\T f$ really exists for each self-similar profile. More precisely, Theorem~\ref{Thm:asymptotics} ensures that $\ee^{-q\cdot} f(\cdot)$ is integrable at infinity, while it still remains to show, that we also have integrability at zero. This will be done in Section~\ref{Sec:uniform:convergence} (see in particular Lemma~\ref{Lem:moment:est:2}).

For an integral kernel $\widetilde{K}\in\{2,W\}$ let us formally define the (bi-)linear forms
\begin{equation*}
 \begin{split}
  A_{\rho}(g)\vcc=\frac{1-\rho}{x^2}\int_{0}^{x}yg(y)\dy\quad \text{and}\quad
  B_{\widetilde{K}}(g,h)\vcc=\frac{1}{x^2}\int_{0}^{x}\int_{x-y}^{\infty}y\widetilde{K}(y,z)g(y)h(z)\dz\dy.
 \end{split}
\end{equation*}
With this, we may rewrite~\eqref{eq:self:sim} as
\begin{equation}\label{eq:abstract:self:sim}
 f=A_{\rho}(f)+ B_{2}(f,f)+\eps  B_{W}(f,f).
\end{equation}
Formally, we can then transform this equation by applying the Laplace transform $\T$. Considering the expressions on the right-hand side separately, we find together with $\ee^{-qx}x^{-2}=\int_{q}^{\infty}\int_{p}^{\infty}\ee^{-rx}\dr\dd{p}$, Fubini's Theorem and $-(\T f)'(r)=\int_{0}^{\infty}yf(y)\ee^{-ry}\dy$ that 
\begin{multline*}
 \bigl(\T A_{\rho}(f)\bigr)(q)=(1-\rho)\int_{q}^{\infty}\int_{p}^{\infty}\int_{0}^{\infty}yf(y)\int_{y}^{\infty}\ee^{-rx}\dx\dy\dr\dd{p}\\*
 =(1-\rho)\int_{q}^{\infty}\int_{p}^{\infty}\int_{0}^{\infty}yf(y)\int_{y}^{\infty}\ee^{-rx}\dx\dy\dr\dd{p}\\*
 =-(1-\rho)\int_{q}^{\infty}\int_{p}^{\infty}\frac{1}{r}(\T f)'(r)\dr\dd{p}.
\end{multline*}
Similarly, we obtain
\begin{multline*}
 \bigl(\T B_{2}(g,h)\bigr)=\int_{q}^{\infty}\int_{p}^{\infty}\frac{1}{r}\int_{0}^{\infty}\int_{0}^{\infty}2yg(y)h(z)\ee^{-ry}\bigl(1-\ee^{-rz}\bigr)\dz\dy\dr\dd{p}\\*
 =-\int_{q}^{\infty}\int_{p}^{\infty}\frac{2}{r}(\T g)'(r)\bigl((\T h)(0)-(\T h)(r)\bigr)\dr\dd{p}.
\end{multline*}
This motivates to define
\begin{align*}
 \mathcal{A}_{\rho}(G)(q)&\vcc=-\int_{q}^{\infty}\int_{p}^{\infty}\frac{1-\rho}{r}G'(r)\dr\dd{p}\\
 \shortintertext{and}
 \mathcal{B}_{2}(G,H)(q)&\vcc=-\int_{q}^{\infty}\int_{p}^{\infty}\frac{2}{r}G'(r)\bigl(H(0)-H(r)\bigr)\dr\dd{p}.
\end{align*}
For $B_{W}(g,h)$ it follows by an analogous calculation that
\begin{equation*}
 \bigl(\T B_{W}(g,h)\bigr)(q)=\int_{q}^{\infty}\int_{p}^{\infty}\frac{1}{r}\int_{0}^{\infty}\int_{0}^{\infty}yW(y,z)g(y)h(z)\ee^{-ry}\bigl(1-\ee^{-rz}\bigr)\dz\dy\dr\dd{p}.
\end{equation*}
Together with Proposition~\ref{Prop:W:representation} this can be further rewritten as
\begin{multline*}
 \bigl(\T B_{W}(g,h)\bigr)(q)=\int_{q}^{\infty}\int_{p}^{\infty}\frac{1}{r}\int_{0}^{\infty}\int_{0}^{\infty}\Ker(\xi,\eta)\cdot\\*
 \cdot\biggl(\int_{0}^{\infty}\int_{0}^{\infty}y(y+z)g(y)h(z)\ee^{-(r+\xi)y}\bigl(\ee^{-\eta z}-\ee^{-(\eta+r)z}\bigr)\dz\dy\biggr)\deta\dxi\dr\dd{p}.
\end{multline*}
In terms of the Laplace transforms $(\T g)$ and $(\T h)$ this reads as
\begin{multline*}
 \bigl(\T B_{W}(g,h)\bigr)(q)\\*
 =\int_{q}^{\infty}\int_{p}^{\infty}\frac{1}{r}\int_{0}^{\infty}\int_{0}^{\infty}\Ker(\xi,\eta)\Bigl((\T g)''(r+\xi)\bigl((\T h)(\eta)-(\T h)(\eta+r)\bigr)\\*
 +(\T g)'(r+\xi)\bigl((\T h)'(\eta)-(\T h)'(\eta+r)\bigr)\Bigr)\dxi\deta.
\end{multline*}
This then motivates to define
\begin{equation}\label{eq:kernel:BW}
 \begin{aligned}
  N_{1}[G,H](r)&\vcc=\frac{1}{r}\int_{0}^{\infty}\int_{0}^{\infty}\Ker(\xi,\eta)G''(r+\xi)\bigl(H(\eta)-H(\eta+r)\bigr)\dxi\deta\\
  N_{2}[G,H](r)&\vcc=\frac{1}{r}\int_{0}^{\infty}\int_{0}^{\infty}\Ker(\xi,\eta)G'(r+\xi)\bigl(H'(\eta)-H'(\eta+r)\bigr)\dxi\deta\\
  N[G,H]&\vcc=N_{1}[G,H]+N_{2}[G,H].
 \end{aligned}
\end{equation}
as well as
\begin{equation}\label{eq:BW:with:kernel}
 \mathcal{B}_{W}(G,H)=\int_{q}^{\infty}\int_{p}^{\infty}N[G,H](r)\dr\dd{p}.
\end{equation}
We then note the relations
\begin{equation}\label{eq:transformed:forms}
 \T \bigl(A_{\rho}(g)\bigr)=\mathcal{A}_{\rho}(\T g),\quad \T\bigl(B_{2}(g,h)\bigr)=\mathcal{B}_{2}(\T g,\T h)\quad \text{and}\quad  \T\bigl(B_{W}(g,h)\bigr)=\mathcal{B}_{W}(\T g,\T h).
\end{equation}
Moreover, equation~\eqref{eq:self:sim} for a self-similar profile $f$ can be written in terms of the Laplace transform $(\T f)$ as
\begin{equation}\label{eq:self:sim:Lap}
 (\T f)(q)=\mathcal{A}_{\rho}(\T f)+\mathcal{B}_{2}(\T f,\T f)+\eps \mathcal{B}_{W}(\T f,\T f).
\end{equation}
We will also have to consider the linearised operator in self-similar variables which is given by
\begin{equation*}
  L(g)\vcc=g-A_{\rho}(g)- B_{2}(\bar{f},g)- B_{2}(g,\bar{f})
\end{equation*}
while the corresponding Laplace-transformed operator reads as
\begin{equation}\label{eq:def:lin:op}
 \bigl(\LL(G)\bigr)(q)\vcc=G(q)-\mathcal{A}_{\rho}(G)(q)-\mathcal{B}_{2}(\bar{F},G)(q)-\mathcal{B}_{2}(G,\bar{F})(q).
\end{equation}

\begin{remark}\label{Rem:transform:justification}
 As already mentioned, these calculations at this stage are purely formal. However, the considerations in Section~\ref{Sec:uniform:convergence} will later justify the previous manipulations rigorously, at least for all solutions $f$ of~\eqref{eq:self:sim} and finite linear combinations of such solutions, provided $\eps$ is sufficiently small. In this context we refer especially to \cref{Rem:existence:Laplace:transform,Lem:moment:est:2} which provide the existence of the Laplace transform for each self-similar profile and the required integrability properties for the previous manipulations.
\end{remark}

\subsection{Notation and elementary properties of $\T$}

We introduce the notation 
\begin{equation*}
 \zeta(x)\vcc=\ee^{-x}
\end{equation*}
since this function will turn out to be quite useful in dealing with Laplace transforms in several places. Moreover, we note the elementary relations
\begin{equation}\label{eq:Laplace:easy:shift}
 \T (\zeta g)(\cdot)=(\T g)(\cdot+1)\quad \text{and}\qquad \T\bigl((1-\zeta)g\bigr)(\cdot)=(\T g)(\cdot)-(\T g)(\cdot+1).
\end{equation}

\begin{remark}\label{Rem:norms:behave:well:under:shifts}
 We note that the norms $\lfnorm*{k}{\mu}{\chi}{\cdot}$ behave well under these shifts in the sense that we have the following estimates
 \begin{align*}
  \lfnorm*{0}{-\rho}{\chi}{\T(\zeta g)} &\leq \lfnorm*{0}{-\rho}{\chi}{\T g} &\text{and} && \lfnorm*{0}{-\rho}{\chi}{\T\bigl((1-\zeta) g\bigr)} &\leq 2\lfnorm*{0}{-\rho}{\chi}{\T g}\\
  \lsnorm*{k}{\mu}{\chi}{\T(\zeta g)} &\leq \lsnorm*{k}{\mu}{\chi}{\T g} &\text{and} && \lsnorm*{k}{\mu}{\chi}{\T\bigl((1-\zeta) g\bigr)} &\leq 2\lsnorm*{k}{\mu}{\chi}{\T g}
 \end{align*}
 for each $g\in\M^{\text{fin}}(0,\infty)$ with $(\T g)\in \X{k}{\mu}{\chi}$.
\end{remark}
To simplify the notation at several places, we define the function $\weight{a}{b}\colon (0,\infty)\to (0,\infty)$ as
\begin{equation*}
 \weight{a}{b}(q)=\begin{cases}
                   q^{a} & \text{if }q\leq 1\\
                   q^{-b} & \text{if } q\geq 1.
                  \end{cases}
\end{equation*}
Furthermore, we introduce the moment of order $\gamma$ of a self-similar profile $f$ as
\begin{equation}\label{eq:def:moment}
 \mom_{\gamma}\vcc=\int_{0}^{\infty}x^{\gamma}f(x)\dx.
\end{equation}
 We note, that it will be shown in Lemma~\ref{Lem:moment:est:2} that this quantity is well-defined for $\gamma\in(-\rho,\rho)$ provided that we choose $\eps>0$ sufficiently small. Moreover, the corresponding moment of order $\gamma$ of the profile $\bar{f}$, i.e.\@ for $\eps=0$ will be denoted by $\bar{\mom}_{\gamma}$. Note that we might also use the notation $\mom_{\gamma}(f)$ to stress the dependence on a certain self-similar profile $f$ is certain places.

\section{Uniqueness of profiles -- Proof of Theorem~\ref{Thm:main}}\label{Sec:proof:main}

\subsection{Key ingredients for the proof}

In this section we will collect the main estimates that we will need to prove Theorem~\ref{Thm:main}, while we postpone the corresponding proofs then to the remaining sections.

The statements in \cref{Lem:est:Arho,Lem:est:B2,Prop:est:BW,Prop:improved:regularity,Prop:linearised:operator} deal with abstract mapping properties of the (bi-)linear operators $\mathcal{A}_{\rho}$, $\mathcal{B}_{2}$, $\mathcal{B}_{W}$ and $\LL$ with respect to our function spaces $\X{k}{\mu}{\chi}$. On the other hand, \cref{Prop:norm:boundedness,Prop:closeness:two:norm,Prop:boundary:layer} provide precise estimates on solutions of~\eqref{eq:self:sim} within our functional setup.

\begin{lemma}\label{Lem:est:Arho}
 For each $\mu\in[0,\mu_{*})$ and $\chi>0$ the operator $\mathcal{A}_{\rho}\colon \X{1}{\mu}{\chi}\to\X{2}{\mu}{\chi}$ is well-defined and continuous, i.e.\@ we have
 \begin{equation}\label{eq:est:Arho}
  \lfnorm*{2}{\mu}{\chi}{\mathcal{A}_{\rho}(G)}\leq C\lfnorm*{1}{\mu}{\chi}{G}
 \end{equation}
 for every $G\in\X{1}{\mu}{\chi}$.
\end{lemma}

\begin{remark}
 Note that we gain an additional derivative here, i.e.\@ the operator $\mathcal{A}_{\rho}$ is regularising.
\end{remark}

\begin{lemma}\label{Lem:est:B2}
 For each $\mu\in[0,\mu_{*})$ and $\chi>0$ the map $\mathcal{B}_{2}\colon \X{1}{0}{\chi}\times \X{1}{0}{\chi}\to \X{2}{\mu}{\chi}$ is well-defined and continuous in the sense that it holds
 \begin{equation}\label{eq:est:B2}
  \lfnorm*{2}{\mu}{\chi}{\mathcal{B}_{2}(G,H)}\leq C\lsnorm*{1}{0}{\chi}{G}\lfnorm*{1}{0}{\chi}{H} 
 \end{equation}
 for all $G,H\in\X{1}{0}{\chi}$.
\end{lemma}
 \begin{remark}
  Note that there is again a regularising effect of the operator in the sense that we gain an additional derivative on the right-hand side. But additionally, we also obtain additional regularity of order $q^{\mu}$ close to zero.
 \end{remark}
 
 The general goal would now be to obtain an analogous estimate to that one in Lemma~\ref{Lem:est:B2} also for $\mathcal{B}_{W}$. Unfortunately, this is not directly possible due to the singular behaviour of the kernel $W$ which causes some loss of decay of order $q^{\alpha}$ at infinity which is measured by the third parameter in the norm. Precisely, we have the following statement.
 
 \begin{proposition}\label{Prop:est:BW}
 For each $\alpha\in(0,1)$ and $\mu\in[0,\mu_{*})$ the map $\mathcal{B}_{W}\colon\X{2}{0}{\theta}\times \X{2}{0}{\theta}\to \X{2}{\mu}{\theta-\alpha}$ is well-defined and continuous, i.e.\@ there exists a constant $C_{\mu,\alpha}>0$ such that it holds
 \begin{equation}\label{eq:continuity:BW}
  \lfnorm*{2}{\mu}{\theta-\alpha}{\mathcal{B}_{W}(G,H)}\leq C_{\mu,\alpha}\lfnorm*{2}{0}{\theta}{G}\lfnorm*{2}{0}{\theta}{H}
 \end{equation}
 for all $G,H\in\X{2}{0}{\theta}$.
\end{proposition}

 The loss of decay at infinity that one encounters in Proposition~\ref{Prop:est:BW} can be compensated if one considers differences as stated in the following result.
 
\begin{proposition}\label{Prop:improved:regularity}
 For all $\alpha\in(0,1)$ and $\mu\in[0,\mu_{*})$ there exists a constant $C_{\mu}>0$ such that it holds
 \begin{equation*}
  \lfnorm*{2}{\mu}{\theta}{\mathcal{B}_{W}(G,H)(\cdot)-\mathcal{B}_{W}(G,H)(\cdot+1)}\leq C_{\mu}\lfnorm*{2}{0}{\theta}{G}\lfnorm*{2}{0}{\theta}{H}
 \end{equation*}
 for all $G,H\in\X{2}{0}{\theta}$.
\end{proposition}

 The next proposition summarises certain properties of the linearised operator $\LL$, namely that it maps $\X{k}{\mu}{\chi}$ continuously into itself for appropriate choices of $k$, $\mu$ and $\chi$ and moreover, $\LL$ is in particular invertible with bounded inverse. 
 
 \begin{proposition}\label{Prop:linearised:operator}
  For each $\mu\in(0,\mu_{*})$, $\chi\in(0,\rho)$ and $k=1,2$ the operator $\LL\colon \X{k}{\mu}{\chi}\to \X{k}{\mu}{\chi}$ is well-defined and bounded. Moreover, the operator $\LL$ is also invertible with bounded inverse $\LL^{-1}\colon \X{k}{\mu}{\chi}\to \X{k}{\mu}{\chi}$. Precisely, this means that we have the estimates
  \begin{equation*}
   \lfnorm*{k}{\mu}{\chi}{\LL G}\leq C_1 \lfnorm*{k}{\mu}{\chi}{G}\quad \text{and}\quad \lfnorm*{k}{\mu}{\chi}{\LL^{-1}H}\leq C_{2}\lfnorm*{k}{\mu}{\chi}{H}
  \end{equation*}
  for all $G,H\in\X{k}{\mu}{\chi}$ and constants $C_1,C_2>0$.
 \end{proposition}
 
 The remaining three statements now provide the necessary estimates of solutions to~\eqref{eq:self:sim} with respect to the norms $\lfnorm*{k}{\mu}{\theta}{\cdot}$. The first proposition gives the uniform boundedness of $(\T f)$ in $\X{2}{0}{\theta}$ for each self-similar profile $f$.

 \begin{proposition}\label{Prop:norm:boundedness}
 For sufficiently small $\eps>0$ there exists a constant $C>0$ such that it holds
 \begin{equation*}
  \lfnorm*{2}{0}{\theta}{\T f}<C
 \end{equation*}
 for any solution $f$ of~\eqref{eq:self:sim}. In particular this shows $\T f\in\X{2}{0}{\theta}$.
\end{proposition}

The next proposition states that for $\eps\to 0$ the Laplace transform $\T f$ of a self-similar profile converges to $\T \bar{f}$ with respect to $\lfnorm*{2}{\mu}{\theta}{\cdot}$ for certain $\mu>0$, i.e.\@ we gain certain regularity close to zero in Laplace variables.

 \begin{proposition}\label{Prop:closeness:two:norm}
  For all $\mu\in[0,\mu_{*})$ and $\delta>0$ it holds for sufficiently small $\eps>0$ that 
  \begin{equation*}
   \lfnorm*{2}{\mu}{\theta}{\T (f-\bar{f})}\leq \delta
  \end{equation*}
 for each solution $f$ of~\eqref{eq:self:sim}. In particular, we have $\T (f-\bar{f})\in\X{2}{\mu}{\theta}$.
 \end{proposition}
 
 The next statement gives an estimate on the boundary layer and the corresponding proof will be the by far most technical and lengthy part of this work. In principle it states that the remainder term, emerging from the application of Proposition~\ref{Prop:improved:regularity} in order to obtain enough decay at infinity, can still be controlled in a suitable way to obtain a global contraction estimate.
 
 \begin{proposition}\label{Prop:boundary:layer}
 For all $\alpha\in(0,1/2)$, $\rho\in(1/2,1)$ and $\mu\in[0,1]$ it holds that for each $\delta_{*}>0$ there exists a constant $C_{\delta_{*}}>0$ such that for sufficiently small $\eps$ the estimate
 \begin{equation*}
  \lfnorm*{2}{\mu}{\theta}{\T\bigl(\zeta(f_1-f_2)\bigr)}\leq \delta_{*}\lfnorm*{2}{0}{\theta}{\T(f_1-f_2)}+C_{\delta_{*}}\lfnorm*{1}{0}{\theta}{\T\bigl((1-\zeta)(f_1-f_2)\bigr)}
 \end{equation*}
 holds for all solutions $f_1$, $f_2$ of~\eqref{eq:self:sim}.
\end{proposition}

\subsection{Proof of Theorem~\ref{Thm:main}}

Based on the key results collected in the previous section, we will now give the proof of our main result, Theorem~\ref{Thm:main}. Therefore let us first introduce some notation to simplify the presentation. Precisely, we will consider two solutions $f_1$ and $f_2$ of~\eqref{eq:self:sim} and we more precisely look on their difference from the explicitly known solution $\bar{f}$ in the case $\eps=0$, i.e.\@ we define $m_{j}\vcc=f_{j}-\bar{f}$ for $j=1,2$. Moreover, the difference of these two deviations will be denoted by $m$, i.e.\@ $m\vcc=m_1-m_2$ and we also note that it holds $m=f_1-f_2$. 

Taking into account that $\bar{f}$ solves $\bar{f}=A_{\rho}(\bar{f})+ B_{2}(\bar{f},\bar{f})$ and recalling  that $f_{j}$ solves~\eqref{eq:abstract:self:sim} for $j=1,2$ one obtains, by taking the difference of these equations, that $m_j$ satisfies
 \begin{equation*}
  m_{j}=A_{\rho}(m_j)+ B_{2}(\bar{f},m_{j})+ B_{2}(m_j,\bar{f})+ B_{2}(m_j,m_j)+\eps B_{W}(f_j,f_j).
 \end{equation*}
 Subtracting this equation for $j=2$ from that one for $j=1$ it follows
 \begin{equation*}
  m=A_{\rho}(m)+B_{2}(\bar{f},m)+B_{2}(m,\bar{f})+B_{2}(m,m_1)+B_{2}(m_2,m)+\eps B_{W}(m,f_1)+\eps B_{W}(f_2,m).
 \end{equation*}
 Together with the operator $L$ this can be further rearranged as
 \begin{equation*}
  L(m)=B_{2}(m,m_1)+ B_{2}(m_2,m)+\eps B_{W}(m,f_1)+\eps B_{W}(f_2,m).
 \end{equation*}
 We then apply first the Laplace transform operator $\T$ to the equation, use $\T( L(m))=\LL(\T m)$ and apply also $\LL^{-1}$ which yields
 \begin{equation*}
  \T m =\LL^{-1}\T \bigl( B_{2}(m,m_1)+ B_{2}(m_2,m)\bigr)+\eps \bigl( B_{W}(m,f_1)+\eps B_{W}(f_2,m)\bigr).
 \end{equation*}
Note also Remark~\ref{Rem:transform:justification} which comments on the justification of this step. Using~\eqref{eq:Laplace:easy:shift}, we can take the difference of the previous equation and the same relation shifted by one which can then be written as
\begin{multline*}
 \T\bigl((1-\zeta)m\bigr)=\LL^{-1}\T\Bigl((1-\zeta)\bigl( B_{2}(m,m_1)+ B_{2}(m_2,m)\bigr)\Bigr)\\*
 +\eps\LL^{-1}\T \Bigl((1-\zeta)\bigl( B_{W}(m,f_1)+\eps B_{W}(f_2,m)\bigr)\Bigr).
\end{multline*}
For any fixed $\mu\in(0,\mu_{*})$ we now apply $\lfnorm*{2}{\mu}{\theta}{\cdot}$ on both sides which yields together with the boundedness of $\LL^{-1}$ as given by Proposition~\ref{Prop:linearised:operator} that
\begin{multline*}
 \lfnorm*{2}{\mu}{\theta}{\T\bigl((1-\zeta)m\bigr)}\leq C\Bigl(\lfnorm*{2}{\mu}{\theta}{\T\bigl((1-\zeta) B_{2}(m,m_1)\bigr)}+\lfnorm*{2}{\mu}{\theta}{\T\bigl((1-\zeta) B_{2}(m_2,m)\bigr)}\Bigr)\\*
 +C\eps\Bigl(\lfnorm*{2}{\mu}{\theta}{\T\bigl((1-\zeta) B_{W}(m,f_1)\bigr)}+\lfnorm*{2}{\mu}{\theta}{\T\bigl((1-\zeta) B_{W}(f_2,m)\bigr)}\Bigr).
\end{multline*}
The first two terms on the right-hand side can be bounded due to \cref{Lem:est:B2} together with \cref{Rem:norms:behave:well:under:shifts,eq:transformed:forms}, while for the remaining two terms we rely on \cref{Prop:improved:regularity,eq:transformed:forms,eq:Laplace:easy:shift} to find
\begin{multline*}
 \lfnorm*{2}{\mu}{\theta}{\T\bigl((1-\zeta)m\bigr)}\leq C\Bigl(\lfnorm*{1}{0}{\theta}{\T m_1}+\lfnorm*{1}{0}{\theta}{\T m_2}\Bigr)\lfnorm*{1}{0}{\theta}{\T m}\\*
 +C\eps \Bigl(\lfnorm*{2}{0}{\theta}{\T f_1}+\lfnorm*{2}{0}{\theta}{\T f_2}\Bigr)\lfnorm*{2}{0}{\theta}{\T m}.
\end{multline*}
From \cref{Prop:norm:boundedness,Prop:closeness:two:norm} we then deduce
\begin{equation}\label{eq:uniqueness:easy:part}
 \lfnorm*{2}{\mu}{\theta}{\T\bigl((1-\zeta)m\bigr)}\leq C(\delta_{\eps}+\eps)\lfnorm*{2}{0}{\theta}{\T m} \quad \text{with }\delta_{\eps}\to 0 \text{ for }\eps\to 0.
\end{equation}
To conclude the proof, we split the quantity $m$ as $m=\zeta m+(1-\zeta)m$ which yields
\begin{equation*}
 \lfnorm*{2}{\mu}{\theta}{\T m}\leq \lfnorm*{2}{\mu}{\theta}{\T(\zeta m)}+\lfnorm*{2}{\mu}{\theta}{\T \bigl((1-\zeta)m\bigr)}.
\end{equation*}
Since we also have $m=f_1-f_2$ the boundary layer estimate from Proposition~\ref{Prop:boundary:layer} together with Lemma~\ref{Lem:comparison:norms} implies
\begin{equation*}
 \lfnorm*{2}{\mu}{\theta}{\T m}\leq \delta_{*}\lfnorm*{2}{0}{\theta}{\T m}+(1+C_{\delta_{*}})\lfnorm*{2}{\mu}{\theta}{\T\bigl((1-\zeta)m\bigr)}
\end{equation*}
for all $\delta_{*}>0$ provided $\eps>0$ is sufficiently small. If we combine this with~\eqref{eq:uniqueness:easy:part} we can conclude that
\begin{equation*}
 \lfnorm*{2}{\mu}{\theta}{\T m}\leq \bigl(C(1+C_{\delta_{*}})(\delta_{\eps}+\eps)+\delta_{*}\bigr)\lfnorm*{2}{0}{\theta}{\T m}.
\end{equation*}
The claim finally follows by choosing first $\delta_{*}<1/4$ and then $\eps>0$ sufficiently small such that $C(1+C_{\delta_{*}})(\delta_{\eps}+\eps)<1/4$ which together with Lemma~\ref{Lem:comparison:norms} leads to $\lfnorm*{2}{\mu}{\theta}{\T m}\leq 1/2 \lfnorm*{2}{\mu}{\theta}{\T m}$. This shows $\T m=0$ and thus also $m=0$ due to the fact that the Laplace transform defines a finite measure uniquely. Since $m=f_1-f_2$ this then implies uniqueness of solutions to~\eqref{eq:self:sim} and finishes the proof of Theorem~\ref{Thm:main}.

\section{Continuity estimates}\label{Sec:continuity:estimates}

In this section, we establish several estimates on the (bi-)linear forms $\mathcal{A}_{\rho}$, $\mathcal{B}_{2}$ and $\mathcal{B}_{W}$ and show that they are well-defined and continuous on the spaces $\X{k}{\mu}{\chi}$ for appropriate choices of the parameters $k,\mu,\chi$. In particular, we will give the proofs of \cref{Lem:est:Arho,Lem:est:B2,Prop:est:BW,Prop:improved:regularity}. 

\subsection{Proof of \cref{Lem:est:Arho,Lem:est:B2}}

The statements of \cref{Lem:est:Arho,Lem:est:B2} can be derived rather directly from the definitions of $\mathcal{A}_{\rho}$ and $\mathcal{B}_{2}$.

\begin{proof}[Proof of Lemma~\ref{Lem:est:Arho}]
 Due to the definition of $\mathcal{A}_{\rho}$ it suffices to show the estimate~\eqref{eq:est:Arho}, while we conclude from Lemma~\ref{Lem:seminorm:improved} that it is in fact enough to show $\lsnorm*{2}{\mu}{\theta}{\mathcal{A}_{\rho}(G)}\leq C\lfnorm*{1}{\mu}{\chi}{G}$. The latter however follows immediately, i.e.\@ we have
 \begin{equation*}
  \abs*{\del_{q}^{2}\mathcal{A}_{\rho}(G)}=\abs*{\frac{1-\rho}{r}G'(r)}\leq \frac{(1-\rho)\lsnorm*{1}{\mu}{\chi}{G}}{r^{2-\rho-\mu}(r+1)^{\chi+\rho+\mu}}
 \end{equation*}
 which concludes the proof due to the definition of the norm.
\end{proof}

\begin{proof}[Proof of Lemma~\ref{Lem:est:B2}]
 The proof is similar to that one of Lemma~\ref{Lem:est:Arho}, i.e.\@ due to the structure of $\mathcal{B}_{2}$ it suffices to prove only the estimate~\eqref{eq:est:B2} which, due to Lemma~\ref{Lem:seminorm:improved} and~\eqref{eq:est:weight} reduces to showing
 \begin{equation*}
  \frac{1}{r}\abs*{G'(r)\bigl(H(0)-H(r)\bigr)}\leq C\lsnorm*{1}{0}{\chi}{G}\lsnorm*{1}{0}{\chi}{H}\weight{\rho+\mu}{\chi}(r)r^{-2}.
 \end{equation*}
 The latter estimate now follows from~\cref{eq:weight:mult,eq:est:by:norm,eq:weight:parmono,Lem:difference:regularised}, i.e.\@ we have
 \begin{equation*}
  \begin{split}
   \frac{1}{r}\abs*{G'(r)\bigl(H(0)-H(r)\bigr)}&\leq C\lsnorm*{1}{0}{\chi}{G}\lsnorm*{1}{0}{\chi}{H}\weight{\rho}{\chi}(r)\weight{\rho}{0}(r)r^{-2}=C\lsnorm*{1}{0}{\chi}{G}\lsnorm*{1}{0}{\chi}{H}\weight{2\rho}{\chi}(r)r^{-2}\\
   &\leq C\lsnorm*{1}{0}{\chi}{G}\lsnorm*{1}{0}{\chi}{H}\weight{\rho+\mu}{\chi}(r)r^{-2}
  \end{split}
 \end{equation*}
 which ends the proof.
\end{proof}

\subsection{Proof of Proposition~\ref{Prop:est:BW}}

Before we come to the proof of Proposition~\ref{Prop:est:BW}, we will collect several preliminary estimates.

\begin{lemma}\label{Lem:est:N:small}
 For every $\alpha\in(0,1)$ and $\mu\in[0,\mu_{*})$ there exists a constant $C_{\mu,\alpha}>0$ such that it holds
 \begin{equation*}
  \sup_{0<r<1}r^{2-\rho-\mu}\abs*{N[G,H](r)}\leq C_{\mu,\alpha}\lfnorm*{2}{0}{\theta}{G}\lfnorm*{2}{0}{\theta}{H}
 \end{equation*}
 for all $G,H\in \X{2}{0}{\theta}$.
\end{lemma}

\begin{proof}
 We assume $r\in (0,1)$ and recall from~\eqref{eq:kernel:BW} that $N=N_1+N_2$ such that we can treat $N_1$ and $N_2$ separately. For $N_{1}$ we find together with~\eqref{eq:est:norm:difference} that
 \begin{equation*}
  \abs*{G''(\xi+r)}\leq \frac{\lsnorm*{2}{0}{\theta}{G}}{(\xi+r)^{2-\rho}(\xi+r+1)^{\theta+\rho}}\quad \text{and}\quad \frac{\abs*{H(\eta)-H(\eta+r)}}{r}\leq C\frac{\lsnorm*{1}{0}{\theta}{H}}{\eta^{1-\rho}(\eta+1)^{\theta+\rho}}.
 \end{equation*}
 From this we deduce together with the definition of $N_{1}$ and $(\xi+r)^{\rho-2}\leq r^{-1}(\xi+r)^{\rho-1}$ that
 \begin{equation*}
  \begin{split}
   \abs*{N_{1}[G,H](r)}&\leq \int_{0}^{\infty}\int_{0}^{\infty}\abs*{\Ker(\xi,\eta)}\abs*{G''(\xi+r)}\frac{\abs*{H(\eta)-H(\eta+r)}}{r}\deta\dxi\\
   &\leq C\lfnorm*{2}{0}{\theta}{G}\lfnorm*{1}{0}{\theta}{H}\int_{0}^{\infty}\int_{0}^{\infty}\frac{\abs*{\Ker(\xi,\eta)}}{(\xi+r)^{2-\rho}(\xi+r+1)^{\theta+\rho}\eta^{1-\rho}(\eta+1)^{\theta+\rho}}\deta\dxi\\
   &\leq C\frac{\lfnorm*{2}{0}{\theta}{G}\lfnorm*{1}{0}{\theta}{H}}{r}\int_{0}^{\infty}\int_{0}^{\infty}\frac{\abs*{\Ker(\xi,\eta)}}{(\xi+r)^{1-\rho}(\xi+r+1)^{\theta+\rho}\eta^{1-\rho}(\eta+1)^{\theta+\rho}}\deta\dxi.
  \end{split}
 \end{equation*}
Lemma~\ref{Lem:kernel:est:small:1} then implies for $\mu\in[0,\mu_{*})$ that
\begin{equation}\label{eq:est:N1:small}
 \abs*{N_{1}[G,H](r)}\leq C_{\mu,\alpha}\lfnorm*{2}{0}{\theta}{G}\lfnorm*{1}{0}{\theta}{H} r^{\rho+\mu-2}.
\end{equation}
From~\eqref{eq:difference:simple} one obtains that it holds
\begin{equation*}
 \abs*{G'(\xi+r)}\leq \frac{\lsnorm*{1}{0}{\theta}{G}}{(\xi+r)^{1-\rho}(\xi+r+1)^{\theta+\rho}}\quad \text{and}\quad \abs*{H'(\eta)-H'(\eta+\tau)}\leq C\frac{\lsnorm*{1}{0}{\theta}{H}}{\eta^{1-\rho}(\eta+1)^{\theta+\rho}}.
\end{equation*}
With these estimates, one can deduce for each $\mu\in [0,\mu_{*})$ in the same way as for $N_{1}[G,H]$ that
\begin{equation*}
 \abs*{N_{2}[G,H](r)}\leq C_{\alpha,\mu}\lfnorm*{1}{0}{\theta}{G}\lfnorm*{1}{0}{\theta}{H}r^{\rho+\mu-2}.
\end{equation*}
Combining this with~\eqref{eq:est:N1:small} the claim directly follows.
\end{proof}

\begin{lemma}\label{Lem:N:est:large}
 For each $\alpha\in(0,1)$ there exists a constant $C_{\alpha}>0$ such that it holds
 \begin{equation*}
  \sup_{r\geq 1}r^{2+\theta-\alpha}\abs*{N[G,H](r)}\leq C_{\alpha}\lfnorm*{2}{0}{\theta}{G}\lfnorm*{2}{0}{\theta}{H}
 \end{equation*}
for all $G,H\in\X{2}{0}{\theta}$.
\end{lemma}

\begin{proof}
 We only have to consider $r\geq 1$ and due to~\cref{eq:est:weight,eq:est:by:norm,eq:difference:simple} we get the estimates
 \begin{align*}
  \abs*{G''(\xi+r)}&\leq \frac{\lsnorm*{2}{0}{\theta}{G}}{(\xi+r)^{2+\theta}},& \abs*{G'(\xi+r)}&\leq \frac{\lsnorm*{1}{0}{\theta}{G}}{(\xi+r)^{1+\theta}}\\
  \abs*{H(\eta)-H(\eta+r)}&\leq C\frac{\lfnorm*{0}{-\rho}{\theta}{H}}{(\eta+1)^{\theta}},& \abs*{H'(\eta)-H'(\eta+r)}&\leq C\frac{\lsnorm*{1}{0}{\theta}{H}}{(1+\eta)^{\theta+\rho}\eta^{1-\rho}}.
 \end{align*}
 Recalling~\eqref{eq:kernel:BW} we thus deduce
 \begin{multline*}  
  \abs*{N[G,H](r)}\\*
  \leq \frac{\lfnorm*{2}{0}{\theta}{G}\lfnorm*{2}{0}{\theta}{H}}{r}\int_{0}^{\infty}\int_{0}^{\infty}\biggl(\frac{\abs*{\Ker(\xi,\eta)}}{(\xi+r)^{2+\theta}(\eta+1)^{\theta}}+\frac{\abs*{\Ker(\xi,\eta}}{(\xi+r)^{1+\theta}(1+\eta)^{\theta+\rho}\eta^{1-\rho}}\biggr)\deta\dxi.
 \end{multline*}
Finally, it follows from \cref{Lem:Ker:est:large:1,Lem:Ker:est:large:2} and $r\geq 1$ that
\begin{equation*}
 \abs*{N[G,H](r)}\leq C\frac{\lfnorm*{2}{0}{\theta}{G}\lfnorm*{2}{0}{\theta}{H}}{r}\biggl(\frac{1}{r^{1+\theta}}+\frac{1}{r^{1+\theta-\alpha}}\biggr)\leq C\lfnorm*{2}{0}{\theta}{G}\lfnorm*{2}{0}{\theta}{H}r^{\alpha-\theta-2}
\end{equation*}
which ends the proof.
\end{proof}

\begin{proof}[Proof of Proposition~\ref{Prop:est:BW}]
 It suffices to prove that~\eqref{eq:continuity:BW} holds, which is however and immediate consequence of \cref{Lem:N:est:large,Lem:est:N:small} taking also \cref{eq:est:weight,Lem:seminorm:improved} and the structure of $\mathcal{B}_{W}$ into account.
\end{proof}

\subsection{Estimates for differences -- Proof of Proposition~\ref{Prop:improved:regularity}}

Again, we collect several preliminary estimates that will be used in the proof of Proposition~\ref{Prop:improved:regularity}.

\begin{lemma}\label{Lem:N1:difference}
 For each $\alpha\in(0,1)$ there exists a constant $C_{\alpha}>0$ such that it holds
 \begin{equation*}
  \sup_{r\geq 1}\Bigl(r^{2+\theta}\abs*{N_{1}[G,H](r)-N_{1}[G,H](r+1)}\Bigr)\leq C_{\alpha}\lfnorm*{2}{0}{\theta}{G}\lfnorm*{2}{0}{\theta}{H},
 \end{equation*}
for all $G,H\in\X{2}{0}{\theta}$.
\end{lemma}

\begin{proof}
 We assume $r\geq 1$ and note that an elementary calculation shows that we can rewrite
 \begin{multline*}
  \phantom{{}={}}N_{1}[G,H](r)-N_{1}[G,H](r+1)\\*
  \shoveleft{=\frac{1}{r}\int_{0}^{\infty}\int_{0}^{\infty}\Ker(\xi,\eta)\Bigl(\bigl(G''(\xi+r)-G''(\xi+r+1)\bigr)\bigl(H(\eta)-H(\eta+r)\bigr)\Bigr)\deta\dxi}\\*
  +\frac{1}{r}\int_{0}^{\infty}\int_{0}^{\infty}\Ker(\xi,\eta)G''(\xi+r+1)\bigl(H(\eta+r+1)-H(\eta+r)\bigr)\deta\dxi\\*
  +\frac{1}{r(r+1)}\int_{0}^{\infty}\int_{0}^{\infty}\Ker(\xi,\eta)G''(\xi+r+1)\bigl(H(\eta)-H(\eta+r+1)\bigr)\deta\dxi=\vcc(I)+(II)+(III).
 \end{multline*}
 To simplify the presentation, we estimate the three terms on the right-hand side separately. To start with $(I)$ we note that~\cref{eq:est:weight,eq:difference:simple} together with $r\geq 1$ yields
 \begin{equation*}
  \abs*{G''(\xi+r)-G''(\xi+r+1)}\leq C\frac{\lsnorm*{2}{0}{\theta}{G}}{(\xi+r)^{2+\theta}}\quad \text{and}\quad \abs*{H(\eta)-H(\eta+r)}\leq C\frac{\lfnorm*{0}{-\rho}{\theta}{H}}{(\eta+1)^{\theta}}.
 \end{equation*}
 Combining this with Lemma~\ref{Lem:Ker:est:large:2} it follows
 \begin{equation}\label{eq:N1:diff:1}
  \abs*{(I)}\leq C\frac{\lfnorm*{2}{0}{\theta}{G}\lfnorm*{0}{-\rho}{\theta}{H}}{r}\int_{0}^{\infty}\int_{0}^{\infty}\frac{\abs*{\Ker(\xi,\eta)}}{(\xi+r)^{2+\theta}(\eta+1)^{\theta}}\deta\dxi\leq C\frac{\lfnorm*{2}{0}{\theta}{G}\lfnorm*{0}{-\rho}{\theta}{H}}{r^{2+\theta}}.
 \end{equation}
 We next consider $(II)$ and note that~\eqref{eq:est:norm:difference} additionally implies
 \begin{equation*}
  \abs*{H(\eta+r+1)-H(\eta+r)}\leq C\lsnorm*{1}{0}{\theta}{H}(\eta+r)^{-1-\theta}.
 \end{equation*}
 To estimate $(II)$ we now change variables $\xi\mapsto r\xi$ and $\eta\mapsto r\eta$, use the homogeneity of $\Ker$ and invoke Lemma~\ref{Lem:Ker:est:most:general} which together yields
 \begin{equation}\label{eq:N1:diff:2}
  \begin{split}
   \abs*{(II)}&\leq C\frac{\lfnorm*{2}{0}{\theta}{G}\lfnorm*{1}{0}{\theta}{H}}{r}\int_{0}^{\infty}\int_{0}^{\infty}\frac{\abs*{\Ker(\xi,\eta)}}{(\xi+r)^{2+\theta}(\eta+r)^{1+\theta}}\deta\dxi\\
   &\leq C\frac{\lfnorm*{2}{0}{\theta}{G}\lfnorm*{1}{0}{\theta}{H}}{r^{3+2\theta}}\int_{0}^{\infty}\int_{0}^{\infty}\frac{\abs*{\Ker(\xi,\eta)}}{(\xi+1)^{2+\theta}(\eta+1)^{1+\theta}}\deta\dxi\leq C\frac{\lfnorm*{2}{0}{\theta}{G}\lfnorm*{1}{0}{\theta}{H}}{r^{3+2\theta}}.
  \end{split}
 \end{equation}
 It remains to consider $(III)$ for which we first recall from~\eqref{eq:est:norm:difference} that
 \begin{equation*}
  \abs*{H(\eta)-H(\eta+r+1)}\leq C\lfnorm*{0}{-\rho}{\theta}{H}(\eta+1)^{-\theta}.
 \end{equation*}
 Since $r\geq 1$ we thus obtain together with Lemma~\ref{Lem:Ker:est:large:2} that
 \begin{equation}\label{eq:N1:diff:3}
  \abs*{(III)}\leq C\frac{\lfnorm*{2}{0}{\theta}{G}\lfnorm*{0}{-\rho}{\theta}{H}}{r^2}\int_{0}^{\infty}\int_{0}^{\infty}\frac{\abs*{\Ker(\xi,\eta)}}{(\xi+r)^{2+\theta}(\eta+1)^{\theta}}\deta\dxi\leq C\frac{\lfnorm*{2}{0}{\theta}{G}\lfnorm*{0}{-\rho}{\theta}{H}}{r^{3+\theta}}.
 \end{equation}
 If we combine~\cref{eq:N1:diff:1,eq:N1:diff:2,eq:N1:diff:3} the claim follows because $\theta>0$.
\end{proof}

\begin{lemma}\label{Lem:N2:difference}
 For each $\alpha\in(0,1)$ there exists a constant $C_{\alpha}>0$ such that it holds
 \begin{equation*}
  \sup_{r\geq 1}\Bigl(r^{2+\theta}\abs*{N_{2}[G,H](r)-N_{2}[G,H](r+1)}\Bigr)\leq C_{\alpha}\lfnorm*{2}{0}{\theta}{G}\lfnorm*{2}{0}{\theta}{H},
 \end{equation*}
for all $G,H\in\X{2}{0}{\theta}$.
\end{lemma}

\begin{proof}
 We assume $r\geq 1$ and recall~\eqref{eq:kernel:BW} to rewrite
 \begin{equation}\label{eq:N2:diff:0}
  \begin{split}
   &\phantom{{}={}}N_{2}[G,H](r)-N_{2}[G,H](r+1)\\
   &=\int_{0}^{\infty}\int_{0}^{\infty}\Ker(\xi,\eta)\bigl(G'(\xi+r)-G'(\xi+r+1)\bigr)\frac{H'(\eta)-H'(\eta+r)}{r}\deta\dxi\\
   &\qquad +\int_{0}^{\infty}\int_{0}^{\infty}\Ker(\xi,\eta)G'(\xi+r+1)\biggl(\frac{H'(\eta)-H'(\eta+r)}{r}-\frac{H'(\eta)-H'(\eta+r+1)}{r+1}\biggr)\deta\dxi\\
   &=\vcc (I)+(II).
  \end{split}
 \end{equation}
To estimate the first integral $(I)$ on the right-hand side, we note that due to~\cref{eq:difference:simple,eq:est:norm:difference} it holds
\begin{equation*}
 \abs*{G'(\xi+r+1)-G'(\xi+r)}\leq C\frac{\lsnorm*{2}{0}{\theta}{G}}{(\xi+r)^{2+\theta}}\quad \text{and}\quad \abs*{H'(\eta)-H'(\eta+r)}\leq C\frac{\lsnorm*{1}{0}{\theta}{H}}{\eta^{1-\rho}(1+\eta)^{\theta+\rho}}.
\end{equation*}
Together with Lemma~\ref{Lem:Ker:est:large:1} this now yields
\begin{equation}\label{eq:N2:diff:1}
 \begin{split}
  \abs*{(I)}&\leq \frac{\lfnorm*{2}{0}{\theta}{G}\lfnorm*{1}{0}{\theta}{H}}{r}\int_{0}^{\infty}\int_{0}^{\infty}\frac{\abs*{\Ker(\xi,\eta)}}{(\xi+r)^{2+\theta}\eta^{1-\rho}(\eta+1)^{\theta+\rho}}\deta\dxi\leq \frac{\lfnorm*{2}{0}{\theta}{G}\lfnorm*{1}{0}{\theta}{H}}{r^{3+\theta-\alpha}}.
 \end{split}
\end{equation}
To estimate the second expression, i.e.\@ $(II)$, we have to rewrite further. Precisely, it holds
\begin{multline*}
 \frac{H'(\eta)-H'(\eta+r)}{r}-\frac{H'(\eta)-H'(\eta+r+1)}{r+1}\\*
 =\int_{\eta+r}^{\eta+r+1}\del_{s}\biggl(\frac{H'(s)}{s-\eta}\biggr)\ds+\frac{H'(\eta)}{r(r+1)}=\int_{\eta+r}^{\eta+r+1}\frac{H''(s)}{s-\eta}-\frac{H'(s)}{(s-\eta)^2}\ds+\frac{H'(\eta)}{r(r+1)}.
\end{multline*}
Together with~\cref{eq:est:weight,eq:est:by:norm} we thus find
\begin{multline*}
 \phantom{{}\leq{}}\abs*{\frac{H'(\eta)-H'(\eta+r)}{r}-\frac{H'(\eta)-H'(\eta+r+1)}{r+1}}\\*
 \shoveleft{\leq \lfnorm*{2}{0}{\theta}{H}\int_{\eta+r}^{\eta+r+1}\frac{1}{(s-\eta)s^{2+\theta}}+\frac{1}{(s-\eta)^{2}s^{1+\theta}}\ds+C\frac{\lsnorm*{1}{0}{\theta}{H}}{r^2\eta^{1-\rho}(1+\eta)^{\theta+\rho}}}\\*
 \leq C\lfnorm*{2}{0}{\theta}{H}\biggl(\frac{1}{r(\eta+r)^{2+\theta}}+\frac{1}{r^{2}(\eta+r)^{1+\theta}}+\frac{1}{r^{2}\eta^{1-\rho}(\eta+1)^{\theta+\rho}}\biggr).
\end{multline*}
Together with $(\eta+r)^{-1}\leq r^{-1}$ and $(\eta+r)^{-1-\theta}\leq (\eta+1)^{-\theta-\rho}\eta^{\rho-1}$ for $r\geq 1$ and $\rho<1$ we finally get
\begin{equation*}
 \abs*{\frac{H'(\eta)-H'(\eta+r)}{r}-\frac{H'(\eta)-H'(\eta+r+1)}{r+1}}\leq C\frac{\lfnorm*{2}{0}{\theta}{H}}{r^2}\frac{1}{\eta^{1-\rho}(\eta+1)^{\theta+\rho}}.
\end{equation*}
To estimate the integral $(II)$ we take further into account that~\eqref{eq:est:by:norm} yields $\abs*{G'(\xi+r+1)}\leq C\lsnorm*{1}{0}{\theta}{G}(\xi+r)^{-1-\theta}$ for $r\geq 1$ such that we can conclude with Lemma~\ref{Lem:Ker:est:large:1} that
\begin{equation*}
 \begin{split}
  \abs*{(II)}\leq C\frac{\lsnorm*{1}{0}{\theta}{G}\lfnorm*{2}{0}{\theta}{H}}{r^2}\int_{0}^{\infty}\int_{0}^{\infty}\frac{\abs*{\Ker(\xi,\eta)}}{(\xi+r)^{1+\theta}\eta^{1-\rho}(\eta+1)^{\theta+\rho}}\deta\dxi\leq C\frac{\lsnorm*{1}{0}{\theta}{G}\lfnorm*{2}{0}{\theta}{H}}{r^{3+\theta-\alpha}}.
 \end{split}
\end{equation*}
Since $\alpha<1$ and $r\geq 1$ the claim now follows if we combine this with~\cref{eq:N2:diff:0,eq:N2:diff:1}.
\end{proof}

\begin{proof}[Proof of Proposition~\ref{Prop:improved:regularity}]
 If we recall~\eqref{eq:BW:with:kernel} one easily checks that 
 \begin{equation*}
  \mathcal{B}_{W}(G,H)(q)-\mathcal{B}_{W}(G,H)(q+1)=\int_{q}^{\infty}\int_{p}^{\infty}\bigl(N[G,H](r)-N[G,H](r+1)\bigr)\dr\dd{p}.
 \end{equation*}
 From the structure of the right-hand side together with Lemma~\ref{Lem:seminorm:improved} we deduce that it suffices to prove the estimate
 \begin{equation*}
  \abs*{N[G,H](r)-N[G,H](r+1)}\leq C_{\nu}\lfnorm*{2}{0}{\theta}{G}\lfnorm*{2}{0}{\theta}{H}\weight{\rho+\mu-2}{\theta+2}(r)\quad \text{for all }r>0.
 \end{equation*}
 To do so, we consider $r\leq 1$ and $r>1$ separately and rely on \cref{Lem:est:N:small,Lem:N:est:large,Lem:N1:difference,Lem:N2:difference}. In more detail, for $r\leq 1$ it follows from \cref{Lem:est:N:small,Lem:N:est:large} and estimating by the most singular term that
 \begin{equation*}
  \begin{split}
   \abs*{N[G,H](r)-N[G,H](r+1)}&\leq C_{\mu}\lfnorm*{2}{0}{\theta}{G}\lfnorm*{2}{0}{\theta}{H}\bigl(r^{\rho+\mu-2}+(r+1)^{\alpha-\theta-2}\bigr)\\
   &\leq C_{\mu}\lfnorm*{2}{0}{\theta}{G}\lfnorm*{2}{0}{\theta}{H}r^{\rho+\mu-2}.
  \end{split}
 \end{equation*}
 Conversely, \cref{Lem:N1:difference,Lem:N2:difference} together already yield
 \begin{equation*}
  \abs*{N[G,H](r)-N[G,H](r+1)}\leq C_{\mu}\lfnorm*{2}{0}{\theta}{G}\lfnorm*{2}{0}{\theta}{H}r^{-2-\theta}\quad \text{for }r\geq 1.
 \end{equation*}
 This then finishes the proof.
\end{proof}

\section{Linearised coagulation operator -- Proof of Proposition~\ref{Prop:linearised:operator}}\label{Sec:linearised:operator}

 This section is devoted to the proof of Proposition~\ref{Prop:linearised:operator} which states that $\LL$ is continuous and invertible with bounded inverse.

 \begin{proof}[Proof of Proposition~\ref{Prop:linearised:operator}]
  We split the proof in three main steps. First, we will show that $\LL$ is a well defined continuous operator from $\X{k}{\mu}{\chi}$ into itself. In the second step, we will compute an explicit formula for the inverse operator $\LL^{-1}$ while in the last step we then show that the inverse is again bounded as operator from $\X{k}{\mu}{\chi}$ into itself.
  
  \paragraph{Step 1:}    
   The boundedness and well-definedness of $\LL$ will be a straightforward consequence of~\cref{Lem:est:Arho,Lem:est:B2,Lem:comparison:norms}. Precisely, we have
   \begin{equation*}
    \begin{split}
     \lfnorm*{k}{\mu}{\chi}{\LL(G)}&\leq \lfnorm*{k}{\mu}{\chi}{G}+\lfnorm*{k}{\mu}{\chi}{\mathcal{A}_{\rho}(G)}+\lfnorm*{k}{\mu}{\chi}{\mathcal{B}_{2}(\bar{F},G)}+\lfnorm*{k}{\mu}{\chi}{\mathcal{B}_{2}(G,\bar{F})}\\
     &\leq C\Bigl(1+\lfnorm*{1}{0}{\chi}{\bar{F}}\Bigr)\lfnorm*{k}{\mu}{\chi}{G}.
    \end{split}
   \end{equation*}
 Since $\lfnorm*{1}{0}{\chi}{\bar{F}}\leq C$ due to Remark~\ref{Rem:explicit:profiles}, it follows
 \begin{equation*}
  \lfnorm*{k}{\mu}{\chi}{\LL(G)}\leq C\lfnorm*{k}{\mu}{\chi}{G}.
 \end{equation*}
 This shows that $\LL$ is well-defined and bounded.
 
 \paragraph{Step 2:} We will now derive and explicit formula for the inverse of $\LL$. More precisely, we will solve the equation $\LL(G)=H$. To do so we recall the definition of $\LL$ in~\eqref{eq:def:lin:op}, plug in the integral expressions for $\mathcal{A}_{\rho}$ and $\mathcal{B}_{2}$ and take the second derivative on both sides which leads to the following non-local ordinary differential equation
 \begin{equation}\label{eq:inversion:ode}
  G''(q)+\frac{1-\rho}{q}G'(q)-2\rho^2\frac{q^{\rho-2}}{(1+q^{\rho})^2}\bigl(G(0)-G(q)\bigr)+2\rho\frac{q^{\rho-1}}{1+q^{\rho}}G'(q)=H''(q).
 \end{equation}
 As already mentioned, the general strategy to get a formula for $\LL^{-1}$ will be to solve this equation for $G$. However, at this stage, this is only possible on a formal level, i.e.\@ to solve this equation, we have to impose certain regularity/boundary conditions on $G$ to obtain a consistent solution in $\X{k}{\mu}{\chi}$ while we have to verify later on in Step~3 that these assumptions are really satisfied by our formula.

 We first define the new unknown function $U(q)\vcc=G(q)-G(0)$ and note that by elementary manipulations~\eqref{eq:inversion:ode} can be rewritten as
 \begin{equation*}
  \frac{\dd}{\dd{q}}\bigl(qU'(q)\bigr)-\rho U'(q)+2\rho\frac{\dd}{\dd{q}}\biggl(\frac{q^{\rho}}{1+q^{\rho}}U(q)\biggr)=\frac{\dd}{\dd{q}}\bigl(qH'(q)\bigr)-H'(q).
 \end{equation*}
 The assumption $G\in\X{k}{\mu}{\chi}$ yields in particular that $qU'(q),U(q)\to 0$ for $q\to 0$ which allows to integrate the previous equation over $[0,q]$ and we thus find
 \begin{equation}\label{eq:inversion:ode:2}
  qU'(q)-\rho U(q)+2\rho \frac{q^{\rho}}{1+q^{\rho}}U(q)=qG'(q)-G(q)+G(0).
 \end{equation}
 The integrating factor $q^{-1-\rho}(1+q^{\rho})^2$ allows to rewrite this further as
 \begin{equation*}
  \frac{\dd}{\dd{q}}\biggl(\frac{(q^{\rho}+1)^{2}}{q^{\rho}}U(q)\biggr)=\frac{(q^{\rho}+1)^2}{q^{\rho}}H'(q)+\frac{(q^{\rho}+1)^{2}}{q^{\rho+1}}\bigl(H(q)-H(0)\bigr).
 \end{equation*}
 Since $\mu>0$ we can again exploit the assumption $G\in\X{k}{\mu}{\chi}$ to see that $q^{-\rho}U(q)\to 0$ for $q\to 0$ such that we can integrate once more over $[0,q]$ which yields
 \begin{equation}\label{eq:inversion:ode:3}
  U(q)=\frac{q^{\rho}}{(1+q^{\rho})^2}\int_{0}^{q}\frac{(1+r^{\rho})^2}{r^{\rho}}H'(r)+\frac{(1+r^{\rho})^2}{r^{1+\rho}}\bigl(H(0)-H(r)\bigr)\dr.
 \end{equation}
 We note that for $G\in\X{k}{\mu}{\chi}$ the integral on the right-hand side is well-defined as will be shown in Step~3. This is not yet the desired expression for $\LL^{-1}$ since we still have to recover from this a formula for $G(q)$. This however, requires to determine the value of $G(0)$. To do so, we note that by assumption $H\in\X{k}{\mu}{\chi}$ which yields $qH'(q),H(q)\to 0$ for $q\to\infty$. Since we want $G$ to satisfy $G\in\X{k}{\mu}{\chi}$, we expect in the same way that $qG'(q),G(q)\to 0$ for $q\to \infty$. Thus, it holds by definition of $U$ that $U(q)\to -G(0)$ for $q\to\infty$. If we now note that $U'=G'$ and take the limit $q\to\infty$ in~\eqref{eq:inversion:ode:2} it follows that for a consistent solutions we should choose
 \begin{equation*}
  G(0)=-\frac{H(0)}{\rho}.
 \end{equation*}
 With this, $U(q)=G(q)-G(0)$ and~\eqref{eq:inversion:ode:3} it follows that the inverse operator $\LL^{-1}$ is formally given as
 \begin{equation*}
  \LL^{-1}(H)=-\frac{H(0)}{\rho}+\frac{q^{\rho}}{(1+q^{\rho})^2}\int_{0}^{q}\frac{(1+r^{\rho})^2}{r^{\rho}}H'(r)+\frac{(1+r^{\rho})^2}{r^{1+\rho}}\bigl(H(0)-H(r)\bigr)\dr.
 \end{equation*}
 We now further rewrite this expression because we have have to exploit some cancellation that takes place for large values of $q$ when we want to show the boundedness of $\LL^{-1}$. More precisely, we split the prefactor $(1+r^{\rho})^2=r^{2\rho}+(2r^{\rho}+1)$ in front of $(H(0)-H(r))$ which yields after some elementary rearrangement that
 \begin{multline}\label{eq:inverse:operator}
  \LL^{-1}(H)(q)=-\frac{H(0)}{\rho}\frac{2q^{\rho}+1}{(1+q^{\rho})^2}\\*
  +\frac{q^{\rho}}{(1+q^{\rho})^2}\int_{0}^{q}\frac{(1+r^{\rho})^2}{r^{\rho}}H'(r)-r^{\rho-1}H(r)+\frac{2r^{\rho}+1}{r^{1+\rho}}\bigl(H(0)-H(r)\bigr)\dr.
 \end{multline}

 \paragraph{Step 3:} 
 To conclude the proof we show that the formula~\eqref{eq:inverse:operator} that we derived in the previous step is well-defined and really defines a bounded operator on $\X{k}{\mu}{\chi}$. For this aim we introduce some notation, namely
 \begin{equation}\label{eq:inverse:def:A:B}
  A(q)\vcc=\frac{2q^{\rho}+1}{(1+q^{\rho})^2}\quad \text{and}\quad B(q)\vcc=\frac{q^{\rho}}{(1+q^{\rho})^2}.
 \end{equation}
 Moreover, we note that we have the following estimates which can be verified by an explicit computation. First of all it holds
 \begin{equation}\label{eq:invers:A:1}
  \abs*{A(q)}\leq C\weight{0}{\rho}(q)\quad \text{and}\quad \abs{\del_{q}^{k}A(q)}\leq C\weight{2\rho-k}{\rho+k}(q)\quad \text{for }k=1,2.
 \end{equation}
 Moreover, we have
 \begin{equation}\label{eq:inverse:B:1}
  \abs[\big]{\bigl(B(q)\bigr)^{-1}}\leq C\weight{-\rho}{-\rho}(q)\quad \text{and}\quad \abs*{\del_{q}^{k}B(q)}\leq C\weight{\rho-k}{\rho+k}(q)\quad \text{for }k=0,1,2.
 \end{equation}
 In terms of $A$ and $B$, we can now rewrite~\eqref{eq:inverse:operator} as
 \begin{equation}\label{eq:inverse:operator:simple}
  \LL^{-1}(H)(q)=-\frac{H(0)}{\rho}A(q)+B(q)\int_{0}^{q}\bigl(B(r)\bigr)^{-1}H'(r)-r^{\rho-1}H(r)+\frac{2r^{\rho}+1}{r^{1+\rho}}\bigl(H(0)-H(r)\bigr)\dr.
 \end{equation}
 To show that $\LL^{-1}$ is well-defined and bounded, it suffices to prove that
 \begin{equation*}
  \abs{\LL^{-1}(H)(q)}\leq C\lfnorm*{1}{\mu}{\chi}{H}\weight{0}{\chi}(q)\quad \text{and}\quad \abs{\del_{q}^{k}\LL^{-1}(H)(q)}\leq C\lfnorm*{k}{\mu}{\chi}{H}\weight{\rho+\mu-k}{\chi+k}(q)
 \end{equation*}
 hold true for $k=1,2$ because of~\eqref{eq:est:weight}.
 
 Due to~\eqref{eq:invers:A:1} we immediately get
 \begin{equation*}
  \abs*{\frac{H(0)}{\rho}A(q)}\leq C\lfnorm*{0}{-\rho}{\chi}{H}\weight{0}{\rho}(q)\; \text{ and }\; \abs*{\frac{H(0)}{\rho}\del_{q}^{k}A(q)}\leq C\lfnorm*{0}{-\rho}{\chi}{H}\weight{2\rho-k}{\rho+k}(q)\;\text{for }k=1,2.
 \end{equation*}
 In view of Lemm~\ref{Lem:comparison:norms} and the assumptions $\mu\in(0,\mu_{*})$ and $\chi\in(0,\rho)$ we only have to consider the integral on the right-hand side of~\eqref{eq:inverse:operator:simple}. More precisely, it is sufficient if we prove the estimate
 \begin{multline}\label{eq:inverse:bounded:1}
   \abs*{\del_{q}^{k}\biggl(B(q)\int_{0}^{q}\bigl(B(r)\bigr)^{-1}H'(r)-r^{\rho-1}H(r)+\frac{2r^{\rho}+1}{r^{1+\rho}}\bigl(H(0)-H(r)\bigr)\dr\biggr)}\\*
   \leq C\lfnorm*{\max\{1,k\}}{\mu}{\chi}{H}\weight{\rho+\mu-k}{\chi+k}(q)
 \end{multline}
for $k=0,1,2$. We also remark that for $k=0$ it would even be enough to estimate the left-hand side by $C\lfnorm*{1}{\mu}{\chi}{H}\weight{0}{\chi}(q)$ but it turns out that we in fact get the better estimate shown in~\eqref{eq:inverse:bounded:1}.

In order to verify~\eqref{eq:inverse:bounded:1} we note that~\cref{eq:weight:mult,eq:est:weight,eq:inverse:B:1,Lem:difference:regularised} imply
\begin{equation}\label{eq:inverse:est:integrand}
 \begin{split}
  &\phantom{{}\leq{}}\abs*{\bigl(B(r)\bigr)^{-1}H'(r)-r^{\rho-1}H(r)+\frac{2r^{\rho}+1}{r^{1+\rho}}\bigl(H(0)-H(r)\bigr)}\\
  &\leq C\lfnorm*{1}{\mu}{\chi}{H}\Bigl(\weight{-\rho}{-\rho}(r)\weight{\rho+\mu-1}{\chi+1}(r)+\weight{\rho-1}{1-\rho}(r)\weight{0}{\chi}(r)+\weight{-\rho-1}{1}(r)\weight{\rho+\mu}{0}(r)\Bigr)\\
  &=C\lfnorm*{1}{\mu}{\chi}{H}\Bigl(\weight{\mu-1}{\chi+1-\rho}(r)+\weight{\rho-1}{\chi+1-\rho}(r)+\weight{\mu-1}{1}(r)\Bigr)\leq C\lfnorm*{1}{\mu}{\chi}{H}\weight{\mu-1}{\chi+1-\rho}(r).
 \end{split}
\end{equation}
In the last step we additionally used~\eqref{eq:weight:monotonicity} as well as $\mu\in(0,\mu_{*})$ and $\chi\in(0,\rho)$. Together with~\cref{eq:weight:int:low,eq:weight:mult,eq:inverse:B:1} we thus deduce that it holds
\begin{multline*}
 \abs*{B(q)\int_{0}^{q}\bigl(B(r)\bigr)^{-1}H'(r)-r^{\rho-1}H(r)+\frac{2r^{\rho}+1}{r^{1+\rho}}\bigl(H(0)-H(r)\bigr)\dr}\\*
 \leq C\lfnorm*{1}{\mu}{\chi}{H}\weight{\rho}{\rho}(q)\weight{\mu}{\chi-\rho}(q)\leq C\lfnorm*{1}{\mu}{\chi}{H}\weight{\rho+\mu}{\chi}(q).
\end{multline*}
This then shows~\eqref{eq:inverse:bounded:1} for $k=0$.

To treat the case $k=1$, we compute the first derivative and recall also~\eqref{eq:inverse:def:A:B} to obtain
\begin{multline*}
 \del_{q}\biggl(B(q)\int_{0}^{q}\bigl(B(r)\bigr)^{-1}H'(r)-r^{\rho-1}H(r)+\frac{2r^{\rho}+1}{r^{1+\rho}}\bigl(H(0)-H(r)\bigr)\dr\biggr)\\*
 =B'(q)\int_{0}^{q}\bigl(B(r)\bigr)^{-1}H'(r)-r^{\rho-1}H(r)+\frac{2r^{\rho}+1}{r^{1+\rho}}\bigl(H(0)-H(r)\bigr)\dr\\*
 +H'(q)-\frac{q^{2\rho-2}}{(1+q^{\rho})^2}H(q)+\frac{2q^{\rho}+1}{q(1+q^{\rho})^2}\bigl(H(0)-H(q)\bigr).
\end{multline*}
The estimate~\eqref{eq:inverse:bounded:1} follows now similarly as in the case $k=0$ by means of~\cref{eq:weight:mult,eq:est:by:norm,eq:inverse:B:1,eq:inverse:est:integrand,Lem:difference:regularised}.

The case $k=2$ can be obtained in the same way, i.e.\@ one first computes the second derivative explicitly and then estimates the different terms separately.  
 \end{proof}

\section{Uniform bounds on self-similar profiles}\label{Sec:uniform:convergence}

In this section we will show the uniform convergence of perturbed self-similar profiles to the unperturbed one in Laplace variables. For this we follow a similar approach as in~\cite{NiV14a} (see also~\cite{NTV15}) and therefore we also also keep most of the notation as there. Precisely, for a solution $f$ of~\eqref{eq:self:sim} we denote by $Q$ the desingularised Laplace transform, which is defined by
\begin{equation*}
 Q(q)=\int_{0}^{\infty}(1-\ee^{-qx})f(x)\dx.
\end{equation*}
Due to Definition~\ref{Def:profile} the function $Q$ is well-defined for $q\in[0,\infty)$ and moreover twice differentiable for $q\in(0,\infty)$ with derivatives
\begin{equation}\label{eq:des:Lap:trans:der}
 Q'(q)=\int_{0}^{\infty}\ee^{-qx}xf(x)\dx\quad \text{and}\quad Q''(q)=-\int_{0}^{\infty}\ee^{-qx}x^2f(x)\dx.
\end{equation}
For later use, we also note that $Q(0)=0$ and $Q(q),Q'(q)>0$ for all $q>0$, i.e.\@ $Q$ is in particular strictly increasing.

We also recall the analogous definition for $\bar{f}$ in~\eqref{eq:explicit:bar:quantities}. Moreover, we recall the definition of the Laplace transform $\T$ in~\eqref{eq:Def:T} and note that it holds $\bar{F}=\T \bar{f}$.
\begin{remark}
 Note that at this point, i.e.\@ from Definition~\ref{Def:profile} it is not clear that $\T f$ exists for each self-similar profile $f$ but we will show in Lemma~\ref{Lem:apriori} that this is really the case.
\end{remark}

Finally we introduce the following non-linear operator
\begin{equation*}
 \Pert(f,f)(q)\vcc=\frac{1}{2}\int_{0}^{\infty}\int_{0}^{\infty}W(x,y)f(x)f(y)(1-\ee^{-qx})(1-\ee^{-qy})\dy\dx,
\end{equation*}
which arises naturally when one applies the Laplace transform to equation~\eqref{eq:self:sim}. Again, it follows immediately from Definition~\ref{Def:profile} that $\Pert(f,f)$ is well-defined for any self-similar profile and it holds $\Pert(f,f)\geq 0$.

\subsection{A priori estimates for self-similar profiles}

As a first step, we derive an estimate on the non-linear operator $\Pert(f,f)$ in Laplace variables.

\begin{lemma}\label{Lem:Pert:estimate:Laplace:var}
 For each $\nu\in(0,\rho-\alpha)$ there exists a constant $C_{\nu}>0$ such that it holds
 \begin{equation*}
  \Pert(f,f)(q)\leq C_{\nu}q^{\min\{2\rho-2\nu,1+\rho-\alpha-\nu\}} \quad \text{for all }q>0
 \end{equation*}
 and for each solution $f$ of~\eqref{eq:self:sim}. In particular it holds $\Pert(f,f)(q)\to 0$ for $q\to 0$.
\end{lemma}

\begin{proof}
 Due to the non-negativity of $f$ and~\eqref{eq:Ass:basic} it follows by means of a symmetry argument that
 \begin{equation*}
  \begin{split}
   \Pert(f,f)&\leq \frac{1}{2}\int_{0}^{\infty}\int_{0}^{\infty}\biggl(\Bigl(\frac{x}{y}\Bigr)^{\alpha}+\Bigl(\frac{y}{x}\Bigr)^{\alpha}\biggr)f(x)f(y)(1-\ee^{-qx})(1-\ee^{-qy})\dy\dx\\
   &=\int_{0}^{\infty}x^{\alpha}f(x)(1-\ee^{-qx})\dx\int_{0}^{\infty}y^{-\alpha}f(y)(1-\ee^{-qy})\dy.
  \end{split}
 \end{equation*}
 The estimate $(1-\ee^{-qz})\leq (qz)^{\beta}$ one time with $\beta=\rho-\alpha-\nu$ and one time with $\beta=\min\{\rho+\alpha-\nu,1\}$ then yields together with Lemma~\ref{Lem:moment:est:1} that
 \begin{equation*}
  \begin{split}
   \Pert(f,f)(q)&\leq q^{\min\{2\rho-2\nu,1+\rho-\alpha-\nu\}}\int_{0}^{\infty}x^{\rho-\nu}f(x)\dx\int_{0}^{\infty}y^{-\alpha+\min\{\rho+\alpha-\nu,1\}}f(y)\dy\\
   &\leq C_{\nu}q^{\min\{2\rho-2\nu,1+\rho-\alpha-\nu\}}
  \end{split}
 \end{equation*}
which shows the claim.
\end{proof}

\begin{remark}\label{Rem:Q:asymptotics}
 Due to the normalisation of self-similar profiles according to their tail-behaviour at infinity, as given in Theorem~\ref{Thm:asymptotics}, it is well-known that this directly translates to the asymptotic behaviour at zero in Laplace variables. Precisely, it holds $Q'(q)\sim \rho^{2}q^{\rho-1}$ as $q\to 0$.
\end{remark}

We now continue, by applying the Laplace transform to~\eqref{eq:self:sim} which yields a differential equation for $Q$.

\begin{lemma}\label{Lem:equation:desingualarised}
 For any solution $f$ of~\eqref{eq:self:sim} the corresponding desingualarised Laplace transform $Q$ satisfies the equation
 \begin{equation}\label{eq:Q}
  -qQ'(q)=-\rho Q(q)+Q^2(q)+\eps\Pert(f,f)(q)
 \end{equation}
 pointwise for each $q>0$.
\end{lemma}

\begin{proof}
 The proof is a straightforward computation, i.e.\@ we first multiply~\eqref{eq:self:sim} by $\ee^{-qx}$ and integrate over $(0,\infty)$ to find
 \begin{multline*}
  \int_{0}^{\infty} x^2f(x)\ee^{-qx}\dx=(1-\rho)\int_{0}^{\infty}\ee^{-qx}\int_{0}^{x}yf(y)\dy\dx\\*
  +\int_{0}^{\infty}\ee^{-qx}\int_{0}^{x}\int_{x-y}^{\infty}yK(x-y)f(y)f(z)\dz\dy\dx.
 \end{multline*}
 Fubini's Theorem together with~\eqref{eq:des:Lap:trans:der} yields then after evaluating the $x$-integrals that
 \begin{equation*}
  -Q''(q)=\frac{1-\rho}{q}\int_{0}^{\infty}yf(y)\dy+\frac{1}{q}\int_{0}^{\infty}\int_{0}^{\infty}yK(y,z)f(y)f(z)\ee^{-qy}(1-\ee^{-qz})\dz\dy.
 \end{equation*}
 Taking the representation $K=2+\eps W$ into account it follows once more from~\eqref{eq:des:Lap:trans:der} that
 \begin{equation*}
  -qQ''(q)=(1-\rho)Q'(q)+2Q'(q)Q(q)+\eps\int_{0}^{\infty}\int_{0}^{\infty}yW(y,z)f(y)f(z)\ee^{-qy}(1-\ee^{-qz})\dz\dy.
 \end{equation*}
 From the symmetry of the kernel $W$ it follows
 \begin{equation*}
  \frac{\dd}{\dd{q}}\Pert(f,f)(q)=\int_{0}^{\infty}\int_{0}^{\infty}yW(y,z)f(y)f(z)\ee^{-qy}(1-\ee^{-qz})\dz\dy
 \end{equation*}
 which, together with $\del_{q}\bigl(qQ'(q)\bigr)=Q'(q)+qQ''(q)$ yields
 \begin{equation*}
  -\del_{q}\bigl(qQ'(q)\bigr)=-\rho Q'(q)+\del_{q}\bigl(Q^2(q)\bigr)+\eps\del_{q}\Pert(f,f)(q).
 \end{equation*}
Now we use that $Q(0)=0$ and $\Pert(f,f)(q)\to 0$ for $q\to 0$, according to Lemma~\ref{Lem:Pert:estimate:Laplace:var}, as well as $\lim_{q\to 0} qQ'(q)=0$, according to Remark~\ref{Rem:Q:asymptotics}, to integrate this equation over $[0,q]$. It then follows
\begin{equation*}
 -qQ'(q)=-\rho Q(q)+Q^2(q)+\eps \Pert(f,f)(q),
\end{equation*}
 which is the claimed relation.
\end{proof}

As a consequence of Lemma~\ref{Lem:equation:desingualarised} we can now derive several a priori estimates for self-similar profiles which are summarised in the following lemma.

\begin{lemma}\label{Lem:apriori}
 Let $f$ be a solution of~\eqref{eq:self:sim}. The corresponding desingularised Laplace transform $Q$ satisfies
 \begin{itemize}
  \item $Q(q)\leq \frac{\rho q^{\rho}}{1+q^{\rho}}\leq \min\{\rho,\rho q^{\rho}\}$,
  \item $\lim_{q\to\infty} Q(q)\leq \rho$ and thus in particular $\int_{0}^{\infty}f(x)\dx\leq \rho$,
  \item $\sup_{q>0}\abs{q^{1-\rho}Q'(q)}\leq\rho^2$ and $\sup_{q>0}\abs{qQ'(q)}\leq \rho^2$,
  \item $\int_{0}^{\infty}\int_{0}^{\infty}K(x,y)f(x)f(y)\dx\dy<\infty$,
  \item $\lim_{q\to \infty}\Pert(f,f)(q)=\frac{1}{2}\int_{0}^{\infty}\int_{0}^{\infty}W(x,y)f(x)f(y)\dx\dy<\infty$.
 \end{itemize}
\end{lemma}

\begin{remark}
 Note that these estimates are independent of the value of $\eps\geq 0$ and thus are also true for $\bar{Q}$, i.e.\@ for $\eps=0$.
\end{remark}

\begin{proof}[Proof of Lemma~\ref{Lem:apriori}]
 Plugging the explicit form of $\Pert(f,f)$ into~\eqref{eq:Q} the function $Q$ satisfies
 \begin{equation*}
  -qQ'(q)=-\rho Q(q)+Q^2(q)+\frac{\eps}{2}\int_{0}^{\infty}\int_{0}^{\infty}W(x,y)f(x)f(y)(1-\ee^{-qx})(1-\ee^{-qy})\dy\dx.
 \end{equation*}
 Since $\eps W(x,y)=K(x,y)-2$ and $Q^2(q)=\int_{0}^{\infty}\int_{0}^{\infty}f(x)f(y)(1-\ee^{-qx})(1-\ee^{-qy})\dy\dx$, this equation can be rewritten as
 \begin{equation}\label{eq:apriori:0}
  -\rho Q'(q)=-\rho Q(q)+\frac{1}{2}\int_{0}^{\infty}\int_{0}^{\infty}K(x,y)f(x)f(y)(1-\ee^{-qx})(1-\ee^{-qy})\dy\dx.
 \end{equation}
 Since $f$ is non-negative and $K$ satisfies the lower bound $K\geq 2$ we deduce from this that
 \begin{equation}\label{eq:apriori:1}
  -\rho Q'(q)\geq -\rho Q(q)+\int_{0}^{\infty}\int_{0}^{\infty}f(x)f(y)(1-\ee^{-qx})(1-\ee^{-qy})\dy\dx=-\rho Q(q)+Q^2(q).
 \end{equation}
 With the integrating factor $q^{-\rho-1}$ we rewrite this first as $-\del_{q}(q^{-\rho}Q(q))\geq q^{-\rho-1}Q^2(q)$. Exploiting the fact that $Q$ is strictly positive and strictly increasing we finally get
 \begin{equation*}
  \del_{q}\Bigl(\bigl(q^{-\rho}Q(q)\bigr)^{-1}\Bigr)\geq q^{\rho-1}.
 \end{equation*}
 Due to the normalisation of the self-similar profiles, an application of l'H{\^o}pital's rule yields $q^{\rho}Q(q)\to \rho$ for $q\to 0$ such that an integration over $[0,q]$ of the previous inequality gives
 \begin{equation*}
  \bigl(q^{-\rho}Q(q)\bigr)^{-1}-\frac{1}{\rho}\geq \frac{1}{\rho}q^{\rho}.
 \end{equation*}
 This can now be rearranged such that we end up with
 \begin{equation}\label{eq:apriori:2}
  Q(q)\leq \frac{\rho q^{\rho}}{1+q^{\rho}}\leq \min\{\rho,\rho q^{\rho}\}
 \end{equation}
 which gives the uniform boundedness of $Q$. Note, that $Q$ is bounded by zero from below due to the non-negativity of $f$. 
 
 Since $Q$ is also strictly increasing we additionally get that the limit $\lim_{q\to\infty} Q(q)$ exists and is finite which yields by monotone convergence that the integral of $f$ is finite, i.e.\@ $\int_{0}^{\infty}f(x)\dx<\infty$.
 
 From~\eqref{eq:apriori:1} together with $Q^2(q)\geq 0$ we further deduce that $Q'(q)\leq \frac{\rho}{q}Q(q)$. Since $Q'$ is non-negative the uniform bound~\eqref{eq:apriori:2} yields
 \begin{equation*}
  \sup_{q>0}\abs{q^{1-\rho}Q'(q)}\leq q^{1-\rho}\frac{\rho}{q}\rho q^{\rho}=2\rho^2\quad \text{and}\quad \sup_{q>0}\abs{qQ'(q)}\leq \rho^2.
 \end{equation*}
 Moreover, these estimates together with~\cref{eq:apriori:0,eq:apriori:2} and the non-negativity of $f$ yield
 \begin{equation}\label{eq:apriori:3}
  0\leq \int_{0}^{\infty}\int_{0}^{\infty}K(x,y)f(x)f(y)(1-\ee^{-qx})(1-\ee^{-qy})\dy\dx<C\quad \text{for all }q\geq 0.
 \end{equation}
 Monotone convergence then implies that
 \begin{equation*}
  \int_{0}^{\infty}\int_{0}^{\infty}K(x,y)f(x)f(y)\dx\dy\leq C.
 \end{equation*}
 Finally we note that the relation
 \begin{equation*}
  \Pert(f,f)(q)=\frac{1}{2}\int_{0}^{\infty}\int_{0}^{\infty}K(x,y)f(x)f(y)(1-\ee^{-qx})(1-\ee^{-qy})\dy\dx-Q^{2}(q)
 \end{equation*}
 together with~\cref{eq:apriori:2,eq:apriori:3} gives that $\Pert(f,f)$ is uniformly bounded and $\lim_{q\to\infty}\Pert(f,f)(q)$ exists.
\end{proof}

\begin{remark}\label{Rem:existence:Laplace:transform}
 By dominated convergence, Lemma~\ref{Lem:apriori} in particular shows that the Laplace transform $F=\T f$ exists on $[0,\infty)$ for any self-similar profile. 
\end{remark}

\subsection{Uniform convergence in Laplace variables}

In this section we will show on the one hand that the desingularised Laplace transform $Q$ of any solution $f$ of~\eqref{eq:self:sim} converges  uniformly to $\bar{Q}$ for $\eps\to 0$. Additionally, we will obtain an improvement of the moment estimate in Lemma~\ref{Lem:moment:est:1}.

As a first technical step we will show the locally uniform convergence of $Q$ to $\bar{Q}$.

\begin{lemma}\label{Lem:local:uniform:conergence}
 For any $\nu\in(0,\min\{\rho/2,\rho-\alpha\})$ there exists a constant $C_{\nu}>0$ such that it holds
 \begin{equation}\label{eq:Gronwall:bound}
  \abs*{Q(q)-\bar{Q}(q)}\leq \frac{C_{\nu}\eps}{\min\{\rho-2\nu,1-\alpha-\nu\}}q^{\min\{2\rho-2\nu,1+\rho-\alpha-\nu\}}\exp(2q^{\rho})\quad \text{for all }q\geq 0
 \end{equation}
 for the desingularised Laplace transform $Q$ of any solution $f$ of~\eqref{eq:self:sim}.
 
 In particular this implies that $Q$ converges locally uniformly to $\bar{Q}$ for $\eps\to 0$ which precisely means that for any compact set $D\subset[0,\infty)$  and $\delta>0$ there exists $\eps_{\delta}>0$ such that
 \begin{equation*}
  \sup_{q\in D}\abs{Q(q)-\bar{Q}(q)}\leq \delta \quad \text{for }\eps\leq \eps_{\delta}.
 \end{equation*}
\end{lemma}

\begin{proof}
 The claim follows essentially by an application of Gr{\"o}nwall's inequality. More precisely, we take the difference of the equations $-qQ'(q)=-\rho Q(q)+Q^2(q)+\eps \Pert(f,f)(q)$ and $-q\bar{Q}'(q)=-\rho \bar{Q}(q)+\bar{Q}^2(q)$ satisfied by $Q$ and $\bar{Q}$ and obtain after some elementary rearrangement that
 \begin{equation*}
  \frac{\dd}{\dd{q}}\Bigl(q^{-\rho}\bigl(Q(q)-\bar{Q}(q)\bigr)\Bigr)=q^{-\rho-1}\bigl(Q(q)-\bar{Q}(q)\bigr)\bigl(Q(q)+\bar{Q}(q)\bigr)-\eps q^{-\rho-1}\Pert(f,f)(q).
 \end{equation*}
 The normalisation of the profiles yields that $q^{-\rho}(Q(q)-\bar{Q}(q))\to 0$ for $q\to 0$ such that an integration over $[0,q]$ gives
 \begin{equation*}
  \frac{Q(q)-\bar{Q}(q)}{q^{\rho}}=-\int_{0}^{q}\frac{(Q(r)-\bar{Q}(r))(Q(r)+\bar{Q}(r))}{r^{1+\rho}}\dr-\eps\int_{0}^{q}\frac{\Pert(f,f)(r)}{r^{1+\rho}}\dr.
 \end{equation*}
 We now recall \cref{Lem:Pert:estimate:Laplace:var,Lem:apriori} which imply $Q(q)+\bar{Q}(q)\leq 2\rho q^{\rho}$ as well as $\Pert(f,f)(q)\leq C_{\nu}q^{\min\{2\rho-2\nu,1+\rho-\alpha-\nu\}}$ for every $\nu\in(0,\rho-\alpha)$. Thus, the previous integrals are well-defined and we obtain
 \begin{equation*}
  \frac{\abs{Q(q)-\bar{Q}(q)}}{q^{\rho}}\leq 2\rho\int_{0}^{q}\frac{\abs{Q(r)-\bar{Q}(r)}}{r^{\rho}}r^{\rho-1}\dr+C_{\nu}\eps\int_{0}^{q}r^{\min\{\rho-2\nu-1,-\alpha-\nu\}}\dr.
 \end{equation*}
 For $0<\nu<\min\{\rho/2,1-\alpha\}$ we obtain $\int_{0}^{q}r^{\min\{\rho-2\nu-1,-\alpha-\nu\}}\dr\leq C_{\nu,\rho,\alpha}q^{\min\{\rho-2\nu,1-\alpha-\nu\}}$ such that an application of Gr{\"o}nwall's inequality finally yields the desired estimate~\eqref{eq:Gronwall:bound}.
\end{proof}

The following lemma gives a bound on the decay of the Laplace transform $(\T f)(q)$ of a self-similar profile $f$ for large values of $q$. 

\begin{lemma}\label{Lem:improvement:zero:Laplace}
 For any $\nu>0$ there exist constants $q_{\nu},C_{\nu}>0$ such that it holds
 \begin{equation*}
  (\T f)(q)\leq C_{\nu}q^{\nu-\rho}\quad \text{if } q\geq q_{\nu}
 \end{equation*}
 for each solution $f$ of~\eqref{eq:self:sim}. Moreover, we obtain that for each $\delta>0$ there exists $\eps_{\delta}>0$ such that
 \begin{equation*}
  \abs*{(\T f)(0)-(\T \bar{f})(0)}\leq \delta\quad \text{for }\eps\leq \eps_{\delta}.
 \end{equation*}
 In terms of the desingualarised Laplace transform the latter result can equivalently be written as $\abs*{Q(\infty)-\bar{Q}(\infty)}\leq \delta$ for $\eps\leq \eps_{\delta}$.
\end{lemma}

\begin{proof}
 We first note that the Laplace transform $\T f$ exists for any self-similar profile according to Remark~\ref{Rem:existence:Laplace:transform} and we have the relation $(\T f)(q)=Q(\infty)-Q(q)$. The proof is divided in two main steps, namely we first show that $(\T f)(q)$ becomes uniformly small for large values of $q$, while in the second step, we derive a differential inequality for $F$ which we integrate explicitly to obtain the desired bound.
 
 \paragraph{Step 1:}
 For $\delta>0$ given, we fix a constant $q_{\delta}>0$ such that $\abs*{\bar{Q}(\infty)-\bar{Q}(q)}\leq \delta/2$ for $q\geq q_{\delta}$. Moreover, we choose $\eps_{\delta}>0$ according to Lemma~\ref{Lem:local:uniform:conergence} such that $\abs*{\bar{Q}(q_{\delta})-Q(q_{\delta})}\leq \delta/2$ for $\eps\leq \eps_{\delta}$. For $q\geq q_{\delta}$ the monotonicity of $Q$ and Lemma~\ref{Lem:apriori} imply that $Q(q_{\delta})\leq Q(q)\leq Q(\infty)\leq \bar{Q}(\infty)=\rho$ such that we obtain for $q\geq q_{\delta}$ and $\eps\leq \eps_{\delta}$ that
 \begin{equation}\label{eq:sandwich:argument}
  \abs*{(\T f)(q)}=\abs*{Q(\infty)-Q(q)}\leq \abs{\bar{Q}(\infty)-Q(q_{\delta})}\leq \abs{\bar{Q}(\infty)-\bar{Q}(q_{\delta})}+\abs{\bar{Q}(q_{\delta})-Q(q_{\delta})}\leq \delta.
 \end{equation}
This yields in particular the estimate
\begin{equation}\label{eq:desing:Lap:trans:clos:infinity}
 \abs{Q(\infty)-\bar{Q}(\infty)}\leq \delta \quad \text{for }\eps\leq \eps_{\delta},
\end{equation}
i.e.\@ the second part of the statement.

\paragraph{Step 2:}

In order to derive a differential inequality for $(\T f)$, we rewrite~\eqref{eq:Q} in terms of $(\T f)$ by substituting $Q(q)=(\T f)(0)-(\T f)(q)=Q(\infty)-(\T f)(q)$, which, after some rearrangement, leads to 
\begin{multline}\label{eq:diff:ODE:F:1}
 q(\T f)'(q)=(\T f)^2(q)+\bigl(\rho-2Q(\infty)\bigr)(\T f)(q)+\eps\bigl(\Pert(f,f)(q)-\Pert(f,f)(\infty)\bigr)\\*
 +\bigl[-\rho Q(\infty)+Q^2(\infty)+\eps\Pert(f,f)(\infty)\bigr].
\end{multline}
We now note that due to Lebesgue's Theorem it holds $qQ'(q)=\int_{0}^{\infty}qx\ee^{-qx}f(x)\to 0$ for $q\to\infty$ such that we can pass to the limit $q\to\infty$ in~\eqref{eq:Q} to find $-\rho Q(\infty)+Q^2(\infty)+\eps\Pert(f,f)(\infty)=0$, i.e.\@ the term in brackets on the right-hand side of~\eqref{eq:diff:ODE:F:1} vanishes. Moreover, since $f$ is non-negative, $\Pert(f,f)(\cdot)$ is monotonously increasing such that it holds $\Pert(f,f)(q)-\Pert(f,f)(\infty)\leq 0$. Thus, together it follows from~\eqref{eq:diff:ODE:F:1} that
\begin{equation*}
 q(\T f)'(q)\leq \bigl(\rho-2Q(\infty)+(\T f)(q)\bigr)(\T f)(q).
\end{equation*}
To estimate this further, we note that $\bar{Q}(\infty)=\rho$ such that it follows from~\eqref{eq:desing:Lap:trans:clos:infinity} for $q\geq q_{\delta}$ and $\eps<\eps_{\delta}$ that $\rho-2Q(\infty)=2(\bar{Q}(\infty)-Q(\infty))-\rho\leq 2\delta-\rho$. Together with~\eqref{eq:sandwich:argument} we thus obtain under the same conditions on $q$ and $\eps$ that
\begin{equation*}
 q(\T f)'(q)\leq (3\delta-\rho)(\T f)(q).
\end{equation*}
This inequality however can be integrated readily over $[q_{\delta},q]$ and one obtains after some rearrangement that $(\T f)(q)\leq q_{\delta}^{\rho-3\delta}(\T f)(q_{\delta})q^{3\delta-\rho}$. From~\eqref{eq:sandwich:argument} it follows that $q_{\delta}^{\rho-3\delta}(\T f)(q_{\delta})\leq \delta q_{\delta}^{\rho-3\delta}\vcc=C_{\delta}$ such that we finally have
\begin{equation*}
 (\T f)(q)\leq C_{\delta}q^{3\delta-\rho}\quad\text{for }q\geq q_{\delta}\text{ and }\eps\leq \eps_{\delta}.
\end{equation*}
To finish the proof it then suffices to choose $\delta=\nu/3$.
\end{proof}

We now use the previous result to show that in the limit $\eps\to 0$ self-similar profiles behave at least in average not worse than $x^{\rho-1}$ close to zero.

\begin{lemma}\label{Lem:improved:reg:zero}
 For every $\nu>0$ there exist constants $r_{\nu}>0$, $\eps_{\nu}>0$ and $C_{\nu}>0$ such that it holds
 \begin{equation*}
  \int_{r}^{2r}f(x)\dx\leq C_{\nu}r^{\rho-\nu}\quad \text{for all }r\in(0,r_{\nu}]
 \end{equation*}
 and each solution $f$ of~\eqref{eq:self:sim}, provided $\eps\leq \eps_{\nu}$.
\end{lemma}

\begin{proof}
 We use the notation of Lemma~\ref{Lem:improvement:zero:Laplace} and define $r_{\nu}\vcc=1/q_{\nu}$. Then Lemma~\ref{Lem:improvement:zero:Laplace} implies for $q=1/r$ with $r\leq r_{\nu}$ that
 \begin{equation*}
  \int_{r}^{2r}f(x)\dx=\int_{r}^{2r}f(x)\ee^{-\frac{x}{r}}\ee^{\frac{x}{r}}\dx\leq \ee^{2}\int_{0}^{\infty}f(x)\ee^{-\frac{x}{r}}\dx\leq C_{\nu}r^{\rho-\nu}
 \end{equation*}
 which finishes the proof.
\end{proof}

\begin{remark}\label{Rem:regularity:improvement}
 Since we have already seen in Lemma~\ref{Lem:apriori} that $\int_{0}^{\infty}f(x)\dx<C$, we can always choose $r_{\nu}=1$ in the previous statement, if we also enlarge the constant $C_{\nu}$ if necessary.
\end{remark}

We are now prepared to give an improved moment estimate for self-similar profiles which will turn out to be quite useful.

\begin{lemma}\label{Lem:moment:est:2}
 Let $\gamma\in(-\rho,\rho)$. For sufficiently small $\eps>0$ there exists a constant $C>0$ such that it holds
 \begin{equation*}
  \int_{0}^{\infty}x^{\gamma}f(x)\dx<C
 \end{equation*}
 for every solution $f$ of~\eqref{eq:self:sim}.
\end{lemma}

\begin{proof}
 The proof essentially relies on a dyadic argument together with Lemma~\ref{Lem:improved:reg:zero} which has been used frequently in the context of self-similarity for Smoluchowski's equation (see for example~\cite{NTV16a,NiV13a}). Nevertheless, we sketch the argument here once for completeness. 
 
 We fix $\nu<\rho+\gamma$ and split the integral to consider as
 \begin{equation*}
  \int_{0}^{\infty}x^{\gamma}f(x)\dx=\int_{0}^{1}(\cdots)\dx+\int_{1}^{\infty}(\cdots)\dx.
 \end{equation*}
Together with \cref{Lem:improved:reg:zero,Rem:regularity:improvement,Thm:asymptotics} a dyadic decomposition then yields
\begin{equation*}
 \int_{0}^{1}x^{\gamma}f(x)\dx=\sum_{\ell=0}^{\infty}\int_{2^{-(\ell+1}}^{2^{-\ell}}x^{\gamma}f(x)\dx\leq C_{\nu}\sum_{\ell=0}^{\infty}2^{-\ell(\rho+\gamma-\nu)}\leq C(\gamma,\nu,\rho)
\end{equation*}
and
\begin{equation*}
 \int_{1}^{\infty}x^{\gamma}f(x)\dx=\sum_{\ell=}^{\infty}\int_{2^{\ell}}^{2^{\ell+1}}x^{\gamma-1}xf(x)\dx\leq C\sum_{\ell=0}^{\infty}2^{\ell(\gamma-\rho)}\leq C(\gamma,\rho).
\end{equation*}
Note that the constants are independent of the specific solution $f$ due to the normalisation condition.
\end{proof}

We can now show that the Laplace transform $\T f$ of any self-similar profile converges uniformly to $\bar{F}=\T\bar{f}$.

\begin{proposition}\label{Prop:uniform:convergence:Laplace:transform}
 For each $\delta>0$ it holds for sufficiently small $\eps>0$ that
 \begin{equation*}
  \sup_{q>0}\abs*{(\T f)(q)-(\T \bar{f})(q)}\leq \delta
 \end{equation*}
 for all solutions $f$ of~\eqref{eq:self:sim}.
\end{proposition}

\begin{proof}
 We recall from the proof of Lemma~\ref{Lem:improvement:zero:Laplace} that for each $\delta>0$ there exist constants $q_{\delta}>0$ and $\eps_{\delta}>0$ such that 
 \begin{equation*}
  \abs*{Q(q)-Q(\infty)}\leq \delta,\quad \abs{\bar{Q}(q)-\bar{Q}(\infty)}\leq \delta\quad \text{and}\quad \abs{\bar{Q}(\infty)-Q(\infty)}\leq \delta\quad \text{if }q\geq q_{\delta}\text{ and }\eps\leq \eps_{\delta}.
 \end{equation*}
 Thus, the triangle inequality yields
 \begin{equation*}
  \abs{Q(q)-\bar{Q}(q)}\leq 3\delta\quad \text{for } q\geq q_{\delta}\text{ and }\eps\leq \eps_{\delta}.
 \end{equation*}
 By choosing $\eps_{\delta}$ maybe even smaller, we obtain additionally from Lemma~\ref{Lem:local:uniform:conergence} that
 \begin{equation*}
  \sup_{q\leq q_{\delta}}\abs{Q(q)-\bar{Q}(q)}\leq \delta\quad \text{for }\eps\leq \eps_{\delta}.
 \end{equation*}
Together, this yields that $\sup_{q>0}\abs{Q(q)-\bar{Q}(q)}\leq 3\delta$ for $\eps\leq \eps_{\delta}$. Since $(\T f)(q)=(\T f)(0)-Q(q)$, we can conclude together with Lemma~\ref{Lem:improvement:zero:Laplace} that
\begin{equation*}
 \sup_{q>0}\abs*{(\T f)(q)-(\T \bar{f})(q)}\leq \abs*{(\T f)(0)-(\T \bar{f})(0)}+\sup_{q>0}\abs*{Q(q)-\bar{Q}(q)}\leq 4\delta.
\end{equation*}
If we replace $\delta$ by $\delta/4$ in the proof above, the claim follows.
\end{proof}

\subsection{Proof of \cref{Prop:norm:boundedness,Prop:closeness:two:norm}}

In this section we will further improve the convergence result of self-similar profiles that we obtained in the previous section by showing that we in fact have convergence with respect to the norm $\lfnorm*{2}{\mu}{\theta}{\cdot}$. This will be done iteratively, starting with $\lfnorm*{0}{\mu}{\theta}{\cdot}$ and then extending to the derivatives of order one and two. Additionally, we will obtain the boundedness of self-similar profiles in the corresponding norms with parameter $\mu=0$. 

As a preliminary step, let us first show the following lemma which provides an estimate on the difference $(\T f)(q)-(\T\bar{f})(q)$ and the corresponding derivatives for large values of $q$.

\begin{lemma}\label{Lem:help:norm:convergence}
For every $\delta>0$ there exists a constant $q_{*}=q_{*}(\delta)>0$ such that it holds for sufficiently small $\eps$ that
\begin{equation*}
 \sup_{q\geq q_{*}}\Bigl(q^{k+\theta}\abs*{\del_{q}^{k}(\T f)(q)-\del_{q}^{k}(\T\bar{f})(q)}\Bigr)\leq \delta \quad \text{for }k=0,1,2
\end{equation*}
and all solutions $f$ of~\eqref{eq:self:sim}.
\end{lemma}

\begin{proof}
 From Lemma~\ref{Lem:improvement:zero:Laplace} we know that for any $\nu>0$ and $\eps$ sufficiently small it holds
\begin{equation}\label{eq:help:norm:convergence:1}
 (\T f)(q)\leq C_{\nu}q^{\nu-\rho}\quad \text{for }q\geq q_{\nu}.
\end{equation}
Moreover, we have $(\T\bar{f})(q)=\rho(1+q^{\rho})^{-1}$ which also yields $(\T \bar{f})(q)\leq q^{-\rho}$ for all $q>0$. From Lemma~\ref{Lem:non:neg:measures} we thus infer for sufficiently small $\eps>0$ that 
\begin{equation*}
 \abs*{\del_{q}^{k}(\T f)(q)}\leq C_{\nu}q^{\nu-\rho-k}\quad \text{and}\quad \abs*{\del_{q}^{k}(\T \bar{f})(q)}\leq Cq^{-\rho-k}\quad \text{for }q\geq q_{\nu} \text{ and }k=0,1,2.
\end{equation*}
Since $\theta<\rho$, it follows that for a given constant $\delta>0$ it holds
\begin{equation*}
 \abs*{\del_{q}^{k}(\T f)(q)-\del_{q}^{k}(\T \bar{f})(q)}\leq C_{\nu}q^{\nu-\rho-k}\leq \delta q^{-\theta-k}\quad \text{for } q\geq q_{\nu,\delta}
\end{equation*}
provided that we first choose $\nu>0$ sufficiently small, i.e.\@ $\nu\in(0,\rho-\theta)$ and then $q_{\nu,\delta}$ sufficiently large, i.e.\@ such that $q_{\nu,\delta}^{\nu+\theta-\rho}\leq \delta$. The claim then holds with $q_{*}=q_{\nu,\delta}$.
\end{proof}

\begin{lemma}\label{Lem:convergence:zero:norm}
 For every $\delta>0$ there exists a constant $\eps_{\delta}>0$ such that it holds
 \begin{equation*}
  \lfnorm*{0}{-\rho}{\theta}{\T f-\T \bar{f}}\leq \delta
 \end{equation*}
for every solution $f$ of~\eqref{eq:self:sim} provided $\eps\leq \eps_{\delta}$.
\end{lemma}

\begin{proof}
 From the definition of $\lfnorm*{0}{-\rho}{\theta}{\cdot}$ we have to show that
 \begin{equation*}
  \sup_{q>0}(1+q)^{\theta}\abs*{(\T f)(q)-(\T\bar{f})(q)}\leq \delta\quad \text{for } \eps\leq \eps_{\delta}.
 \end{equation*}
 To see this, we recall from Lemma~\ref{Lem:help:norm:convergence} that we can fix a constant $q_{*}>0$ sufficiently large such that it holds
 \begin{equation*}
  \sup_{q>q_{*}}(1+q)^{\theta}\abs*{(\T f)(q)-(\T \bar{f})(q)}\leq \delta/2 \quad \text{for }\eps \text{ sufficiently small.}
 \end{equation*}
 On the other hand, for $q\leq q_{*}$ and $\eps$ maybe even smaller it follows from Proposition~\ref{Prop:uniform:convergence:Laplace:transform} that 
 \begin{equation*}
  \sup_{q\leq q_{*}}(1+q)^{\theta}\abs*{(\T f)(q)-(\T \bar{f})(q)}\leq (1+q_{*})^{\theta}\sup_{q>0}\abs*{(\T f)(q)-(\T \bar{f})(q)}\leq \delta/2.
 \end{equation*}
 Taking both estimates together the claim follows.
\end{proof}

We are now already prepared to prove the boundedness of $\T f$ with respect to $\lfnorm*{2}{0}{\theta}{\cdot}$ for self-similar profiles $f$.

\begin{proof}[Proof of Proposition~\ref{Prop:norm:boundedness}]
 We recall from Remark~\ref{Rem:explicit:profiles} that $\lfnorm*{2}{0}{\theta}{\T\bar{f}}\leq C$ and $\T\bar{f}\in\X{2}{0}{\theta}$. From Lemma~\ref{Lem:convergence:zero:norm} we thus deduce that
 \begin{equation}\label{eq:norm:boundedness:1}
  \lfnorm*{0}{-\rho}{\theta}{\T f}\leq \lfnorm*{0}{-\rho}{\theta}{\T f-\T \bar{f}}+\lfnorm*{0}{-\rho}{\theta}{\T\bar{f}}\leq C \quad \text{for }\eps\text{ sufficiently small.}
 \end{equation}
 If we denote again by $Q$ the desingularised Laplace transform of $f$ it holds $(\T f)'(q)=-Q'(q)$ such that we infer from Lemma~\ref{Lem:apriori} that 
 \begin{equation}\label{eq:norm:boundedness:2}
  \sup_{0<q\leq 1}(1+q)^{\theta+\rho}q^{1-\rho}\abs*{(\T f)'(q)}\leq 2^{\theta+\rho}\sup_{0<q\leq 1}\abs*{q^{1-\rho}Q'(q)}\leq 2^{\theta+\rho}\rho^2.
 \end{equation}
 Since $f$ is non-negative, we note that Remark~\ref{Rem:est:der:by:zero:norm} gives
 \begin{equation}\label{eq:norm:boundedness:3}
  \sup_{q\geq 1}(1+q)^{\theta}q^{1-\rho}\abs*{(\T f)'(q)}\leq 2^{\rho}\sup_{q\geq 1}q(1+q)^{\theta}\abs*{(\T f)'(q)}\leq 2^{\rho+\theta}\lfnorm*{0}{-\rho}{\theta}{\T f}.
 \end{equation}
 Summarising~\cref{eq:norm:boundedness:1,eq:norm:boundedness:2,eq:norm:boundedness:3} it follows for sufficiently small $\eps>0$ that $\lfnorm*{1}{0}{\theta}{\T f}\leq C$ with a uniform constant $C>0$. The claimed statement follows then finally from Lemma~\ref{Lem:norm:non:neg:meas} since $f$ is non-negative.
 \end{proof}

 \begin{remark}
  As an immediate consequence of Proposition~\ref{Prop:norm:boundedness} we obtain that $\lfnorm*{2}{0}{\theta}{\T f-\T\bar{f}}$ is also uniformly bounded for any solution $f$ of~\eqref{eq:self:sim}.
 \end{remark}
 
 \begin{lemma}\label{Lem:convergence:one:norm}
  For each $\delta>0$ and $\mu\in[0,\mu_{*})$ there exists a constant $\eps_{\delta,\mu}>0$ such that 
  \begin{equation*}
   \lfnorm*{1}{\mu}{\theta}{\T f-\T\bar{f}}\leq \delta
  \end{equation*}
 holds for all solutions $f$ of~\eqref{eq:self:sim} provided that $\eps\leq \eps_{\delta,\mu}$.
 \end{lemma}

 \begin{proof}
  In view of Lemma~\ref{Lem:convergence:zero:norm} it only remains to show that $\lsnorm*{1}{\mu}{\theta}{\T f-\T \bar{f}}\leq \delta$ for $\eps\leq\eps_{\delta,\mu}$. As in the proof of Lemma~\ref{Lem:convergence:zero:norm} we rely on Lemma~\ref{Lem:help:norm:convergence} and fix a constant $q_{*}=q_{*}(\delta)>0$ such that it holds
  \begin{equation}\label{eq:convergence:one:norm:1}
   (1+q)^{\theta+\rho+\mu}q^{1-\rho-\mu}\abs*{(\T f)(q)-(\T \bar{f})(q)}\leq \delta \quad \text{for all }q\geq q_{*}.
  \end{equation}
  It remains to estimate the difference $(\T f)(q)-(\T \bar{f})(q)$ for values $q\leq q_{*}$ for which we rely again on~\eqref{eq:Q}. Therefore let $Q$ and $\bar{Q}$ denote again the desingualarised Laplace transforms of $f$ and $\bar{f}$. Due to the relations $Q'=-(\T f)'$ and $\bar{Q}'=-(\T\bar{f})'$ we recall from~\eqref{eq:Q} that
  \begin{equation*}
   \begin{split}
    (\T f)'(q)&=-\frac{\rho}{q}Q(q)+\frac{1}{q}Q^2(q)+\frac{\eps}{q}\Pert(f,f)(q)\\
    (\T\bar{f})'(q)&=-\frac{\rho}{q}\bar{Q}(q)+\frac{1}{q}\bar{Q}^2(q).
   \end{split}
  \end{equation*}
 We take the difference of these two equation and rearrange to find
 \begin{equation*}
  (\T f)'(q)-(\T\bar{f})'(q)=\frac{1}{q}\bigl(Q(q)-\bar{Q}(q)\bigr)\bigl(Q(q)+\bar{Q}(q)-\rho\bigr)+\frac{\eps}{q}\Pert(f,f)(q).
 \end{equation*}
 Due to Lemma~\ref{Lem:apriori} the expressions $Q$ and $\bar{Q}$ are uniformly bounded. If we use this, fix some $\nu\in(0,\rho-\alpha)$ and recall \cref{Lem:Pert:estimate:Laplace:var,Lem:local:uniform:conergence} it follows
 \begin{equation*}
  \abs*{(\T f)'(q)-(\T\bar{f})'(q)}\leq \frac{C_{\nu}(q_{*})\eps}{q}q^{\min\{2\rho-2\nu,1+\rho-\alpha-\nu\}}\quad \text{for all }q\in(0,q_{*}). 
 \end{equation*}
 Thus, if we fix $\nu$ sufficiently small depending on $\mu\in(0,\mu_{*})$ it follows for sufficiently small $\eps>0$ that
 \begin{equation*}
  (1+q)^{\theta+\rho+\mu}q^{1-\rho-\mu}\abs*{(\T f)'(q)-(\T\bar{f})'(q)}\leq \delta\quad \text{for }q\leq q_{*}.
 \end{equation*}
  Combining this with~\eqref{eq:convergence:one:norm:1} the claim follows in view of~\eqref{eq:est:weight}.
 \end{proof}

 We are now prepared to prove that $\T f$ is close to $\T \bar{f}$ for each self-similar profile $f$ with respect to $\lfnorm*{2}{\mu}{\theta}{\cdot}$.

\begin{proof}[Proof of Proposition~\ref{Prop:closeness:two:norm}]
 To simplify the notation, we define $m\vcc=f-\bar{f}$ and note that in view of Lemma~\ref{Lem:convergence:one:norm} it suffices to prove
 \begin{equation*}
  \lsnorm*{2}{\mu}{\theta}{\T m}\leq \delta \quad \text{for }\eps \text{ sufficiently small.}
 \end{equation*}
Due to Lemma~\ref{Lem:help:norm:convergence} we can fix a constant $q_{*}=q_{*}(\delta)>0$ such that
\begin{equation}\label{eq:closenenss:two:norm:1}
 (1+q)^{\theta+\rho+\mu}q^{2-\rho-\mu}\abs*{\del_{q}^{2}(\T m)(q)}\leq \delta\quad \text{for }q\geq q_{*}.
\end{equation}
To treat the region $q<q_{*}$ we note that in terms of $m$ we can rewrite~\eqref{eq:self:sim:Lap} as
\begin{multline*}
 (\T m)(q)=\mathcal{A}_{\rho}(\T m)(q)+\mathcal{B}_{2}(\T \bar{f},\T m)(q)+\mathcal{B}_{2}(\T m,\T\bar{f})\\*
 +\mathcal{B}_{2}(\T m,\T m)(q)+\eps \mathcal{B}_{W}(\T f,\T f)(q).
\end{multline*}
Note that we used additionally that $\T\bar{f}$ solves~\eqref{eq:self:sim:Lap} with $\eps=0$. We then differentiate this equation twice and recall \cref{Lem:est:Arho,Lem:est:B2,Prop:est:BW} as well as~\cref{eq:est:weight,eq:weight:parmono} to get for $q<q_{*}$ that
\begin{multline*}
  \abs*{\del_{q}^{2}(\T m)(q)}\leq \frac{C \lsnorm*{1}{\mu}{\theta}{\T m}}{q^{2-\rho-\mu}(1+q)^{\theta+\rho+\mu}}\\*
  +C\frac{(\lfnorm{1}{0}{\theta}{\T\bar{f}}+\lfnorm{1}{0}{\theta}{\T m})\lfnorm*{1}{0}{\theta}{\T m}}{q^{2-\rho-\mu}(1+q)^{\theta+\rho+\mu}}+C_{\mu}\eps\frac{\lfnorm*{2}{0}{\theta}{\T f}^2}{q^{2-\rho-\mu}(1+q)^{\theta-\alpha+\rho-\mu}}.
\end{multline*}
Due to Lemma~\ref{Lem:comparison:norms} and $q<q_{*}$ we further find
\begin{multline*}
 \abs*{\del_{q}^{2}(\T m)(q)}\leq C_{\mu}\bigl(1+\lfnorm{1}{0}{\theta}{\T\bar{f}}+\lfnorm*{1}{0}{\theta}{\T m}\bigr)\lfnorm*{1}{\mu}{\theta}{\T m}\frac{1}{q^{2-\rho-\mu}(1+q)^{\theta+\rho+\mu}}\\*
 +C_{\mu}(q_{*})\eps \lfnorm*{2}{0}{\theta}{\T f}^{2}\frac{1}{q^{2-\rho-\mu}(1+q)^{\theta+\rho+\mu}}.
\end{multline*}
The uniform boundedness of $\lfnorm*{1}{0}{\theta}{\T\bar{f}}$, $\lfnorm*{1}{0}{\theta}{\T m}$ and $\lfnorm*{2}{0}{\theta}{\T f}$ provided by \cref{Prop:norm:boundedness,Rem:explicit:profiles} yields in combination with Lemma~\ref{Lem:convergence:one:norm} that
\begin{equation*}
 \sup_{0<q<q_{*}}\Bigl((1+q)^{\theta+\rho+\mu}q^{2-\rho-\mu}\abs*{\del_{q}^{2}(\T m)(q)}\Bigr)\leq C_{\mu}\lfnorm*{1}{\mu}{\theta}{\T m}+C_{\mu}(q_{*})\eps\leq \delta,
\end{equation*}
provided that we choose $\eps>0$ sufficiently small. Combining this estimate with~\eqref{eq:closenenss:two:norm:1} the claim follows.
\end{proof}

\section{The boundary layer estimate}\label{Sec:bd:layer}

In this section, we will give the proof of Proposition~\ref{Prop:boundary:layer}. However, as already indicated before, this proof is rather technical and will essentially cover the rest of this work since the proofs of certain auxiliary results (e.g.\@ \cref{Prop:rep:H0,Prop:Q0:int:est}) have to be postponed to \cref{Sec:H0:representation,Sec:Q0:est,Sec:asymptotics}.

\subsection{Boundary layer equation}

Before we start with the proof itself let us first introduce some notation, while we mainly keep the notation already used in~\cite{NTV15,Thr16}. Precisely, we will consider two solutions $f_1$ and $f_2$ of~\eqref{eq:self:sim} and we define
\begin{equation}\label{eq:def:beta:and:R}
 \beta_{K}(y,f_{j})\vcc=\int_{0}^{\infty}K(y,z)f_{j}(z)\dz\quad \text{and} \quad R_{j}(x)\vcc =\int_{0}^{x}\int_{0}^{x-y}K(y,z)yf_{j}(y)f_{j}(z)\dz\dy.
\end{equation}
This enables us to rewrite~\eqref{eq:self:sim} as
\begin{equation}\label{eq:self:sim:bd:layer}
 x^2f_{j}(x)=(1-\rho)\int_{0}^{x}yf_{j}(y)\dy+\int_{0}^{x}\beta_{K}(y,f_{j})yf_{j}(y)\dy-R_{j}(x).
\end{equation}
Moreover, we introduce the exponent
\begin{equation}\label{eq:def:kappa}
 \kappa_{j}\vcc=\beta_{2}(f_{j})-2\rho
\end{equation}
that will appear frequently in the following calculations. Note that $\beta_{2}(f_{j})$ is a constant, i.e.\@ independent of $x$ such that we can omit the 'spatial' dependence.
\begin{remark}\label{Rem:kappa:smallness}
 We note that it holds $\kappa_{j}=2\T (f_{j}-\bar{f})(0)$. Proposition~\ref{Prop:closeness:two:norm} thus implies
 \begin{equation*}
  \abs*{\kappa_{j}}\leq 2\lfnorm{0}{-\rho}{0}{\T (f_{j}-\bar{f})}\longrightarrow 0\quad \text{for }\eps\longrightarrow 0,
 \end{equation*}
 i.e.\@ we can make the absolute value of $\kappa_{j}$ as small as we need provided we choose $\eps>0$ sufficiently small.
\end{remark}
Finally, since in the following, we have to consider frequently the difference $f_1-f_2$, we also define
\begin{equation*}
 \Delta f\vcc=f_1-f_2.
\end{equation*}
We now rewrite~\eqref{eq:self:sim} by differentiating the equation and plugging in $K=2+\eps W$, which leads to
\begin{equation*}
 \del_{x}\bigl(x^2 f_{j}(x)\bigr)=\frac{1-\rho}{x}x^2f_{j}(x)+\frac{\beta_{2}(f_{j})}{x}x^2f_{j}(x)+\eps\frac{\beta_{W}(x,f_{j})}{x}x^2 f_{j}(x)-\del_{x} R_{j}(x).
\end{equation*}
Next, we use the splitting $1=\ee^{-x}+1-\ee^{-x}$ in the term containing $\eps$ and also rewrite the first two terms on the right-hand side by means of~\eqref{eq:def:kappa} to obtain
\begin{multline}\label{eq:self:sim:bd:layer:1}
 \del_{x}\bigl(x^2 f_{j}(x)\bigr)=\frac{1+\rho+\kappa_{j}}{x}x^2f_{j}(x)+\eps\frac{\beta_{W}(x,f_{j})}{x}x^2 f_{j}(x)\ee^{-x}\\*
 +\eps\beta_{W}(x,f_{j})x f_{j}(x)(1-\ee^{-x})-\del_{x} R_{j}(x).
\end{multline}
It turns out to be convenient to introduce moreover the function
\begin{equation}\label{eq:def:Phi}
 \Phi(x,f_{j})\vcc=\eps\int_{x}^{\infty}\frac{\beta_{W}(y,f_{j})}{y}\ee^{-y}\dy\quad \text{with the abbreviation }\Phi_{j}(\cdot)\vcc=\Phi(\cdot,f_{j}).
\end{equation}
\begin{remark}\label{Rem:properties:beta:and:Phi}
 One immediately checks that it holds
 \begin{equation*}
  \abs*{\beta_{W}(x,f_{j})}\leq C\weight{-\alpha}{\alpha}(x)\quad \text{and}\quad \abs*{\Phi(x,f_{j})}\leq C\weight{-\alpha}{1-\alpha}(x)
 \end{equation*}
 for each solution $f_{j}$ of~\eqref{eq:self:sim} with uniform constants.
\end{remark}
We note that \cref{Rem:kappa:smallness,Rem:properties:beta:and:Phi} together with Theorem~\ref{Thm:asymptotics} imply for sufficiently small $\eps>0$ that $\lim_{x\to \infty}(x^{1+\rho-\kappa_{j}}\exp(\Phi(x,f_{j}))x^2f_{j}(x))=0$. Thus, by means of the integrating factor $(x^{-(1+\rho+\kappa_{j})}\exp(\Phi(x,f_{j})))$, we can integrate~\eqref{eq:self:sim:bd:layer:1} to get
\begin{multline*}
 x^{-1-\rho-\kappa_{j}}\exp(\Phi(x,f_{j}))x^2f_{j}(x)=-\int_{x}^{\infty}\eps \beta_{W}(z,f_{j})(1-\ee^{-z})f_{j}(z)z^{-\rho-\kappa_{j}}\exp(\Phi(z,f_{j}))\dz\\*
 +\int_{x}^{\infty}z^{-(1+\rho+\kappa_{j})}\exp(\Phi(z,f_{j}))\del_{z}R_{j}(z)\dz.
\end{multline*}
Since $\del_{z}R_{j}(z)=\int_{0}^{x}K(y,z-y)yf_{j}(y)f_{j}(z-y)\dy$ it follows after some rearrangement that 
\begin{multline}\label{eq:self:sim:bd:layer:intermediate}
  x^{2}f_{j}(x)=-\eps x^{1+\rho}\int_{x}^{\infty}\Bigl(\frac{x}{z}\Bigr)^{\kappa_{j}}\exp\bigl(\Phi(z,f_{j})-\Phi(x,f_{j})\bigr)\beta_{W}(z,f_{j})\frac{1-\ee^{-z}}{z^{\rho}}f_{j}(z)\dz\\*
  +x^{1+\rho}\int_{x}^{\infty}\Bigl(\frac{x}{z}\Bigr)^{\kappa_{j}}\frac{1}{z^{1+\rho}}\exp\bigl(\Phi(z,f_{j})-\Phi(x,f_{j})\bigr)\int_{0}^{z}K(y,z-y)yf_{j}(y)f_{j}(z-y)\dy\dz.
 \end{multline}
Multiplying the equation by $\zeta(x)\ee^{-qx}$ and integrating over $(0,\infty)$ we obtain together with Fubini's Theorem that
\begin{equation*}
 \begin{split}
  \bigl(\T (\zeta f_{j})\bigr)''(q)=&-\eps \int_{0}^{\infty}\int_{0}^{z}x^{1+\rho}\ee^{-(q+1)x}\Bigl(\frac{x}{z}\Bigr)^{\kappa_{j}}\frac{\ee^{\Phi_{j}(z)}}{\ee^{\Phi_{j}(x)}}\beta_{W}(z,f_{j})\frac{1-\ee^{-z}}{z^{\rho}}f_{j}(z)\dx\dz\\
  &+\int_{0}^{\infty}\int_{y}^{\infty}\int_{0}^{z}\frac{x^{1+\rho}\ee^{-(q+1)x}}{z^{1+\rho}}\Bigl(\frac{x}{z}\Bigr)^{\kappa_{j}}\frac{\ee^{\Phi_{j}(z)}}{\ee^{\Phi_{j}(x)}}K(y,z-y)yf_{j}(y)f_{j}(z-y)\dx\dz\dy.
 \end{split}
\end{equation*}
We finally change variables $z\mapsto z+y$ in the second integral on the right-hand side which leads after a further rearrangement to
\begin{multline*}
  \bigl(\T(\zeta f_{j})\bigr)''(q)=-\eps \int_{0}^{\infty}\int_{0}^{z}x^{1+\rho}\ee^{-(q+1)x}\Bigl(\frac{x}{z}\Bigr)^{\kappa_{j}}\frac{\ee^{\Phi_{j}(z)}}{\ee^{\Phi_{j}(x)}}\beta_{W}(z,f_{j})\frac{1-\ee^{-z}}{z^{\rho}}f_{j}(z)\dx\dz\\*
  +\int_{0}^{\infty}\int_{0}^{\infty}\int_{0}^{y+z}\Bigl(\frac{x}{y+z}\Bigr)^{\kappa_{j}}\frac{x^{1+\rho}}{(y+z)^{1+\rho}}\ee^{-(q+1)x}\frac{\ee^{\Phi_{j}(y+z)}}{\ee^{\Phi_{j}(x)}}K(y,z)yf_{j}(y)f_{j}(z)\dx\dz\dy.
 \end{multline*}
For two solutions $f_1$ and $f_2$ of ~\eqref{eq:self:sim} we take now the difference of the previous equation by itself for $j=1$ and $j=2$ which yields that
\begin{equation}\label{eq:bd:layer:splitting:1}
 \Bigl(\T\bigl(\zeta (f_{1}-f_{2})\bigr)\Bigr)''(q)=-\eps (K_1+K_2+K_3+K_4)+(J_1+J_2+J_3)
\end{equation}
with the terms
\begin{align*}
 K_{1}&=\int_{0}^{\infty}\int_{0}^{z}x^{1+\rho}\ee^{-(q+1)x}\biggl(\Bigl(\frac{x}{z}\Bigr)^{\kappa_{1}}-\Bigl(\frac{x}{z}\Bigr)^{\kappa_{2}}\biggr)\frac{\ee^{\Phi_{1}(z)}}{\ee^{\Phi_{1}(x)}}\beta_{W}(z,f_{1})\frac{1-\ee^{-z}}{z^{\rho}}f_{1}(z)\dx\dz,\\
 K_{2}&=\int_{0}^{\infty}\int_{0}^{z}x^{1+\rho}\ee^{-(q+1)x}\Bigl(\frac{x}{z}\Bigr)^{\kappa_{2}}\biggl(\frac{\ee^{\Phi_{1}(z)}}{\ee^{\Phi_{1}(x)}}-\frac{\ee^{\Phi_{2}(z)}}{\ee^{\Phi_{2}(x)}}\biggr)\beta_{W}(z,f_{j})\frac{1-\ee^{-z}}{z^{\rho}}f_{j}(z)\dx\dz,\\
 K_{3}&=\int_{0}^{\infty}\int_{0}^{z}x^{1+\rho}\ee^{-(q+1)x}\Bigl(\frac{x}{z}\Bigr)^{\kappa_{2}}\frac{\ee^{\Phi_{2}(z)}}{\ee^{\Phi_{2}(x)}}\bigl(\beta_{W}(z,f_{1})-\beta_{W}(z,f_{2})\bigr)\frac{1-\ee^{-z}}{z^{\rho}}f_{j}(z)\dx\dz,\\
 K_{4}&=\int_{0}^{\infty}\int_{0}^{z}x^{1+\rho}\ee^{-(q+1)x}\Bigl(\frac{x}{z}\Bigr)^{\kappa_{2}}\frac{\ee^{\Phi_{2}(z)}}{\ee^{\Phi_{2}(x)}}\beta_{W}(z,f_{2})\frac{1-\ee^{-z}}{z^{\rho}}\bigl(f_{1}(z)-f_{2}(z)\bigr)\dx\dz
\end{align*}
and, together with the abbreviation $\int_{\Sigma}(\cdots)=\int_{0}^{\infty}\int_{0}^{\infty}\int_{0}^{y+z}(\cdots)\dx\dz\dy$, furthermore
\begin{align*}
 J_{1}&=\int_{\Sigma}\frac{x^{1+\rho}}{(y+z)^{1+\rho}}\ee^{-(q+1)x}\biggl(\Bigl(\frac{x}{y+z}\Bigr)^{\kappa_{1}}-\Bigl(\frac{x}{y+z}\Bigr)^{\kappa_{2}}\biggr)\frac{\ee^{\Phi_{1}(y+z)}}{\ee^{\Phi_{1}(x)}}K(y,z)yf_{1}(y)f_{1}(z),\\
 J_{2}&=\int_{\Sigma}\frac{x^{1+\rho}}{(y+z)^{1+\rho}}\ee^{-(q+1)x}\Bigl(\frac{x}{y+z}\Bigr)^{\kappa_{2}}\biggl(\frac{\ee^{\Phi_{1}(y+z)}}{\ee^{\Phi_{1}(x)}}-\frac{\ee^{\Phi_{2}(y+z)}}{\ee^{\Phi_{2}(x)}}\biggr)K(y,z)yf_{1}(y)f_{1}(z),\\
 J_{3}&=\int_{\Sigma}\frac{x^{1+\rho}}{(y+z)^{1+\rho}}\ee^{-(q+1)x}\Bigl(\frac{x}{y+z}\Bigr)^{\kappa_{2}}\frac{\ee^{\Phi_{2}(y+z)}}{\ee^{\Phi_{2}(x)}}K(y,z)y\bigl(f_{1}(y)f_{1}(z)-f_{2}(y)f_{2}(z)\bigr).
\end{align*}
It turns out to be convenient to define the function
\begin{equation}\label{eq:def:H}
 H(y,q)\vcc=\frac{1}{y^{\rho}(1+y)}\int_{0}^{y}x^{1+\rho}\ee^{-(q+1)x}\Bigl(\frac{x}{y}\Bigr)^{\kappa_{2}}\frac{\ee^{\Phi_{2}(y)}}{\ee^{\Phi_{2}(x)}}\dx
\end{equation}
such that $K_4$ reads as
\begin{equation*}
 K_{4}=\int_{0}^{\infty}(1-\ee^{-z})(1+z)H(z,q)\beta_{W}(z,f_{2})\bigl(f_{1}(z)-f_{2}(z)\bigr)\dz.
\end{equation*}
Making additionally use of the symmetry in $y$ and $z$, the expression $J_{3}$ can be written by means of $H$ as 
\begin{equation*}
 J_{3}=\frac{1}{2}\int_{0}^{\infty}\int_{0}^{\infty}H(y+z,q)(1+y+z)K(y,z)\bigl(f_{1}(y)f_{1}(z)-f_{2}(y)f_{2}(z)\bigr)\dz\dy.
\end{equation*}
It will turn out that we have to rewrite the function $H(\cdot,q)$ further i.e.\@ integration by parts yields that
\begin{equation}\label{eq:H:splitting:1}
 \begin{split}
  H(y,q)&=-\frac{1}{(q+1)y^{\rho}(1+y)}\int_{0}^{y}\del_{x}\Bigl(\ee^{-(q+1)x}\Bigr)\biggl(\frac{x^{1+\rho+\kappa_{2}}}{y^{\kappa_{2}}}\frac{\ee^{\Phi_{2}(y)}}{\ee^{\Phi_{2}(x)}}\biggr)\dx\\
  &=-\frac{\ee^{-(q+1)y}}{q+1}\frac{y}{1+y}+H_{0}(y,q)
 \end{split}
\end{equation}
with the terms
\begin{equation}\label{eq:H:splitting:2}
 \begin{split}
  H_{0}(y,q)&\vcc=H_{0,1}(y,q)+H_{0,2}(y,q)\\
  H_{0,1}(y,q)&\vcc=\frac{1+\rho+\kappa_{2}}{(q+1)y^{\rho}(1+y)}\int_{0}^{y}\ee^{-(q+1)x}x^{\rho}\Bigl(\frac{x}{y}\Bigr)^{\kappa_{2}}\frac{\ee^{\Phi_{2}(y)}}{\ee^{\Phi_{2}(x)}}\dx\\
  H_{0,2}(y,q)&\vcc=-\frac{1}{(q+1)y^{\rho}(1+y)}\int_{0}^{y}\ee^{-(q+1)x}x^{1+\rho}\Phi_{2}'(x)\Bigl(\frac{x}{y}\Bigr)^{\kappa_{2}}\frac{\ee^{\Phi_{2}(y)}}{\ee^{\Phi_{2}(x)}}\dx.
 \end{split}
\end{equation}
The splitting of $H$ in~\eqref{eq:H:splitting:1} gives a corresponding separation of the terms $K_{4}$ and $J_{3}$ as 
\begin{equation}\label{eq:bd:layer:splitting:2}
 K_4=K_{4,1}+K_{4,0}\quad \text{and}\quad J_3=J_{3,1}+J_{3,0}
\end{equation}
with
\begin{align*}
 K_{4,1}&\vcc=-\frac{1}{q+1}\int_{0}^{\infty}(1-\ee^{-y})y\beta_{W}(y,f_{2})\bigl(f_{1}(y)-f_{2}(y)\bigr)\ee^{-(q+1)y}\dy,\\
 K_{4,0}&\vcc=\int_{0}^{\infty}(1-\ee^{-y})(1+y)H_{0}(y,q)\beta_{W}(y,f_{2})\bigl(f_{1}(y)-f_{2}(y)\bigr)\dy,\\
 J_{3,1}&\vcc=-\frac{1}{2(q+1)}\int_{0}^{\infty}\int_{0}^{\infty}\ee^{-(q+1)(y+z)}(y+z)K(y,z)\bigl(f_{1}(y)f_{1}(z)-f_{2}(y)f_{2}(z)\bigr)\dz\dy,\\
 J_{3,0}&\vcc=\frac{1}{2}\int_{0}^{\infty}\int_{0}^{\infty}H_{0}(y+z,q)(1+y+z)K(y,z)\bigl(f_{1}(y)f_{1}(z)-f_{2}(y)f_{2}(z)\bigr)\dz\dy.
\end{align*}
In order to prove Proposition~\ref{Prop:boundary:layer} we will estimate the expressions $K_1$--$K_3$, $J_1$, $J_2$, $K_{4,1}$, $K_{4,0}$, $J_{3,1}$ and $J_{3,0}$ separately. This will be done in Section~\ref{Sec:Proof:bd:layer} but before, we will collect several preliminary estimates that we will use.

\subsection{Preliminary estimates}

\begin{lemma}\label{Lem:bd:layer:est:f:close:zero}
 There exists a constant $C>0$ such that for all $\eps>0$ sufficiently small the estimate
 \begin{equation*}
  f_{j}(x)\leq C x^{\kappa_{j}+\rho-1} \quad \text{for all }x>0
 \end{equation*}
 holds for each solution $f_{j}$ of~\eqref{eq:self:sim}.
\end{lemma}

\begin{proof}
 We recall from Remark~\ref{Rem:properties:beta:and:Phi} that $\abs*{\beta_{W}(z,f)}\leq C(z^{-\alpha}+z^{\alpha})$ and note that $z^{-\rho}(1-\ee^{-z})\leq 1$. Then, the monotonicity of $\Phi(\cdot,f_{j})$, Fubini's Theorem and the change of variables $z\mapsto z+y$ together with~\eqref{eq:self:sim:bd:layer:intermediate} imply that
 \begin{multline*}
  f_{j}(x)\leq C\eps x^{\rho-1+\kappa_{j}}\int_{0}^{\infty}z^{-\kappa_{j}}(z^{-\alpha}+z^{\alpha})f_{j}(z)\dz\\*
  +Cx^{\rho-1+\kappa_{j}}\int_{0}^{\infty}\int_{y}^{\infty}(y+z)^{-\rho-1-\kappa_{j}}K(y,z)yf_{j}(y)f_{j}(z)\dz\dy.
 \end{multline*}
 It thus suffices to show that the integrals on the right-hand side can be bounded uniformly. This is however a direct consequence of Lemma~\ref{Lem:moment:est:2} and the estimate $K(y,z)\leq C((y/z)^{\alpha}+(z/y)^{\alpha})$.
\end{proof}

\begin{lemma}\label{Lem:bd:layer:int:est:1}
 Let $\alpha\in(0,1)$ be given. For each $\gamma\in(-\rho,\rho)$ and each $p\in[1,\infty)$ with dual exponent $p'=p/(p-1)$ there exists a constant $C>0$ such that for sufficiently $\eps>0$ the estimate
 \begin{equation*}
  \eps \int_{x}^{\infty}z^{\gamma-\kappa_{2}}f_{1}(z)\ee^{-\min\{\Phi_{j}(x)-\Phi_{j}(z)\;|\; j=1,2\}}\dz\leq C\eps^{1/p'}x^{\alpha/p}\quad \text{for all }x>0
 \end{equation*}
 holds true for any pair $f_1$ and $f_2$ of solutions to~\eqref{eq:self:sim}.
\end{lemma}

\begin{remark}\label{Rem:bd:layer:reg:effect}
 From \cref{Lem:asymptotics:Phi,Lem:asymptotics:beta,Lem:analyticity:Phi,Lem:reg:exponetial:decay}, one can derive for each $\alpha\in(0,1)$ that it holds
 \begin{equation}\label{eq:bd:layer:exp:decay}
  \ee^{-\min\{\Phi_{j}(x)-\Phi_{j}(z)\;|\; j=1,2\}}\leq C\ee^{-\frac{d\eps}{x^{\alpha}}}\quad \text{if }x\leq 1\text{ and }z\geq 1
 \end{equation}
where $C,d>0$ are constants. Moreover, one has for $j=1,2$ that
\begin{equation}\label{eq:bd:layer:reg:difference}
 \Phi_{j}(x+z)-\Phi_{j}(x)\leq\begin{cases}
                               -\frac{B\eps}{x^{\alpha}}+C &\text{if }x\leq z\\
                               -\frac{b\eps}{x^{1+\alpha}}z +C &\text{if }x\geq z,
                              \end{cases}
\end{equation}
holds for all $0<x,z<2$. 
\end{remark}

\begin{proof}[Proof of Lemma~\ref{Lem:bd:layer:int:est:1}]
 We define $\mathcal{I}\vcc=\int_{x}^{\infty}z^{\gamma-\kappa_{2}}f_{1}(z)\ee^{-\min\{\Phi_{j}(x)-\Phi_{j}(z)\;|\; j=1,2\}}\dz$ to shorten the notation and we note that the monotonicity of $\Phi_{j}$ and Lemma~\ref{Lem:moment:est:2} directly imply that $\mathcal{I}\leq C$. Thus, the claim already holds for $x\geq 1$. 
 
 We thus consider $x\leq 1$ and split $\mathcal{I}$ as
 \begin{equation}\label{eq:bd:layer:int:est:1:1}
  \mathcal{I}=\int_{x}^{2}(\cdots)\dz+\int_{2}^{\infty}(\cdots)\dz=\vcc\mathcal{I}_1+\mathcal{I}_2.
 \end{equation}
For the expression $\mathcal{I}_2$ we deduce with \cref{eq:bd:layer:exp:decay,Lem:moment:est:2,Lem:help:exponential:regularising} that
\begin{equation}\label{eq:bd:layer:int:est:1:2}
 \eps \mathcal{I}_{2}\leq C\eps \ee^{-\frac{d\eps}{x^{\alpha}}}\int_{2}^{\infty}z^{\gamma-\kappa_{2}}f_{1}(z)\dz\leq C\eps \ee^{-\frac{d\eps}{x^{\alpha}}}\leq C\eps^{1/p'}x^{\alpha/p}.
\end{equation}
The expression $\mathcal{I}_1$ is slightly more complicated, i.e.\@ we introduce $\Delta\kappa\vcc=\kappa_1-\kappa_2$, change variables $z\mapsto z+x$, split the integral further and we invoke \cref{Lem:bd:layer:est:f:close:zero,eq:bd:layer:reg:difference} to get
\begin{equation*}
 \begin{split}
  \eps \mathcal{I}_{1}&\leq C\eps \int_{0}^{2}(x+z)^{\gamma+\rho+\Delta\kappa-1}\ee^{-\min\{\Phi_{j}(x)-\Phi_{j}(x+z)\;|\; j=1,2\}}\dz\\
  &\leq C\eps\int_{0}^{x}(x+z)^{\gamma+\rho+\Delta\kappa-1}\ee^{-\frac{b\eps}{x^{1+\alpha}}z}\dz+C\eps\ee^{-\frac{B\eps}{x^{\alpha}}}\int_{x}^{2}(x+z)^{\gamma+\rho+\Delta\kappa-1}\dz.
 \end{split}
\end{equation*}
Computing the second integral on the right-hand side explicitly and using $\gamma+\rho+\Delta\kappa-1=(\gamma+\rho+\Delta\kappa-1)/p+(\gamma+\rho+\Delta\kappa-1)/p'$ for the first one, yields
\begin{multline*}
 \eps\mathcal{I}_{1}\leq  C\eps\biggl(\int_{0}^{x}(x+z)^{\frac{\gamma+\rho+\Delta\kappa-1}{p}}\ee^{-\frac{b\eps}{x^{1+\alpha}}z}(x+z)^{\frac{\gamma+\rho+\Delta\kappa-1}{p'}}\dz\biggr)\\*
 +\frac{C\eps}{\gamma+\rho+\Delta\kappa}\ee^{-\frac{B\eps}{x^{\alpha}}}\bigl((2+x)^{\gamma+\rho+\Delta\kappa}-x^{\gamma+\rho+\Delta\kappa}\bigr).
\end{multline*}
Note that according to Remark~\ref{Rem:kappa:smallness} the expression $\Delta\kappa$ is small if we choose $\eps$ small enough. In the second term on the right-hand side we use that $x\leq 1$ to bound the term in parenthesis by a constant, while in the first term, we apply Hölder's inequality with $p$ and $p'$ which yields
\begin{equation*}
 \eps\mathcal{I}_{1}\leq C\eps\biggl(\int_{0}^{x}(x+z)^{\gamma+\rho+\Delta\kappa-1}\ee^{-\frac{-\rho b\eps}{x^{1+\alpha}}z}\dz\biggr)^{1/p}\biggl(\int_{0}^{x}(x+z)^{\gamma+\rho+\Delta\kappa-1}\dz\biggr)^{1/p'}+\frac{C\eps}{\gamma+\rho+\Delta\kappa}\ee^{-\frac{B\eps}{x^{\alpha}}}.
\end{equation*}
We now change variables $z\mapsto \frac{x^{1+\alpha}}{p b\eps}z$ in the first integral, estimate the second one by $Cx^{\frac{\rho+\gamma\Delta\kappa}{p'}}$ and use Lemma~\ref{Lem:help:exponential:regularising} in the last expression on the right-hand side which yields
\begin{equation*}
 \eps\mathcal{I}_{1}\leq C\eps\biggl(\frac{x^{1+\alpha}}{pb\eps}\int_{0}^{\frac{bpb\eps}{x^{\alpha}}}\biggl(x+\frac{x^{1+\alpha}}{pb\eps}z\biggr)^{\gamma+\rho+\Delta\kappa-1}\ee^{-z}\dz\biggr)^{1/p} x^{\frac{\rho+\gamma\Delta\kappa}{p'}}+ C\eps^{1/p'}x^{\alpha/p}.
\end{equation*}
Since $\left(x+\frac{x}{r}z\right)^{\gamma+\rho+\Delta\kappa-1}\leq C x^{\gamma+\rho+\Delta\kappa-1}$ for all $z\in(0,r)$ as well as $1/p+1/p'=1$ and $\rho+\gamma+\Delta\kappa>0$ for $\eps$ sufficiently small, it follows
\begin{multline*}
  \eps\mathcal{I}_{1}\leq C\eps^{1-\frac{1}{p}}x^{\frac{\rho+\alpha+\gamma+\Delta\kappa}{p}}x^{\frac{\rho+\gamma+\Delta\kappa}{p'}}\int_{0}^{\infty}\ee^{-z}\dz+C\eps^{1/p'}x^{\alpha/p}\\*
  =C\eps^{1/p'}x^{\rho+\gamma+\Delta\kappa+\frac{\alpha}{p}}+C\eps^{1/p'}x^{\alpha/p}\leq C\eps^{1/p'}x^{\alpha/p}.
 \end{multline*}
Together with~\eqref{eq:bd:layer:int:est:1:2} the claim follows.
\end{proof}

\begin{lemma}\label{Lem:bd:layer:int:est:2}
 For $\alpha\in(0,1)$ and each $p\in[1,\infty)$ with dual exponent $p'=p/(p-1)$ there exists a constant $C>0$ such that it holds for $\eps>0$ sufficiently small and all $x>0$ that 
 \begin{equation*}
  \eps \int_{0}^{\infty}\int_{\max\{0,x-y\}}^{\infty}\frac{f_1(y)f_1(z)}{(y+z)^{\rho+\kappa_2}}K(y,z)\ee^{-\min\{\Phi_j(x)-\Phi_j(y+z)\;|\; j=1,2\}}\dz\dy\leq C\eps^{1/p'}x^{\alpha/p}
 \end{equation*}
  for all pairs of solutions $f_1$ and $f_2$ of~\eqref{eq:self:sim}.
\end{lemma}

\begin{proof}
 The proof is similar to that one of Lemma~\ref{Lem:bd:layer:est:f:close:zero}.
\end{proof}

\begin{lemma}\label{Lem:kappa:difference}
 For every $\delta>0$ there exists a constant $C_{\delta}>0$ such that it holds for sufficiently small $\eps>0$ that
 \begin{equation*}
  \abs*{\kappa_1-\kappa_2}\leq \delta \lfnorm*{0}{-\rho}{\theta}{\T(f_1-f_2)}+C_{\delta}\lfnorm*{0}{-\rho}{\theta}{\T\bigl((1-\zeta)(f_1-f_2)\bigr)}
 \end{equation*}
 for all pairs $f_1$ and $f_2$ of solutions to~\eqref{eq:self:sim}.
\end{lemma}

\begin{proof}
 Note that $\kappa_1-\kappa_2=2\T (f_1-f_2)(0)$, which means that it suffices to prove the stated estimate with the left-hand side replaced by $\abs*{\T (f_1-f_2)(0)}$. Together with $1=\ee^{-nz}+(1-\ee^{-nz})$ we thus first rewrite and estimate
 \begin{multline*}
  \abs*{\T (f_1-f_2)(0)}=\abs*{\int_{0}^{\infty}f_1(z)-f_2(z)\dz}\\*
  \leq \abs*{\int_{0}^{\infty}\bigl(f_1(z)-f_2(z)\bigr)\ee^{-nz}\dz}+\abs*{\int_{0}^{\infty}\bigl(f_1(z)-f_2(z)\bigr)(1-\ee^{-nz})\dz}\\*
  \leq \frac{\lfnorm*{0}{-\rho}{\theta}{\T(f_1-f_2)}}{(1+n)^{\theta}}+\abs*{\int_{0}^{\infty}\bigl(f_1(z)-f_2(z)\bigr)(1-\zeta(z))\frac{1-\ee^{-nz}}{1-\ee^{-z}}\dz}.
 \end{multline*}
Since $(1-\ee^{-nz})/(1-\ee^{-z})=\sum_{k=0}^{n-1}\ee^{-kz}$ it further follows
\begin{multline*}
 \abs*{\T (f_1-f_2)(0)}\leq \frac{\lfnorm*{0}{-\rho}{\theta}{\T(f_1-f_2)}}{(1+n)^{\theta}}+\sum_{k=0}^{n-1}\abs*{\int_{0}^{\infty}\bigl(f_1(z)-f_2(z)\bigr)(1-\zeta(z))\ee^{-kz}\dz}\\*
 \leq \frac{\lfnorm*{0}{-\rho}{\theta}{\T(f_1-f_2)}}{(1+n)^{\theta}}+\lfnorm*{0}{-\rho}{\theta}{\T\bigl((1-\zeta)(f_1-f_2)\bigr)}\sum_{k=0}^{n-1}\frac{1}{(k+1)^{\theta}}.
\end{multline*}
We then fix $n_{\delta}\in\N$ sufficiently large such that it holds $(1+n_{\delta})^{-\theta}\leq \delta$ which implies the claim with $C_{\delta}=\sum_{k=0}^{n_{\delta}-1}(k+1)^{-\theta}$.
\end{proof}

\begin{lemma}\label{Lem:bd:layer:Phi:difference}
 There exists a constant $C>0$ such that it holds
 \begin{equation*}
  \abs*{\Phi(x,f_1)-\Phi(x,f_2)}\leq C\eps \lfnorm*{1}{0}{\theta}{\T(f_1-f_2)}x^{-\alpha}\ee^{-x/2}
 \end{equation*}
 for any two solutions $f_1$ and $f_2$ of~\eqref{eq:self:sim}.
\end{lemma}

\begin{proof}
 From the definitions of $\Phi_j$ and $\beta_{W}$ in~\cref{eq:def:Phi,eq:def:beta:and:R} one computes together with Proposition~\ref{Prop:W:representation} that it holds
 \begin{equation*}
  \Phi_{j}(x)=\eps\int_{x}^{\infty}\int_{0}^{\infty}\int_{0}^{\infty}\int_{0}^{\infty}\Ker(s,\eta)\ee^{-s y-\eta z}\ds\deta f_{j}(z)\ee^{-y}\Bigl(1+\frac{z}{y}\Bigr)\dz\dy.
 \end{equation*}
Using the relation $y^{-1}\ee^{(s+1)y}=\int_{s}^{\infty}\ee^{-(\xi+1)y}$, applying Fubini's Theorem and evaluating the integral in $y$, this can be further rearranged to obtain
\begin{multline*}
 \Phi_{j}(x)=\eps \int_{0}^{\infty}\int_{0}^{\infty}\frac{\ee^{-(\xi+1)x}}{\xi+1}\Ker(\xi,\eta)(\T f_j)(\eta)\deta\dxi\\*
 -\eps\int_{0}^{\infty}\int_{0}^{\infty}\frac{\ee^{-(\xi+1)x}}{\xi+1}(\T f_j)'(\eta)\int_{0}^{\xi}\Ker(s,\eta)\ds\dxi\deta.
\end{multline*}
For the difference $\Phi_1-\Phi_2$ it thus follows
\begin{multline*}
 \Phi_1(x)-\Phi_2(x)=\eps\int_{0}^{\infty}\int_{0}^{\infty}\frac{\ee^{-(\xi+1)x}}{\xi+1}\Ker(\xi,\eta)\bigl(\T(f_1-f_2)\bigr)(\eta)\deta\dxi\\*
 -\eps\int_{0}^{\infty}\int_{0}^{\infty}\frac{\ee^{-(\xi+1)x}}{\xi+1}\bigl(\T(f_1-f_2)\bigr)'(\eta)\int_{0}^{\xi}\Ker(s,\eta)\ds\deta\dxi.
\end{multline*}
This then yields the estimate
\begin{multline}\label{eq:bd:layer:difference:Phi:1}
 \abs*{ \Phi_1(x)-\Phi_2(x)}\leq \eps\lfnorm*{0}{-\rho}{\theta}{\T(f_1-f_2)}\int_{0}^{\infty}\frac{\ee^{-(\xi+1)x}}{\xi+1}\int_{0}^{\infty}\frac{\abs*{\Ker(\xi,\eta)}}{(\eta+1)^{\theta}}\deta\dxi\\*
 +\eps\lsnorm*{1}{0}{\theta}{\T(f_1-f_2)}\int_{0}^{\infty}\frac{\ee^{-(\xi+1)x}}{\xi+1}\int_{0}^{\infty}\int_{0}^{\xi}\abs*{\Ker(s,\eta)}\ds\frac{1}{(\eta+1)^{\theta+\rho}\eta^{1-\rho}}\deta\dxi.
\end{multline}
For the first integral on the right-hand side we obtain from Lemma~\ref{Lem:Ker:est:most:general} and $\ee^{-(\xi+1)x}\leq \ee^{-x}$ that
\begin{equation*}
 \int_{0}^{\infty}\frac{\ee^{-(\xi+1)x}}{\xi+1}\int_{0}^{\infty}\frac{\abs*{\Ker(\xi,\eta)}}{(\eta+1)^{\theta}}\deta\dxi\leq C\ee^{-x}.
\end{equation*}
For the second integral on the right-hand side of~\eqref{eq:bd:layer:difference:Phi:1} we use Lemma~\ref{Lem:kernel:est:primitive} and $(\xi+1)^{-1}(\xi^{\alpha}\eta^{-\alpha}+1)\leq \xi^{\alpha-1}(\eta^{-\alpha}+1)$ to get
\begin{multline*}
 \int_{0}^{\infty}\frac{\ee^{-(\xi+1)x}}{\xi+1}\int_{0}^{\infty}\int_{0}^{\xi}\abs*{\Ker(s,\eta)}\ds\frac{1}{(\eta+1)^{\theta+\rho}\eta^{1-\rho}}\deta\dxi\\*
 \leq C\ee^{-x}\int_{0}^{\infty}\frac{\ee^{-x\xi}(\xi^{\alpha}\eta^{-\alpha}+1)}{(\xi+1)(\eta+1)^{\theta+\rho}\eta^{1-\rho}}\deta\dxi\\*
 \leq C\ee^{-x}\int_{0}^{\infty}\xi^{\alpha-1}\ee^{-x\xi}\dxi\int_{0}^{\infty}\frac{\eta^{-\alpha}+1}{(\eta+1)^{\theta+\rho}\eta^{1-\rho}}\deta.
\end{multline*}
Estimating the integral in $\eta$ by a constant and changing variables $\xi\mapsto x^{-1}\xi$ in the integral in $\xi$ one deduces that
\begin{equation*}
 \int_{0}^{\infty}\frac{\ee^{-(\xi+1)x}}{\xi+1}\int_{0}^{\infty}\int_{0}^{\xi}\abs*{\Ker(s,\eta)}\ds\frac{1}{(\eta+1)^{\theta+\rho}\eta^{1-\rho}}\deta\dxi\leq Cx^{-\alpha}\ee^{-x}.
\end{equation*}
Since $\ee^{-x}\leq Cx^{-\alpha}\ee^{-x/2}$ for all $x>0$ the claim follows together with~\eqref{eq:bd:layer:difference:Phi:1}.
\end{proof}

\begin{lemma}\label{Lem:bd:layer:beta:difference}
 There exists a constant $C>0$ such that it holds
 \begin{equation*}
  \abs*{\beta_{W}(y,f_1)-\beta_{W}(y,f_2)}\leq C\lfnorm*{1}{0}{\theta}{\T(\Delta f)}\bigl(y^{-\alpha}+y^{\alpha}+y^{\theta-\alpha}\bigr)
 \end{equation*}
 for each pair $f_1$ and $f_2$ of solutions to~\eqref{eq:self:sim}.
\end{lemma}

\begin{proof}
 From the definition of $\beta_{W}$ in~\eqref{eq:def:beta:and:R} and Proposition~\ref{Prop:W:representation} we obtain for $j=1,2$ after some rearrangement that it holds
 \begin{equation}\label{eq:beta:rewritten}
  \beta_{W}(y,f_j)=\int_{0}^{\infty}\ee^{-\xi y}\int_{0}^{\infty}\Ker(\xi,\eta)\bigl(y(\T f_j)(\eta)-(\T f_j)'(\eta)\bigr)\deta\dxi.
 \end{equation}
 For the difference $\beta_{W}(y,f_1)-\beta_{W}(y,f_2)$ we may thus estimate
 \begin{equation}\label{eq:bd:layer:beta:difference:1}
  \begin{split}
   \abs*{\beta_{W}(y,f_1)-\beta_{W}(y,f_2)}&\leq \lfnorm*{0}{-\rho}{\theta}{\T(\Delta f)}\int_{0}^{\infty}\ee^{-\xi y}\int_{0}^{\infty}\abs*{\Ker(\xi,\eta)}\frac{y}{(\eta+1)^{\theta}}\deta\dxi\\
   &\phantom{{}\leq{}}+\lsnorm*{1}{0}{\theta}{\T(\Delta f)}\int_{0}^{\infty}\ee^{-\xi y}\int_{0}^{\infty}\abs*{\Ker(\xi,\eta)}\frac{1}{(\eta+1)^{\theta+\rho}\eta^{1-\rho}}\deta\dxi.
  \end{split}
 \end{equation}
We continue by estimating the two integrals on the right-hand side separately, while we obtain for the first one together with Lemma~\ref{Lem:kernel:est:partial:0} and a change of variables $\xi\mapsto \xi/y$ that 
\begin{multline}\label{eq:bd:layer:beta:difference:2}
  \int_{0}^{\infty}\ee^{-\xi y}\int_{0}^{\infty}\abs*{\Ker(\xi,\eta)}\frac{y}{(\eta+1)^{\theta}}\deta\dxi\leq C\int_{0}^{\infty}y\ee^{-\xi y}(\xi^{-\alpha}+\xi^{\alpha-\theta})\dxi\\*
  \leq C(y^{\alpha}+y^{\theta-\alpha})\int_{0}^{\infty}\ee^{-\xi}(\xi^{-\alpha}+\xi^{\alpha-\theta})\dxi\leq C(y^{\alpha}+y^{\theta-\alpha}).
 \end{multline}
On the other hand, it follows from Lemma~\ref{Lem:kernel:est:partial:1} that 
\begin{multline*}
 \int_{0}^{\infty}\ee^{-\xi y}\int_{0}^{\infty}\frac{\abs*{\Ker(\xi,\eta)}}{(\eta+1)^{\theta+\rho}\eta^{1-\rho}}\deta\dxi\\*
 \leq C\int_{0}^{1}\ee^{-y\xi}(\xi^{\rho-1}+\xi^{-\alpha})\abs*{\log(\xi)}\dxi+C\int_{1}^{\infty}\ee^{-y\xi}\xi^{\alpha-1}\dxi.
\end{multline*}
Since $\ee^{-y\xi}\leq 1$, we can bound the first integral on the right-hand side just by a constant, while in the second one, we change variables $\xi\mapsto \xi/y$ which yields
\begin{equation*}
 \int_{0}^{\infty}\ee^{-\xi y}\int_{0}^{\infty}\frac{\abs*{\Ker(\xi,\eta)}}{(\eta+1)^{\theta+\rho}\eta^{1-\rho}}\deta\dxi\leq C+Cy^{-\alpha}\int_{0}^{\infty}\ee^{-\xi}\xi^{\alpha-1}\dxi\leq C(1+y^{-\alpha}).
\end{equation*}
If we use this estimate together with~\eqref{eq:bd:layer:beta:difference:2} in~\eqref{eq:bd:layer:beta:difference:1} the claim follows immediately since $1\leq y^{-\alpha}+y^{\alpha}$ for all $y>0$.
\end{proof}

\subsection{Proof of Proposition~\ref{Prop:boundary:layer}}\label{Sec:Proof:bd:layer}

The general strategy will be to consider the expressions $K_1$ \& $J_{1}$, $K_{2}$ \& $J_{2}$, $K_3$, as well as $K_{4,1}$, $K_{4,0}$, $J_{3,1}$ and $J_{3,0}$ separately which will be done in own subsections. Additionally, we mention that we will make frequent use of the fact that the function $\Phi_{j}(\cdot)$ is monotonously decreasing and thus it holds for $j=1,2$ that
\begin{equation*}
 \Phi_{j}(z)-\Phi_j(x)\leq 0\quad \text{and}\quad \exp\bigl(\Phi_{j}(z)-\Phi_{j}(x)\bigr)\leq 1\quad \text{if }x\leq z.
\end{equation*}
We will also often use that, according to Remark~\ref{Rem:kappa:smallness}, for a given $\nu>0$ it holds $\abs*{\kappa_j}\leq \nu$ if $\eps>0$ is small enough.

\subsubsection{Estimates for $K_1$ and $J_1$}

In this subsection we will show that it holds
\begin{equation}\label{eq:bd:layer:est:K1:J1}
 \eps\abs*{K_1}+\abs*{J_1}\leq \Bigl(\delta\lfnorm*{0}{-\rho}{\theta}{\T(\Delta f)}+C_{\delta}\lfnorm*{0}{-\rho}{\theta}{\T\bigl((1-\zeta)\Delta f\bigr)}(q+1)^{-2-\theta}.
\end{equation}
To prove this estimate, we use that $(1-\ee^{-z})z^{-\rho}\leq 1$ for all $z>0$ as well as the elementary bound $\abs*{\ee^{-a}-\ee^{-b}}\leq \abs*{a-b}$ for $a,b>0$ which yields for any small $\nu>0$ that 
\begin{equation*}
 \abs*{\Bigl(\frac{x}{z}\Bigr)^{\kappa_{1}}-\Bigl(\frac{x}{z}\Bigr)^{\kappa_{2}}}\leq \Bigl(\frac{z}{x}\Bigr)^{\nu}\log\Bigl(\frac{z}{x}\Bigr)\abs*{\kappa_{1}-\kappa_{2}}\leq C_{\nu}\Bigl(\frac{z}{x}\Bigr)^{2\nu}\abs*{\kappa_{1}-\kappa_{2}}.
\end{equation*}
Together with Remark~\ref{Rem:properties:beta:and:Phi} we thus deduce that the expression $K_1$ can be estimated by
\begin{equation*}
 \abs*{K_{1}}\leq C\abs*{\kappa_1-\kappa_2}\int_{0}^{\infty}z^{2\nu}(z^{-\alpha}+z^{\alpha})f_1(z)\int_{0}^{z}x^{1+\rho-2\nu}\ee^{-(q+1)x}\dx\dz.
\end{equation*}
For $\nu$ sufficiently small it holds $\int_{0}^{z}x^{1+\rho-2\nu}\ee^{-(q+1)x}\dx\leq \Gamma(2+\rho-2\delta)(q+1)^{2+\rho-2\delta}$. Therefore, if we take $\nu\leq (\rho-\theta)/2$, we conclude with Lemma~\ref{Lem:moment:est:2} that
\begin{equation}\label{eq:bd:layer:K1:1}
 \abs*{K_1}\leq C\abs*{\kappa_1-\kappa_2}(q+1)^{2+\theta}.
\end{equation}
The expression $J_1$ can be estimated similarly, i.e.\@ with the estimates $K(y,z)\leq C((y/z)^{\alpha}+(z/y)^{\alpha})$ and $y/(y+z)\leq 1$ one obtains
\begin{equation*}
 \abs*{J_1}\leq \frac{C}{(q+1)^{2+\rho-2\nu}}\abs*{\kappa_1-\kappa_2}\int_{0}^{\infty}\int_{0}^{\infty}\frac{y^{\alpha}z^{-\alpha}+y^{-\alpha}z^{\alpha}}{(y+z)^{\rho-2\nu}}f_1(y)f_1(z)\dz\dy.
\end{equation*}
If we choose again $\nu\leq(\rho-\theta)/2$, it follows from the symmetry of the integrand together with Lemma~\ref{Lem:moment:est:2} that 
\begin{equation}\label{eq:bd:layer:J1:1}
 \abs*{J_1}\leq \frac{C}{(q+1)^{2+\theta}}\abs*{\kappa_1-\kappa_2}\int_{0}^{\infty}y^{\alpha+2\nu-\rho}f_1(y)\dy+\int_{0}^{\infty}z^{-\alpha}f_1(z)\dz\leq \frac{C}{(q+1)^{2+\theta}}\abs*{\kappa_1-\kappa_2}.
\end{equation}
The proof of~\eqref{eq:bd:layer:est:K1:J1} now follows immediately by combining \cref{eq:bd:layer:K1:1,eq:bd:layer:J1:1,Lem:kappa:difference}.

\subsubsection{Estimates for $K_2$ and $J_2$}

In this subsection, we will show that there exists $p'\in(1,\infty)$ such that it holds
\begin{equation}\label{eq:bd:layer:est:K2:J2}
 \eps\abs*{K_2}+\abs*{J_2}\leq C\eps^{1/p'}\lfnorm*{2}{0}{\theta}{\T(\Delta f)}(q+1)^{-2-\theta}.
\end{equation}
To derive this estimate, we note that it holds $\abs*{\ee^{-a}-\ee^{-b}}\leq \ee^{-\min\{a,b\}}\abs*{a-b}$ for all $a,b\geq 0$. This then allows to deduce for $x\leq z$ that
\begin{multline}\label{eq:bd:layer:K2:J2:help}
 \abs*{\ee^{-(\Phi_1(x)-\Phi_1(z))}-\ee^{-(\Phi_2(x)-\Phi_2(z))}}\\*
 \leq \ee^{-\min\{\Phi_{j}(x)-\Phi_{j}(z)\;|\; j=1,2\}}\bigl(\abs*{\Phi_1(z)-\Phi_2(z)}+\abs*{\Phi_1(x)-\Phi_2(x)}\bigr).
\end{multline}
Since $z^{-\rho}(1-\ee^{-z})\leq 1$ for all $z\geq 0$, and taking into account \cref{Rem:properties:beta:and:Phi,Lem:bd:layer:Phi:difference} we obtain
\begin{multline*}
 \abs*{K_2}\leq C\eps\lfnorm*{1}{0}{\theta}{\T(\Delta f)}\int_{0}^{\infty}\int_{0}^{\xi}x^{1+\rho+\kappa_{2}}\ee^{-(q+1)x}\ee^{-\min\{\Phi_{j}(x)-\Phi_{j}(z)\;|\; j=1,2\}}\cdot\\*
 \cdot \biggl(\frac{\ee^{-x/2}}{x^{\alpha}}+\frac{\ee^{-z/2}}{z^{\alpha}}\biggr)z^{-\kappa_{2}}(z^{-\alpha}+z^{\alpha})f_{1}(z)\dz\dx.
\end{multline*}
If we apply now Fubini's Theorem, and note that $z^{-\alpha}\ee^{-z/2}\leq x^{-\alpha}\ee^{-x/2}$ for $x\leq z$ we obtain together with Lemma~\ref{Lem:bd:layer:int:est:1} for dual exponents $p,p'\in(1,\infty)$ with $1/p+1/p'=1$ that 
\begin{equation*}
 \begin{split}
  \abs*{K_2}&\leq C\eps\lfnorm*{1}{0}{\theta}{\T(\Delta f)}\int_{0}^{\infty}x^{1+\rho-\alpha+\kappa_{2}}\ee^{-\left(q+\frac{3}{2}\right)x}\int_{x}^{\infty}\frac{z^{-\kappa_2}(z^{-\alpha}+z^{\alpha})f_{1}(z)}{\ee^{\min\{\Phi_{j}(x)-\Phi_{j}(z)\;|\; j=1,2\}}}\dz\dx\\
  &\leq C\eps^{1/p'}\lfnorm*{1}{0}{\theta}{\T(\Delta F)}\int_{0}^{\infty}x^{1+\rho-\alpha+\frac{\alpha}{p}+\kappa_{2}}\ee^{-\left(q+\frac{3}{2}\right)}\dx\leq C\eps^{1/p'}\frac{\lfnorm*{1}{0}{\theta}{\T(\Delta f)}}{(q+3/2)^{2+\rho-\alpha/p'+\kappa_{2}}}.
 \end{split}
\end{equation*}
We now choose $p'$ sufficiently large and $\abs*{\kappa_{2}}$ sufficiently small such that we finally get
\begin{equation}\label{eq:bd:layer:K2:1}
 \abs*{K_2}\leq C\eps^{1/p'}\lfnorm*{1}{0}{\theta}{\T(\Delta f)}(q+1)^{-2-\theta}.
\end{equation}
 To estimate $J_2$ we proceed in the same way, i.e.\@ we first note that \cref{eq:bd:layer:K2:J2:help,Lem:bd:layer:Phi:difference} as well as $(y+z)^{-\alpha}\ee^{-(y+z)/2}\leq x^{-\alpha}\ee^{-x/2}$ for $x\leq y+z$ and $y/(y+z)\leq 1$ together with Fubini's Theorem imply
 \begin{multline*}
  \abs*{J_2}\leq C\eps \lfnorm*{1}{0}{\theta}{\T(\Delta f)}\int_{0}^{\infty}x^{1+\rho-\alpha+\kappa_{2}}\ee^{-\left(q+\frac{3}{2}\right)x}\int_{0}^{\infty}\int_{\max\{0,x-y\}}^{\infty}K(y,z)\frac{f_{1}(y)f_{1}(z)}{(y+z)^{\rho+\kappa_{2}}}\cdot\\*
  \cdot \ee^{-\min\{\Phi_{j}(x)-\Phi_{j}(y+z)\;|\; j=1,2\}}\dz\dy\dx.
 \end{multline*}
 Together with Lemma~\ref{Lem:bd:layer:int:est:2} it follows in the same way as above for $p'$ sufficiently large and $\abs*{\kappa_{2}}$ sufficiently small that 
 \begin{equation*}
  \abs*{J_{2}}\leq C\eps^{1/p'}\lfnorm*{1}{0}{\theta}{\T(\Delta f)}(q+1)^{-2-\theta}.
 \end{equation*}
 Together with~\eqref{eq:bd:layer:K2:1} this shows the claimed estimate~\eqref{eq:bd:layer:est:K2:J2} while we also note that the additional factor $\eps$ in front of $K_2$ is not needed here.
 
 \subsubsection{Estimate for $K_3$}

 In this subsection we will show that it holds
 \begin{equation}\label{eq:bd:layer:est:K3}
  \eps\abs*{K_3}\leq C\eps\lfnorm*{1}{0}{\theta}{\T(\Delta f)}(q+1)^{-2-\theta}.
 \end{equation}

 To see this, we use $z^{-\rho}(1-\ee^{-z})\leq 1$, the monotonicity of $\Phi_{2}$ as well as \cref{Lem:bd:layer:beta:difference,Lem:moment:est:2} to deduce for $\abs*{\kappa_{2}}$ sufficiently small that
 \begin{equation*}
  \begin{split}
   \abs*{K_3}&\leq C\lfnorm*{1}{0}{\theta}{\T(\Delta f)}\int_{0}^{\infty}x^{1+\rho+\kappa_{2}}\ee^{-(q+1)x}\dx\int_{0}^{\infty}z^{-\kappa_2}\bigl(z^{-\alpha}+z^{\alpha}+z^{\theta-\alpha}\bigr)f_{1}(z)\dz\\
   &\leq C\lfnorm*{1}{0}{\theta}{\T(\Delta f)}(q+1)^{-2-\rho-\kappa_{2}}\leq C\lfnorm*{1}{0}{\theta}{\T(\Delta f)}(q+1)^{-2-\theta}.
  \end{split}
 \end{equation*}
 This already proves~\eqref{eq:bd:layer:est:K3}.
 
 \subsubsection{Estimate for $J_{3,1}$}\label{Sec:est:J31}
 In this subsection, we will show that it holds
 \begin{equation}\label{eq:bd:layer:est:J:31}
  \abs*{J_{3,1}}\leq \Bigl(\delta \lfnorm*{2}{0}{\theta}{\T(\Delta f)}+C_{\delta}\lfnorm*{1}{0}{\theta}{\T\bigl((1-\zeta)(\Delta f)\bigr)}(q+1)^{-2-\theta}.
 \end{equation}
 As a first step, we recall that $K=2+\eps W$ such that $J_{3,1}$ can be rewritten by means of Proposition~\ref{Prop:W:representation} and the symmetry of the integrand. Precisely, we get
 \begin{equation}\label{eq:bd:layer:J31:splitting}
 \begin{split}
  J_{3,1}&=\frac{1}{q+1}\frac{\dd}{\dd{q}}\biggl(\int_{0}^{\infty}\int_{0}^{\infty}\ee^{-(q+1)(y+z)}\bigl(f_{1}(y)-f_{2}(y)\bigr)\bigl(f_{1}(z)+f_{2}(z)\bigr)\dz\dy\biggr)\\
  &-\frac{\eps}{2(q+1)}\int_{0}^{\infty}\int_{0}^{\infty}\Ker(\xi,\eta)(\del_{\xi}+\del_{\eta})^{2}\Bigl((\T(\Delta f))(\xi+q+1)\bigl(\T(f_1+f_{2})\bigr)(\eta+q+1)\Bigr)\deta\dxi\\
  &=\vcc J_{3,1,1}+J_{3,1,2}.
 \end{split}
\end{equation} 
 To simplify the presentation, we will treat the terms $J_{3,1,1}$ and $J_{3,1,2}$ individually, while we consider first $J_{3,1,1}$. Since this expression does not contain a prefactor $\eps$, it is necessary to invoke some interpolation argument, i.e.\@ for $n\in\N$ we squeeze in the factor
 \begin{equation}\label{eq:bd:layer:exp:splitting}
  1=\ee^{-yn}\ee^{-zn}+(1-\ee^{-yn})\ee^{-zn}+\ee^{-yn}(1-\ee^{-zn})+(1-\ee^{-yn})(1-\ee^{-zn}).
 \end{equation}
This then yields
\begin{multline*}
 J_{3,1,1}=\frac{1}{q+1}\frac{\dd}{\dd{q}}\biggl(\bigl(\T(\Delta f)\bigr)(q+n+1)\bigl(\T(f_1+f_2)\bigr)(q+n+1)\\*
 +\Bigl(\T\bigl((1-\ee^{-n\cdot})(\Delta f)\bigr)\Bigr)(q+1)\bigl(\T(f_1+f_2)\bigr)(q+n+1)\\*
 +\bigl(\T(\Delta f)\bigr)(q+n+1)\Bigl(\T\bigl((1-\ee^{-n\cdot})(f_1+f_2)\bigr)\Bigr)(q+1)\\*
 +\Bigl(\T\bigl((1-\ee^{-n\cdot})(\Delta f)\bigr)\Bigr)(q+1)\Bigl(\T\bigl((1-\ee^{-n\cdot})(f_1+f_2)\bigr)\Bigr)(q+1)\biggr).
\end{multline*}
Due to this relation, we can estimate
\begin{multline}\label{eq:bd:layer:est:J31:0}
 \abs*{J_{3,1,1}}\leq \frac{1}{q+1}\Biggl(\frac{\lsnorm*{1}{0}{\theta}{\T(\Delta f)}(\lfnorm*{0}{-\rho}{\theta}{\T f_1}+\lfnorm*{0}{-\rho}{\theta}{\T f_2})}{(q+n+2)^{2\theta+\rho}(q+n+1)^{1-\rho}}\\*
 \shoveright{+\frac{\lfnorm*{0}{-\rho}{\theta}{\T(\Delta f)}(\lsnorm*{1}{0}{\theta}{\T f_1}+\lsnorm*{1}{0}{\theta}{\T f_2})}{(q+n+2)^{2\theta+\rho}(q+n+1)^{1-\rho}}\Biggr)}\\*
 \shoveleft{\phantom{\abs*{J_{3,1,1}}\leq}+\frac{1}{q+1}\Biggl(\frac{\lsnorm*{1}{0}{\theta}{\T((1-\ee^{-n\cdot})(\Delta f))}(\lfnorm*{0}{-\rho}{\theta}{\T f_1}+\lfnorm*{0}{-\rho}{\theta}{\T f_2})}{(q+2)^{\theta+\rho}(q+1)^{1-\rho}(q+n+2)^{\theta}}}\\*
 +\frac{\lfnorm*{0}{-\rho}{\theta}{\T((1-\ee^{-n\cdot})(\Delta f))}(\lsnorm*{1}{0}{\theta}{\T f_1}+\lsnorm*{1}{0}{\theta}{\T f_2})}{(q+2)^{\theta}(q+n+2)^{\theta+\rho}(q+n+1)^{1-\rho}}\\*
 +\frac{\lsnorm*{1}{0}{\theta}{\T(\Delta f)}(\lfnorm*{0}{-\rho}{\theta}{\T((1-\ee^{-n\cdot}) f_1)}+\lfnorm*{0}{-\rho}{\theta}{\T((1-\ee^{-n\cdot}) f_2)})}{(q+n+2)^{\theta+\rho}(q+n+1)^{1-\rho}(q+2)^{\theta}}\\*
 \shoveright{+\frac{\lfnorm*{0}{-\rho}{\theta}{\T(\Delta f)}(\lsnorm*{1}{0}{\theta}{\T((1-\ee^{-n\cdot}) f_1)}+\lsnorm*{1}{0}{\theta}{\T((1-\ee^{-n\cdot}) f_2)}}{(q+n+2)^{\theta}(q+2)^{\theta+\rho}(q+1)^{1-\rho}}\Biggr)}\\*
 +\frac{1}{q+1}\Biggl(\frac{\lsnorm*{1}{0}{\theta}{\T((1-\ee^{-n\cdot})(\Delta f))}(\lfnorm*{0}{-\rho}{\theta}{\T((1-\ee^{-n\cdot}) f_1)}+\lfnorm*{0}{-\rho}{\theta}{\T((1-\ee^{-n\cdot}) f_2)})}{(q+2)^{2\theta+\rho}(q+1)^{1-\rho}}\\*
 +\frac{\lfnorm*{0}{-\rho}{\theta}{\T((1-\ee^{-n\cdot})(\Delta f))}(\lsnorm*{1}{0}{\theta}{\T((1-\ee^{-n\cdot}) f_1)}+\lsnorm*{1}{0}{\theta}{\T((1-\ee^{-n\cdot}) f_2)})}{(q+2)^{2\theta+\rho}(q+1)^{1-\rho}}\Biggr).
\end{multline}
Together with \cref{Lem:norm:shift,Lem:norm:shift:by:shift,Prop:norm:boundedness} this simplifies as
\begin{equation*}
 \begin{split}
  \abs*{J_{3,1,1}}&\leq C\Bigl(\lfnorm*{1}{0}{\theta}{\T f_1}+\lfnorm*{1}{0}{\theta}{\T f_2}+\lfnorm*{1}{0}{\theta}{\T((1-\ee^{-n\cdot}) f_1)}+\lfnorm*{1}{0}{\theta}{\T((1-\ee^{-n\cdot}) f_2)}\Bigr)\cdot\\
  &\qquad\qquad\cdot \biggl(\frac{\lfnorm*{1}{0}{\theta}{\T(\Delta f)}}{(q+1)^{2+\theta}n^{\theta}}+\frac{\lfnorm*{1}{0}{\theta}{\T((1-\ee^{-n\cdot})(\Delta f))}}{(q+1)^{2+2\theta}}\biggr)\\
  &\leq \frac{C}{n^{\theta}}\lfnorm*{1}{0}{\theta}{\T(\Delta f)}(q+1)^{-2-\theta}+C(n)\lfnorm*{1}{0}{\theta}{\T((1-\zeta)(\Delta f))}(q+1)^{-2-\theta}.
 \end{split}
\end{equation*}
For a given $\delta>0$ we then fix $n_{\delta}$ large enough such that it holds
\begin{equation}\label{eq:bd:layer:est:J:31:1}
 \abs*{J_{3,1,1}}\leq \biggl(\frac{\delta}{2}\lfnorm*{1}{0}{\theta}{\T(\Delta f)}+C_{n_{\delta}}\lfnorm*{1}{0}{\theta}{\T((1-\zeta)(\Delta f))}\biggr)\frac{1}{(q+1)^{2+\theta}}.
\end{equation}
Since the expression $J_{3,1,2}$ appears already with a factor $\eps$ this is slightly simpler to estimate and we find, by expanding the derivative $(\del_{\xi}+\del_{\eta})^{2}$, that it holds
\begin{multline*}
 \abs*{J_{3,1,2}}=\frac{\eps}{2(q+1)}\int_{0}^{\infty}\int_{0}^{\infty}\abs*{\Ker(\xi,\eta)}\Biggl(\frac{\lsnorm*{2}{0}{\theta}{\T(\Delta f)}(\lfnorm*{0}{-\rho}{\theta}{\T f_1}+\lfnorm*{0}{-\rho}{\theta}{\T f_2})}{(\xi+q+2)^{\theta+\rho}(\xi+q+1)^{2-\rho}(\eta+q+2)^{\theta}}\\*
 +2\frac{\lsnorm*{1}{0}{\theta}{\T(\Delta f)}(\lsnorm*{1}{0}{\theta}{\T f_1}+\lsnorm*{1}{0}{\theta}{\T f_2})}{(\xi+q+2)^{\theta+\rho}(\xi+q+1)^{1-\rho}(\eta+q+2)^{\theta+\rho}(\eta+q+1)^{1-\rho}}\\*
 +\frac{\lfnorm*{0}{-\rho}{\theta}{\T(\Delta f)}(\lsnorm*{2}{0}{\theta}{\T f_1}+\lsnorm*{2}{0}{\theta}{\T f_2})}{(\xi+q+2)^{\theta}(\eta+q+2)^{\theta+\rho}(\eta+q+1)^{2-\rho}}\Biggr).
 \end{multline*}
Together with~\eqref{eq:est:weight} this can be further simplified such that we get
\begin{equation*}
 \begin{split}
  \abs*{J_{3,1,2}}&\leq C\eps\frac{\lfnorm*{2}{0}{\theta}{\T(\Delta f)}(\lfnorm*{2}{0}{\theta}{\T f_1}+\lfnorm*{2}{0}{\theta}{\T f_2})}{q+1}\cdot\\
  &\cdot \int_{0}^{\infty}\int_{0}^{\infty}\abs*{\Ker(\xi,\eta)}\Bigl((\xi+q+1)^{-2-\theta}(\eta+q+1)^{-\theta}\\
  &+(\xi+q+1)^{-1-\theta}(\eta+q+1)^{-1-\theta}+(\xi+q+1)^{-\theta}(\eta+q+1)^{-2-\theta}\Bigr)\deta\dxi.
 \end{split}
\end{equation*}
The change of variables $\xi\mapsto \xi(q+1)$ and $\eta\mapsto \eta(q+1)$ together with the homogeneity of $\Ker$ as well as \cref{Prop:norm:boundedness,Lem:Ker:est:most:general} then yields
\begin{multline*}
  \abs*{J_{3,1,2}}\leq C\eps\frac{\lfnorm*{2}{0}{\theta}{\T(\Delta f)}}{(q+1)^{2+2\theta}}\int_{0}^{\infty}\int_{0}^{\infty}\abs*{\Ker(\xi,\eta)}\Bigl((\xi+1)^{-2-\theta}(\eta+1)^{-\theta}\\*
  +(\xi+1)^{-1-\theta}(\eta+1)^{-1-\theta}+(\xi+1)^{-\theta}(\eta+1)^{-2-\theta}\Bigr)\deta\dxi\\
  \leq C\eps \lfnorm*{2}{0}{\theta}{\T(\Delta f)}(q+1)^{-2-\theta}.
 \end{multline*}
We then recall~\eqref{eq:bd:layer:est:J:31:1} such that the claimed estimate~\eqref{eq:bd:layer:est:J:31} follows if we choose $\eps>0$ sufficiently small.

\subsubsection{Estimate for $K_{4,1}$}
 
In this subsection, we will show that it holds
\begin{equation}\label{eq:bd:layer:est:K41}
 \eps\abs*{K_{4,1}}\leq C\eps \lfnorm*{2}{0}{\theta}{\T(\Delta f)}(q+1)^{-2-\theta}.
\end{equation}
We first recall the representation~\eqref{eq:beta:rewritten} for $\beta_{W}$ to find after some elementary manipulations that $K_{4,1}$ can be rewritten as
\begin{equation}\label{eq:bd:layer:rep:K41}
 \begin{split}
  K_{4,1}&=-\frac{1}{q+1}\int_{0}^{\infty}\int_{0}^{\infty}\Ker(\xi,\eta)(\T f_2)(\eta)\bigl(\T((1-\zeta)\Delta f)\bigr)''(q+1+\xi)\dxi\deta\\
  &\phantom{{}={}}-\frac{1}{q+1}\int_{0}^{\infty}\int_{0}^{\infty}\Ker(\xi,\eta)(\T f_2)'(\eta)\bigl(\T((1-\zeta)\Delta f)\bigr)'(q+1+\xi)\dxi\deta.
 \end{split}
\end{equation}
To estimate the right-hand side, we note that~\cref{Rem:est:difference:Laplace,eq:est:weight} imply
\begin{equation*}
 \abs[\big]{\bigl(\T((1-\zeta)\Delta f)\bigr)'(q+1+\xi)}\leq C\lsnorm*{2}{0}{\theta}{\T(\Delta f)}(\xi+q+1)^{-2-\theta}.
\end{equation*}
Thus, together with~\eqref{eq:difference:simple} we conclude from~\eqref{eq:bd:layer:rep:K41} that
\begin{multline*}
 \abs*{K_{4,1}}\leq \frac{C}{q+1}\lfnorm*{0}{-\rho}{\theta}{\T f_2}\lsnorm*{2}{0}{\theta}{\T(\Delta f)}\int_{0}^{\infty}\int_{0}^{\infty}\frac{\abs*{\Ker(\xi,\eta)}}{(\xi+q+2)^{\theta+\rho}(\xi+q+1)^{2-\rho}(\eta+1)^{\theta}}\deta\dxi\\*
 +\frac{C}{q+1}\lsnorm*{1}{0}{\theta}{\T f_{2}}\lsnorm*{2}{0}{\theta}{\T(\Delta f)}\int_{0}^{\infty}\int_{0}^{\infty}\frac{\abs*{\Ker(\xi,\eta)}}{(\xi+q+1)^{2+\theta}(\eta+1)^{\theta+\rho}\eta^{1-\rho}}\deta\dxi.
\end{multline*}
We proceed now similarly as in Section~\ref{Sec:est:J31}, i.e.\@ we estimate $(\xi+q+2)^{-\theta-\rho}\leq (\xi+q+1)^{-\theta-\rho}$ and change variables $\xi\mapsto \xi(q+1)$ and $\eta\mapsto \eta(q+1)$ such that the homogeneity of $\Ker$ and Proposition~\ref{Prop:norm:boundedness} imply
\begin{multline*}
 \abs*{K_{4,1}}\leq \frac{C}{(q+1)^{2+\theta}}\lsnorm*{2}{0}{\theta}{\T(\Delta f)}\int_{0}^{\infty}\int_{0}^{\infty}\frac{\abs*{\Ker(\xi,\eta)}}{(\xi+1)^{2+\theta}(\eta(q+1)+1)^{\theta}}\deta\dxi\\*
 +\frac{C}{(q+1)^{3+\theta-\rho}}\lsnorm*{2}{0}{\theta}{\T(\Delta f)}\int_{0}^{\infty}\int_{0}^{\infty}\frac{\abs*{\Ker(\xi,\eta)}}{(\xi+1)^{2+\theta}(\eta(q+1)+1)^{\theta+\rho}\eta^{1-\rho}}\deta\dxi.
\end{multline*}
To estimate the first integral on the right-hand side, we note that $(\eta(q+1)+1)^{-\theta}\leq (\eta+1)^{-\theta}$ and we use Lemma~\ref{Lem:Ker:est:most:general}, while for the second one, we similarly use $(\eta(q+1)+1)^{-\theta-\rho}\leq (\eta+1)^{-\theta-\rho}$ such that Lemma~\ref{Lem:Ker:est:large:1} applies in the special case $r=1$. Together it then follows
\begin{equation*}
 \abs*{K_{4,1}}\leq C\lfnorm*{2}{0}{\theta}{\T(\Delta f)}(q+1)^{-2-\theta}
\end{equation*}
because we have $\rho<1$. From this, we can immediately deduce the claimed estimate~\eqref{eq:bd:layer:est:K41}.

\subsubsection{A representation formula for $H_{0}$}\label{Sec:rep:H0}

In order to estimate the expressions $J_{3,0}$ and $K_{4,0}$, we have to write the function $H_{0}(\cdot,q)$ as Laplace transform of a certain function $\Qo_{0}(\cdot,q)$. The next proposition states that such a representation is in fact possible for all $q\geq 0$.

\begin{proposition}\label{Prop:rep:H0}
 For each $q\geq 0$ there exists a function $\Qo_{0}(\cdot,q)$ such that $H_{0}(\cdot,q)$, which is defined in~\eqref{eq:H:splitting:2}, is the Laplace transform of $\Qo_{0}(\cdot,q)$, i.e.\@ it holds
 \begin{equation*}
  H_{0}(y,q)=\int_{0}^{\infty}\Qo_{0}(\xi,q)\ee^{-\xi y}\dxi.
 \end{equation*}
 The function $\Qo_{0}(\cdot,q)$ is given by the representation formula
 \begin{equation*}
  \Qo_{0}(\xi,q)=\frac{1}{2\pi \im}\lim_{R\to \infty}\int_{-\im R}^{\im R}H_{0}(y,q)\ee^{y\xi}\dy
 \end{equation*}
 where this expression has to be understood as a limit in $L^2$, i.e.\@ $\Qo_{0}(\cdot,q)\in L^{2}(\R)$ for all $q\geq 0$.
\end{proposition}

Moreover, we will need a suitable estimate for the representation kernel $\Qo_{0}(\cdot,q)$ which will be provided by the following proposition.
 
 \begin{proposition}\label{Prop:Q0:int:est}
  For $\Qo_{0}(\cdot ,q)$ as given by Proposition~\ref{Prop:rep:H0} and $\mathfrak{a}\in[0,2\theta]$, $\mathfrak{b}\in[2\theta,\infty)$ there exists a parameter $\nu_{*}>0$ and a constant $C=C_{\nu_{*},\blp}>0$ such that it holds for sufficiently small $\eps>0$ that
  \begin{equation*}
   \int_{0}^{\infty}\weight{-\mathfrak{a}}{\mathfrak{b}-\nu}(\xi)\abs*{\Qo_{0}(\xi,q)}\dxi\leq C (q+1)^{-2-\theta}
  \end{equation*}
 for every $\nu\in[0,\nu_{*}]$ and all $q\geq 0$.
 \end{proposition}

 With these two results, we will now continue to estimate the expressions $J_{3,0}$ and $K_{4,0}$ while the rather technical proofs of \cref{Prop:rep:H0,Prop:Q0:int:est} are postponed to \cref{Sec:H0:representation,Sec:Q0:est} respectively.
 
 \subsubsection{Estimate for $J_{3,0}$}
 In this subsection, we will show that it holds
 \begin{equation}\label{eq:bd:layer:est:J30}
  \abs*{J_{3,0}}\leq \Bigl(\delta \lfnorm*{2}{0}{\theta}{\T(\Delta f)}+C_{\delta}\lfnorm*{1}{0}{\theta}{\T\bigl((1-\zeta)(\Delta f)\bigr)}\Bigr)(q+1)^{-2-\theta}.
 \end{equation}
 We first note that the splitting $K=2+\eps W$ together with the symmetry of the integrand and Proposition~\ref{Prop:rep:H0} allow to rewrite $J_{3,0}$ as 
 \begin{equation*}
  J_{3,0}=J_{3,0,1}+J_{3,0,2}
 \end{equation*}
 with 
 \begin{equation*}
  \begin{split}
   J_{3,0,1}&\vcc=\int_{0}^{\infty}\Qo_{0}(\xi,q)\int_{0}^{\infty}\int_{0}^{\infty}\ee^{-\xi(y+z)}(1+y+z)\Delta f(y)\bigl(f_{1}(z)+f_{2}(z)\bigr)\dz\dy\dxi,\\
   J_{3,0,2}&\vcc=\frac{\eps}{2}\int_{0}^{\infty}\Qo_{0}(\xi,q)\int_{0}^{\infty}\int_{0}^{\infty}\ee^{-\xi(y+z)}(1+y+z)W(y,z)\Delta f(y)\bigl(f_{1}(z)+f_{2}(z)\bigr)\dz\dy\dxi.
  \end{split}
 \end{equation*}
 We estimate the two expressions again individually while we start with $J_{3,0,1}$ and the approach will be the same as in Section~\ref{Sec:est:J31}. More precisely, this means that one introduces the factor~\eqref{eq:bd:layer:exp:splitting}, estimates the resulting expression by the norms as in~\eqref{eq:bd:layer:est:J31:0} such that together with \cref{Lem:norm:shift,Lem:norm:shift:by:shift,Prop:norm:boundedness} one obtains
 \begin{multline*}
  \abs*{J_{3,0,1}}\leq C\lfnorm*{1}{0}{\theta}{\T(\Delta f)}\int_{0}^{\infty}\abs*{\Qo_{0}(\xi,q)}\cdot\\*
  \cdot\biggl(\frac{1}{(\xi+n)^{2\theta}}+\frac{1}{(\xi+n+1)^{\theta}(\xi+1)^{\theta}}+\frac{1}{(\xi+n+1)^{\theta}(\xi+1)^{\theta+\rho}\xi^{1-\rho}}\biggr)\dxi\\*
  \shoveleft{\quad+C(n)\lfnorm*{1}{0}{\theta}{\T\bigl((1-\zeta)(\Delta f)\bigr)}\int_{0}^{\infty}\abs*{\Qo_{0}(\xi,q)}\cdot}\\*
  \cdot\biggl(\frac{1}{(\xi+1)^{\theta}(\xi+n+1)^{\theta}}+\frac{1}{(\xi+n+1)^{\theta}(\xi+1)^{\theta+\rho}\xi^{1-\rho}}\\*
  +\frac{1}{(\xi+n)^{2\theta}}+\frac{1}{(\xi+1)^{2\theta+\rho}\xi^{1-\rho}}\biggr)\dxi.
 \end{multline*}
For a small parameter $\nu>0$ in the sense of Proposition~\ref{Prop:Q0:int:est}, we may thus further estimate
\begin{multline*}
 \abs*{J_{3,0,1}}\leq \frac{C}{n^{\nu}}\lfnorm*{1}{0}{\theta}{\T(\Delta f)}\biggl(\int_{0}^{\infty}\frac{\abs*{\Qo_{0}(\xi,q)}}{(\xi+1)^{2\theta-\nu}}+\frac{\abs*{\Qo_{0}(\xi,q)}}{(\xi+1)^{2\theta+\rho-\nu}\xi^{1-\rho}}\dxi\biggr)\\*
 +C(n)\lfnorm*{1}{0}{\theta}{\T\bigl((1-\zeta)(\Delta f)\bigr)}\biggl(\int_{0}^{\infty}\frac{\abs*{\Qo_{0}(\xi,q)}}{(\xi+1)^{2\theta}}+\frac{\abs*{\Qo_{0}(\xi,q)}}{(\xi+1)^{2\theta+\rho}\xi^{1-\rho}}\dxi\biggr).
\end{multline*}
Thus, the integrals on the right-hand side can be estimated due to Proposition~\ref{Prop:Q0:int:est} and it follows for any given $\delta>0$ by choosing $n=n_{\delta}\in\N$ sufficiently large that it holds
\begin{equation}\label{eq:bd:layer:est:J30:1}
 \abs*{J_{3,0,1}}\leq \biggl(\frac{\delta}{2}\lfnorm*{1}{0}{\theta}{\T(\Delta f)}+C(n_{\delta})\lfnorm*{1}{0}{\theta}{\T\bigl((1-\zeta)(\Delta f)\bigr)}\biggr)\frac{1}{(q+1)^{2+\theta}}.
\end{equation}
To estimate $J_{3,0,2}$ we note that by means of Proposition~\ref{Prop:W:representation} we can rewrite
\begin{equation*}
 J_{3,0,2}=-\frac{\eps}{2}\int_{0}^{\infty}\int_{0}^{\infty}\Ker(\sigma,\eta)\int_{0}^{\infty}\Qo_{0}(\xi,q)(1-\del_{\xi})\del_{\xi}\Bigl(\T(\Delta f)(\xi+\sigma)\bigl(\T(f_1+f_2)\bigr)(\xi+\eta)\Bigr)\dxi\deta\dsig.
\end{equation*}
This then gives the estimate
\begin{multline*}
 \abs*{J_{3,0,2}}\leq \frac{\eps}{2}\int_{0}^{\infty}\int_{0}^{\infty}\int_{0}^{\infty}\abs*{\Ker(\sigma,\eta)}\abs*{\Qo_{0}(\xi,q)}\Biggl(\frac{\lsnorm*{1}{0}{\theta}{\T(\Delta f)}(\lfnorm*{0}{-\rho}{\theta}{\T f_1}+\lfnorm*{0}{-\rho}{\theta}{\T f_2})}{(\xi+\sigma+1)^{\theta+\rho}(\xi+\sigma)^{1-\rho}(\xi+\eta+1)^{\theta}}\\*
 +\frac{\lfnorm*{0}{-\rho}{\theta}{\T(\Delta f)}(\lsnorm*{1}{0}{\theta}{\T f_1}+\lsnorm*{1}{0}{\theta}{\T f_2})}{(\xi+\sigma+1)^{\theta}(\xi+\eta+1)^{\theta+\rho}(\xi+\eta)^{1-\rho}}+\frac{\lsnorm*{2}{0}{\theta}{\T(\Delta f)}(\lfnorm*{0}{-\rho}{\theta}{\T f_1}+\lfnorm*{0}{-\rho}{\theta}{\T f_2})}{(\xi+\sigma+1)^{\theta+\rho}(\xi+\sigma)^{2-\rho}(\xi+\eta+1)^{\theta}}\\*
 +\frac{\lsnorm*{1}{0}{\theta}{\T(\Delta f)}(\lsnorm*{1}{0}{\theta}{\T f_1}+\lsnorm*{1}{0}{\theta}{\T f_2})}{(\xi+\sigma+1)^{\theta+\rho}(\xi+\sigma)^{1-\rho}(\xi+\eta+1)^{\theta+\rho}(\xi+\eta)^{1-\rho}}\\*
 +\frac{\lfnorm*{0}{-\rho}{\theta}{\T(\Delta f)}(\lsnorm*{2}{0}{\theta}{\T f_1}+\lsnorm*{2}{0}{\theta}{\T f_2})}{(\xi+\sigma+1)^{\theta}(\xi+\eta+1)^{\theta+\rho}(\xi+\eta)^{2-\rho}}\Biggr)\dxi\deta\dsig.
\end{multline*}
With Proposition~\ref{Prop:norm:boundedness} and the elementary estimates $(\xi+\tau+1)^{-\theta-\rho}(\xi+\tau)^{\rho-1}\leq (\xi+\tau+1)^{-\frac{\theta+\rho}{2}}(\xi+\tau)^{-\frac{\theta+\rho}{2}+\rho-1}$ for $\tau=\sigma$ and $\tau=\eta$ and furthermore $(\xi+\sigma+1)^{-\frac{\theta+\rho}{2}}(\xi+\eta+1)^{-\frac{\theta+\rho}{2}}\leq (\xi+1)^{-\theta-\rho}$ and $(x+\tau)^{-a}\leq x^{-a}$ for $x,\tau,a>0$ we can further simplify the previous estimate to get
\begin{multline*}
 \abs*{J_{3,0,2}}\leq C\eps \lfnorm*{2}{0}{\theta}{\T(\Delta f)}\int_{0}^{\infty}\int_{0}^{\infty}\int_{0}^{\infty}\abs*{\Ker(\sigma,\eta)}\abs*{\Qo_{0}(\xi,q)}\Biggl(\frac{1}{(\xi+1+\sigma)^{\theta+\rho}\sigma^{1-\rho}(\xi+1+\eta)^{\theta}}\\*
 +\frac{1}{(\xi+1+\sigma)^{\theta}(\xi+1+\eta)^{1-\rho}\eta^{1-\rho}}+\frac{1}{(\xi+1)^{\theta+\rho}(\xi+\sigma)^{2-\rho}(\xi+\eta)^{\theta}}\\*
 +\frac{1}{(\xi+1)^{\theta+\rho}(\xi+\sigma)^{\frac{\theta+\rho}{2}+1-\rho}(\xi+\eta)^{\frac{\theta+\rho}{2}+1-\rho}}+\frac{1}{(\xi+1)^{\theta+\rho}(\xi+\sigma)^{\theta}(\xi+\eta)^{2-\rho}}\Biggr)\dxi\deta\dsig.
\end{multline*}
In the first two integrals on the right-hand side we change now variables $\sigma\mapsto (\xi+1)\sigma$ and $\eta\mapsto (\xi+1)\eta$ while in the remaining ones, we use $\sigma \mapsto \xi\sigma$ and $\eta\mapsto \xi\eta$ which yields together with \cref{Lem:kernel:est:small:1,Lem:Ker:est:most:general} that 
 \begin{equation*}
  \abs*{J_{3,0,2}}\leq C\eps\lfnorm*{2}{0}{\theta}{\T(\Delta f)}\biggl(\int_{0}^{\infty}\frac{\abs*{\Qo_{0}(\xi,q)}}{(\xi+1)^{2\theta}}+\int_{0}^{\infty}\frac{\abs*{\Qo_{0}(\xi,q)}}{(\xi+1)^{1+\theta-\rho}(\xi+1)^{\theta+\rho}}\biggr).
 \end{equation*}
From Proposition~\ref{Prop:Q0:int:est} we thus conclude that
 \begin{equation*}
  \abs*{J_{3,0,2}}\leq C\eps\lfnorm*{2}{0}{\theta}{\T(\Delta f)}(q+1)^{-2-\theta}.
 \end{equation*}
 Combining this with~\eqref{eq:bd:layer:est:J30:1}, we note that for $\eps>0$ small enough the claimed estimate~\eqref{eq:bd:layer:est:J30} holds.
 
 \subsubsection{Estimate for $K_{4,0}$}
 
 In this subsection we will show that it holds
 \begin{equation}\label{eq:bd:layer:est:K40}
  \eps\abs*{K_{4,0}}\leq C\eps\lfnorm*{2}{0}{\theta}{\T(\Delta f)}(q+1)^{-2-\theta}.
 \end{equation}
 For this, we first rewrite $K_{4,0}$ by means of \cref{Prop:rep:H0,eq:beta:rewritten} to obtain after some rearrangement that it holds
 \begin{equation*}
  \begin{split}
   K_{4,0}=&-\int_{0}^{\infty}\int_{0}^{\infty}\Ker(\sigma,\eta)(\T f_{2})(\eta)\int_{0}^{\infty}\Qo_{0}(\xi,q)(1-\del_{\xi})\del_{\xi}\bigl(\T((1-\zeta)\Delta f)(\xi+q)\bigr)\dxi\dsig\deta\\
   & -\int_{0}^{\infty}\int_{0}^{\infty}\Ker(\sigma,\eta)(\T f_{2})'(\eta)\int_{0}^{\infty}\Qo_{0}(\xi,q)(1-\del_{\xi})\bigl(\T((1-\zeta)\Delta f)(\xi+q)\bigr)\dxi\dsig\deta.
  \end{split}
 \end{equation*}
 Together with \cref{Rem:est:difference:Laplace,Lem:norm:shift}, we can bound the previous expression as
 \begin{equation*}
  \begin{split}
   \abs*{K_{4,0}}&\leq C\int_{0}^{\infty}\int_{0}^{\infty}\int_{0}^{\infty}\abs*{\Ker(\sigma,\eta)}\abs*{\Qo_{0}(\xi,q)}\frac{\lfnorm*{0}{-\rho}{\theta}{\T f_{2}}}{(\eta+1)^{\theta}}\frac{\lsnorm*{2}{0}{\theta}{\T(\Delta f)}}{(\xi+\sigma+1)^{\theta+\rho}(\xi+\sigma)^{2-\rho}}\dxi\deta\dsig\\
   &+C\int_{0}^{\infty}\int_{0}^{\infty}\int_{0}^{\infty}\abs*{\Ker(\sigma,\eta)}\abs*{\Qo_{0}(\xi,q)}\frac{\lsnorm*{1}{0}{\theta}{\T f_{2}}}{(\eta+1)^{\theta+\rho}\eta^{1-\rho}}\frac{\lsnorm*{1}{0}{\theta}{\T(\Delta f)}}{(\xi+\sigma+1)^{\theta+\rho}(\xi+\sigma)^{1-\rho}}\dxi\deta\dsig.
  \end{split}
 \end{equation*}
 From the choice of $\theta$ in~\eqref{eq:choice:theta:bl} one immediately verifies that it holds $(\xi+\sigma+1)^{-\theta-\rho}(\xi+\sigma)^{\rho-2}\leq (\sigma+1)^{-\theta-\rho}\sigma^{2\theta+\rho-2}\xi^{-2\theta}$ and $(\xi+\sigma+1)^{-\theta-\rho}(\xi+\sigma)^{\rho-1}\leq (\sigma+1)^{\theta-1}(\xi+1)^{1-\rho-2\theta}\xi^{\rho-1}$. Thus, together with Proposition~\ref{Prop:norm:boundedness} we deduce that
 \begin{multline*}
   \abs*{K_{4,0}}\leq C\lfnorm*{2}{0}{\theta}{\T(\Delta f)}\biggl(\int_{0}^{\infty}\frac{\abs*{\Qo_{0}(\xi,q)}}{\xi^{2\theta}}\int_{0}^{\infty}\int_{0}^{\infty}\frac{\abs*{\Ker(\sigma,\eta)}}{(\eta+1)^{\theta}(\sigma+1)^{\theta+\rho}\sigma^{2-2\theta-\rho}}\dsig\deta\dxi\\*
   +\int_{0}^{\infty}\frac{\abs*{\Qo_{0}(\xi,q)}}{\xi^{1-\rho}(\xi+1)^{2\theta+\rho-1}}\int_{0}^{\infty}\int_{0}^{\infty}\frac{\abs*{\Ker(\sigma,\eta)}}{(\eta+1)^{\theta+\rho}\eta^{1-\rho}(\sigma+1)^{1-\theta}}\dsig\deta\dxi\biggr).
  \end{multline*}
 Due to the choice of $\theta$ in~\eqref{eq:choice:theta:bl} the integrals in $\sigma$ and $\eta$ can be estimated uniformly by some constant according to Lemma~\ref{Lem:Ker:est:most:general}. On the other hand, Proposition~\ref{Prop:Q0:int:est} yields an estimate for the integral in $\xi$ such that we conclude $\abs*{K_{4,0}}\leq C\lfnorm*{2}{0}{\theta}{\T(\Delta f)}(q+1)^{-2-\theta}$. From this, \eqref{eq:bd:layer:est:K40} directly follows.
 
 \subsubsection{Conclusion of the proof}
 
  From~\cref{eq:bd:layer:splitting:1,eq:bd:layer:splitting:2} we recall that
  \begin{equation*}
   \Bigl(\T\bigl(\zeta (f_{1}-f_{2})\bigr)\Bigr)''(q)=-\eps (K_1+K_2+K_3+K_{4,1}+K_{4,0})+(J_1+J_2+J_{3,1}+J_{3,0}).
  \end{equation*}
  If we collect now the estimates~\cref{eq:bd:layer:est:K1:J1,eq:bd:layer:est:K2:J2,eq:bd:layer:est:K3,eq:bd:layer:est:J:31,eq:bd:layer:est:K41,eq:bd:layer:est:J30,eq:bd:layer:est:K40} the claim of Proposition~\ref{Prop:boundary:layer} follows.

\section{The representation formula for $H_{0}(\cdot,q)$}\label{Sec:H0:representation}

This section is devoted to the proof of Proposition~\ref{Prop:rep:H0} which will be essentially a consequence of the Paley-Wiener Theorem. To formulate this statement, we need some notation while we follow~\cite{Yos78} and denote by $H^2(0)$ the Hardy-Lebesgue class which consists of all functions $\varphi$ satisfying
\begin{enumerate}
 \item $\varphi$ is holomorphic in the right half-plane,
 \item for each fixed value $x>0$ the function $z\mapsto \varphi(x+\im z)$ is contained in $L^{2}(\R)$ and satisfies
 \begin{equation}\label{eq:cond:Hardy:Lebesgue}
  \sup_{x>0}\biggl(\int_{-\infty}^{\infty}\abs*{\varphi(x+\im z)}^2\dz\biggr)<\infty.
 \end{equation}
 Under these conditions we have the Paley-Wiener Theorem which is shown in~\cite[Ch.\@ VI, 4, Theorem~2]{Yos78} and can be rephrased as follows.
 \begin{theorem}\label{Thm:Paley:Wiener}
  For each $\varphi\in H^2(0)$ the function $z\mapsto\varphi(\im z)$ exists in $L^{2}(\R)$ in the sense that 
  \begin{equation*}
   \lim_{x\downarrow 0}\int_{-\infty}^{\infty}\abs*{\varphi(\im z)-\varphi(x+\im z)}^2\dz=0.
  \end{equation*}
  Moreover, the inverse Fourier transform
 \begin{equation*}
  \psi(t)=\frac{1}{2\pi}\lim_{N\to \infty}\int_{-N}^{N}\varphi(\im z)\ee^{\im t z}\dz
 \end{equation*}
 vanishes for $t<0$ and $\varphi$ can be represented as the Laplace transform of $\psi$.
 \end{theorem}
\end{enumerate}

\subsection{Analyticity properties}

The general strategy will be to apply Theorem~\ref{Thm:Paley:Wiener}, in order to construct the representation kernel $\Qo_{0}(\cdot,q)$. However, the conditions in this statement require certain analyticity properties for the function $H_{0}(\cdot,q)$ to hold. As one can see from the definition of $H_{0}$ in~\cref{eq:H:splitting:1,eq:H:splitting:2}, this boils down to verify analyticity of the functions $\beta_{W}(\cdot,f_j)$ and $\Phi_{j}(\cdot)$. Since we will consider in the following exclusively the case of a fixed self-similar profile $f_{j}$, we drop all indices indicating a specific profile and just write $\kappa$, $\beta_{W}(\cdot)$ and $\Phi(\cdot)$ instead of $\kappa_{j}$, $\beta_{W}(\cdot,f_j)$ and $\Phi_{j}(\cdot)$. We also emphasise that the results in this subsection are universal in the sense that they are more or less independent of the self-similar profiles. More precisely this means, that we rely heavily on the analyticity properties of the kernel $W$, while for the self-similar profiles we essentially only need the uniform boundedness of certain moments which is guaranteed by Lemma~\ref{Lem:moment:est:2}. Therefore, the following proofs are quite similar to the corresponding ones in~\cite{NTV15}.

 \begin{figure}
  \centering
  \begin{minipage}{0.48\textwidth}
    \centering
     \begin{tikzpicture}[contour/.style={postaction={decorate, decoration={markings,
mark=at position 1.7cm with {\arrow[line width=1pt]{>}}
}}},
contourtwo/.style={postaction={decorate, decoration={markings,
mark=at position 1.5cm with {\arrow[line width=1pt]{>}}
}}},
contourthree/.style={postaction={decorate, decoration={markings,
mark=at position 2cm with {\arrow[line width=1pt]{>}}
}}},
    interface/.style={postaction={draw,decorate,decoration={border,angle=45,
                    amplitude=0.3cm,segment length=2mm}}},
]
\draw[->] (0,0) -- (3,0) node[below] {$\Re$};
\draw[->] (0,-2.5) -- (0,4) node[left] {$\Im$};

\draw[line width=.8pt,interface](0,0)--(-3,0);

\path[draw,line width=0.8pt,contour] (-2,1)  -- (-2,3.5);
\path[draw,dashed,line width=0.8pt,contourthree] (1.5,0)  -- (1.5,3.5);
\path[draw,dashed,line width=0.8pt,contourtwo] (1.5,0)  -- (1.5,-2);

\draw[black,fill=black] (-2,1) circle (.4ex);
\draw[black, fill=black] (1.5,0) circle (.4ex);
\node[below right] at (-2,1) {$x$};
\node[below right] at (1.5,0) {$z$};
\node[right] at (-2,2) {$\gamma_{x}$};
\node[right] at (1.5,1.25) {$\gamma_{z}$};
\node[right] at (1.5,-1.5) {$\gamma_{z}$};
\end{tikzpicture}
\subcaption{Path for $\int_{x}^{\sgn(\Im x)\im\infty}(\cdots)\dt$}
  \end{minipage}
  \hfill
  \begin{minipage}{0.51\textwidth}
    \centering
     \begin{tikzpicture}[contour/.style={postaction={decorate, decoration={markings,
mark=at position 4.5cm with {\arrow[line width=1pt]{>}},
mark=at position 2.5cm with {\arrow[line width=1pt]{>}}
}}},
    interface/.style={postaction={draw,decorate,decoration={border,angle=45,
                    amplitude=0.3cm,segment length=2mm}}},
]
\draw[->] (0,0) -- (3,0) node[below] {$\Re$};
\draw[->] (0,-2.5) -- (0,4) node[left] {$\Im$};

\draw[line width=.8pt,interface](0,0)--(-3.8,0);

\path[draw,line width=0.8pt,contour] (-3,1.5)  -- (2.5,1.5);

\draw[black,fill=black] (-3,1.5) circle (.4ex);
\node[below right] at (-3,1.5) {$x$};
\node[above] at (-1.3,1.5) {$\gamma_{x}$};
\end{tikzpicture}
      \subcaption{Path for $\int_{x}^{\infty}(\cdots)\dt$}
  \end{minipage}
  \caption{Integration path $\gamma_{x}$}
  \label{fig:gamma}
\end{figure}
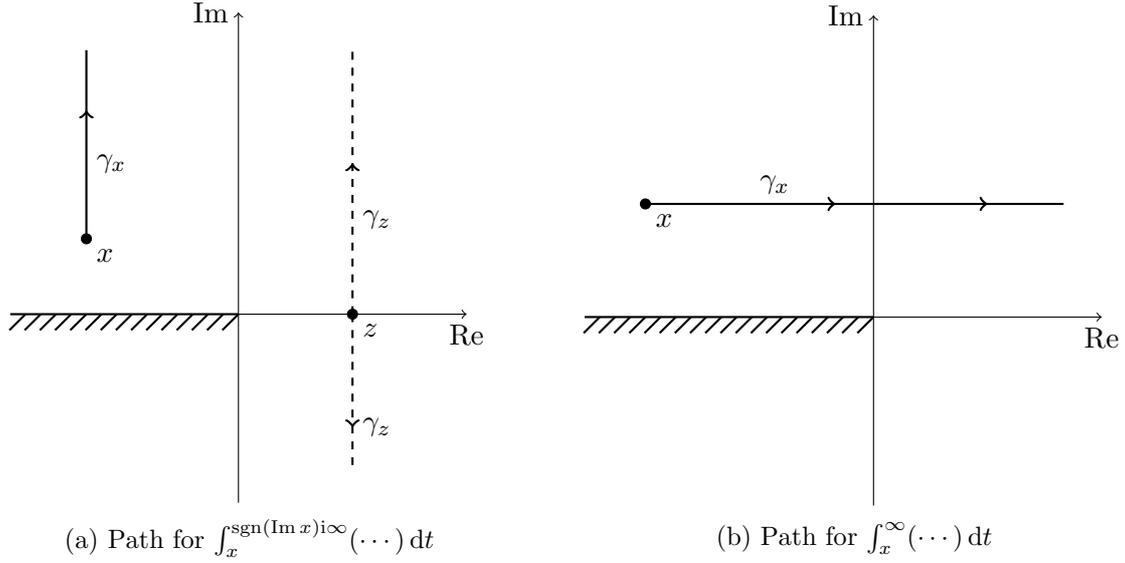

\begin{lemma}\label{Lem:analyticity:beta}
 For each $\alpha\in(0,1)$ the function $\beta_{W}$ as given by~\eqref{eq:def:beta:and:R} has an analytic extension to $\Ccut$ and there it holds
 \begin{align}
  \abs*{\beta_{W}(x)}&\leq C\bigl(\abs*{x}^{-\alpha}+\abs*{x}^{\alpha}\bigr)\label{eq:analyticity:beta:est}\\
  \abs*{\del_{x}^{k}\beta_{W}(x)}&\leq C_{k}\bigl(\abs*{x}^{-k-\alpha}+\abs*{x}^{\alpha-k}\bigr) && \text{for }k\in\N \text{ if }\Re(x)\geq 0\label{eq:analyticity:beta:est:der}\\
  \Re(\beta_{W}(x))&\geq 0 && \text{if }\Re(x)\geq 0.\label{eq:analyticity:beta:real:part}
 \end{align}
\end{lemma}

\begin{proof}
 The homogeneity of the kernel $W$ allows to rewrite $\beta_{W}(x)=\int_{0}^{\infty}W(x/z,1)f(z)\dz$. Since $W$ can be extended analytically to $\Ccut$ we then see at once that the same is also true for $\beta_{W}$. Moreover, \eqref{eq:Ass:analytic} together with Lemma~\ref{Lem:moment:est:2} directly implies that 
 \begin{equation*}
  \abs*{\beta_{W}(x)}\leq C\int_{0}^{\infty}\biggl(\frac{\abs*{x}^{\alpha}}{z^{\alpha}}+\frac{z^{\alpha}}{\abs*{x}^{\alpha}}\biggr)f(z)\dz\leq C\bigl(\abs*{x}^{\alpha}+\abs*{x}^{-\alpha}\bigr).
 \end{equation*}
 This proves~\eqref{eq:analyticity:beta:est}, while~\eqref{eq:analyticity:beta:est:der} follows then by Cauchy estimates. The property~\eqref{eq:analyticity:beta:real:part} is again a direct consequence of the definition of $\beta_{W}$ and $\Re(W(\xi,1))\geq 0$ for $\Re(\xi)\geq 0$.
\end{proof}

\begin{remark}
 In the following, we have to deal with contour integrals in the complex plane. However, since in most cases the contour is just given by a straight (half-)line we use the following notation. For $x\in\Ccut$ the integral $\int_{x}^{\infty}(\cdots)\dy$ should be interpreted as integral $\int_{\gamma_{x}}(\cdots)\ds$ with contour $\gamma_{x}=\{s+\im \Im(x)\;|\; s\in(\Re(x),\infty)\}$. 
 
 Similarly, $\int_{x}^{y}(\cdots)\dz$ for $x,y\in\C$ denotes integration along the segment connecting $x$ and $y$, while in our considerations below this will never cross the negative real axis $(-\infty,0)$.
 
 Finally, we write $\int_{x}^{\sgn(\Im x)\im\infty}(\cdots)\dy$ for the integral $\int_{\gamma_{x}}(\cdots)\ds$ along the path $\gamma_{x}=\{\Re(x)+\im \sgn(\Im x)s\; |\; s\in(\Im(x),\infty)\}$ provided $\Im (x)\neq 0$. If however $\Im(x)=0$, then we can choose between both signs equivalently, i.e.\@ $\gamma_{x}=\{\Re(x)\pm \im s\;| s\in (0,\infty)\}$. See also Figure~\ref{fig:gamma} for an illustration.
\end{remark}

\begin{lemma}\label{Lem:analyticity:Phi}
 For each $\alpha\in(0,1)$ the function $\Phi$, as defined in~\eqref{eq:def:Phi}, has an analytic extension to $\Ccut$ while on this domain, we have the equivalent representation
 \begin{equation}\label{eq:analyticity:Phi:further:rep}
  \Phi(x)=\eps\frac{\beta_{W}(x)}{x}\ee^{-x}+\eps\int_{x}^{\infty}\frac{\dd}{\dy}\biggl(\frac{\beta_{W}(y)}{y}\biggr)\ee^{-y}\dy.
 \end{equation}
 Moreover, it holds for $x\in\Ccut$ with $\Re(x)\geq 0$ that
 \begin{align}
  \abs*{\del_{x}^{k}\Phi(x)}&\leq C_{k}\eps \bigl(\abs*{x}^{-k-\alpha}+\abs*{x}^{\alpha-1}\bigr) && \text{for all }k\in\N, \label{eq:analyticity:Phi:est:der}\\
  \abs*{\Phi(x)}&\leq C\eps\min\bigl\{\abs*{x}^{-\alpha},\abs*{x}^{\alpha-1}\bigr\}\label{eq:analyticity:Phi:est}.
 \end{align}
 For $y\in\Ccut$ with $\Re(y)\geq 0$ and $x\in\{\lambda y\;|\; \lambda\in[0,1]\}$ it further holds that
 \begin{equation}\label{eq:analyticity:Phi:difference}
  \Re\bigl(\Phi(y)-\Phi(x)\bigr)\leq C\eps\bigl(\abs*{y}^{1+\alpha}+\abs*{y}^{1-\alpha}\bigr).
 \end{equation}
 In particular this means that $\Re(\Phi(y)-\Phi(x))$ is uniformly bounded for all $x\in\{\lambda y\;|\; \lambda\in[0,1]\}$.
\end{lemma}

\begin{proof}
 From Lemma~\ref{Lem:analyticity:beta} and the definition of $\Phi$ we directly deduce that $\Phi$ has an analytic extension to $\Ccut$ while~\eqref{eq:analyticity:Phi:further:rep} follows from integration by parts. The estimate~\eqref{eq:analyticity:Phi:est:der} for $k=1$ directly follow from the formula $\Phi'(x)=-\eps x^{-1}\beta_{W}(x)\ee^{-x}$ and Lemma~\ref{Lem:analyticity:beta}. For $k\geq 2$ one repeatedly differentiates and uses Lemma~\ref{Lem:analyticity:beta} to conclude.
 
 To obtain the bound~\eqref{eq:analyticity:Phi:est}, we have to estimate the integral in the definition of $\Phi$ directly. For this, one can consider the cases $\abs*{x}\leq 1$ and $\abs*{x}\geq 1$ separately. For $\abs*{x}\geq 1$, the desired bound can be derived directly from the integral representation of $\Phi$ together with Lemma~\ref{Lem:analyticity:beta}. If $\abs*{x}\leq 1$ we can use the analyticity of $\beta_{W}$ to deform the contour as shown in Figure~\ref{fig:Phi} which then also allows to derive the bound~\eqref{eq:analyticity:Phi:est}. More details are contained in~\cite{Thr16}.
 
\begin{figure}
 \begin{center}
  \begin{tikzpicture}[contour/.style={postaction={decorate, decoration={markings,
mark=at position 0.8cm with {\arrow[line width=1pt]{>}},
mark=at position 2.8cm with {\arrow[line width=1pt]{>}},
mark=at position 5.5cm with {\arrow[line width=1pt]{>}}
}}},
    interface/.style={postaction={draw,decorate,decoration={border,angle=45,
                    amplitude=0.3cm,segment length=2mm}}},
]

\draw[dashed] (0,0) circle (2cm);

\draw[->] (0,0) -- (7,0) node[below] {$\Re$};
\draw[->] (0,-2.5) -- (0,3) node[left] {$\Im$};

\draw[line width=.8pt,interface](0,0)--(-4.5,0);

\path[draw,line width=0.8pt,contour] (0.25,0.5)  -- (0.894427,1.788854) node[above]{$\frac{x}{\abs*{x}}$} -- (6,1.788854);

\draw[black,fill=black] (0.25,0.5) circle (.4ex);
\node[below right] at (0.25,0.5) {$x$};
\node[below right] at (2,0) {$1$};
\end{tikzpicture}
  \caption{Contour to estimate $\Phi$ for $x\leq 1$}
 \label{fig:Phi}
 \end{center}
\end{figure}
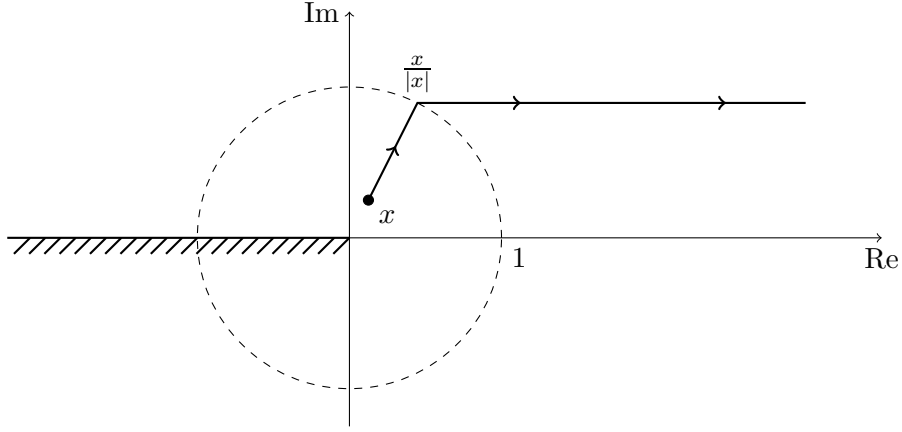
 
 To derive the estimate~\eqref{eq:analyticity:Phi:difference}, we first note that it holds by contour deformation and the homogeneity of $W$ that $\Phi(y)-\Phi(x)=-\eps \int_{0}^{\infty}f(t)\int_{x}^{y}W(z/t,1)z^{-1}\ee^{-z}\dz\dt$. If we substitute $\ee^{-z}=1+(\ee^{-z}-1)$, it follows
 \begin{multline*}
  \Phi(y)-\Phi(x)=-\eps\int_{0}^{\infty}f(t)\int_{x}^{y}\frac{1}{z}W\Bigl(\frac{z}{t},1\Bigr)\dz\dt\\*
  -\eps\int_{0}^{\infty}f(t)\int_{x}^{y}W\Bigl(\frac{z}{t},1\Bigr)\frac{\ee^{-z}-1}{z}\dz\dt=\vcc(I)+(II).
 \end{multline*}
Since $\abs*{z^{-1}(\ee^{-z}-1)}\leq 1$ for $\Re(z)\geq 0$ one obtains rather directly from~\eqref{eq:Ass:analytic} and Lemma~\ref{Lem:moment:est:2} that $\abs*{(II)}\leq C\eps (\abs*{y}^{1+\alpha}+\abs*{y}^{1-\alpha})$. On the other hand, we can rewrite $(I)$ with $x=\lambda y$ and $y=\abs*{y}\ee^{\im \vartheta}$ for some $\vartheta\in[-\pi/2,\pi/2]$ which yields
\begin{equation*}
 (I)=-\eps \int_{0}^{\infty}f(t)\int_{\lambda \abs*{y}}^{\abs*{y}}\frac{1}{s}W\Bigl(\frac{s}{t}\ee^{\im\vartheta},1\Bigr)\ds\dt.
\end{equation*}
 Since $\Re(W(\xi,1))\geq 0$ if $\Re(\xi)\geq 0$, due to~\eqref{eq:Ass:analytic}, the non-negativity of $f$ implies that $\Re(I)\leq 0$. In combination with the estimate on $(II)$ this finishes the proof while we refer again to~\cite{Thr16,NTV15} for slightly more details.
\end{proof}

\subsection{Proof of Proposition~\ref{Prop:rep:H0}}

With the preparations above we can now give the proof of the representation formula for $H_{0}(\cdot, q)$.

\begin{proof}[Proof of Proposition~\ref{Prop:rep:H0}]
 Due to Remark~\ref{Rem:kappa:smallness} we can always assume in the following that we have for any $\delta>0$ that $\abs*{\kappa}\leq \delta$ if we take $\eps>0$ small enough. According to Theorem~\ref{Thm:Paley:Wiener} it suffices to verify that $H_{0}(\cdot,q)\in H^2(0)$ for all $q\geq 0$. From the definition of $H_{0}(\cdot,q)$ in~\cref{eq:H:splitting:1,eq:H:splitting:2} one easily checks together with Lemma~\ref{Lem:analyticity:Phi} that $H_{0}(\cdot,q)$ can be extended analytically to $\Ccut$. Thus, it suffices to verify that $H_{0}(\cdot,q)$ also satisfies~\eqref{eq:cond:Hardy:Lebesgue}.
 
 From the definition of $H_{0}(\cdot,q)$ one derives together with~\cref{eq:analyticity:Phi:est:der,eq:analyticity:Phi:difference} that 
 \begin{equation}\label{eq:proof:rep:bd:1}
  \abs*{H_{0}(y,q)}\leq \frac{C}{q+1}\quad \text{if }\abs*{y}\leq 1 \text{ and }\Re(y)\geq 0.
 \end{equation}
On the other hand, if $\abs*{y}\geq 1$, we recall the splitting $H_{0}=H_{0,1}+H_{0,2}$ and first consider $H_{0,1}(\cdot,q)$ which we may rewrite as
\begin{multline*}
 H_{0,1}(y,q)=\frac{1+\rho+\kappa}{(q+1)y^{\rho}(1+y)}\int_{0}^{y/\abs*{y}}\ee^{-(q+1)x}x^{\rho}\Bigl(\frac{x}{y}\Bigr)\frac{\ee^{\Phi(y)}}{\ee^{\Phi(x)}}\dx\\*
 +\frac{1+\rho+\kappa}{(q+1)y^{\rho}(1+y)}\int_{y/\abs*{y}}^{y}\ee^{-(q+1)x}x^{\rho}\Bigl(\frac{x}{y}\Bigr)\frac{\ee^{\Phi(y)}}{\ee^{\Phi(x)}}\dx=\vcc H_{0,1,1}(y,q)+H_{0,1,2}(y,q).
\end{multline*}
 Due to Lemma~\ref{Lem:analyticity:Phi} it holds for $\abs*{y}\geq 1$ that $\Re(\Phi(y)-\Phi(x))=\Re(\Phi(y/\abs*{y})-\Phi(x))+\Re(\Phi(y)-\Phi(y/\abs*{y}))\leq C$ such that the expression $H_{0,1,1}$ can be estimated as
 \begin{equation}\label{eq:proof:rep:bd:2}
  \abs*{H_{0,1,1}(y,q)}\leq \frac{C}{(q+1)\abs*{y}^{\rho-\delta}\abs*{1+y}}\quad \text{for }\abs*{y}\geq 1 \text{ and }\Re(y)\geq 0.
 \end{equation}
 In order to see that $H_{0,1,2}(y,q)$ also decays sufficiently fast in $y$, we rewrite the expression $\ee^{-(q+1)x}=-(q+1)^{-1}\del_{x}(\ee^{-(q+1)x})$ and integrate by parts. This then yields
 \begin{multline*}
  H_{0,1,2}(y,q)=-\frac{1+\rho+\kappa}{(q+1)^2y^{\rho}(1+y)}\biggl(\ee^{-(q+1)y}y^{\rho}-\frac{y^{\rho}}{\abs*{y}^{\rho}}\frac{\ee^{-(q+1)\frac{y}{\abs*{y}}}}{\abs*{y}^{\kappa}}\frac{\ee^{\Phi(y)}}{\ee^{\Phi(y/\abs*{y})}}\biggr)\\*
  +\frac{1+\rho+\kappa}{(q+1)^2y^{\rho}(1+y)}\int_{y/\abs*{y}}^{y}\ee^{-(q+1)x}\Bigl(\frac{\rho+\kappa}{x}-\Phi'(x)\Bigr)\frac{x^{\rho+\kappa}}{y^{\kappa}}\frac{\ee^{\Phi(y)}}{\ee^{\Phi(x)}}\dx.
 \end{multline*}
 If we parametrise the path of integration as $s\mapsto sy/\abs*{y}$ we obtain together with Lemma~\ref{Lem:analyticity:Phi} that the previous expression can be estimated as
 \begin{multline*}
  \abs*{H_{0,1,2}(y,q)}\leq \frac{C}{(q+1)^2\abs*{1+y}}+\frac{C}{(q+1)^2\abs*{y}^{\rho-\delta}\abs*{1+y}}\\*
  +\frac{C}{(q+1)^2\abs*{y}^{\rho+\kappa}\abs*{1+y}}\int_{1}^{\abs*{y}}(s^{-1}+s^{\alpha})s^{\rho+\kappa}\ds.
 \end{multline*}
 Since $\int_{1}^{\abs*{y}}(s^{-1}+s^{\alpha})s^{\rho+\kappa}\ds\leq C\abs*{y}^{\rho+\alpha+\kappa}$ for $\abs*{y}\geq 1$ we can estimate by the most dominant terms to get
 \begin{equation}\label{eq:proof:rep:bd:3}
  \abs*{H_{0,1,2}(y,q)}\leq\frac{C}{(q+1)^2}\frac{\abs*{y}^{\alpha}}{\abs*{1+y}}\quad \text{if }\abs*{y}\geq 1 \text{ and }\Re(y)\geq 0.
 \end{equation}
The expression $H_{0,2}(\cdot,q)$ can be treated similarly, while here we have to use that $\Phi'(x)=-\eps \beta_{W}(x)x^{-1}\ee^{-x}$ to absorb the oscillating term $\ee^{-x}$ into $\ee^{-(q+1)x}$ before we integrate by parts. Then, following in principle the same procedure as before for $H_{0,1}$, we can split the integral, integrate by parts and estimate while here, in one term we have to integrate by parts twice to get enough decay for large values of $\abs*{y}$. Some more of the details that we omit here can be found in~\cite{Thr16,NTV15}. If one proceeds in the prescribed manner, one can finally show that for $\abs*{y}\geq 1$ with $\Re(y)\geq 0$ it holds that
\begin{equation*}
 \abs*{H_{0,2}(y,q)}\leq \frac{C\eps\abs*{y}^{\alpha+\delta}}{(q+1)^2\abs*{1+y}}+\frac{C\eps^2}{(q+1)^3}\biggl(\frac{\abs*{y}^{3\alpha+\delta}}{\abs*{y}\abs*{1+y}}+\frac{1}{\abs*{y}^{\rho-\delta}\abs*{1+y}}\biggr).
\end{equation*}
Combining this with~\cref{eq:proof:rep:bd:2,eq:proof:rep:bd:3} we further deduce
\begin{equation}\label{eq:proof:rep:bd:4}
 \abs*{H_{0}(y,q)}\leq C\Bigl(\abs*{y}^{\delta-\rho-1}+\abs*{y}^{\alpha-1}+\abs*{y}^{3\alpha+\delta-2}\Bigr)\quad \text{for }\abs*{y}\geq 1\text{ with }\Re(y)\geq 0.
\end{equation}
Note that the constant in the previous estimate is independent of $q$ and $\eps$ if $\eps$ is sufficiently small. Moreover, the latter estimate together with~\eqref{eq:proof:rep:bd:1} yields 
\begin{equation}\label{eq:proof:rep:bd:5}
 \abs*{H_{0}(y,q)}\leq C \quad \text{if } \abs*{y}\leq 2 \text{ and }\Re(y)\geq 0.
\end{equation}
We can now show that $H_{0}(\cdot,q)$ satisfies~\eqref{eq:cond:Hardy:Lebesgue}. For this, we write $y=x+\im z$ and consider the cases $x\leq 1$ and $x\geq 1$ separately. For $x\leq 1$ we find together with~\cref{eq:proof:rep:bd:4,eq:proof:rep:bd:5} that
\begin{equation*}
 \begin{split}
  \int_{-\infty}^{\infty}\abs*{H_{0}(x+\im z)}^2\dz&=\int_{-1}^{1}\abs*{H_{0}(x+\im z)}^2\dz+\int_{\R\setminus (-1,1)}\abs*{H_{0}(x+\im z)}^2\dz\\
  &\leq C+C\int_{1}^{\infty}\Bigl((x^2+z^2)^{\frac{\delta-\rho-1}{2}}+(x^2+z^2)^{\frac{\alpha-1}{2}}+(x^2+z^2)^{\frac{3\alpha+\delta-2}{2}}\Bigr)^2\dz\\
  &\leq C+C\int_{1}^{\infty}z^{2(\delta-1-\rho}+z^{\alpha-1}+z^{2(3\alpha+\delta-2)}\dz\leq C.
 \end{split}
\end{equation*}
Note that we used $x\leq z$ for $x\leq 1$ and $z\in(1,\infty)$. 
 For $x\geq 1$ we proceed similarly, i.e.\@ we split the integral $\int_{0}^{\infty}(\cdots)\dz=\int_{0}^{1}(\cdots)\dz+\int_{1}^{\infty}\dz$ and use $(x^2+z^2)^{-1}\leq x^{-2}$ and $(x^2+z^2)^{-1}\leq z^{-2}$ respectively to find
 \begin{equation*}
  \begin{split}
   &\phantom{{}\leq {}} \int_{-\infty}^{\infty}\abs*{H_{0}(x+\im z)}^2\dz\\
   &\leq C\Bigl(x^{2(\delta-1-\rho)}+x^{2(\alpha-1)}+x^{2(3\alpha+\delta-2)}\Bigr)+C\int_{1}^{\infty}z^{2(\delta-1-\rho)}+z^{2(\alpha-1)}+z^{2(3\alpha+\delta-2)}\dz\leq C.
  \end{split}
 \end{equation*}
Since the constants in the previous estimates can be chosen independently of $x$ and $q$ this gives $H_{0}(\cdot,q)\in H^2(0)$ and thus concludes the proof by an application of Theorem~\ref{Thm:Paley:Wiener}.
\end{proof}

\section{Integral estimate on $\Qo_{0}(\cdot,q)$}\label{Sec:Q0:est}

This section is devoted to the proof of Proposition~\ref{Prop:Q0:int:est} which is the most technical part of this article. Before we come to the proof itself we will consider tow auxiliary lemmas.

\subsection{Auxiliary results to estimate integrals of $\Qo_{0}(\cdot,q)$}

\begin{lemma}\label{Lem:contribution:small:xi}
 Let parameters $\mathfrak{a}\in[0,2\theta]$, $\mathfrak{b}\in[2\theta,\infty)$ and $\delta\in(0,\rho+\theta-1)$ be given. For each $a_{*}\in(\max\{1-\rho+\delta,\alpha\},\theta)$ and $\nu_{*}\in(0,\theta-a_{*})$ it holds 
 \begin{align}
  (q+1)^{a-3}\int_{0}^{q+2}\weight{-\mathfrak{a}}{\mathfrak{b}-\nu}(\xi)\dxi&\leq C(q+1)^{-2-\theta}  &&\text{for all }a\in[0,a_{*}]\label{eq:contribution:small:xi:1}
  \shortintertext{and}
  (q+1)^{-2}\int_{0}^{q+2}\weight{-\mathfrak{a}}{\mathfrak{b}-\nu}(\xi)\abs*{\xi-(q+1)}^{a-1}\dxi&\leq C_{a}(q+1)^{-2-\theta} &&\text{for }a\in(0,a_{*}]\label{eq:contribution:small:xi:2}
 \end{align}
 for all $q\geq 0$ and $\nu\in[0,\nu_{*})$.
\end{lemma}

\begin{proof}
 We note that due to the assumptions on $\mathfrak{a}$ and $\mathfrak{b}$, one deduces from~\eqref{eq:est:weight} that it holds $\weight{-\mathfrak{a}}{\mathfrak{b}-\nu}(\xi)\leq \xi^{-2\theta}(\xi+1)^{\nu}$. Thus, it follows for any $a\in[0,a_{*}]$ and $\nu\in[0,\nu_{*}]$ that
 \begin{equation*}
  \begin{split}
   (q+1)^{a-3}\int_{0}^{q+2}\weight{-\mathfrak{a}}{\mathfrak{b}-\nu}(\xi)\dxi&\leq (q+1)^{a-3}\int_{0}^{q+2}\xi^{-2\theta}(\xi+1)^{\nu}\dxi\\
   &\leq C(q+1)^{a-3}(q+3)^{\nu}(q+2)^{1-2\theta}\leq C(q+1)^{-2-\theta}.
  \end{split}
 \end{equation*}
 This proves~\eqref{eq:contribution:small:xi:1}. To show~\eqref{eq:contribution:small:xi:2} we first obtain similarly as before that 
 \begin{multline*}
  (q+1)^{-2}\int_{0}^{q+2}\weight{-\mathfrak{a}}{\mathfrak{b}-\nu}(\xi)\abs*{\xi-(q+1)}^{a-1}\dxi\\*
  \leq (q+1)^{-2}(q+3)^{\nu}\int_{0}^{q+2}\xi^{-2\theta}\abs*{\xi-(q+1)}^{a-1}\dxi.
 \end{multline*}
 The change of variables $\xi\mapsto (q+1)\xi$ together with $(q+2)/(q+1)\leq 2$ then yields 
 \begin{multline*}
   (q+1)^{-2}\int_{0}^{q+2}\weight{-\mathfrak{a}}{\mathfrak{b}-\nu}(\xi)\abs*{\xi-(q+1)}^{a-1}\dxi\\*
   \leq (q+1)^{a-2-2\theta}(q+3)^{\nu}\int_{0}^{2}\xi^{-2\theta}\abs*{\xi-1}^{a-1}\dxi\leq C_{a}(q+1)^{-2-\theta}
 \end{multline*}
 for each $a\in(0,a_{*}]$.
\end{proof}

\begin{lemma}\label{Lem:contribution:large:xi}
Let parameters $\mathfrak{a}\in[0,2\theta]$, $\mathfrak{b}\in[2\theta,\infty)$ and $\delta\in(0,\rho+\theta-1)$ be given. For each $a_{*}\in(\max\{1-\rho+\delta,\alpha\},\theta)$ and $\nu_{*}\in(0,\theta-a_{*})$ it holds 
 \begin{align}
  (q+1)^{-2}\int_{q+2}^{\infty}\xi^{-1}\weight{-\mathfrak{a}}{\mathfrak{b}-\nu}(\xi)\dxi&\leq C(q+1)^{-2-\theta}  \label{eq:contribution:large:xi:1}
  \shortintertext{and}
  (q+1)^{-2}\int_{q+2}^{\infty}\weight{-\mathfrak{a}}{\mathfrak{b}-\nu}(\xi)\abs*{\xi-(q+1)}^{a-1}\dxi&\leq C(q+1)^{-2-\theta} \label{eq:contribution:large:xi:2}
 \end{align}
 for all $q\geq 0$ and all $a\in[0,a_{*}]$, $\nu\in[0,\nu_{*})$.
 
 Moreover, we have the estimates
  \begin{equation}\label{eq:contribution:large:xi:3}
  \frac{1}{(q+1)(q+2)}\int_{q+2}^{\infty}\weight{-\mathfrak{a}}{\mathfrak{b}-\nu}(\xi)\abs*{\xi-(q+2)}^{a-1}\dxi\leq C_{a}(q+1)^{-2-\theta}
 \end{equation}
 for all $q\geq 0$ and $a\in(0,a_{*}]$, $\nu\in[0,\nu_{*})$ as well as
 \begin{equation}\label{eq:contribution:large:xi:4}
  (q+1)^{-2}\int_{q+1}^{\infty}\weight{-\mathfrak{a}}{\mathfrak{b}-\nu}(\xi)\ee^{-(\xi-(q+1))}\dxi\leq C(q+1)^{-2-\theta}
 \end{equation}
 for all $q\geq 0$ and all $\nu\in[0,\nu_{*})$.
\end{lemma}

\begin{proof}
 The proof of \cref{eq:contribution:large:xi:1,eq:contribution:large:xi:2,eq:contribution:large:xi:3} can be done similarly as in Lemma~\ref{Lem:contribution:small:xi} if we use that it holds $\weight{-\mathfrak{a}}{\mathfrak{b}-\nu}(\xi)\leq \xi^{\nu-2\theta}$ for $\xi\geq q+2$. Note also that the singularity of $\abs*{\xi-(q+1)}^{a-1}$ at $\xi=q+1$ does not cause any problem since we only integrate over $[q+2,\infty)$.
 
 To prove~\eqref{eq:contribution:large:xi:4} we again use $\weight{-\mathfrak{a}}{\mathfrak{b}-\nu}(\xi)\leq \xi^{\nu-2\theta}$ for $\xi\geq q+2$ together with the shift $\xi\mapsto \xi+(q+1)$ and $(\xi+q+1)^{\nu-2\theta}\leq (q+1)^{\nu-2\theta}$ which yields
 \begin{equation*}
  \begin{split}
   (q+1)^{-2}\int_{q+1}^{\infty}\weight{-\mathfrak{a}}{\mathfrak{b}-\nu}(\xi)\ee^{-(\xi-(q+1))}\dxi&\leq  C(q+1)^{-2}\int_{0}^{\infty}(\xi+q+1)\ee^{-\xi}\dxi\\
   &\leq C(q+1)^{\nu-2-2\theta}\int_{0}^{\infty}\ee^{-\xi}\dxi\leq C(q+1)^{-2-\theta}.
  \end{split}
 \end{equation*}
This then concludes the proof.
\end{proof}

\subsection{Proof of Proposition~\ref{Prop:Q0:int:est}}

\begin{proof}[Proof of Proposition~\ref{Prop:Q0:int:est}]
 The general task for this proof is to obtain enough decay in $q$ for $\Qo_{0}(\cdot ,q)$ as $q\to \infty$ in a suitable integral sense. The main strategy to show this, will be to use the relation $\ee^{-(q+1)x}=-(q+1)^{-1}\del_{x}\ee^{-(q+1)x}$ and integrate by parts. In general, this approach is rather straightforward, however, there are two main technical problems arising. First of all, in order to obtain the desired decay in $q$, we have to integrate by parts at least two times which produces a plenty of several terms that have to be estimated. Moreover, as we have to estimate $\Qo_{0}(\cdot ,q)$ in an integral sense, we also have to show that this quantity is integrable both at zero and at infinity. We already note here that there will in general arise problems in both cases, namely, due to the repeated integration by parts, we produce a singularity at zero which worsens during each step of integrating by parts. However, for $\rho>1/2$ it turns out that this can be overcome by a careful analysis and the precise prescription of the asymptotic behaviour of the kernel $W$ (see \cref{Lem:gain:of:alpha:coup,Lem:gain:of:alpha:cut}). On the other hand, in order to obtain enough decay with respect to $\xi$ such that $\Qo_{0}(\xi,q)$ is integrable at infinity, we have to integrate by parts repeatedly in several terms while it turns out that for this approach that we use here the assumption $\alpha<1/2$ is essential to estimate the most singular expressions.
 
 Before we explain the general approach more closely, we will introduce some notation and collect some general properties that we will need in the following. First of all, we denote by 
 \begin{equation*}
  V(q)\vcc=\int_{0}^{\infty}\weight{-\mathfrak{a}}{\mathfrak{b}-\nu}(\xi)\abs*{\Qo_{0}(\xi,q)}\dxi
 \end{equation*}
 the integral expression that we have to estimate. Additionally, we will need two cut-off functions $\cut,\coup\in C^{\infty}(\R)$ which satisfy
 \begin{equation}\label{eq:cutoffs}
  \begin{aligned}
   0&\leq \coup\leq 1 & \coup(-s)&=\coup(s), & \coup(s)&=0\text{ if } \abs{s}\in[0,1/2], & \coup(s)&=1\text{ if }\abs{s}\in [1,\infty),\\
   0&\leq \cut\leq 1 & \cut(-t)&=\coup(t), & \cut(t)&=1\text{ if } \abs{t}\in[0,1], & \cut(t)&=0\text{ if }\abs{t}\in [2,\infty).
  \end{aligned}
 \end{equation}
 For $R>0$, and the function $H_{0,k}(\cdot,q)$ as given by~\eqref{eq:H:splitting:2} we denote
 \begin{equation*}
  \Qo_{k,R}(\xi,q)=\frac{1}{2\pi\im}\int_{-\im R}^{\im R}\ee^{y\xi}H_{0,k}(y,q)\dy\quad \text{for }k=1,2.
 \end{equation*}
The general approach to prove the statement will be to use a certain splitting
\begin{equation}\label{eq:Q0:est:gen:approach}
 \Qo_{k,R}(\xi,q)=M_{k}(\xi,q)+J_{k,R}(\xi,q)
\end{equation}
with a function $M_{k}(\xi,q)$ which behaves nice in the sense that we obtain the desired decay properties in $\xi$ and $q$. On the other hand, for the remainder $J_{k,R}(\xi,q)$, we will show $J_{k,R}(\xi,q)\to 0$ for almost every $\xi$ as $R\to \infty$ uniformly with respect to $q$. 

Due to Proposition~\ref{Prop:rep:H0}, for each $q\geq 0$, we may extract a sequence $R_{n}(q)\to \infty$ for $n\to\infty$ such that it holds
\begin{equation*}
 \Qo_{0}(\xi,q)=\lim_{n\to\infty}\bigl(\Qo_{1,R_{n}(q)}(\xi,q)+\Qo_{2,R_{n}(q)}(\xi,q)\bigr) \quad \text{for almost every }\xi\in\R_{+}.
\end{equation*}
The representation~\eqref{eq:Q0:est:gen:approach} together with the convergence of $J_{k,R}$ then yields
\begin{multline*}
 \abs*{\Qo_{0}(\xi,q)}\leq \bigl(\abs*{M_{1}(\xi,q)}+\abs*{M_{2}(\xi,q)}\bigr)+\lim_{n\to\infty}\bigl(J_{1,R_{n}(q)}(\xi,q)+J_{2,R_{n}(q)}(\xi,q)\bigr)\\*
 =\abs*{M_{1}(\xi,q)}+\abs*{M_{2}(\xi,q)}.
\end{multline*}
Therefore, we see that in order to estimate the expression $V(q)$, it suffices to bound the corresponding integral where we replace $\Qo_{0}(\xi,q)$ by $M_{k}(\xi,q)$.

As already announced above, the general approach to derive the necessary estimates will consist in integrating by parts. More precisely, this means that we first integrate by parts with respect to $x$ in the representation formula for $H_{0,k}(y,q)$. Then, we have to distinguish whether it holds $\xi<q+2$ or $\xi>q+2$. Namely, in the first case, we integrate by parts once more with respect to $x$ while in the second one we instead integrate by parts with respect to $y$ in the integral representation of $\Qo_{k,R}(\xi,q)$. In principle, each integration by parts with respect to $x$ or $y$ yields either a factor $(q+1)^{-1}$ or instead a factor $\xi^{-1}$ or $(\xi-(q+1))^{-1}$ respectively. However, there is a further technical problem arising. Namely, in order to get uniform estimates on the resulting integrals, we have to weaken several singularities using the estimate $\abs*{1-\ee^{-x}}\leq C \abs{x}^{a}$ for some $a\in[0,1]$ as explained in more detail below. Thus, we will generally obtain estimates of the form 
\begin{equation*}
 \abs*{M_{k}(\xi,q)}\leq C(q+1)^{a-3} \quad \text{and}\quad \abs*{M_{k}(\xi,q)}\leq C(q+1)^{a-2}\abs*{\xi-(q+1)}^{a-1}\quad \text{for }\xi<q+2 
\end{equation*}
as well as
\begin{equation*}
 \abs*{M_{k}(\xi,q)}\leq C(q+1)^{-2}\xi^{-1} \quad \text{and}\quad \abs*{M_{k}(\xi,q)}\leq C(q+1)^{-2}\abs*{\xi-(q+1)}^{a-1}\quad \text{for }\xi>q+2.
\end{equation*}
The corresponding contribution to the integral $V(q)$ that we have to estimate is then controlled by \cref{Lem:contribution:small:xi,Lem:contribution:large:xi}.

In order to keep track over all the different terms, we will consider each quantity arising after the first integration by parts in a separate subsection.

As indicated above, we have to exploit frequently the regularising effect of $1-\ee^{-x}$ for small $\abs{x}$ which requires a careful adjustment of certain parameters. More precisely, the choice of $\theta$ in~\eqref{eq:choice:theta:bl} yields $1-\rho<\theta$ such that it is always possible to fix a small parameter $\delta>0$ which satisfies 
\begin{equation*}
 1-\rho+\delta<\theta.
\end{equation*}
Then, we can furthermore take $\eps>0$ sufficiently small such that it holds $\abs*{\kappa}< \delta$ as guaranteed by Remark~\ref{Rem:kappa:smallness}. Throughout the proof we will then make frequent use of the estimate
\begin{equation*}
 \abs*{\ee^{-x}-1}\leq C\abs*{x}^{a}\quad \text{for } a\in[0,1]\text{ and all }x\text{ with }\Re(x)\geq 0.
\end{equation*}
Typically, we apply this estimate with $a_{*}\in(\max\{1-\rho+\delta,\alpha\},\theta)$ while this choice is possible if we take $\delta<\rho-\alpha$ and recall that $\alpha<\rho$.

We also note here once that we will use the relation
\begin{equation*}
 \int_{-R}^{R}\int_{0}^{t}(\cdots)\ds\dt=\int_{-R}^{R}\int_{s}^{\sgn(s)R}(\cdots)\dt\ds
\end{equation*}
which is a consequence of Fubini's Theorem.

Finally, in order to avoid any additional technical complication, we note that we can always redefine the value of $\theta$ or make $\abs*{\kappa}$ slightly smaller such that none of the integrals that we will consider in the following yields a logarithmic expression and we only have to deal with power laws.

\subsection{Rewriting $H_{0}(y,q)$}

We recall from~\eqref{eq:H:splitting:2} that $H_{0}(y,q)=H_{0,1}(y,q)+H_{0,2}(y,q)$ and we use the relations $\Phi'(x)=-\eps\beta_{W}(x)x^{-1}\ee^{-x}$ and $\ee^{-(q+n)x}=-(q+n)^{-1}\del_{x}\ee^{-(q+n)x}$ for $n=1,2$ to rewrite 
\begin{equation*}
 \begin{split}
  H_{0,1}(y,q)&=-\frac{1+\rho+\kappa}{(q+1)^{2}y^{\rho}(1+y)}\int_{0}^{y}\del_{x}\Bigl(\ee^{-(q+1)x}\Bigr)\frac{x^{\rho+\kappa}}{y^{\kappa}}\frac{\ee^{\Phi(y)}}{\ee^{\Phi(x)}}\dx,\\
  H_{0,2}(y,q)&=-\frac{1}{(q+1)(q+2)y^{\rho+\kappa}(1+y)}\int_{0}^{y}\del_{x}\Bigl(\ee^{-(q+2)x}\Bigr)\beta_{W}(x)x^{\rho+\kappa}\frac{\ee^{\Phi(y)}}{\ee^{\Phi(x)}}\dx.
 \end{split}
 \end{equation*} 
Thus, integration by parts yields that
\begin{equation*}
 \begin{split}
  H_{0,1}(y,q)=\U_{1,0}(y,q)+\U_{1,1}(y,q)+\U_{1,2}(y,q),\\
  H_{0,2}(y,q)=\U_{2,0}(y,q)+\U_{2,1}(y,q)+\U_{2,2}(y,q).
 \end{split}
\end{equation*}
with 
\begin{equation*}
 \begin{split}
  \U_{1,0}(y,q)&=-\frac{1+\rho+\kappa}{(q+1)^{2}(1+y)}\ee^{-(q+1)y},\\
  \U_{1,1}(y,q)&=\frac{(1+\rho+\kappa)(\rho+\kappa)}{(q+1)^{2}y^{\rho}(1+y)}\int_{0}^{y}\ee^{-(q+1)x}\frac{x^{\rho+\kappa-1}}{y^{\kappa}}\frac{\ee^{\Phi(y)}}{\ee^{\Phi(x)}}\dx,\\
  \U_{1,2}(y,q)&=-\frac{1+\rho+\kappa}{(q+1)^{2}y^{\rho}(1+y)}\int_{0}^{y}\ee^{-(q+1)x}\Phi'(x)\frac{x^{\rho+\kappa}}{y^{\kappa}}\frac{\ee^{\Phi(y)}}{\ee^{\Phi(x)}}\dx
 \end{split}
\end{equation*}
as well as
\begin{equation*}
 \begin{split}
  \U_{2,0}(y,q)&=-\frac{\eps\ee^{-(q+2)y}\beta_{W}(y)}{(q+1)(q+2)(1+y)},\\
  \U_{2,1}(y,q)&=\frac{\eps}{(q+1)(q+2)y^{\rho+\kappa}(1+y)}\int_{0}^{y}\ee^{-(q+2)x}\del_{x}\Bigl(x^{\rho+\kappa}\beta_{W}(x)\Bigr)\frac{\ee^{\Phi(y)}}{\ee^{\Phi(x)}}\dx,\\
  \U_{2,2}(y,q)&=\frac{1}{(q+1)(q+2)y^{\rho+\kappa}(1+y)}\int_{0}^{y}\ee^{-(q+1)x}x^{1+\rho+\kappa}\bigl(\Phi'(x)\bigr)^2\frac{\ee^{\Phi(y)}}{\ee^{\Phi(x)}}\dx.
 \end{split}
\end{equation*}
As outlined above, the contribution of each of these six terms has to be estimated separately and moreover, for each expression we have to consider the regions $\xi<q+2$ and $\xi>q+2$. However, since the approach for $\U_{1,1}$, $\U_{1,2}$, $\U_{2,1}$ and $\U_{2,2}$ is essentially the same, we will only consider for illustrative purpose the expression $\U_{2,2}(y,q)$ which is the most delicate one since it contains the most singular terms. The terms $\U_{1,1}$, $\U_{1,2}$ and $\U_{2,1}$ can then be treated similarly, while some of the estimates in these cases are slightly easier. More details on how to estimate these terms might also be found in~\cite{Thr16}. On the other hand, the 'boundary' expressions $\U_{1,0}(y,q)$ and $\U_{2,0}(y,q)$ behave slightly different and thus have to be treated separately and we will therefore also give the precise estimates here.

\subsection{Contribution of $\U_{1,0}(y,q)$}

We estimate first the expression $\U_{1,0}(y,q)$ which is exceptional since the contribution to $\Qo_{1,R}$ can be computed rather explicitly using contour deformation. Precisely, we have to consider $(2\pi\im)^{-1}\int_{-\im R}^{\im R} \ee^{y\xi}\U_{1,0}(y,q)\dy$. With the explicit expression for $\U_{1,0}(y,q)$ we can rewrite this and integrate by parts which yields
\begin{multline}\label{eq:W10:1}
 -\frac{1+\rho+\kappa}{2\pi\im(q+1)^{2}}\int_{-\im R}^{\im R}\frac{\ee^{y(\xi-(q+1))}}{1+y}\dy=-\frac{1+\rho+\kappa}{2\pi\im(q+1)^{2}}\frac{1}{\xi-(q+1)}\int_{\im R}^{\im R}\frac{\del_{y}(\ee^{y(\xi-(q+1))}-1)}{1+y}\dy\\*
 =-\frac{1+\rho+\kappa}{2\pi\im(q+1)^2}\frac{1}{\xi-(q+1)}\biggl(\frac{\ee^{y(\xi-(q+1))}}{1+y}\biggr)\bigg|_{y=-\im R}^{y=\im R}\\*
 -\frac{1+\rho+\kappa}{2\pi\im(q+1)^2}\frac{1}{\xi-(q+1)}\int_{-\im R}^{\im R}\frac{\ee^{y(\xi-(q+1))}-1}{(1+y)^2}\dy.
\end{multline}
For the boundary term, i.e.\@ the first expression on the right-hand side we immediately find
\begin{equation*}
 \abs*{\frac{1+\rho+\kappa}{2\pi\im(q+1)^2}\frac{1}{\xi-(q+1)}\biggl(\frac{\ee^{y(\xi-(q+1))}}{1+y}\biggr)\bigg|_{y=-\im R}^{y=\im R}}\longrightarrow 0\quad \text{for }R\longrightarrow \infty.
\end{equation*}
For the second term on the right-hand side of~\eqref{eq:W10:1} we consider the two cases $\xi<q+1$ and $\xi>q+1$ separately and we will compute the integral explicitly by deforming the contour as shown in Figure~\ref{fig:W10}. Similarly as for the boundary term above, one can readily show that the integral over the half-circle in the right (for $\xi<q+1$) or left half-plane (for $\xi>q+1$) vanishes in the limit $R\to\infty$ with a rate as $1/R$. Since the integrand $(1+y)^{-2}(\ee^{y(\xi-(q+1))}-1)$ is analytic in the right half-plane, this already shows that for $\xi<q+1$ also the integral $\int_{-\im R}^{\im R}(1+y)^{-2}(\ee^{y(\xi-(q+1))}-1)\dy$ vanishes in the limit $R\to\infty$. 

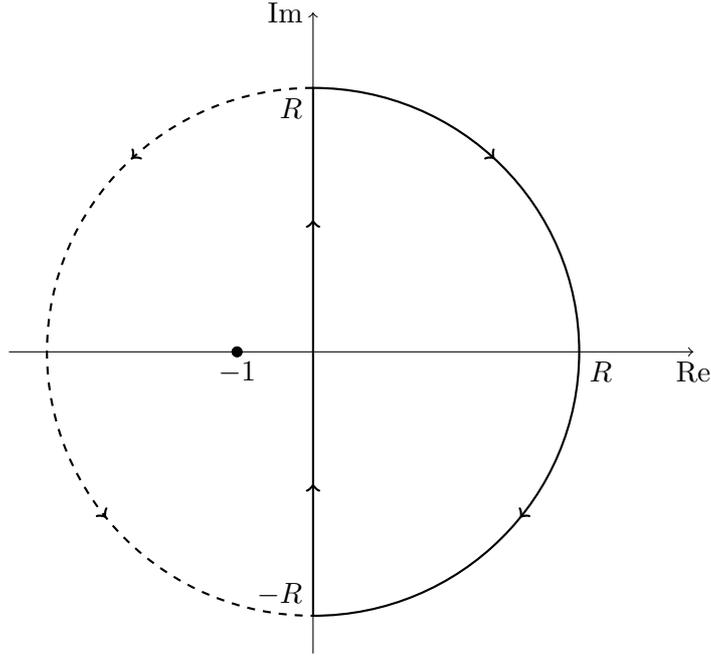
\begin{figure}
\begin{center}
  \begin{tikzpicture}[contour/.style={postaction={decorate, decoration={markings,
mark=at position 1.75cm with {\arrow[line width=1pt]{>}},
mark=at position 5.25cm with {\arrow[line width=1pt]{>}},
mark=at position 9.625cm with {\arrow[line width=1pt]{>}},
mark=at position 14.875cm with {\arrow[line width=1pt]{>}}}}},
contourtwo/.style={postaction={decorate, decoration={markings, mark=at position 2.625cm with {\arrow[line width=1pt]{>}},mark=at position 7.875cm with {\arrow[line width=1pt]{>}}}}}]

\draw[->] (-4,0) -- (5,0) node[below] {$\Re$};
\draw[->] (0,-4) -- (0,4.5) node[left]  {$\Im$};

\path[draw,line width=0.8pt, contour] (0,-3.5) node[above left]{$-R$}  -- (0,3.5) node[below left]{$R$}  arc (90:-90:3.5);

\path[draw, dashed, line width=0.8pt,contourtwo] (0,3.5) arc (90:270:3.5);

\node[below right] at (3.5,0) {$R$};
\draw[black,fill=black] (-1,0) circle (.4ex);
\node[below] at (-1,0){$-1$};
\end{tikzpicture}
\caption{Contour for $\U_{1,0}$}
\label{fig:W10}
 \end{center}
\end{figure}

On the other hand, for $\xi>q+1$, we close the contour in the left half-plane and the Residue Theorem thus yields
\begin{equation*}
 -\frac{1+\rho+\kappa}{2\pi\im(q+1)^2}\frac{1}{\xi-(q+1)}\int_{-\im R}^{\im R}\frac{\ee^{y(\xi-(q+1))}-1}{(1+y)^2}\dy\longrightarrow -\frac{1+\rho+\kappa}{(q+1)^{2}}\ee^{-(\xi-(q+1))}
\end{equation*}
for $R\to \infty$. Therefore, the only contribution of $\U_{1,0}$ to $V(q)$ stems from the region $\xi>q+1$ and we obtain the desired estimate due to Lemma~\ref{Lem:contribution:large:xi}.

\subsection{Contribution of $\U_{2,0}(y,q)$}

In order to estimate the contribution coming from $\U_{2,0}(y,q)$ we have to consider the expression
\begin{equation*}
 \Qo_{2,0,R}\vcc=-\frac{1}{2\pi\im}\frac{\eps}{(q+1)(q+2)}\int_{-\im R}^{\im R}\frac{\ee^{y(\xi-(q+2))}}{1+y}\beta_{W}(y)\dy.
\end{equation*}
To get a bound on this which is uniform in $R$, we use the relation $\ee^{y(\xi-(q+2))}=(\xi-(q+2))^{-1}\del_{y}(\ee^{y(\xi-(q+2))}-1)$ and integrate by parts in $y$ which yields the splitting $\Qo_{2,0,R}=(I)+(II)$ with 
\begin{equation*}
 \begin{split}
  (I)&=-\frac{1}{2\pi\im}\frac{\eps}{(q+1)(q+2)(\xi-(q+2))}\biggl(\Bigl(\ee^{y(\xi-(q+2))}-1\Bigr)\frac{\beta_{W}(y)}{1+y}\biggr)\bigg|_{y=-\im R}^{y=\im R}\\
  (II)&=\frac{1}{2\pi\im}\frac{\eps}{(q+1)(q+2)(\xi-(q+2))}\int_{-\im R}^{\im R}\Bigl(\ee^{y(\xi-(q+2))}-1\Bigr)\biggl(-\frac{\beta_{W}(y)}{(1+y)^2}+\frac{\beta_{W}'(y)}{1+y}\biggr)\dy.
 \end{split}
\end{equation*}
From Lemma~\ref{Lem:analyticity:beta} one immediately deduces that for each $a\in(0,1-\alpha)$ it holds $\abs*{(I)}\leq C\abs*{\xi-(q+2)}^{a-1}R^{\alpha+a-1}\to 0$ for $R\to\infty$. Thus, $\U_{2,0}(y,q)$ contributes only via the expression $(II)$ which will be estimated next and which requires to consider the regions $\xi<q+2$ and $\xi>q+2$ separately. 

\subsubsection{Contribution of $\U_{2,0}(y,q)$ for $\xi<q+2$}

For $\xi<q+2$ one can easily show that due to the analyticity of $\beta_{W}$ in the right half-plane and the estimates provided by Lemma~\ref{Lem:analyticity:beta} that the contour of integration of the integral in $(II)$ can be deformed to a half-circle of radius $r_{0}\in(0,1/2)$ around the origin such that we prevent the singularity at zero of the integrand (see Figure~\ref{fig:W20}). 

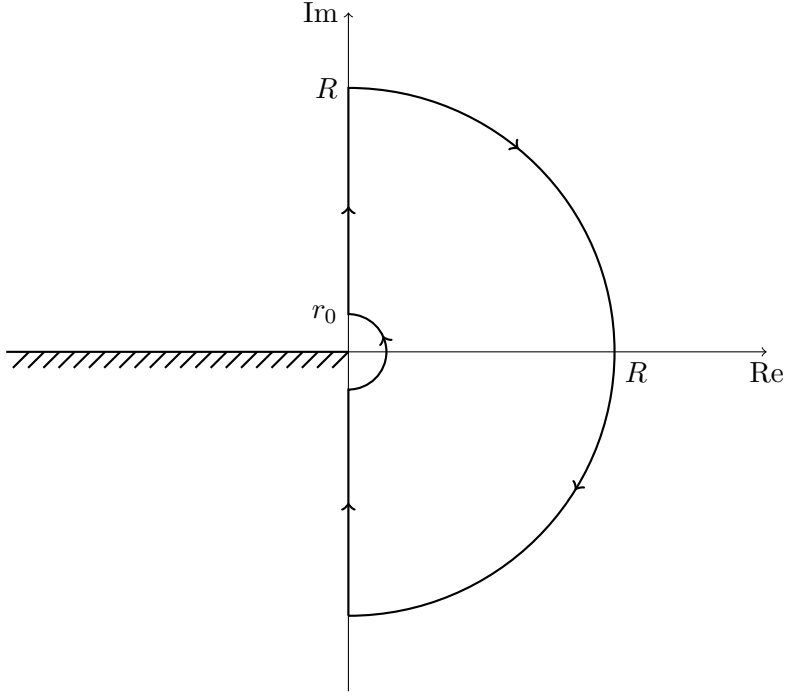
\begin{figure}
\begin{center}
  \begin{tikzpicture}[contour/.style={postaction={decorate, decoration={markings,
mark=at position 1.5cm with {\arrow[line width=1pt]{>}},
mark=at position 4cm with {\arrow[line width=1pt]{>}},
mark=at position 6cm with {\arrow[line width=1pt]{>}},
mark=at position 10cm with {\arrow[line width=1pt]{>}},
mark=at position 15cm with {\arrow[line width=1pt]{>}}}}},
    interface/.style={postaction={draw,decorate,decoration={border,angle=45,
                    amplitude=0.3cm,segment length=2mm}}}]
                    
\draw[->] (0,0) -- (5.5,0) node[below] {$\Re$};
\draw[->] (0,-4.5) -- (0,4.5) node[left]{$\Im$};

\draw[line width=.8pt,interface](0,0)--(-4.5,0);

\path[draw,line width=0.8pt,contour] (0,-3.5)  -- (0,-0.5)  arc (-90:90:0.5) node[left] {$r_{0}$} -- (0, 3.5)node[left] {$R$} arc (90:-90:3.5);

\node[below right] at (3.5,0) {$R$};
\end{tikzpicture}
\caption{Contour for $\U_{2,0}$ if $\xi<q+2$}
\label{fig:W20}
 \end{center}
\end{figure}

In the same way, we can moreover show that if we take instead the half-circle of radius $R$ in the right half-plane as contour for the integral in $(II)$, this converges to zero for $R\to 0$. Thus, since we have deformed the original contour in $(II)$ to the region where the integrand is analytic, we may conclude from Cauchy's Theorem that $\lim_{R\to\infty}\abs*{(II)}=0$ if $\xi<q+2$, i.e.\@ there is no contribution of $\U_{2,0}(y,q)$ stemming from the region $\xi<q+2$.

\subsubsection{Contribution of $\U_{2,0}(y,q)$ for $\xi>q+2$}

To estimate the contribution in the region $\xi>q+2$ we cannot use contour deformation due to the sign in the exponential factor. Therefore, we have to directly estimate the integral where we now take the original contour along the imaginary axis. Together with Lemma~\ref{Lem:analyticity:beta} and estimating by the most dominant terms, both at zero and infinity, we find
\begin{equation*}
 \abs*{(II)}\leq \frac{C}{(q+1)(q+2)\abs*{\xi-(q+2)}}\int_{-R}^{R}\abs*{\ee^{-y(\xi-(q+2))}-1}\frac{\abs{y}^{-\alpha-1}+\abs{y}^{\alpha-1}}{1+\abs*{y}}\dy.
\end{equation*}
In order to bound the integral on the right-hand side uniformly, we use $\abs*{\ee^{-y(\xi-(q+2))}-1}\leq \abs*{\xi-(q+2)}^{a_{*}}\abs*{y}^{a_{*}}$ in order to weaken on the one hand the singularity $\abs*{y}^{-\alpha-1}$ at zero and on the other hand the singularity at $q-2$ of $\abs*{\xi-(q+2)}^{-1}$. This then yields
\begin{equation*}
 \begin{split}
  \abs*{(II)}&\leq \frac{C}{(q+1)(q+2)\abs*{\xi-(q+2)}^{1-a_{*}}}\int_{-\infty}^{\infty}\frac{\abs{y}^{a_{*}-\alpha-1}}{(1+\abs{y})}+\frac{\abs{y}^{a_{*}+\alpha-1}}{(1+\abs{y})}\dy\\
  &\leq \frac{C}{(q+1)(q+2)}\abs*{\xi-(q+2)}^{a_{*}-1}.
 \end{split}
\end{equation*}
The contribution to the integral $V(q)$ can be estimated in the desired way due to Lemma~\ref{Lem:contribution:large:xi}.

\subsection{Contribution of $\U_{2,2}(y,q)$ for $\xi<q+2$}

Together with the relation $\Phi'(x)=-\eps\beta_{W}(x)x^{-1}\ee^{-x}$ and the cut-off function $\coup$ we can rewrite and split $\U_{2,2}(y,q)=\U_{2,2,1}(y,q)+\U_{2,2,2}(y,q)$ with
\begin{equation*}
 \begin{split}
  \U_{2,2,1}(y,q)&=\frac{1}{(q+1)(q+2)y^{\rho+\kappa}(1+y)}\int_{0}^{y}\ee^{-(q+3)x}\bigl(\eps \beta_{W}(x)\bigr)^{2}x^{\rho-1+\kappa}\frac{\ee^{\Phi(y)}}{\ee^{\Phi(x)}}\coup(\im x)\dx,\\
  \U_{2,2,2}(y,q)&=\frac{1}{(q+1)(q+2)y^{\rho+\kappa}(1+y)}\int_{0}^{y}\ee^{-(q+1)x}\bigl(\Phi'(x)\bigr)^{2}x^{\rho-1+\kappa}\frac{\ee^{\Phi(y)}}{\ee^{\Phi(x)}}(1-\coup(\im x)\dx.
 \end{split}
\end{equation*}

\subsubsection{Contribution of $\U_{2,2,1}(y,q)$}

Due to the cut-off $\coup$ we avoid the singularity of the integrand of $\U_{2,2,1}(\cdot,q)$ at the origin such that we can integrate by parts again. Precisely, we use that $\ee^{-(q+3)x}=-(q+3)^{-1}\del_{x}(\ee^{-(q+3)x})$ and the relation $\Phi'(x)=-\eps\beta_{W}(x)x^{-1}\ee^{-x}$ such that integration by parts readily implies that $\U_{2,2,1}(y,q)=\U_{d}(y,q)+\U_{it}(y,q)$ with 
\begin{equation*}
 \begin{split}
 \U_{d}(y,q)&=-D(q)\frac{\ee^{-(q+3)y}(\eps\beta_{W}(y))^2}{y(1+y)}\coup(\im y)\\
 &\qquad +\frac{D(q)}{y^{\rho+\kappa}(1+y)}\int_{0}^{y}\ee^{-(q+3)x}\del_{x}\Bigl(\bigl(\eps\beta_{W}(x)\bigr)^2 x^{\rho-1+\kappa}\coup(\im x)\Bigr)\frac{\ee^{\Phi(y)}}{\ee^{\Phi(y)}}\dx,\\
 \U_{it}(y,q)&=\frac{D(q)}{y^{\rho+\kappa}(1+y)}\int_{0}^{y}\ee^{-(q+4)x}\bigl(\eps\beta_{W}(x)\bigr)^{3} x^{\rho-2+\kappa}\coup(\im x)\frac{\ee^{\Phi(y)}}{\ee^{\Phi(y)}}\dx,
 \end{split}
\end{equation*}
where we also use the abbreviation $D(q)=((q+1)(q+2)(q+3))^{-1}$.
The goal is then to bound the expressions on the right-hand side by a function in $y$ which has a finite integral over $\R$. This works well for $\U_{d}(y,q)$ and due to the cut-off $\coup$ the region close to zero does not cause any problems here but unfortunately, for $\alpha$ too close to $1/2$ the decay of $\U_{it}(y,q)$ is not strong enough for large values of $\abs{y}$. Therefore, we have to integrate by parts once more in the expression $\U_{it}(y,q)$ which in the most dominant term yields an additional factor $\abs{y}^{\alpha-1}$ which then suffices. More precisely, we obtain together with \cref{Lem:analyticity:beta,Lem:analyticity:Phi} and the properties of $\coup$ that 
\begin{multline*}
 \abs*{y^{-\rho-\kappa}(1+y)^{-1}\int_{0}^{y}\ee^{-(q+3)x}\del_{x}\Bigl(\bigl(\eps\beta_{W}(x)\bigr)^2 x^{\rho-1+\kappa}\coup(\im x)\Bigr)\frac{\ee^{\Phi(y)}}{\ee^{\Phi(y)}}\dx}\\*
 \leq C\max\bigl\{\abs*{y}^{2\alpha-2},\abs*{y}^{-(\rho+1+\kappa)}\bigr\}\chi_{\{\abs{y}\geq 1/2\}}.
\end{multline*}
Since the right-hand side is uniformly integrable in $y$, we conclude that the contribution to $\Qo_{0}(\xi,q)$ stemming from $\U_{d}(y,q)$, i.e.\@ $\Qo_{d}=(2\pi\im)^{-1}\int_{-\im\infty}^{\im\infty}\ee^{y\xi}\U_{d}(y,q)\dy$ can be estimated as
\begin{equation}\label{eq:Q0:est:Qd}
 \begin{split}
  \abs*{\Qo_{d}}\leq C(q+1)^{-3} \int_{-\infty}^{\infty}\max\bigl\{\abs*{y}^{2\alpha-2},\abs*{y}^{-(\rho+1+\kappa}\bigr\}\chi_{\{\abs{y}\geq 1/2\}}\dy\leq C(q+1)^{-3}.
 \end{split}
\end{equation}
We next indicate how we conclude for the expression $\U_{it}(y,q)$. Together with \cref{Lem:analyticity:beta,Lem:analyticity:Phi} we can verify similarly as for $\U_{d}(y,q)$ that it holds
\begin{equation*}
 \abs*{\U_{it}(y,q)}\leq \frac{C}{(q+1)^3}\max\{\abs{y}^{3\alpha-2},\abs{y}^{-(\rho+1+\kappa)}\}\chi_{\{\abs{y}\geq 1/2\}}.
\end{equation*}
Thus, if $\alpha<1/3$ we already obtain an integrable function on the right-hand side and we can conclude. On the other hand, if $\alpha\geq 1/3$, we have to iterate the procedure of integration by parts once more to get
\begin{equation}\label{eq:bound:Witn}
 \abs*{\U_{it}(y,q)}\leq \frac{C}{(q+1)^4}\max\{\abs{y}^{3\alpha-3},\abs{y}^{-(\rho+1+\kappa)}\}\chi_{\{\abs{y}\geq 1/2\}}.
\end{equation}
From the bound~\eqref{eq:bound:Witn} we can now estimate the contribution of $\U_{it}(y,q)$ to the function $\Qo_{0}(\xi,q)$ which is given by
\begin{equation*}
 \begin{split}
 \abs[\bigg]{\frac{1}{2\pi\im}\int_{-\im\infty}^{\im\infty}\ee^{y\xi}\U_{it}(y,q)\dy}\leq \frac{C}{(q+1)^4}\int_{-\infty}^{\infty}\abs{y}^{\max\{3\alpha-3,-(\rho+1+\kappa)\}}\chi_{\{\abs{y}\geq 1/2\}}\dy\leq \frac{C}{(q+1)^3}.
 \end{split}
\end{equation*}
Together with~\eqref{eq:Q0:est:Qd} we thus have $\abs*{\Qo_{2,2,1,R}}\leq C(q+1)^{-3}$ uniformly with respect to $R$. The contribution to the integral $V(q)$ can then be controlled by Lemma~\ref{Lem:contribution:small:xi}.

\subsubsection{Contribution of $\U_{2,2,2}(y,q)$}\label{Sec:W222:small}

We have to estimate the expression $\Qo_{2,2,2}^{<}\vcc=\lim_{R\to\infty}(2\pi\im)^{-1}\int_{-\im R}^{\im R}\ee^{y\xi}\U_{2,2,2}(y,q)\dy$ while the limit exists due to the cut-off $1-\coup$ which is supported only in the region close to zero. If we introduce another cut-off function $\cut$ and change also to real variables we can split $\Qo_{2,2,2}^{<}=\Qo_{2,2,2,1}^{<}+\Qo_{2,2,2,2}^{<}$ with
\begin{multline*}
  \Qo_{2,2,2,1}^{<}\vcc=\frac{\im}{2\pi(q+1)(q+2)}\int_{-\infty}^{\infty}\ee^{-\im s(q+1)}(\im s)^{1+\rho+\kappa}\bigl(\Phi'(\im s)\bigr)^{2}\bigl(1-\coup(s)\bigr)\cdot\\*
  \shoveright{\cdot\int_{s}^{\sgn(s)\infty}\frac{\ee^{\im t\xi}\cut(t)}{(\im t)^{\rho+\kappa}(1+\im t)}\frac{\ee^{\Phi(\im t)}}{\ee^{\Phi(\im s)}}\dt\ds}\\*
  \shoveleft{\Qo_{2,2,2,2}^{<}\vcc=\frac{\im}{2\pi(q+1)(q+2)}\int_{-\infty}^{\infty}\ee^{-\im s(q+1)}(\im s)^{1+\rho+\kappa}\bigl(\Phi'(\im s)\bigr)^2\bigl(1-\coup(s)\bigr)\ee^{-\Phi(\im s)}\cdot}\\*
  \cdot\int_{s}^{\sgn(s)\infty}\frac{\ee^{\im t\xi}(1-\cut(t))}{(\im t)^{\rho+\kappa}(1+\im t)}\ee^{\Phi(\im t)}\dt\ds.
 \end{multline*}
To estimate $\Qo_{2,2,2,2}^{<}$, we introduce the abbreviation $D(q)=(q+1)^{-2}(q+2)^{-1}$ and we use $\ee^{-\im s(q+1)}=\im(q+1)^{-1}\del_{s}(\ee^{-\im s(q+1)}-1)$ to integrate by parts once more which gives
\begin{multline*}
  \Qo_{2,2,2,2}^{<}=\frac{D(q)}{2\pi}\int_{-\infty}^{\infty}\biggl(\Bigl(\ee^{-\im s(q+1)}-1\Bigr)\biggl[\del_{s}\Bigl((\im s)^{1+\rho+\kappa}\bigl(\Phi'(\im s)\bigr)^2(1-\coup(s))\Bigr)\\*
  \shoveright{ -\im(\im s)^{1+\rho+\kappa}\bigl(\Phi'(\im s)\bigr)^3(1-\coup(s))\biggr]\ee^{-\Phi(\im s)}\int_{s}^{\sgn(s)\infty}(\cdots)\dt\biggr)\ds}\\*
   -\frac{D(q)}{2\pi}\int_{-\infty}^{\infty}\Bigl(\ee^{-\im s(q+1)}-1\Bigr)\im s\bigl(\Phi'(\im s)\bigr)^{2}(1-\coup(s))\frac{\ee^{\im s\xi}(1-\cut(s))}{(1+\im s)}\ds.
 \end{multline*}
Together with Lemma~\ref{Lem:analyticity:Phi} one can deduce that 
\begin{equation*}
 \abs[\bigg]{\int_{s}^{\sgn(s)\infty}\frac{\ee^{\im t\xi}(1-\cut(t))}{(\im t)^{\rho+\kappa}(1+\im t)}\ee^{\Phi(\im t)}\dt}\leq C\quad \text{for all }s\in\R.
\end{equation*}
Moreover, taking also \cref{Lem:analyticity:Phi,Lem:reg:exponetial:decay,Lem:help:exponential:regularising} into account, it further follows
\begin{equation*}
 \begin{split}
  \abs*{\del_{s}\Bigl((\im s)^{1+\rho+\kappa}\bigl(\Phi'(\im s)\bigr)^2(1-\coup(s))\Bigr)\ee^{-\Phi(\im s)}}\leq C\eps^2\abs*{s}^{\rho-2\alpha-2+\kappa}\ee^{-\frac{d\eps}{\abs{s}^{\alpha}}}\chi_{\{\abs{s}\leq 1\}}\leq C\abs{s}^{\rho-2+\kappa}\chi_{\{\abs{s}\leq 1\}},\\
  \abs*{(\im s)^{1+\rho+\kappa}\bigl(\Phi'(\im s)\bigr)^3(1-\coup(s))\ee^{-\Phi(\im s)}}\leq C\eps^3\abs*{s}^{\rho-3\alpha-2+\kappa}\ee^{-\frac{d\eps}{\abs{s}^{\alpha}}}\chi_{\{\abs{s}\leq 1\}}\leq C\abs{s}^{\rho-2+\kappa}\chi_{\{\abs{s}\leq 1\}}.
 \end{split}
\end{equation*}
Thus, for a parameter $a\in(1-\rho+\delta,\theta)$ and $D(q)\leq (q+1)^{-3}$ it follows
\begin{equation*}
 \abs*{\Qo_{2,2,2,2}^{<}}\leq (q+1)^{a-3}\int_{-1}^{1}\abs*{s}^{a+\rho-2+\kappa}\ds\leq C(q+1)^{a-3}.
\end{equation*}
The contribution of this expression to $V(q)$ is then controlled in the desired manner due to Lemma~\ref{Lem:contribution:small:xi}.

To bound the term $\Qo_{2,2,2,1}^{<}$ we change variables $t\mapsto s+t$ and introduce the abbreviation $\widehat{D}(q,\xi)=((q+1)(q+2)(\xi-(q+1)))^{-1}$ which allows to rewrite
\begin{multline*}
 \Qo_{2,2,2,1}^{<}=\frac{\widehat{D}(q,\xi)}{2\pi}\int_{-\infty}^{\infty}\del_{s}\Bigl(\ee^{\im s(\xi-(q+1))}-1\Bigr)(\im s)^{1+\rho+\kappa}\bigl(\Phi'(\im s)\bigr)^{2}(1-\coup(s))\cdot\\*
 \cdot \int_{0}^{\sgn(s)\infty}\frac{\ee^{\im t\xi}\cut(s+t)}{(\im(s+t))^{\rho+\kappa}(1+\im(s+t))}\frac{\ee^{\Phi(\im(s+t))}}{\ee^{\Phi(\im s)}}\dt\ds.
\end{multline*}
An integration by parts thus yields
\begin{multline*}
  \Qo_{2,2,2,1}^{<}=-\frac{\widehat{D}(q,\xi)}{2\pi}\int_{-\infty}^{\infty}\Bigl(\ee^{\im s(\xi-(q+1))}-1\Bigr)\del_{s}\Bigl((\im s)^{1+\rho+\kappa}\bigl(\Phi'(\im s)\bigr)^{2}(1-\coup(s))\Bigr)\cdot\\*
  \shoveright{\cdot \int_{0}^{\sgn(s)\infty}\frac{\ee^{\im t\xi}\cut(s+t)}{(\im(s+t))^{\rho+\kappa}(1+\im(s+t))}\frac{\ee^{\Phi(\im(s+t))}}{\ee^{\Phi(\im s)}}\dt\ds}\\*
  \shoveleft{\phantom{\Qo_{2,2,2,1}^{<}{}={}}-\frac{\widehat{D}(q,\xi)}{2\pi}\int_{-\infty}^{\infty}\Bigl(\ee^{\im s(\xi-(q+1))}-1\Bigr)(\im s)^{1+\rho+\kappa}\bigl(\Phi'(\im s)\bigr)^{2}(1-\coup(s))\cdot}\\*
  \shoveright{\cdot \int_{0}^{\sgn(s)\infty}\del_{s}\biggl(\frac{\ee^{\im t\xi}\cut(s+t)}{(\im(s+t))^{\rho+\kappa}(1+\im(s+t))}\biggr)\frac{\ee^{\Phi(\im(s+t))}}{\ee^{\Phi(\im s)}}\dt\ds}\\*
  \shoveleft{\phantom{\Qo_{2,2,2,1}^{<}{}={}}-\frac{\im \widehat{D}(q,\xi)}{2\pi}\int_{-\infty}^{\infty}\Bigl(\ee^{\im s(\xi-(q+1))}-1\Bigr)(\im s)^{1+\rho+\kappa}\bigl(\Phi'(\im s)\bigr)^{2}(1-\coup(s))\cdot}\\*
  \cdot \int_{0}^{\sgn(s)\infty}\frac{\ee^{\im t\xi}\cut(s+t)}{(\im(s+t))^{\rho+\kappa}(1+\im(s+t))}\Bigl(\Phi'(\im(s+t))-\Phi'(\im s)\Bigr)\frac{\ee^{\Phi(\im(s+t))}}{\ee^{\Phi(\im s)}}\dt\ds.
 \end{multline*}
To estimate the terms on the right-hand side we take \cref{Lem:analyticity:Phi,Lem:reg:exponetial:decay,Lem:help:exponential:regularising} into account to obtain
\begin{equation*}
 \begin{split}
  &\phantom{{}\leq{}}\abs*{\del_{s}\Bigl((\im s)^{1+\rho+\kappa}\bigl(\Phi'(\im s)\bigr)^{2}(1-\coup(s))\Bigr)\int_{0}^{\sgn(s)\infty}\frac{\ee^{\im t\xi}\cut(s+t)}{(\im(s+t))^{\rho+\kappa}(1+\im(s+t))}\frac{\ee^{\Phi(\im(s+t))}}{\ee^{\Phi(\im s)}}\dt}\\
  &\leq C\abs{s}^{\rho-2\alpha-2+\kappa}\Bigl(\eps\abs{s}^{1+\alpha-\rho-\kappa}+\eps^2\ee^{-\frac{B\eps}{\abs{s}^{\alpha}}}\Bigr)\chi_{\{\abs{s}\leq 1\}}\leq C\Bigl(\eps\abs{s}^{-\alpha-1}+\abs{s}^{\rho-2+\kappa}\Bigr)\chi_{\{\abs{s}\leq 1\}}
 \end{split}
\end{equation*}
and
\begin{equation*}
 \begin{split}
  &\phantom{{}\leq{}} \abs*{(\im s)^{1+\rho+\kappa}\bigl(\Phi'(\im s)\bigr)^{2}(1-\coup(s))\int_{0}^{\sgn(s)\infty}\del_{s}\biggl(\frac{\ee^{\im t\xi}\cut(s+t)}{(\im(s+t))^{\rho+\kappa}(1+\im(s+t))}\biggr)\frac{\ee^{\Phi(\im(s+t))}}{\ee^{\Phi(\im s)}}\dt}\\
  &\leq C\abs{s}^{\rho-2\alpha-1+\kappa}\Bigl(\eps\abs{s}^{\alpha-\rho-\kappa}+\eps^2\abs{s}^{-\rho-\kappa}\ee^{-\frac{B\eps}{\abs{s}^{\alpha}}}\Bigr)\chi_{\{\abs{s}\leq 1\}}\leq C\Bigl(\eps\abs{s}^{-\alpha-1}+\abs{s}^{\rho-2+\kappa}\Bigr)\chi_{\{\abs{s}\leq 1\}}
 \end{split}
\end{equation*}
and
\begin{equation*}
 \begin{split}
  &\phantom{{}\leq{}} \abs*{(\im s)^{1+\rho+\kappa}\bigl(\Phi'(\im s)\bigr)^{2}(1-\coup(s))\int_{0}^{\sgn(s)\infty}\frac{\ee^{\im t\xi}\cut(s+t)(\Phi'(\im(s+t))-\Phi'(\im s))}{(\im(s+t))^{\rho+\kappa}(1+\im(s+t))}\frac{\ee^{\Phi(\im(s+t))}}{\ee^{\Phi(\im s)}}\dt}\\
  &\leq C\abs{s}^{\rho-2\alpha-1+\kappa}\Bigl(\eps\abs{s}^{\alpha-\rho-\kappa}+\eps^3\abs{s}^{-1-\alpha}\ee^{-\frac{B\eps}{\abs{s}^{\alpha}}}\Bigr)\chi_{\{\abs{s}\leq 1\}}\leq C\Bigl(\eps\abs{s}^{-\alpha-1}+\abs{s}^{\rho-2+\kappa}\Bigr)\chi_{\{\abs{s}\leq 1\}}.
 \end{split}
\end{equation*}
We then take the parameter $a_{*}\in(\max\{1-\rho+\delta,\alpha\},\theta)$ to obtain together with $\widehat{D}(q,\xi)=((q+1)(q+2)(\xi-(q+1)))^{-1}$ that
\begin{equation*}
 \abs*{\Qo_{2,2,2,1}^{<}}\leq C\frac{\abs*{\xi-(q+1)}^{a_{*}-1}}{(q+1)^2}\int_{-1}^{1}\eps\abs{s}^{a_{*}-1-\alpha}+\abs*{s}^{a_{*}+\rho-2+\kappa}\ds\leq C\frac{\abs*{\xi-(q+1)}^{a_{*}-1}}{(q+1)^2}.
\end{equation*}
Lemma~\ref{Lem:contribution:small:xi} then shows that the contribution to $V(q)$ can be estimated as desired.

\subsection{Contribution of $\U_{2,2}(y,q)$ for $\xi>q+2$}

We consider the expression $\Qo^{>}_{2,2,R}\vcc=(2\pi\im)^{-1}\int_{-\im R}^{\im R}\ee^{y\xi}\U_{2,2}(y,q)\dy$ and use the partition of unity, $1=\coup+(1-\coup)$ to split this as $\Qo^{>}_{2,2,R}=\Qo^{>}_{2,2,1,R}+\Qo^{>}_{2,2,2,R}$ where we also change to real variables such that the terms on the right-hand side read as
\begin{equation*}
 \begin{split}
  \Qo^{>}_{2,2,1,R}&\vcc=\frac{\im}{2\pi(q+1)(q+2)}\int_{-R}^{R}\frac{\ee^{\im t\xi}\coup(t)\ee^{\Phi(\im t)}}{(\im t)^{\rho+\kappa}(1+\im t)}\int_{0}^{t}\ee^{-(q+1)\im s}(\im s)^{1+\rho+\kappa}\bigl(\Phi'(\im s)\bigr)^2\ee^{-\Phi(\im s)}\ds\dt,\\
  \Qo^{>}_{2,2,2,R}&\vcc=\frac{\im}{2\pi(q+1)(q+2)}\int_{-R}^{R}\frac{\ee^{\im t\xi}(1-\coup(t))}{(\im t)^{\rho+\kappa}(1+\im t)}\int_{0}^{t}\ee^{-(q+1)\im s}(\im s)^{1+\rho+\kappa}\bigl(\Phi'(\im s)\bigr)^2\frac{\ee^{\Phi(\im t)}}{\ee^{\Phi(\im s)}}\ds\dt.
 \end{split}
\end{equation*}
In the expression $\Qo^{>}_{2,2,1,R}$ we can use the relation $\ee^{\im t\xi}=(\im\xi)^{-1}\del_{t}(\ee^{\im t\xi})$ and due to the cut-off $\coup$ we may thus integrate by parts to obtain
\begin{equation*}
 \begin{split}
  \Qo^{>}_{2,2,1,R}&=\frac{1}{2\pi(q+1)(q+2)\xi}\biggl(\ee^{\im t\xi}\frac{\coup(t)}{(\im t)^{\rho+\kappa}(1+\im t)}\ee^{\Phi(\im t)}\int_{0}^{t}(\cdots)\ds\biggr)\bigg|_{t=-R}^{t=R}\\
  &\qquad -\frac{1}{2\pi(q+1)(q+2)\xi}\int_{-R}^{R}\ee^{\im t\xi}\del_{t}\biggl(\frac{\coup(t)}{(\im t)^{\rho+\kappa}(1+\im t)}\ee^{\Phi(\im t)}\int_{0}^{t}(\cdots)\ds\dt\\
  &\qquad -\frac{\im}{2\pi(q+1)(q+2)\xi}\int_{-R}^{R}\ee^{\im t(\xi-(q+1))}\frac{\coup(t)}{(1+\im t)}t\bigl(\Phi'(\im t)\bigr)^2\dt\\
  &\qquad -\frac{\im}{2\pi(q+1)(q+2)\xi}\int_{-R}^{R}\ee^{\im t\xi}\frac{\coup(t)\Phi'(\im t)}{(\im t)^{\rho+\kappa}(1+\im t)}\ee^{\Phi(\im t)}\int_{0}^{t}(\cdots)\ds\dt\\
  &=(I)+(II)+(III)+(IV).
 \end{split}
\end{equation*}
To estimate the right-hand side, we first note that Lemma~\ref{Lem:analyticity:Phi} yields
\begin{equation*}
 \abs*{\ee^{-(q+1)\im s}(\im s)^{1+\rho+\kappa}\bigl(\Phi'(\im s)\bigr)^2\ee^{-\Phi(\im s)}}\leq C\eps^2\abs{s}^{\rho+2\alpha-1+\kappa}\quad \text{if }\abs{s}\geq 1.
\end{equation*}
Moreover, in addition with \cref{Lem:reg:exponetial:decay,Lem:help:exponential:regularising} we get
\begin{equation*}
 \abs*{\ee^{-(q+1)\im s}(\im s)^{1+\rho+\kappa}\bigl(\Phi'(\im s)\bigr)^2\ee^{-\Phi(\im s)}}\leq C\eps^2\abs{s}^{\rho-2\alpha-1+\kappa}\ee^{-\frac{d\eps}{\abs{s}^{\alpha}}}\leq C\abs{s}^{\rho-1+\kappa}\quad \text{if }\abs{s}\leq 1.
\end{equation*}
Thus, we immediately conclude that
\begin{equation*}
 \abs*{\int_{0}^{t}\ee^{-(q+1)\im s}(\im s)^{1+\rho+\kappa}\bigl(\Phi'(\im s)\bigr)^2\ee^{-\Phi(\im s)}\ds}\leq C\abs*{t}^{\rho+2\alpha+\kappa}\quad \text{if }t\geq 1/2.
\end{equation*}
With this estimate we then obtain for the expressions $(I)$--$(III)$ that
\begin{equation*}
 \abs*{(I)}\leq C(q+1)^{-2}\xi^{-1}R^{2\alpha-1}\longrightarrow 0\quad \text{if }R\longrightarrow \infty,
\end{equation*}
while we also use $\alpha<1/2$ here. Similarly, estimating by the most dominant terms, we obtain
\begin{equation*}
 \begin{split}
  \abs*{(II)}&\leq C(q+1)^{-2}\xi^{-1}\int_{-\infty}^{\infty}\abs{s}^{2\alpha-2}\chi_{\{\abs{t}\geq 1/2\}}\dt\leq C(q+1)^{-2}\xi^{-1},\\
  \abs*{(III)}&\leq C(q+1)^{-2}\xi^{-1}\int_{-\infty}^{\infty}\abs*{t}^{2\alpha-2}\chi_{\{\abs{t}\geq 1/2\}}\dt\leq C(q+1)^{-2}\xi^{-1}.
 \end{split}
\end{equation*}
The integrand in $(IV)$ can be estimated in the same fashion and we obtain
\begin{equation*}
 \abs*{\ee^{\im t\xi}\frac{\coup(t)\Phi'(\im t)}{(\im t)^{\rho+\kappa}(1+\im t)}\ee^{\Phi(\im t)}\int_{0}^{t}(\cdots)\ds}\leq C(q+1)^{-2}\xi^{-1}\abs{t}^{3\alpha-2}\chi_{\{\abs{t}\geq 1/2\}}.
\end{equation*}
For $\alpha<1/3$ we can proceed as before and obtain the same bound as for $(II)$ and $(III)$. However, if $\alpha\geq 1/3$ the right-hand side is not integrable over $\R$ such that we not yet obtain a uniform estimate with respect to $R$. Instead, we have to iterate the previous procedure on the expression $(IV)$. More precisely, we use the relation $\Phi'(\im t)=-\eps (\im t)^{-1}\beta_{W}(\im t)\ee^{-\im t}$ which allows to rewrite
\begin{equation*}
 (IV)=\frac{\eps}{2\pi(q+1)(q+2)\xi(\xi-1)}\int_{-R}^{R}\del_{t}\Bigl(\ee^{\im t(\xi-1)}-1\Bigr)\frac{\coup(t)\beta_{W}(\im t)}{(\im t)^{1+\rho+\kappa}(1+\im t)}\ee^{\Phi(\im t)}\int_{0}^{t}(\cdots)\ds\dt.
\end{equation*}
Then, proceeding as before, i.e.\@ integrating by parts and estimating the different terms one can conclude that we can take the limit $R\to\infty$ to obtain that
\begin{equation*}
 \abs*{(IV)}\leq C(q+1)^{-2}\xi^{-1}(\xi-1)^{-1}\leq C(q+1)^{-2}\xi^{-1}\quad \text{for }\xi>q+2.
\end{equation*}
The contribution to $V(q)$ is then controlled by Lemma~\ref{Lem:contribution:large:xi}.

We conclude the proof by estimating the contribution stemming from $\Qo^{>}_{2,2,2,R}$ and we note that due to the fact that $1-\coup$ is supported only in the region close to zero, we may directly pass to the limit $R\to\infty$ and just consider $\Qo^{>}_{2,2,2}\vcc=\lim_{R\to\infty}\Qo^{>}_{2,2,2,R}$. We then rewrite $\Qo^{>}_{2,2,2}$ by applying Fubini's Theorem and changing variables $t\mapsto s+t$ which yields
\begin{multline*}
 \Qo^{>}_{2,2,2}=\frac{1}{2\pi(q+1)(q+2)(\xi-(q+1))}\int_{\infty}^{\infty}\del_{s}\Bigl(\ee^{\im (\xi-(q+1))}-1\Bigr)(\im s)^{1+\rho+\kappa}\bigl(\Phi'(\im s)\bigr)^{2}\cdot\\*
 \cdot \int_{0}^{\sgn(s)\infty}\frac{\ee^{\im t\xi}(1-\coup(s+t))}{(\im(s+t))^{\rho+\kappa}(1+\im(s+t))}\frac{\ee^{\Phi(\im (s+t))}}{\ee^{\Phi(\im s)}}\dt\ds.
\end{multline*}
We use the abbreviation $\widehat{D}(q,\xi)=((q+1)(q+2)(\xi-(q+1)))^{-1}$ and integrate by parts to find
\begin{multline*}
  \Qo^{>}_{2,2,2}=-\frac{\widehat{D}(q,\xi)}{2\pi}\int_{\infty}^{\infty}\Bigl(\ee^{\im (\xi-(q+1))}-1\Bigr)\del_{s}\Bigl((\im s)^{1+\rho+\kappa}\bigl(\Phi'(\im s)\bigr)^{2}\Bigr)\cdot\\
\shoveright{\cdot \int_{0}^{\sgn(s)\infty}\frac{\ee^{\im t\xi}(1-\coup(s+t))}{(\im(s+t))^{\rho+\kappa}(1+\im(s+t))}\frac{\ee^{\Phi(\im (s+t))}}{\ee^{\Phi(\im s)}}\dt\ds}\\*
\shoveleft{\phantom{\Qo^{>}_{2,2,2}{}={}} -\frac{\widehat{D}(q,\xi)}{2\pi}\int_{\infty}^{\infty}\Bigl(\ee^{\im (\xi-(q+1))}-1\Bigr)(\im s)^{1+\rho+\kappa}\bigl(\Phi'(\im s)\bigr)^{2}\cdot}\\*
\shoveright{ \cdot \int_{0}^{\sgn(s)\infty}\del_{s}\biggl(\frac{\ee^{\im t\xi}(1-\coup(s+t))}{(\im(s+t))^{\rho+\kappa}(1+\im(s+t))}\biggr)\frac{\ee^{\Phi(\im (s+t))}}{\ee^{\Phi(\im s)}}\dt\ds}\\
\shoveleft{\phantom{\Qo^{>}_{2,2,2}{}={}} -\frac{\im \widehat{D}(q,\xi)}{2\pi}\int_{\infty}^{\infty}\Bigl(\ee^{\im (\xi-(q+1))}-1\Bigr)(\im s)^{1+\rho+\kappa}\bigl(\Phi'(\im s)\bigr)^{2}\cdot}\\*
 \cdot \int_{0}^{\sgn(s)\infty}\frac{\ee^{\im t\xi}(1-\coup(s+t))}{(\im(s+t))^{\rho+\kappa}(1+\im(s+t))}\Bigl(\Phi'(\im (s+t))-\Phi'(\im s)\Bigr)\frac{\ee^{\Phi(\im (s+t))}}{\ee^{\Phi(\im s)}}\dt\ds.
 \end{multline*}
To bound the right-hand side we note that $1-\coup$ and $\cut$ have the same qualitative behaviour at zero. Therefore, we can proceed in exactly the same way as in Section~\ref{Sec:W222:small} to deduce with \cref{Lem:analyticity:Phi,Lem:gain:of:alpha:coup,Lem:help:exponential:regularising} that it holds
\begin{equation*}
 \abs*{\Qo^{>}_{2,2,2}}\leq \frac{C}{(q+1)^2(\xi-(q+1))^{1-a}}\int_{-1}^{1}\eps\abs*{s}^{a_{*}-\alpha-1}+\abs{s}^{a_{*}+\rho-2+\kappa}\ds\leq \frac{C}{(q+1)^2(\xi-(q+1))^{1-a_{*}}}
\end{equation*}
for $a_{*}\in(\max\{\alpha,1-\rho+\delta\},\theta\}$. The contribution to the integral $V(q)$ can then be estimated in the required fashion due to Lemma~\ref{Lem:contribution:large:xi}.
\end{proof}

\section{Asymptotic behaviour of several auxiliary functions}\label{Sec:asymptotics}

In this section, we will collect certain asymptotic properties of the functions $\beta_{W}(\cdot,f)$ and $\Phi(\cdot ,f)$ which have been defined in~\cref{eq:def:beta:and:R,eq:def:Phi} and where $f$ is a self-similar profile, i.e.\@ a solution to~\eqref{eq:self:sim}. It will turn out, that most of the properties that we will derive depend mainly on the existence the moment of order $\alpha$ of $f$ and on the asymptotic behaviour of the kernel $W$ as given by~\eqref{eq:Ass:asymptotic}.

\subsection{Bounds on moments}
 
 In this first subsection we will collect several properties and uniform bounds for certain moments $\mom_{\gamma}$, as given by~\eqref{eq:def:moment}, of self-similar profiles, which will be needed in the following.
 
 \begin{lemma}\label{Lem:moment:representation}
  Let $\gamma\in[0,\rho)$. Then the relation 
  \begin{equation*}
   \int_{0}^{\infty}x^{\gamma} f(x)\dx=-\frac{1}{\Gamma(1-\gamma)}\int_{0}^{\infty}\xi^{-\gamma}(\T f)'(\xi)\dxi
  \end{equation*}
 holds for each solution $f$ of~\eqref{eq:self:sim}.
 \end{lemma}

 \begin{proof}
  We first note that it holds $x^{\gamma-1}=(\Gamma(1-\gamma))^{-1}\int_{0}^{\infty}\xi^{-\gamma}\ee^{-\xi x}\dxi$. Together with Fubini's Theorem this yields 
  \begin{equation*}
   \int_{0}^{\infty}x^{\gamma} f(x)\ee^{-qx}\dx=\int_{0}^{\infty} x^{\gamma-1}xf(x)\ee^{-qx}\dx=-\int_{0}^{\infty}\xi^{-\gamma}(\T f)'(q+\xi)\dxi.
  \end{equation*}
 Due to Lebesgue's Theorem, we can take the limit $q\to 0$ which gives the claim.
 \end{proof}
 
 As an immediate consequence of the preceding lemma, we obtain uniform convergence of moments.
 
 \begin{lemma}\label{Lem:moments:convergence}
  For each $\gamma\in[0,\rho)$ and any $\delta>0$ it holds for sufficiently small $\eps>0$ that
  \begin{equation*}
   \abs*{\mom_{\gamma}-\bar{\mom}_{\gamma}}\leq \delta
  \end{equation*}
 for every solution $f$ of~\eqref{eq:self:sim}.
 \end{lemma}

 \begin{proof}
  Lemma~\ref{Lem:moment:representation} allows to rewrite and estimate $\abs*{\mom_{\gamma}-\bar{\mom}_{\gamma}}\leq C\int_{0}^{\infty}\xi^{-\gamma}\abs[\big]{\bigl(\T(f-\bar{f})\bigr)'(\xi)}\dxi$. From this, we deduce
  \begin{equation*}
   \abs*{\mom_{\gamma}-\bar{\mom}_{\gamma}}\leq C\lsnorm*{1}{0}{\theta}{\T(f-\bar{f})}\int_{0}^{\infty}\xi^{\rho-1-\gamma}(\xi+1)^{-\theta-\rho}\dxi\leq C\lsnorm*{1}{0}{\theta}{\T(f-\bar{f})}.
  \end{equation*}
 The claim then follows from Proposition~\ref{Prop:closeness:two:norm}.
 \end{proof}

 Moreover, we can bound $\mom_{\alpha}$ uniformly from below.
 
 \begin{lemma}\label{Lem:mom:lower:bd}
  For any $\alpha\in(0,\rho)$ there exists a constant $D>0$ such that it holds for sufficiently small $\eps>0$ that $\mom_{\alpha}\geq D$ for each solution $f$ of~\eqref{eq:self:sim}.
 \end{lemma}

 \begin{proof}
  The claim follows immediately from Lemma~\ref{Lem:moments:convergence} since $\bar{\mom}_{\alpha}>0$, i.e.\@ one can choose for example $D=\bar{\mom}_{\alpha}/2$.
 \end{proof}
 
 \subsection{Asymptotic behaviour of $\beta_{W}$ and $\Phi$}
 
 \begin{lemma}\label{Lem:asymptotics:beta}
  For each solution $f$ of~\eqref{eq:self:sim} the function $\beta_{W}(\cdot,f)$ as defined in~\eqref{eq:def:beta:and:R} satisfies
  \begin{equation*}
   \beta_{W}(x,f)\sim C_{W}\mom_{\alpha}x^{-\alpha}\quad \text{on }\Ccut \text{ for  }x\longrightarrow 0.
  \end{equation*}
 Note that the precise meaning of $\sim$ is given in~\eqref{eq:def:asymptotics:complex:plane}.
 \end{lemma}
 
 \begin{proof}
  Due to~\eqref{eq:Ass:asymptotic}, for each $\nu>0$ there exists $\delta_{\nu}>0$ such that $\abs*{W(\xi,1)-C_{W}\xi^{-\alpha}}\leq \nu \abs*{\xi}^{-\alpha}$ for $\xi\in\Ccut$ with $\abs*{\xi}\leq \delta_{\nu}$. Moreover, by the homogeneity of $W$ the integral which defines $\beta_{W}$ can be split as
  \begin{multline*}
   x^{\alpha}\beta_{W}(x)=\int_{0}^{\abs*{x}/\delta_{\nu}}x^{\alpha}W\Bigl(\frac{x}{z},1\Bigr)f(z)\dz\\*
   +\int_{\abs*{x}/\delta_{\nu}}^{\infty}\biggl(x^{\alpha}W\Bigl(\frac{x}{z},1\Bigr)-C_{W}\Bigl(\frac{1}{z}\Bigr)^{-\alpha}\biggr)f(z)\dz+\int_{\abs*{x}/\delta_{\nu}}^{\infty}C_{W}\Bigl(\frac{1}{z}\Bigr)^{-\alpha}f(z)\dz\\*
   =\vcc(I)+(II)+(III).
  \end{multline*}
 Due to the bound on $W$ provided by~\eqref{eq:Ass:analytic} and Lemma~\ref{Lem:moment:est:2} one deduces that $\abs*{(I)}\leq \mom_{-\alpha}\abs*{x}^{2\alpha}+\frac{\mom_{0}}{\delta_{\nu}^{\alpha}}\abs*{x}^{\alpha}$. Moreover, \eqref{eq:Ass:asymptotic} together with the choice of $\delta_{\nu}$ as explained above yields $\abs*{(II)}\leq \nu\mom_{\alpha}$. Finally, an explicit computation gives $(III)=C_{W}\mom_{\alpha}-C_{W}\int_{0}^{\abs*{x}/\delta_{\nu}}z^{\alpha}f(z)\dz$ and $C_{W}\int_{0}^{\abs*{x}/\delta_{\nu}}z^{\alpha}f(z)\dz\leq C_{W}\frac{\mom_{0}}{\delta_{\nu}^{\alpha}}\abs*{x}^{\alpha}$. In summary, we thus have
 \begin{equation*}
  \abs*{x^{\alpha}\beta_{W}(x)-C_{W}\mom_{\alpha}}\leq \mom_{-\alpha}\abs*{x}^{2\alpha}+(1+C_{W})\frac{\mom_{0}}{\delta_{\nu}^{\alpha}}\abs*{x}^{\alpha}+\nu\mom_{\alpha}.
 \end{equation*}
The claim then follows, if we take first $\nu$ and then $\abs*{x}$ small.
 \end{proof}
 
 \begin{lemma}\label{Lem:asymptotics:Re:Phi}
  For any solution $f$ of~\eqref{eq:self:sim} and $\Phi=\Phi(\cdot,f)$ as defined in~\eqref{eq:def:Phi} it holds
  \begin{equation*}
   \Re\bigl(\im \Phi'(\im s)\bigr)\sim -\eps C_{W}\mom_{\alpha}\cos\Bigl(\frac{\alpha \pi}{2}\Bigr)s^{-1}\abs*{s}^{-\alpha}\quad \text{ for } s\in\R \text{ and }s\longrightarrow 0.
  \end{equation*}
 \end{lemma}

 \begin{proof}
 The proof of this statement follows in a straightforward manner from Lemma~\ref{Lem:asymptotics:beta} if one uses that $\Phi'(\im s)=-\eps (\im s)^{-1}\beta_{W}(\im s)\ee^{-\im s}$ and $(\im s)^{-\alpha}=\abs*{s}^{-\alpha}\exp\left(\im \sgn(s)\frac{\alpha \pi}{2}\right)$.
 \end{proof}

\begin{lemma}\label{Lem:asymptotics:Phi}
 For any solution $f$ of~\eqref{eq:self:sim} and $\Phi=\Phi(\cdot,f)$ as defined in~\eqref{eq:def:Phi}, it holds
 \begin{equation*}
  \Phi(x)\sim \frac{C_{W}\mom_{\alpha}\eps}{\alpha}x^{-\alpha}\quad \text{for }x\in\Ccut\text{ and }x\longrightarrow 0
 \end{equation*}
with $\sim$ as explained in~\eqref{eq:def:asymptotics:complex:plane}.
\end{lemma}

\begin{proof}
 The proof proceeds similarly as the one for Lemma~\ref{Lem:asymptotics:beta}, therefore we will only sketch it. Due to Lemma~\ref{Lem:asymptotics:beta} for any $\nu>0$, we may fix $\delta_{\nu}>0$ such that it holds $\abs*{\beta_{W}(z)-C_{W}\mom_{\alpha}z^{-\alpha}}\leq \nu \abs*{z}^{-\alpha}$ for $\abs*{z}\leq \delta_{\nu}$, where we may assume for simplicity also that $\delta_{\nu}<1$. For $x\in\Ccut$ with $\abs*{x}\leq \delta_{\nu}$ we then deform the contour of integration as illustrated in Figure~\ref{fig:asymp}.
 
  \begin{figure}
\begin{center}
  \begin{tikzpicture}[contour/.style={postaction={decorate, decoration={markings,
mark=at position 3cm with {\arrow[line width=1pt]{>}},mark=at position 6.2cm with {\arrow[line width=1pt]{>}}
}}},
    interface/.style={postaction={draw,decorate,decoration={border,angle=45,
                    amplitude=0.3cm,segment length=2mm}}}]

\draw[->] (0,0) -- (7,0) node[below] {$\Re$};
\draw[->] (0,-2.5) -- (0,3) node[left] {$\Im$};
\draw[line width=.8pt,interface](0,0)--(-4.5,0);

\path[draw,line width=0.8pt,contour] (-1.732,1) arc(150:0:2) node[below left]{$\abs*{x}$} -- (3.5,0)node[below]{$\delta_{\nu}$};
\path[draw, dashed, line width=0.8pt, contour] (-1.732,-1) arc(-150:0:2);

\draw[black,fill=black] (-1.732,1) circle (.4ex);
\draw[black,fill=black] (-1.732,-1) circle (.4ex);
\draw[black,fill=black] (3.5,0) circle (.4ex);
\node[below] at (-2,1) {$x$};
\end{tikzpicture}
\caption{Contour connecting $x$ and $\delta_{\nu}$}
\label{fig:asymp}
 \end{center}
\end{figure}
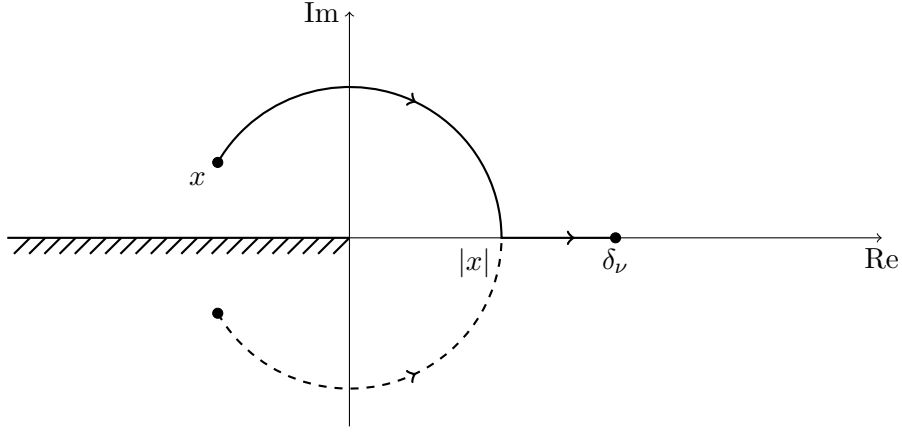
 
 This then yields the splitting of the integral which defines $\Phi$ as
 \begin{multline*}
   \Phi(x)=\eps\int_{x}^{\delta_{\nu}}C_{W}\mom_{\alpha}z^{-\alpha-1}\dz+\eps\int_{x}^{\delta_{\nu}}\frac{\beta_{W}(z)-C_{W}\mom_{\alpha}z^{-\alpha}}{z}\dz\\*
   +\eps\int_{x}^{\delta_{\nu}}\frac{\beta_{W}(z)}{z}(\ee^{-z}-1)\dz+\eps\int_{\delta_{\nu}}^{\infty}\frac{\beta_{W}(z)}{z}\ee^{-z}\dz=(I)+(II)+(III)+(IV).
 \end{multline*}
 An explicit computation shows that $(I)=\frac{C_{W}\mom_{\alpha}\eps}{\alpha}(x^{-\alpha}-\delta_{\nu}^{-\alpha})$. Next, we obtain from the choice of $\delta_{\nu}$ similarly as in the proof of Lemma~\ref{Lem:asymptotics:beta} that $\abs*{(II)}\leq C\nu \eps \abs*{x}^{-\alpha}$. To estimate $(III)$, one can use that $\abs*{1-\ee^{-z}}\leq C z$ and $\delta_{\nu}<1$. Together with Lemma~\ref{Lem:analyticity:beta} a straightforward computation then yields $\abs*{(III)}\leq C\eps$ and moreover also $\abs*{(IV)}\leq C\eps \delta_{\nu}^{-1-\alpha}$. In summary, we have
 \begin{equation*}
  \abs*{x^{\alpha}\Phi(x)-\frac{C_{W}\mom_{\alpha}\eps}{\alpha}}\leq C\eps \biggl(\delta_{\nu}^{-1-\alpha}+1+\frac{C_{W}\mom_{\alpha}}{\alpha}\delta_{\nu}^{-\alpha}\biggr)\abs*{x}^{\alpha}+C\eps \nu.
 \end{equation*}
 To conclude the proof, we first choose $\nu$ and then $\abs*{x}$ sufficiently small.
\end{proof}

\subsection{Regularity properties close to zero}

\begin{lemma}\label{Lem:reg:exponetial:decay}
 There exist constants $C,d>0$ such that it holds
 \begin{equation*}
  \exp\bigl(-\Phi(\im s)\bigr)\leq C\exp\bigl(-d\eps\abs*{s}^{-\alpha}\bigr)\quad \text{if }\abs*{s}\leq 1
 \end{equation*}
 for each solution $f$ of~\eqref{eq:self:sim} with $\Phi=\Phi(\cdot,f)$ as given by~\eqref{eq:def:Phi}. In particular we have $\exp\bigl(-\Phi(\im s)\bigr)\leq C$ for $\abs*{s}\leq 1$.
\end{lemma}

\begin{proof}
 It holds $\Re((\im s)^{-\alpha})=\abs*{s}^{-\alpha}\cos(\alpha\pi/2)$. We define $D_{\alpha}\vcc=\frac{C_{W}\mom_{\alpha}}{\alpha}\cos(\alpha\pi/2)$ and note that according to Lemma~\ref{Lem:asymptotics:Phi} there exists $\delta_{*}>0$ such that $\abs*{\Phi(\im s)-D_{\alpha}(\im s)^{-\alpha}}\leq \frac{D_{\alpha}}{2}\eps\abs*{s}^{-\alpha}$ for $\abs*{s}\leq \delta_{*}$. On the other hand, Lemma~\ref{Lem:analyticity:Phi} ensures that there exists a constant $C_{\delta_{*}}$ such that $\abs*{\Phi(\im s)-D_{\alpha}(\im s)^{-\alpha}}\leq C_{\delta_{*}}\eps$ if $\abs*{s}\geq \delta_{*}$. Together, we find that $\abs*{\Phi(\im s)-D_{\alpha}(\im s)^{-\alpha}}\leq \frac{D_{\alpha}}{2}\eps\abs*{s}^{-\alpha}+C_{\delta_{*}}\eps$ for all $s\neq 0$. We then use that
 \begin{equation*}
  \begin{split}
     \abs*{\exp\bigl(-\Phi(\im s)\bigr)}&=\abs*{\exp(-\Phi(\im s)+D_{\alpha}\eps (\im s)^{\alpha}\bigr)\exp\bigl(-D_{\alpha}(\im s)^{-\alpha}\bigr)}\\
     &\leq \exp\bigl(\abs*{\Phi(\im s)-D_{\alpha}\eps (\im s)^{\alpha}}\bigr)\exp\bigl(-D_{\alpha}\eps \Re((\im s)^{-\alpha})\bigr).
  \end{split}
 \end{equation*}
 This yields the estimate $\abs*{\exp(-\Phi(\im s))}\leq \exp(C_{\delta_{*}})\exp(-(D_{\alpha}/2)\abs*{s}^{-\alpha})$ and thus the claim follows with $C=\exp(C_{\delta_{*}})$ and $d=D_{\alpha}/2$ while we also note that $d>0$ due to Lemma~\ref{Lem:mom:lower:bd}.
\end{proof}

\begin{lemma}\label{Lem:help:exponential:regularising}
 For constants $B>0$ and $a,b\in\R$ such that $a+b>0$ there exists $C>0$ which only depends on $\alpha$, $B$ and $a+b$ such that it holds
 \begin{equation*}
  x^{-a}\ee^{-\frac{B\eps}{x^{\alpha}}}\leq C\eps^{-\frac{a+b}{\alpha}}x^{b}
 \end{equation*}
 for all $x>0$.
\end{lemma}
 
 \begin{proof}
  By an elementary computation one sees that the function $x\mapsto x^{-a-b}\exp(-B\eps x^{-\alpha})$ attains its maximum at $x_{\text{max}}=\left((\alpha B\eps)/(a+b)\right)^{1/\alpha}$ which directly implies the claim.
 \end{proof}

We conclude this section with two auxiliary statements which are essential in the proof of Proposition~\ref{Prop:Q0:int:est}.

\begin{lemma}\label{Lem:gain:of:alpha:cut}
 For the function $\cut$ as given in~\eqref{eq:cutoffs} there exist constants $B,C>0$ such that it holds
 \begin{align*}
  \abs*{\int_{0}^{\sgn(s)\infty}\frac{\cut(s+t)\ee^{\im t \xi}}{(\im(s+t))^{\rho+\kappa}(1+\im(s+t))}\frac{\ee^{\Phi(\im (s+t))}}{\ee^{\Phi(\im s)}}\dt}&\leq C\biggl(\frac{\abs{s}^{1+\alpha-\rho-\kappa}}{\eps}+\ee^{-\frac{B\eps}{\abs*{s}^{\alpha}}}\biggr)\chi_{\{\abs*{s}\leq 1\}},\\
  \abs*{\int_{0}^{\sgn(s)\infty}\del_{s}\Biggl(\frac{\cut(s+t)\ee^{\im t \xi}}{(\im(s+t))^{\rho+\kappa}(1+\im(s+t))}\Biggr)\frac{\ee^{\Phi(\im (s+t))}}{\ee^{\Phi(\im s)}}\dt}&\leq C\biggl(\frac{\abs{s}^{\alpha-\rho-\kappa}}{\eps}+\abs*{s}^{\rho-\kappa}\ee^{-\frac{B\eps}{\abs*{s}^{\alpha}}}\biggr)\chi_{\{\abs*{s}\leq 1\}}
 \end{align*}
 and
  \begin{multline*}
   \abs*{\int_{0}^{\sgn(s)\infty}\frac{\cut(s+t)\ee^{\im t \xi}(\Phi'(\im(s+t))-\Phi'(\im s))}{(\im(s+t))^{\rho+\kappa}(1+\im(s+t))}\frac{\ee^{\Phi(\im (s+t))}}{\ee^{\Phi(\im s)}}\dt}\\*
   \leq C\biggl(\frac{\abs{s}^{\alpha-\rho-\kappa}}{\eps}+\eps\abs*{s}^{-1-\alpha}\ee^{-\frac{B\eps}{\abs*{s}^{\alpha}}}\biggr)\chi_{\{\abs*{s}\leq 1\}}.
  \end{multline*}

\end{lemma}

\begin{proof}
 Since $\cut$ is supported in the interval $[-2,2]$ it suffices to consider $\abs*{s},\abs*{t}\leq 2$. In the following, we will moreover only restrict to the case $s,t>0$ while $s,t<0$ can be treated analogously. From the asymptotic properties of $\Phi$ and $\beta_{W}$ provided by \cref{Lem:asymptotics:Phi,Lem:asymptotics:beta} as well as \cref{Lem:analyticity:Phi,Lem:analyticity:beta} it is straightforward to derive the bound 
 \begin{equation*}
  \Re\bigl(\Phi(\im (s+t))-\Phi(\im s)\bigr)\leq \begin{cases}
                                                  -B\eps s^{-\alpha} +c_{1} &\text{if }s\leq t\\
                                                  -d\eps s^{-1-\alpha} t+c_{2} &\text{if }s\geq t,
                                                 \end{cases}
 \end{equation*}
with constants $B,d,c_1,c_2>0$.

To prove the first estimate of the lemma, we split the integral and use the previous bound to get
\begin{equation}\label{eq:gain:alpha:1}
 \begin{split}
  &\phantom{{}\leq{}} \abs*{\int_{0}^{\sgn(s)\infty}\frac{\cut(s+t)\ee^{\im t \xi}}{(\im(s+t))^{\rho+\kappa}(1+\im(s+t))}\frac{\ee^{\Phi(\im (s+t))}}{\ee^{\Phi(\im s)}}\dt}\leq C\int_{0}^{2}\frac{\exp(\Re(\Phi(\im (s+t))-\Phi(\im s)))}{(s+t)^{\rho+\kappa}}\dt\\
  &\leq C\int_{0}^{s}\frac{\exp(-d\eps s^{-1-\alpha}t)}{(s+t)^{\rho+\kappa}}\dt+C\int_{s}^{2}\frac{\exp(-B\eps s^{-\alpha})}{(s+t)^{\rho+\kappa}}\dt.
 \end{split}
\end{equation}
To estimate the first integral on the right-hand side, we change variables $t\mapsto (d\eps)^{-1}s^{1+\alpha}t$ which yields
\begin{equation*}
 \begin{split}
  C\int_{0}^{s}\frac{\exp(-d\eps s^{-1-\alpha}t)}{(s+t)^{\rho+\kappa}}\dt&=C\frac{s^{1+\alpha}}{d\eps}\int_{0}^{\frac{d\eps}{s^{\alpha}}}\ee^{-t}\biggl(s+\frac{ts^{1+\alpha}}{d\eps}\biggr)^{-\rho-\kappa}\dt\\
  &\leq C\frac{s^{\alpha+1-\rho-\kappa}}{d\eps}\int_{0}^{\infty}\ee^{-t}\dt\leq \frac{C}{\eps}s^{1+\alpha-\rho-\kappa}.
 \end{split}
\end{equation*}
Moreover, the second integral on the right-hand side of~\eqref{eq:gain:alpha:1} can be estimated as
\begin{equation*}
 C\int_{s}^{2}\frac{\exp(-B\eps s^{-\alpha})}{(s+t)^{\rho+\kappa}}\dt\leq C\exp(-B\eps s^{-\alpha})\int_{0}^{2}t^{-\rho-\kappa}\dt\leq C\exp(-B\eps s^{-\alpha}).
\end{equation*}
Note that the last step holds since $\rho<1$ and $\abs{\kappa}$ can be made as small as needed. If we summarise the two previous estimates, the first claim of the statement follows. The two remaining claims, can be shown in a similar fashion if we use for the second one that
\begin{equation*}
 \abs*{\del_{s}\Biggl(\frac{\cut(s+t)\ee^{\im t \xi}}{(\im(s+t))^{\rho+\kappa}(1+\im(s+t))}\Biggr)}\leq \frac{C}{(s+t)^{1+\rho+\kappa}}.
\end{equation*}
On the other hand, to prove the third estimate of the statement, it is also straightforward to establish
\begin{equation*}
 \abs*{\Phi'(\im(s+t))-\Phi'(\im s)}\leq C\eps s^{-1-\alpha} t
\end{equation*}
while Lemma~\ref{Lem:analyticity:Phi} directly yields $\abs*{\Phi'(\im(s+t))-\Phi'(\im s)}\leq C\eps s^{-1-\alpha}$. Then, proceeding similarly as above, the statement readily follows. More details can be found in~\cite{Thr16}.
\end{proof}

\begin{lemma}\label{Lem:gain:of:alpha:coup}
 For the function $\coup$ as given in~\eqref{eq:cutoffs} there exist constants $B,C>0$ such that it holds
 \begin{align*}
  \abs*{\int_{0}^{\sgn(s)\infty}\frac{(1-\coup(s+t))\ee^{\im t \xi}}{(\im(s+t))^{\rho+\kappa}(1+\im(s+t))}\frac{\ee^{\Phi(\im (s+t))}}{\ee^{\Phi(\im s)}}\dt}&\leq C\biggl(\frac{\abs{s}^{1+\alpha-\rho-\kappa}}{\eps}+\ee^{-\frac{B\eps}{\abs*{s}^{\alpha}}}\biggr)\chi_{\{\abs*{s}\leq 1\}},\\
  \abs*{\int_{0}^{\sgn(s)\infty}\del_{s}\Biggl(\frac{(1-\coup(s+t))\ee^{\im t \xi}}{(\im(s+t))^{\rho+\kappa}(1+\im(s+t))}\Biggr)\frac{\ee^{\Phi(\im (s+t))}}{\ee^{\Phi(\im s)}}\dt}&\leq C\biggl(\frac{\abs{s}^{\alpha-\rho-\kappa}}{\eps}+\abs*{s}^{\rho-\kappa}\ee^{-\frac{B\eps}{\abs*{s}^{\alpha}}}\biggr)\chi_{\{\abs*{s}\leq 1\}}
 \end{align*}
 and
 \begin{multline*}
  \abs*{\int_{0}^{\sgn(s)\infty}\frac{(1-\coup(s+t))\ee^{\im t \xi}(\Phi'(\im(s+t))-\Phi'(\im s))}{(\im(s+t))^{\rho+\kappa}(1+\im(s+t))}\frac{\ee^{\Phi(\im (s+t))}}{\ee^{\Phi(\im s)}}\dt}\\*
  \leq C\biggl(\frac{\abs{s}^{\alpha-\rho-\kappa}}{\eps}+\eps\abs*{s}^{-1-\alpha}\ee^{-\frac{B\eps}{\abs*{s}^{\alpha}}}\biggr)\chi_{\{\abs*{s}\leq 1\}}.
 \end{multline*}
\end{lemma}

\begin{proof}
 Note that $1-\coup$ and $\cut$ have the same qualitative behaviour such that up to different constants the same proof as for Lemma~\ref{Lem:gain:of:alpha:cut} applies.
\end{proof}

\appendix

\section{Useful elementary results}\label{Sec:elementary:results}

We collect here some properties of the weight that we used to define the norms on Laplace transforms. These things are rather elementary but some of them will be used frequently throughout the proofs.

For $a,b,a_1,a_2,b_1,b_2,r\in\R$ it holds that
\begin{equation}\label{eq:weight:mult}
 \weight{a_1}{b_1}(q)\weight{a_2}{b_2}(q)=\weight{a_1+a_2}{b_1+b_2}(q),\quad \weight{a}{b}(q)q^{r}=\weight{a+r}{b-r}(q)\; \text{ and }\; \weight{-a}{-b}(q)=\bigl(\weight{a}{b}(q)\bigr)^{-1}.
\end{equation}
Moreover, we have
\begin{equation}\label{eq:weight:parmono}
 \weight{a_2}{b_2}(q)\leq \weight{a_1}{b_1}(q) \quad \text{if } a_1\leq a_2 \text{ and }b_1\leq b_2.
\end{equation}
If $a>-1$ one also finds the integral estimates
\begin{equation}\label{eq:weight:int:low}
 \int_{0}^{q}\weight{a}{b}(r)\dr\leq C\weight{a+1}{b-1}(q)\quad \text{if }b<1\quad \text{and}\quad \int_{0}^{q}\weight{a}{b}(r)\dr\leq C\weight{a+1}{0}(q)\quad \text{if }b>1.
\end{equation}
In the same way we have for $b>1$ that
\begin{equation}\label{eq:weight:int:high}
 \int_{q}^{\infty}\weight{a}{b}(r)\dr\leq C\weight{a+1}{b-1}(q)\quad \text{if }a<-1\quad \text{and}\quad \int_{q}^{\infty}\weight{a}{b}(r)\dr\leq C\weight{0}{b-1}(q)\quad \text{if }a>-1.
\end{equation}
Combining~\eqref{eq:weight:mult} and~\eqref{eq:weight:int:high} it further follows
\begin{equation}\label{eq:weight:int:reduced}
 \int_{q}^{\infty}\frac{\weight{a}{b}(r)}{r}\dr\leq C\weight{a}{b}(q)\quad \text{if } a<0\text{ and }b>0.
\end{equation}
Concerning the weights occurring in the definition of the (semi-)norms in~\eqref{eq:def:seminorm} we have
\begin{equation}\label{eq:est:weight}
 2^{-b-a}\weight{a}{b}(q)q^{-k}\leq \frac{1}{(1+q)^{b+1}q^{k-a}}\leq \weight{a}{b}(q)q^{-k}\quad \text{if }a+b\geq 0.
\end{equation}
From this one immediately deduces that for each $G\in\X{k}{\mu}{\chi}$ it holds
\begin{equation}\label{eq:est:by:norm}
 \abs*{G(q)}\leq \lfnorm*{0}{-\rho}{\chi}{G}\weight{0}{\chi}(q)\quad \text{and}\quad \abs*{\del_{q}^{k}G(q)}\leq \lfnorm*{k}{\mu}{\chi}{G}\weight{\mu+\rho}{\chi}(q)q^{-k}.
\end{equation}
By monotonicity we also have
\begin{equation}\label{eq:weight:monotonicity}
 \frac{1}{(1+q+\tau)^{b}(q+\tau)^{a}}\leq \frac{1}{(1+q)^{b}q^{a}}\quad \text{for }\tau>0\text{ if }a,b\geq 0.
\end{equation}
Moreover, we have for $G\in\X{2}{\mu}{\chi}$ that
\begin{equation}\label{eq:difference:simple}
 \abs*{G(q)-G(q+\tau)}\leq \frac{2\lfnorm*{0}{-\rho}{\chi}{G}}{(q+1)^{\chi}}\quad \text{and}\quad \abs*{\del_{q}^{k}G(q)-\del_{q}^{k}G(q+\tau)}\leq \frac{2\lsnorm*{k}{\mu}{\chi}{G}}{(q+1)^{\chi+\rho+\mu}q^{k-\rho-\mu}}.
\end{equation}
On the other hand, we can also estimate such differences by the norm of the derivative. Precisely, if we use $\del_{q}^{0}G\vcc=G$ we have for $k=0,1$ and $\mu\in[0,1-\rho]$ that
\begin{equation}\label{eq:est:norm:difference}
 \begin{split}
  \abs*{\del_{q}^{k}G(q)-\del_{q}^{k}G(q+\tau)}&=\abs[\bigg]{\int_{q}^{q+\tau}\del_{q}^{k+1}G(s)\ds}\leq \lsnorm*{k+1}{\mu}{\chi}{G}\int_{q}^{q+\tau}\frac{1}{(s+1)^{\chi+\rho+\mu}s^{k+1-\rho-\mu}}\ds\\
  &\leq \lsnorm*{k+1}{\mu}{\chi}{G}\frac{\tau}{(q+1)^{\chi+\rho+\mu}q^{k+1-\rho-\mu}}.
 \end{split}
\end{equation}

\begin{remark}\label{Rem:est:difference:Laplace}
 In terms of the Laplace transform the estimate~\eqref{eq:est:norm:difference} read as
 \begin{equation*}
  \abs[\big]{\del_{q}^{k}\T\bigl((1-\ee^{-\cdot \tau}) g\bigr)(q)}\leq\lsnorm*{k+1}{\mu}{\chi}{\T g}\frac{\tau}{(q+1)^{\chi+\rho+\mu}q^{k+1-\rho-\mu}},
 \end{equation*}
 for $g\in\M^{\text{fin}}(0,\infty)$ with $(\T g)\in \X{k}{\mu}{\chi}$.
\end{remark}

The next lemma states that the norm $\lfnorm{k}{\mu}{\chi}{\cdot}$ is monotonous both with respect to the second and third parameter.

 \begin{lemma}\label{Lem:comparison:norms}
  For $\mu_1,\mu_2\in [0,\mu_{*})$ such that $\mu_1\leq \mu_2$ and $\chi_1,\chi_2>0$ such that $\chi_1\leq \chi_2$ we have
  \begin{equation*}
   \lfnorm*{0}{-\rho}{\chi_1}{G}\leq \lfnorm*{0}{-\rho}{\chi_2}{G}\quad \text{and}\quad \lsnorm*{k}{\mu_1}{\chi_1}{G}\leq \lsnorm*{k}{\mu_2}{\chi_2}{G}
  \end{equation*}
 for all $G\in\X{k}{\mu_2}{\chi_2}$ with $k=1,2$.
 \end{lemma}

 \begin{proof}
  Under the assumptions of $\mu_{\ell}$ and $\chi_{\ell}$ for $\ell=1,2$ one can check that it holds
  \begin{equation*}
   (1+q)^{\rho+\mu_1}q^{-\rho-\mu_1}\leq (1+q)^{\rho+\mu_2}q^{-\rho-\mu_2}\quad\text{and}\quad (1+q)^{\chi_1}\leq (1+q)^{\chi_2} \quad \text{for }q>0.
  \end{equation*}
  From these estimates the claim immediately follows from the definition of the norms.
 \end{proof}
 
 The next lemma states a certain regularising effect for differences in Laplace variables, i.e.\@ we may gain additional regularity of order $q^{\mu}$ at zero.

 \begin{lemma}\label{Lem:difference:regularised}
  For every $\mu\in[0,\mu_{*})$ there exists a constant $C>0$ such that it holds
  \begin{equation*}
   \abs*{G(q)-G(0)}\leq C\lsnorm*{1}{\mu}{\chi}{G}\weight{\mu+\rho}{0}(q)
  \end{equation*}
 for all $\chi>0$ and all $G\in \X{1}{\mu}{\chi}$.
 \end{lemma}

 \begin{proof}
  From~\cref{eq:weight:mult,eq:weight:int:low,eq:est:by:norm} we infer
  \begin{equation*}
   \abs*{G(q)-G(0)}=\abs*{\int_{0}^{q}G'(s)\ds}\leq \lsnorm*{1}{\mu}{\chi}{G}\int_{0}^{q}\weight{\rho+\mu-1}{\chi+1}(s)\ds\leq C\lsnorm*{1}{\mu}{\chi}{G}\weight{\rho+\mu}{0}(q)
  \end{equation*}
 which shows the claim.
 \end{proof}
 
 The next results states that $\lsnorm{k}{\mu}{\chi}{\cdot}$ for $k=1,2$ defines already a norm on the spaces $\X{k}{\mu}{\chi}$.

\begin{lemma}\label{Lem:seminorm:improved}
 For each $\mu\in[0,\mu_{*})$ and $\chi>0$ there exists a constant $C=C_{\mu,\chi}>0$ such that it holds $\lfnorm*{k}{\mu}{\chi}{G}\leq C\lsnorm*{k}{\mu}{\chi}{G}$ for all $G\in\X{k}{\mu}{\chi}$ and all $k=1,2$.
\end{lemma}

\begin{proof}
 We consider first the case $k=1$ and thus assume $G\in\X{1}{\mu}{\chi}$. The estimates~\cref{eq:weight:int:high,eq:est:weight,eq:est:by:norm} then imply
 \begin{equation*}
  \abs*{G(q)}=\abs*{\int_{q}^{\infty}G'(s)\ds}\leq \lsnorm*{1}{\mu}{\chi}{G}\int_{q}^{\infty}\weight{\rho+\mu-1}{\chi+1}(s)\ds\leq C\weight{0}{\chi}(q)\leq C\lsnorm*{1}{\mu}{\chi}{G}(1+q)^{-\chi}.
 \end{equation*}
This implies $\lfnorm*{0}{-\rho}{\chi}{G}\leq C\lsnorm*{1}{\mu}{\chi}{G}$ and thus the claim for $k=1$. The case $k=2$ follows analogously.
\end{proof}

The following lemma shows that the derivatives of the Laplace transform of a non-negative measure can be estimated by the Laplace transform itself.

\begin{lemma}\label{Lem:non:neg:measures}
 For each $ g \in\M_{+}^{\text{fin}}(0,\infty)$ it holds that
 \begin{equation*}
  \abs*{\del_{q}(\T  g )(q)}\leq \frac{\abs*{(\T g )(q/2)}}{q}\quad \text{and}\quad \abs*{\del_{q}^{2}(\T g )(q)}\leq \frac{\abs*{\del_{q}(\T  g )(q/2)}}{q}\quad \text{for all }q>0.
 \end{equation*}
This in particular shows that $\del_{q}^{2}(\T g )(q)$ can be estimated by $(\T g )(q/4)$.
\end{lemma}

\begin{proof}
 The non-negativity of $ g $ implies that
 \begin{equation*}
  \abs*{\del_{q}(\T g )(q)}=\int_{0}^{\infty}x g (x)\ee^{-qx}\dx=\frac{1}{q}\int_{0}^{\infty} g (x)\ee^{-\frac{q}{2}x}qx\ee^{-\frac{q}{2}x}\dx.
 \end{equation*}
 Since $x\ee^{-x/2}\leq 1$ for all $x\geq 0$ it follows that 
 \begin{equation*}
  \abs*{\del_{q}(\T g )(q)}\leq \frac{1}{q}\int_{0}^{\infty} g (x)\ee^{-\frac{q}{2}x}\dx=\frac{\abs*{(\T g )(q/2)}}{q}.
 \end{equation*}
This shows the first part of the claim, while the second one follows in the same way.
\end{proof}

\begin{lemma}\label{Lem:norm:non:neg:meas}
 Let $\mu\in[0,\mu_{*})$, $\chi>0$ and $ g \in\M_{+}^{\text{fin}}(0,\infty)$ such that $\T g \in\X{1}{\mu}{\chi}$. Then there exists $C>0$ such that $\lsnorm*{2}{\mu}{\chi}{\T g }\leq C \lsnorm*{1}{\mu}{\chi}{\T g }$ which in particular implies $\T g \in \X{2}{\mu}{\chi}$.
\end{lemma}

\begin{proof}
 From Lemma~\ref{Lem:non:neg:measures} it follows that
 \begin{equation*}
  \abs*{(\T g )''(q)}\leq \frac{\abs*{(\T g )'(q/2)}}{q}\leq \lsnorm*{1}{\mu}{\chi}{\T g }\frac{1}{q}\frac{1}{\left(\frac{q}{2}\right)^{1-\rho-\mu}\left(1+\frac{q}{2}\right)^{\chi+\rho+\mu}}\leq C\frac{\lsnorm*{1}{\mu}{\chi}{\T g }}{q^{2-\rho-\mu}(1+q)^{\chi+\rho+\mu}}.
 \end{equation*}
 The claim then follows from the definition of the semi-norm.
\end{proof}

\begin{remark}\label{Rem:est:der:by:zero:norm}
 Under the same conditions as in Lemma~\ref{Lem:norm:non:neg:meas} one can show that 
 \begin{equation*}
  \abs*{(\T g )'(q)}\leq 2^{\theta}\lfnorm*{0}{-\rho}{\theta}{\T g }\frac{1}{q(1+q)^{\theta}}\quad \text{for all }q>0.
 \end{equation*}
However, it is not possible to estimate $\lsnorm*{1}{\mu}{\chi}{\T g }$ by $\lfnorm*{0}{-\rho}{\theta}{\T g }$ since the latter does not contain enough information about the regularity for $q$ close to zero.
\end{remark}

The next lemma states that the norms $\lfnorm{k}{\mu}{\chi}{\cdot}$ behave well under shifts in Laplace variables.

\begin{lemma}\label{Lem:norm:shift}
 For all parameters $\mu\in[0,\mu_{*})$, $\chi>0$ and $\tau>0$ it holds for all $g\in\M^{\text{fin}}(0,\infty)$ with $(\T g)\in\X{0}{\mu}{\chi}$ that 
 \begin{equation*}
  \lfnorm*{0}{-\rho}{\chi}{\T\bigl(\ee^{-\tau\cdot}g\bigr)}\leq \lfnorm*{0}{-\rho}{\chi}{\T g} \quad \text{and} \quad \lfnorm*{0}{-\rho}{\chi}{\T\bigl((1-\ee^{-\tau\cdot})g\bigr)}\leq 2\lfnorm*{0}{-\rho}{\chi}{\T g}.
  \end{equation*}
  Moreover, if $(\T g)\in\X{k}{\mu}{\chi}$ we also have
  \begin{equation*}
   \lsnorm*{k}{\mu}{\chi}{\T\bigl(\ee^{-\tau\cdot}g\bigr)}\leq \lfnorm*{k}{\mu}{\chi}{\T g} \quad \text{and} \quad \lfnorm*{k}{\mu}{\chi}{\T\bigl((1-\ee^{-\tau\cdot})g\bigr)}\leq 2\lfnorm*{k}{\mu}{\chi}{\T g}
  \end{equation*}
 for $k=1,2$.
\end{lemma}

\begin{proof}
 The estimates on $\T(\ee^{-\tau\cdot}g(\cdot))$ follow immediately from the definition of the norm and the monotonicity properties of $(1+q)^{\chi}$ and $q^{k-\rho-\mu}(1+q)^{\chi+\rho+\mu}$ as functions of $q$. The estimates on $\T((1-\ee^{-\tau\cdot})g(\cdot))$ are then an immediate consequence.
\end{proof}

The next lemma states that the expression $\T\bigl((1-\ee^{-n\cdot})g\bigr)$ can also be estimates by $\T\bigl((1-\zeta)g\bigr)$ up to constant which grows linearly with $n$.

\begin{lemma}\label{Lem:norm:shift:by:shift}
 For every $\mu\in[0,\mu_{*})$ and $\chi>0$ there exists a constant $C>0$ such that it holds
 \begin{equation*}
  \lfnorm*{k}{\mu}{\chi}{\T\bigl((1-\ee^{-n\cdot})g\bigr)}\leq Cn\lfnorm*{k}{\mu}{\chi}{\T\bigl((1-\zeta)g\bigr)}
 \end{equation*}
 for all $n\in\N$ and all $g\in\M^{\text{fin}}(0,\infty)$ such that $(\T g)\in\X{k}{\mu}{\chi}$ for $k=1,2$.
\end{lemma}

\begin{proof}
 Due to Lemma~\ref{Lem:seminorm:improved} it suffices to show that $\lsnorm*{k}{\mu}{\chi}{\T\bigl((1-\ee^{-n\cdot})g\bigr)}\leq n\lsnorm*{k}{\mu}{\chi}{\T\bigl((1-\zeta)g\bigr)}$ for $k=1,2$. To see this, we first observer that $(1-\ee^{-zn})=(1-\zeta)\sum_{\ell=0}^{n-1}\ee^{-\ell z}$ which yields
 \begin{equation*}
  \del_{q}^{k}\biggl(\int_{0}^{\infty}(1-\ee^{-nz})g(z)\ee^{-qz}\dz\biggr)=\sum_{\ell=0}^{n-1}\del_{q}^{k}\biggl(\int_{0}^{\infty}(1-\zeta(z))g(z)\ee^{-(\ell+q)z}\dz\biggr).
 \end{equation*}
Together with~\eqref{eq:weight:monotonicity} this relation implies
\begin{equation*}
 \abs*{\del_{q}^{k}\biggl(\int_{0}^{\infty}(1-\ee^{-nz})g(z)\ee^{-qz}\dz\biggr)}\leq \sum_{\ell=0}^{n-1}\frac{\lsnorm*{k}{\mu}{\chi}{\T((1-\zeta)g)}}{(\ell+q)^{k-\rho-\mu}(\ell+q+1)^{\chi+\rho+\mu}}\leq \frac{n\lsnorm*{k}{\mu}{\chi}{\T((1-\zeta)g)}}{q^{k-\rho-\mu}(q+1)^{\chi+\rho+\mu}}
\end{equation*}
from which the claim readily follows.
\end{proof}

\section{The representation formula for $W$}\label{Sec:Proof:Gamma}

In this section, we will on the one hand sketch the proof of Proposition~\ref{Prop:W:representation} and moreover, we will provide several integral estimates on the representation kernel $\Ker$.

\subsection{Proof of Proposition~\ref{Prop:W:representation}}

Note that the statement of Proposition~\ref{Prop:W:representation} only depends on the assumptions on the kernel $W$ and it is in particular completely independent of the kind of self-similar profiles that one considers. In particular the same result has already been used and proved in~\cite{NTV15} while a revised proof is also contained in~\cite{Thr16}. For the sake of completeness, we again sketch the argument by indicating the main steps of the proof while more details may be found either in~\cite{NTV15} or~\cite{Thr16}. The following presentation is based on the version which is contained in~\cite{Thr16}. 

\begin{proof}[Proof of Proposition~\ref{Prop:W:representation}]

By abuse of notation we write $\W(z)\vcc=W(z,1)$ and we recall from Section~\ref{Sec:Assumptions:W} that $\W_{\pm}(z)$ is the restriction of $\W$ to $\HH_{\pm}$. Due to~\eqref{eq:Ass:analytic}, the function $\W$ is analytic in $\Ccut$. Moreover, we note that~\eqref{eq:Ass:analytic} yields 
\begin{equation}\label{eq:growth:G}
 \abs*{\W_{\pm}(z)}\leq C \weight{-\alpha}{-\alpha}\bigl(\abs*{z}\bigr)\qquad \text{and}\qquad \abs*{\W'_{\pm}(z)}\leq C \weight{-\alpha-1}{1-\alpha}\bigl(\abs*{z}\bigr) \qquad \text{for all }z\in\overline{\HH}_{\pm}\setminus\{0\}.
\end{equation}
Before we outline how the statement of Proposition~\ref{Prop:W:representation} can be shown, we give a motivation for the later considerations.

\paragraph{Motivation:}

Assume that we already have the existence of a homogeneous measure $\Ker$ as stated in Proposition~\ref{Prop:W:representation} which satisfies $(x+y)^{-1}W(x,y)=\int_{0}^{\infty}\int_{0}^{\infty}\Ker(\xi,\eta)\ee^{-\xi x-\eta y}\dxi\deta$. Due to homogeneity, one immediately checks that the left-hand side can be rewritten as
\begin{equation}\label{eq:representation:W:rewritten}
 \frac{W(x,y)}{x+y}=\frac{1}{y}\frac{\W(x/y)}{x/y+1}.
\end{equation}
On the other hand, performing the change of variables $\xi\mapsto \xi\eta$ and evaluating the integral in $\xi$ explicitly, one gets for the right-hand side that
\begin{multline*}
 \int_{0}^{\infty}\int_{0}^{\infty}\Ker(\xi,\eta)\ee^{-\xi x-\eta y}\dxi\deta=\int_{0}^{\infty}\int_{0}^{\infty}\Ker(1,\eta)\ee^{-\xi(x+\eta y)}\dxi\deta\\*
 =\int_{0}^{\infty}\frac{\Ker(1,\eta)}{x+\eta y}\deta=\frac{1}{y}\int_{0}^{\infty}\frac{\Ker(1,\eta)}{x/y+\eta}\deta.
\end{multline*}
These relations suggest to look for a solution $\phi$ of the integral equation
\begin{equation*}\label{eq:construct:Ker:formal}
 \frac{\W(z)}{z+1}=\int_{0}^{\infty}\frac{\phi(\eta)}{\eta+z}\deta
\end{equation*}
in order to construct $\Ker$. Since our argument relies crucially on methods from complex analysis, it turns out that we also have to remove the pole at $z=-1$ on the right-hand side such that we should look instead for a function $\phi\colon(0,\infty)\to\R$ which satisfies
\begin{equation}\label{eq:prob:G}
 \frac{\W(z)-\W(-1)}{1+z}=\int_{0}^{\infty}\frac{\phi(\eta)}{\eta+z}\deta\quad \text{for all }z\in\C.
\end{equation}
The constant $\W(-1)$ is well-defined due to Remark~\ref{Rem:existence:Wminusone}. Note that it is exactly this term where the singular part of the measure $\Ker$ stems from. More precisely, $\widehat{\phi}(\cdot)=\delta(\cdot-1)$ is the explicit solution to the equation
\begin{equation*}
 \frac{1}{1+z}=\int_{0}^{\infty}\frac{\widehat{\phi}(\eta)}{\eta+z}\deta.
\end{equation*}

In order to detect an appropriate candidate for $\phi$ we note that the Sokhotski-Plemelj formula of complex analysis suggests for sufficiently regular $\phi$ that
\begin{equation*}
 \begin{split}
  \lim_{\nu\to 0^{+}}\frac{\W(z_0+\nu\im)-\W(-1)}{z_0+\nu\im +1}&=-\pi\im\phi\bigl(\abs*{z_0}\bigr)+\mathrm{p.v.}\int_{0}^{\infty}\frac{\phi(\eta)}{\eta-\abs*{z_0}}\deta\\
  \lim_{\nu\to 0^{+}}\frac{\W(z_0-\nu\im)-\W(-1)}{z_0-\nu\im +1}&=\pi\im\phi\bigl(\abs*{z_0}\bigr)+\mathrm{p.v.}\int_{0}^{\infty}\frac{\phi(\eta)}{\eta-\abs*{z_0}}\deta
 \end{split}
\end{equation*}
for $z_{0}\in(-\infty,0)$. Note that $\frac{1}{\eta-\abs*{z_0}+\nu\im}\to -\pi\im \delta(\cdot-\abs*{z_0})+\mathrm{p.v.}\bigl(\frac{1}{\eta-\abs*{z_0}}\bigr)$ for $\nu\to 0$ and $\mathrm{p.v.}$ denotes 'principle value'.

This suggest that $\phi$ is given by
\begin{equation}\label{eq:representation:phi:formal}
 \phi(s)=\frac{1}{2\pi\im(1-s)}\Bigl[\,\lim_{\nu\to 0^{+}}\bigl(\W(-s-\nu\im)-\W(-1)\bigr)-\lim_{\nu\to 0^{+}}\bigl(\W(-s+\nu\im)-\W(-1)\bigr)\,\Bigr].
\end{equation}

The proof of Proposition~\ref{Prop:W:representation} can be divided in essentially four main steps. In the first one, $\phi$ is precisely defined and suitable bounds for small and large values are provided. In the second step one has to show that $\phi$ really has the required regularity to make the formal considerations above precise, i.e.\@ $\phi$ is locally Hölder continuous (in the same sense as in~\cref{eq:Ass:hoelder:1,eq:Ass:hoelder:2}). The third step consists of the verification that $\phi$ really satisfies~\eqref{eq:prob:G} while in the fourth step the proof is concluded by showing that $\Ker$ yields the desired representation formula for $W$ and has the claimed properties.

\paragraph{Step 1:}

Based on~\eqref{eq:representation:phi:formal} we define $\phi\colon (0,\infty)\to\R$ as
\begin{equation}\label{eq:phi:rep:1}
 \phi(s)\vcc=\frac{\W_{-}(-s)-\W_{-}(-1)-(\W_{+}(-s)-\W_{+}(-1))}{2\pi\im (1-s)}=\frac{\W_{-}(-s)-\W_{+}(-s)}{2\pi\im (1-s)}.
\end{equation}
Here we also used that $\W(-1)=\W_{+}(-1)=\W_{-}(-1)$ (see Remark~\ref{Rem:existence:Wminusone}). Note that $\phi$ is well-defined due to the regularity assumptions~\cref{eq:Ass:hoelder:1,eq:Ass:hoelder:2}. Moreover, we immediately obtain that $\phi(1)=\frac{1}{2\pi\im}\bigl(\W'_{-}(-1)-\W'_{+}(-1)\bigr)$.

We claim that $\phi$ satisfies the bound
\begin{equation}\label{eq:phi:decay}
 \abs*{\phi(s)}\leq C\weight{-\alpha}{1-\alpha}(s).
\end{equation}
This estimate can be verified by a straightforward computation from~\eqref{eq:growth:G}. More precisely, one first shows that
 \begin{equation}\label{eq:Gpm:decay}
  \abs*{\frac{\W_{\pm}(z)-\W_{\pm}(-1)}{z+1}}\leq C\weight{-\alpha}{1-\alpha}(\abs*{z})\qquad\text{for all }z\in \overline{\HH}_{\pm}\setminus\{0\}
 \end{equation}
 by considering the regions $\abs*{z}\geq 2$, $\abs*{z+1}\leq 1/2$ and $\{z\in \overline{\HH}_{\pm}\setminus\{0\}\;|\; \abs*{z+1}\geq 1/2 \text{ and  } \abs*{z}\leq 2\}$ separately. The estimate~\eqref{eq:phi:decay} then follows together with~\eqref{eq:phi:rep:1}.

\paragraph{Step 2:}

As the second step one has to verify that $\phi$ satisfies a local Hölder condition. The precise meaning of this is that it holds
\begin{equation}\label{eq:phi:hoel}
 \frac{\abs*{\phi(s)-\phi(t)}}{\abs*{s-t}^{\hol}}\leq C\begin{cases}
                                                        \min\{s,t\}^{-\alpha-\hol} & \text{if } \min\{s,t\}\leq 1\\
                                                        \min\{s,t\}^{\alpha-1-\hol} & \text{if } \min\{s,t\}\geq 1
                                                       \end{cases}
\qquad \text{for }\abs*{s-t}\leq \frac{\min\{s,t\}}{2} \text{ and }s,t>0
\end{equation}
with $\hol\in(0,1)$ as given by~\eqref{eq:Ass:hoelder:1}. This regularity then ensures that $\int_{0}^{\infty}\frac{\phi(\eta)}{z+\eta}\deta$ can also be extended to $z\in(-\infty,0)$ either in $\overline{\HH}_{+}\setminus \{0\}$ or $\overline{\HH}_{-}\setminus \{0\}$ due to the Sokhotski-Plemelj formula.

Again, the verification of~\eqref{eq:phi:hoel} is a lengthy but straightforward computation which essentially relies on the regularity properties~\cref{eq:Ass:hoelder:1,eq:Ass:hoelder:2}. Therefore, we only mention that one has to consider several different cases depending on the values of $s$ and $t$ while moreover the relation 
\begin{multline}\label{eq:phi:diff:rep}
  \frac{\W_{\pm}(-s)-\W_{\pm}(-1)}{1-s}-\frac{\W_{\pm}(-t)-\W_{\pm}(-1)}{1-t}=\int_{t}^{s}\del_{\tau}\biggl(\frac{\W_{\pm}(-\tau)-\W_{\pm}(-1)}{1-\tau}\biggr)\dtau\\
  =\int_{t}^{s}-\frac{\W_{\pm}'(-\tau)}{1-\tau}+\frac{\W_{\pm}(-\tau)-\W_{\pm}(-1)}{(1-\tau)^2}\dtau=\int_{t}^{s}\frac{\int_{-1}^{-\tau}\W_{\pm}'(\xi)-\W_{\pm}'(-\tau)\dxi}{(1-\tau)^2}\dtau
\end{multline}
turns out to be useful. More details may be found in~\cite{NTV15,Thr16}.

\paragraph{Step 3:}

This step is devoted to the verification of the relation~\eqref{eq:prob:G} which will rely essentially on standard arguments of complex analysis. More precisely, we consider the function
\begin{equation*}
 \Phi(z)\vcc=\frac{\W(z)-\W(-1)}{1+z}-\int_{0}^{\infty}\frac{\phi(\eta)}{z+\eta}\deta.
\end{equation*}
The general strategy is then to show on the one hand that $\Phi$ is an entire function, i.e.\@ analytic on $\C$ and on the other hand that $\abs*{\Phi(z)}\leq C\weight{0}{1-\alpha}(\abs*{z})$. Due to Liouville's Theorem we may then conclude that $\Phi$ is constant while the latter estimate for $\abs*{z}\to\infty$ directly yields $\Phi\equiv 0$. This then proves the desired relation~\eqref{eq:prob:G}.

To prove the analyticity of $\Phi$ and the corresponding estimate, we first note that immediately verifies that $\Phi$ is analytic in $\Ccut$ and that $\frac{\W(z)-\W(-1)}{1+z}$ is analytic in $\C\setminus\{0\}$. 

Moreover, we denote by $\Phi_{\pm}$ the restriction of $\Phi$ to $\HH_{\pm}$ and note that the considerations of Step~2 allow to extend $\Phi_{\pm}$ to $\overline{\HH}_{\pm}\setminus\{0\}$ since for $\int_{0}^{\infty}\frac{\phi(\eta)}{z\pm\im\nu+\eta}\deta$, the limit $\nu\to 0^{+}$ exists for all $z\in(-\infty,0)$. Moreover, $\Phi_{+}(z)=\Phi_{-}(z)$ for all $z\in(-\infty,0)$ due to the construction of $\phi$. Morera's Theorem thus yields the analyticity of $\Phi$ in $\C\setminus\{0\}$. 

To conclude it thus suffices to show that $\abs*{\Phi(z)}\leq C\weight{-\alpha}{1-\alpha}(\abs*{z})$. To see that this is really sufficient, we recall that $\Phi$ is analytic in $\C\setminus\{0\}$ and thus Riemann's Theorem together with $\abs*{\Phi(z)}\leq C\weight{-\alpha}{1-\alpha}(\abs*{z})$ yields that $\Phi$ has an analytic extension over zero which thus in addition also gives that we in fact have $\abs*{\Phi(z)}\leq C\weight{0}{1-\alpha}(\abs*{z})$.

Thus, it remains to verify that $\abs*{\Phi(z)}\leq C\weight{-\alpha}{1-\alpha}(\abs*{z})$, while the definition of $\Phi$ together with~\eqref{eq:Gpm:decay} implies that it suffices to prove that
\begin{equation}\label{eq:int:phi:decay}
 \abs*{\int_{0}^{\infty}\frac{\phi(\eta)}{z+\eta}\deta}\leq C\weight{-\alpha}{1-\alpha}(\abs*{z}).
\end{equation}
The verification of this estimate is rather straightforward and relies essentially on~\eqref{eq:phi:decay}. However, since there are again several case distinctions necessary which make the computations lengthy, we omit further details at this point and instead refer to~\cite{NTV15,Thr16}. 
 
 \paragraph{Step 4:}
 
 We can now conclude the proof and for this, we define the representation kernel $\Ker$ as
  \begin{equation*}
  \widetilde{\Ker}(\xi,\eta)\vcc=\frac{\phi(\eta/\xi)}{\xi}\qquad \text{and}\qquad \Ker(\xi,\eta)\vcc=\widetilde{\Ker}(\xi,\eta)+\W(-1)\delta(\xi-\eta).
 \end{equation*}
 By construction of $\phi$ together with~\eqref{eq:Ass:analytic} it follows that $\Ker$ is symmetric. Moreover, $\Ker$ is homogeneous of degree $-1$ by definition. Furthermore, the estimate
 \begin{equation*}
  \abs*{\widetilde{\Ker}(\xi,\eta)}\leq C\frac{1}{(\eta+\xi)^{1-\alpha}}\biggl(\frac{1}{\xi^{\alpha}}+\frac{1}{\eta^{\alpha}}\biggr)
 \end{equation*}
 directly follows from~\eqref{eq:phi:decay}. To conclude the proof we compute the integral
 \begin{equation*}
  \int_{0}^{\infty}\int_{0}^{\infty}\Ker(\xi,\eta)\ee^{-\xi x-\eta y}\deta\dxi=\int_{0}^{\infty}\int_{0}^{\infty}\frac{\phi(\eta/\xi)}{\xi}\ee^{-x\xi-y\eta}\deta\dxi+\W(-1)\int_{0}^{\infty}\ee^{-\xi(x+y)}\dxi.
 \end{equation*}
 The second integral on the right-hand side can be directly evaluated, while in the first one, we first change variables $\eta\mapsto\xi\eta$ and evaluate the integral in $\xi$ which yields
 \begin{equation*}
   \int_{0}^{\infty}\int_{0}^{\infty}\Ker(\xi,\eta)\ee^{-\xi x-\eta y}\deta\dxi=\int_{0}^{\infty}\frac{\phi(\eta)}{x+\eta y}\deta+\frac{\W(-1)}{x+y}=\frac{1}{y}\biggl(\int_{0}^{\infty}\frac{\phi(\eta)}{\frac{x}{y}+\eta}\deta+\frac{\W(-1)}{\frac{x}{y}+1}\biggr).
 \end{equation*}
 Since $\phi$ satisfies~\eqref{eq:prob:G}, we then obtain for $z=x/y$ that
\begin{equation*}
 \begin{split}
  \int_{0}^{\infty}\int_{0}^{\infty}\Ker(\xi,\eta)\ee^{-\xi y-\eta z}\deta\dxi&=\frac{1}{y}\frac{\W(x/y)}{x/y+1}=\frac{\W(x/y)}{x+y}.
 \end{split}
\end{equation*}
The claimed representation formula for $W$ then follows from the relation~\ref{eq:representation:W:rewritten}.
 \end{proof}
 
 \subsection{Integral estimates on $\Ker$}
 
 In this section, we collect several estimates on integrals involving the representation kernel $\Ker$ for $W$ as given by Proposition~\ref{Prop:W:representation}.

\begin{lemma}\label{Lem:Ker:est:most:general}
 Assume $\alpha\in(0,1/2)$ and let $a_1, a_2\in[0,1-\alpha)$ and $b_1, b_2\in(0,\infty)$ be given such that
 \begin{equation*}
  a_1+b_1+a_2+b_2>1,\quad a_1+a_2<1\quad \text{and}\quad a_k+b_k>\alpha\quad \text{for }k=1,2.
 \end{equation*}
 Then, there exists a constant $C>0$ such that it holds
 \begin{equation*}
  \int_{0}^{\infty}\int_{0}^{\infty}\frac{\abs*{\Ker(\xi,\eta)}}{\xi^{a_1}(\xi+1)^{b_1}\eta^{a_2}(\eta+1)^{b_2}}\deta\dxi\leq C
 \end{equation*}
 with $\Ker$ as given by Proposition~\ref{Prop:W:representation}.
\end{lemma}

\begin{proof}
 The estimate provided by Proposition~\ref{Prop:W:representation} already yields
 \begin{multline}\label{eq:Ker:est:most:general}
  \int_{0}^{\infty}\int_{0}^{\infty}\frac{\abs*{\Ker(\xi,\eta)}}{\xi^{a_1}(\xi+1)^{b_1}\eta^{a_2}(\eta+1)^{b_2}}\deta\dxi\\*
  \leq C\int_{0}^{\infty}\frac{1}{\xi^{\alpha+a_1}(\xi+1)^{b_1}}\int_{0}^{\infty}\frac{1}{(\xi+\eta)^{1-\alpha}}\frac{1}{\eta^{a_2}(\eta+1)^{b_2}}\deta\dxi\\*
  +C\int_{0}^{\infty}\frac{1}{\eta^{\alpha+a_2}(\eta+1)^{b_2}}\int_{0}^{\infty}\frac{1}{(\xi+\eta)^{1-\alpha}}\frac{1}{\xi^{a_1}(\xi+1)^{b_1}}\dxi\deta\\*
  +C\int_{0}^{\infty}\frac{1}{\xi^{a_1+a_2}(\xi+1)^{b_1+b_2}}\dxi.
 \end{multline}
 The third integral on the right-hand side can be estimated directly by a constant due to the assumptions on $a_k$ and $b_k$. Moreover, since these conditions are symmetric under the change of indices, we see that the first and second integral on the right-hand side can be estimated in the same way. Thus, it suffices to estimate the first integral and for this, we note that the restrictions on $a_k$ and $b_k$ imply that
 \begin{equation*}
  \max\{1-\alpha-a_1-b_1,a_2-\alpha\}<\min\{a_2+b_2-\alpha,1-\alpha-a_1\}.
 \end{equation*}
 Thus, we may define
 \begin{equation*}
  \lambda\vcc=\frac{1}{2}\Bigl(\max\{1-\alpha-a_1-b_1,a_2-\alpha\}+\min\{a_2+b_2-\alpha,1-\alpha-a_1\}\Bigr)
 \end{equation*}
 and one can verify by a case distinction that it holds $\lambda\in(0,1-\alpha)$. We then use this parameter to split the exponent $1-\alpha=\lambda+(1-\alpha-\lambda)$ which yields the corresponding estimate $(\xi+\eta)^{\alpha-1}\leq \xi^{-\lambda}\eta^{\alpha+\lambda-1}$. With this, we estimate the first integral on the right-hand side of~\eqref{eq:Ker:est:most:general} as
 \begin{multline*}
   \int_{0}^{\infty}\frac{1}{\xi^{\alpha+a_1}(\xi+1)^{b_1}}\int_{0}^{\infty}\frac{1}{(\xi+\eta)^{1-\alpha}}\frac{1}{\eta^{a_2}(\eta+1)^{b_2}}\deta\dxi\\*
   \leq \int_{0}^{\infty}\frac{1}{\xi^{\alpha+a_1+\lambda_1}(\xi+1)^{b_1}}\dxi\int_{0}^{\infty}\frac{1}{\eta^{a_2+(1-\alpha-\lambda)}(\eta+1)^{b_2}}\deta\leq C.
  \end{multline*}
In the last step we used that the choice of $\lambda$ implies that both integrals on the right-hand side are bounded by a constant which only depends on the values $\alpha$ and $a_k,b_k$ for $k=1,2$.
\end{proof}

\begin{lemma}\label{Lem:kernel:est:partial:0}
 For each $\alpha\in(0,1)$ there exists a constant $C>0$ such that it holds
 \begin{equation*}
  \int_{0}^{\infty}\abs*{\Ker(\xi,\eta)}\frac{1}{(\eta+1)^{\theta}}\deta\leq C(\xi^{-\alpha}+\xi^{\alpha-\theta})
 \end{equation*}
 with $\theta$ as given in~\eqref{eq:choice:theta:standard}.
\end{lemma}

\begin{proof}
 With Proposition~\ref{Prop:W:representation} and the elementary estimates $(\xi+\eta)^{\alpha-1}\leq \eta^{\alpha-1}$ and $(\xi+\eta)^{\alpha-1}\leq \xi^{\alpha-\theta}\eta^{\theta-1}$ we find that
 \begin{multline*}
  \int_{0}^{\infty}\abs*{\Ker(\xi,\eta)}\frac{1}{(\eta+1)^{\theta}}\deta\leq C\xi^{-\alpha}\int_{0}^{\infty}\frac{1}{\eta^{1-\alpha}(\eta+1)^{\theta}}\deta\\*
  +C\xi^{\alpha-\theta}\int_{0}^{\infty}\frac{1}{\eta^{1+\alpha-\theta}(\eta+1)^{\theta}}\deta+C(\xi+1)^{-\theta}.
 \end{multline*}
Since the remaining integrals on the right-hand side can be estimated by a constant and $(\xi+1)^{-\theta}\leq \xi^{-\alpha}$ for all $\xi>0$ the claim immediately follows.
\end{proof}

\begin{lemma}\label{Lem:kernel:est:partial:1}
 For each $\alpha\in(0,1)$ there exists a constant $C>0$ such that it holds
 \begin{equation*}
  \int_{0}^{\infty}\abs*{\Ker(\xi,\eta)}\frac{1}{\eta^{1-\rho}(1+\eta)^{\theta+\rho}}\deta\leq C\Bigl(\bigl(\xi^{\rho-1}+\xi^{-\alpha}\bigr)\abs*{\log(\xi)}\chi_{\{\chi<1\}}+\xi^{\alpha-1}\chi_{\{\xi\geq 1\}}\Bigr)
 \end{equation*}
with $\theta$ as given in~\eqref{eq:choice:theta:standard}
\end{lemma}

\begin{proof}
 From Proposition~\ref{Prop:W:representation} we first deduce
 \begin{equation}\label{eq:Ker:single:int:1}
  \int_{0}^{\infty}\frac{\abs*{\Ker(\xi,\eta)}}{\eta^{1-\rho}(1+\eta)^{\theta+\rho}}\deta\leq C\int_{0}^{\infty}\frac{1}{(\xi+\eta)^{1-\alpha}}\biggl(\frac{1}{\xi^{\alpha}}+\frac{1}{\eta^{\alpha}}\biggr)\frac{1}{\eta^{1-\rho}(1+\eta)^{\theta+\rho}}\deta+\frac{C}{\xi^{1-\rho}(1+\xi)^{\theta+\rho}}.
 \end{equation}
 We first treat the case $\xi\geq 1$ which is easier and we exploit $(\xi+\eta)^{\alpha-1}\leq \xi^{\alpha-1}$ and $(\xi+1)^{-\theta-\rho}\leq \xi^{-\theta-\rho}$ to get
 \begin{multline*}
  \int_{0}^{\infty}\abs*{\Ker(\xi,\eta)}\frac{1}{\eta^{1-\rho}(1+\eta)^{\theta+\rho}}\deta\\*
  \leq \frac{C}{\xi}\int_{0}^{\infty}\frac{1}{\eta^{1-\rho}(1+\eta)^{\theta+\rho}}\deta+\frac{C}{\xi^{1-\alpha}}\int_{0}^{\infty}\frac{1}{\eta^{1+\alpha-\rho}(1+\eta)^{\theta+\rho}}\deta+\frac{C}{\xi^{1+\theta}}.
 \end{multline*}
 The integrals on the right-hand side can be estimated by a constant which only depends on the fixed values of $\alpha,\rho$ and $\theta$. Thus, taking also into account that we have $\xi^{-1},\xi^{-\theta-1}\leq \xi^{\alpha-1}$ for $\xi\geq 1$ it follows
 \begin{equation*}
  \int_{0}^{\infty}\abs*{\Ker(\xi,\eta)}\frac{1}{\eta^{1-\rho}(1+\eta)^{\theta+\rho}}\deta\leq C\xi^{\alpha-1}\qquad \text{for }\xi\geq 1.
 \end{equation*}
  This shows the claim for $\xi\geq 1$.
  
  To treat the case $\xi\leq 1$, we have to split the integral on the right-hand side of~\eqref{eq:Ker:single:int:1} which gives together with $(\xi+1)^{-\theta-\rho}\leq 1$ and $\eta^{\rho-1}(\eta+1)^{-\theta-\rho}\leq\weight{\rho-1}{\theta+1}(\eta)$ that 
 \begin{multline*}
  \int_{0}^{\infty}\abs*{\Ker(\xi,\eta)}\frac{1}{\eta^{1-\rho}(1+\eta)^{\theta+\rho}}\deta\\*
   \leq \frac{C}{\xi^{\alpha}}\biggl(\int_{0}^{\xi}\frac{1}{(\xi+\eta)^{1-\alpha}\eta^{1-\rho}}\deta+\int_{\xi}^{1}\eta^{\alpha+\rho-2}\deta+\int_{1}^{\infty}\eta^{\alpha-\theta-2}\deta\biggr)\\*
  +C\biggl(\int_{0}^{1}\frac{1}{(\xi+\eta)^{1-\alpha}}\eta^{1+\alpha-\rho}\deta+\int_{1}^{\infty}\eta^{-2-\theta}\deta\biggr)+C\xi^{\rho-1}.
 \end{multline*}
 The third and fifth integral on the right-hand side can be estimated by a constant, while in the first and fourth one, we make a change of variables $\eta\mapsto \xi\eta$ which leads to
 \begin{multline*}
  \int_{0}^{\infty}\abs*{\Ker(\xi,\eta)}\frac{1}{\eta^{1-\rho}(1+\eta)^{\theta+\rho}}\deta\leq C\xi^{\rho-1}\int_{0}^{1}\frac{\eta^{\rho-1}}{(\eta+1)^{1-\alpha}}\deta+C\xi^{-\alpha}\int_{\xi}^{1}\eta^{\alpha+\rho-2}\deta+C\xi^{-\alpha}\\*
  +C\xi^{\rho-1}\int_{0}^{\infty}\frac{1}{(\eta+1)^{1-\alpha}\eta^{1+\alpha-\rho}}\deta+C+C\xi^{\rho-1}.
 \end{multline*}
 We again estimate the first and third integral on the right-hand side by a constant and we use $\xi\leq 1$ to further simplify the estimate as
 \begin{equation}\label{eq:Ker:single:int:2}
  \int_{0}^{\infty}\abs*{\Ker(\xi,\eta)}\frac{1}{\eta^{1-\rho}(1+\eta)^{\theta+\rho}}\deta\leq C(\xi^{\rho-1}+\xi^{-\alpha})+C\xi^{-\alpha}\int_{\xi}^{1}\eta^{\alpha+\rho-2}\deta.
 \end{equation}
 Depending on whether $\rho+\alpha=1$ or not the remaining integral on the right-hand side has a different scaling and we find
 \begin{equation*}
  \int_{\xi}^{1}\eta^{\alpha+\rho-2}\deta=\abs*{\log(\xi)}\quad \text{if }\rho+\alpha=1 \qquad \text{and}\qquad \int_{\xi}^{1}\eta^{\alpha+\rho-2}\deta\leq C(1+\xi^{\rho+\alpha-1}) \quad \text{if }\rho+\alpha\neq 1.
 \end{equation*}
Combining this with~\eqref{eq:Ker:single:int:2} the claim follows also for $\xi\leq 1$.
\end{proof}

\begin{lemma}\label{Lem:Ker:est:large:1}
 For each $\alpha\in(0,1)$ and $k\in\N_{0}$ there exists a constant $C>0$ such that
 \begin{equation*}
  \int_{0}^{\infty}\int_{0}^{\infty}\frac{\abs*{\Ker(\xi,\eta)}}{(\xi+r)^{k+\theta}\eta^{1-\rho}(1+\eta)^{\theta+\rho}}\deta\dxi\leq \frac{C}{r^{k+\theta-\alpha}}\quad \text{for all }r\geq 1
 \end{equation*}
with $\theta$ as given in~\eqref{eq:choice:theta:standard}
\end{lemma}

\begin{proof}
 Proposition~\ref{Prop:W:representation} implies
 \begin{multline*}
  \int_{0}^{\infty}\int_{0}^{\infty}\frac{\abs*{\Ker(\xi,\eta)}}{(\xi+r)^{k+\theta}\eta^{1-\rho}(1+\eta)^{\theta+\rho}}\deta\dxi\\*
  \leq C\int_{0}^{\infty}\frac{1}{(\xi+r)^{k+\theta}\xi^{\alpha}}\int_{0}^{\infty}\frac{1}{(\xi+\eta)^{1-\alpha}\eta^{1-\rho}(1+\eta)^{\theta+\rho}}\deta\dxi\\*
  +C\int_{0}^{\infty}\frac{1}{(\xi+r)^{k+\theta}}\int_{0}^{\infty}\frac{1}{(\xi+\eta)^{1-\alpha}\eta^{1+\alpha-\rho}(\eta+1)^{\theta+\rho}}\deta\dxi\\*
  +C\int_{0}^{\infty}\frac{1}{(\xi+r)^{k+\theta}\xi^{1-\rho}(1+\xi)^{\theta+\rho}}\dxi.
 \end{multline*}
 The splitting $1-\alpha=\max\{1-2\alpha,0\}+\min\{\alpha,1-\alpha\}$ induces the estimate $(\xi+\eta)^{\alpha-1}\leq \xi^{-\max\{1-2\alpha,0\}}\eta^{-\min\{\alpha,1-\alpha\}}$. Thus, we infer together with $(\xi+\eta)^{\alpha-1}\leq \xi^{\alpha-1}$ and $(\xi+r)^{-k-\theta}\leq r^{-k-\theta}$ that
 \begin{multline*}
  \int_{0}^{\infty}\int_{0}^{\infty}\frac{\abs*{\Ker(\xi,\eta)}}{(\xi+r)^{k+\theta}\eta^{1-\rho}(1+\eta)^{\theta+\rho}}\deta\dxi\\*
  \leq C\int_{0}^{\infty}\frac{1}{(\xi+r)^{k+\theta}\xi^{\max\{1-\alpha,\alpha\}}}\dxi\int_{0}^{\infty}\frac{1}{\eta^{\min\{\alpha,1-\alpha\}+1-\rho}(\eta+1)^{\theta+\rho}}\deta\\*
  +C\int_{0}^{\infty}\frac{1}{(\xi+r)^{k+\theta}\xi^{1-\alpha}}\dxi\int_{0}^{\infty}\frac{1}{\eta^{1+\alpha-\rho}(\eta+1)^{\theta+\rho}}\deta+\frac{1}{r^{k+\theta}}\int_{0}^{\infty}\frac{1}{\xi^{1-\rho}(\xi+1)^{\theta+\rho}}\dxi.
 \end{multline*}
 We estimate the integrals in $\eta$ and the last one in $\xi$ by a constant and change variables $\xi\mapsto r\xi$ in the remaining ones to get
 \begin{equation*}
  \begin{split}
   &\phantom{{}\leq{}}\int_{0}^{\infty}\int_{0}^{\infty}\frac{\abs*{\Ker(\xi,\eta)}}{(\xi+r)^{k+\theta}\eta^{1-\rho}(1+\eta)^{\theta+\rho}}\deta\dxi\\
   &\leq \frac{C}{r^{k+\theta+\max\{1-\alpha,\alpha\}-1}}\int_{0}^{\infty}\frac{1}{(\xi+1)^{k+\theta}\xi^{\max\{1-\alpha,\alpha\}}}\dxi+\frac{C}{r^{k+\theta-\alpha}}\int_{0}^{\infty}\frac{1}{(\xi+1)^{k+\theta}\xi^{1-\alpha}}\dxi+\frac{1}{r^{k+\theta}}\\
   &\leq \frac{C}{r^{k+\theta-\alpha}}.
  \end{split}
 \end{equation*}
 Note that in the last step we exploited that $r\geq 1$ and that the integrals in $\xi$ are finite because $\alpha\in(0,1)$ and $\theta>\alpha$.
\end{proof}

\begin{lemma}\label{Lem:Ker:est:large:2}
 For each $\alpha\in(0,1)$ there exists a constant $C>0$ such that it holds
 \begin{equation*}
  \int_{0}^{\infty}\int_{0}^{\infty}\frac{\abs*{\Ker(\xi,\eta)}}{(\xi+r)^{2+\theta}(\eta+1)^{\theta}}\deta\dxi\leq \frac{C}{r^{1+\theta+\frac{\theta+\alpha}{2}}}\quad \text{for all }r\geq 1
 \end{equation*}
with $\theta$ as given in~\eqref{eq:choice:theta:standard}
\end{lemma}

\begin{proof}
 The estimate follows similarly to that one in Lemma~\ref{Lem:Ker:est:large:1} while here one can use one time the splitting $\alpha-1=-\frac{\theta-\alpha}{2}+\frac{\alpha+\theta}{2}-1$ and one time $\alpha-1=-\frac{\theta}{2}+\alpha+\frac{\theta}{2}-1$ to obtain corresponding estimates for $(\xi+\eta)^{\alpha-1}$.
\end{proof}

\begin{lemma}\label{Lem:kernel:est:small:1}
 For any $\alpha\in(0,1)$ and $\mu\in[0,\mu_{*})$ there exists a constant $C_{\mu,\alpha}>0$ such that it holds
 \begin{equation*}
  \int_{0}^{\infty}\int_{0}^{\infty}\abs*{\Ker(\xi,\eta)}\frac{1}{(\xi+r)^{1-\rho}(\xi+r+1)^{\theta+\rho}}\frac{1}{\eta^{1-\rho}(\eta+1)^{\theta+\rho}}\deta\dxi\leq C_{\mu,\alpha}r^{\rho+\mu-1}
 \end{equation*}
 for every $r\in(0,1)$ and $\theta$ as given in~\eqref{eq:choice:theta:standard}
\end{lemma}

\begin{proof}
 We use $(\xi+r+1)^{-\theta-\rho}\leq 1$ and $(\xi+r)^{\rho-1}(\xi+r+1)^{-\theta-\rho}\leq \xi^{-1-\theta}$ to deduce together with Lemma~\ref{Lem:kernel:est:partial:1} that
 \begin{multline}\label{eq:kernel:est:small:1}
  \int_{0}^{\infty}\int_{0}^{\infty}\abs*{\Ker(\xi,\eta)}\frac{1}{(\xi+r)^{1-\rho}(\xi+r+1)^{\theta+\rho}}\frac{1}{\eta^{1-\rho}(\eta+1)^{\theta+\rho}}\deta\dxi\\*
  \leq C\int_{0}^{1}\frac{1}{(\xi+r)^{1-\rho}}\biggl(\frac{1}{\xi^{1-\rho}}+\frac{1}{\xi^{\alpha}}\biggr)\abs*{\log(\xi)}\dxi+C\int_{1}^{\infty}\frac{1}{\xi^{1+\theta}\xi^{1-\alpha}}\dxi\\*
  \leq C\int_{0}^{1}\frac{1}{(\xi+r)^{1-\rho}}\frac{\abs*{\log(\xi)}}{\xi^{1-\rho}}\dxi+C\int_{0}^{1}\frac{1}{(\xi+r)^{1-\rho}}\frac{\abs*{\log(\xi)}}{\xi^{\alpha}}\dxi+C.
 \end{multline}
The constant term on the right-hand side can be estimated trivially by $r^{\rho+\mu-1}$ since $r\leq 1$ and $\mu<1-\rho$. Thus, it suffices to consider the two remaining integrals. For the second one, one immediately gets
\begin{equation}\label{eq:kernel:est:small:2}
 \int_{0}^{1}\frac{1}{(\xi+r)^{1-\rho}}\frac{\abs*{\log(\xi)}}{\xi^{\alpha}}\dxi\leq C
\end{equation}
since $\alpha<\rho$. The first integral can be estimated for example by
\begin{equation}\label{eq:kernel:est:small:3}
 \int_{0}^{1}\frac{1}{(\xi+r)^{1-\rho}}\frac{\abs*{\log(\xi)}}{\xi^{1-\rho}}\dxi\leq C\frac{1+\abs*{\log(r)}^2}{r^{1-\rho-\mu_{*}}}.
\end{equation}
The derivation of this estimate is in principle elementary but also a bit lengthy since one has to consider the three cases $\rho<1/2$, $\rho=1/2$ and $\rho>1/2$ separately. Therefore, we omit the details which are contained in~\cite{Thr16} and only remark that one can derive this using the change of variables $\xi\mapsto r\xi$ together with the estimate $\abs*{\log(\xi r)}\leq (1+\abs*{\log(\xi)})(1+\abs*{\log(r)})$.

Combining the estimates~\cref{eq:kernel:est:small:1,eq:kernel:est:small:2,eq:kernel:est:small:3} the claim follows readily from the choice of $\mu$ and the fact that $\abs*{\log(r)}$ can be estimated by any small power of $r$ for $r\leq 1$.
\end{proof}

\begin{lemma}\label{Lem:kernel:est:primitive}
 For each $\alpha\in(0,1)$ there exists a constant $C>0$ such that it holds
 \begin{equation*}
  \int_{0}^{\xi}\abs*{\Ker(s,\eta)}\ds\leq C\biggl(\frac{\xi^{\alpha}}{\eta^{\alpha}}+1\biggr).
 \end{equation*}
 for all $\xi,\eta\in(0,\infty)$.
\end{lemma}

\begin{proof}
 Proposition~\ref{Prop:W:representation} yields together with the change of variables $s\mapsto \eta s$ that
 \begin{equation*}
  \begin{split}
   \int_{0}^{\xi}\abs*{\Ker(s,\eta)}\ds&\leq C\int_{0}^{\xi}\frac{1}{(s+\eta)^{1-\alpha}}\biggl(\frac{1}{s^{\alpha}}+\frac{1}{\eta^{\alpha}}\biggr)\ds+C\int_{0}^{\infty}\delta(s-\eta)\ds\\
   &\leq C\int_{0}^{\frac{\xi}{\eta}}\frac{1}{(1+s)^{1-\alpha}}\biggl(\frac{1}{s^{\alpha}}+1\biggr)\ds+C\\
   &=C\int_{0}^{\frac{\xi}{\eta}}\frac{1}{(1+s)^{1-\alpha}s^{\alpha}}+C\biggl(1+\frac{\xi}{\eta}\biggr)^{\alpha}+C.
  \end{split}
 \end{equation*}
 Since $\alpha\in(0,1)$ one immediately verifies that $\int_{0}^{\xi/\eta}(1+s)^{\alpha-1}s^{-\alpha}\ds\leq C+\abs*{\log(\xi/\eta)}$ if one considers the cases $\xi/\eta\leq 1$ and $\xi/\eta\geq 1$ separately. Exploiting that $(1+\xi/\eta)^{\alpha}\leq 1+(\xi/\eta)^{\alpha}$ as well as $\abs*{\log(\xi/\eta)}\leq C(1+(\xi/\eta)^{\alpha})$ the claim directly follows.
\end{proof}

\section*{Acknowledgements} 
This article is based on parts of the PhD thesis of the author which has been supported through the CRC 1060 \textit{The mathematics of emergent effects} at the University of Bonn that is funded through the German Science Foundation (DFG). Moreover, partial support by a Lichtenberg Professorship Grant of the VolkswagenStiftung awarded to Christian Kühn is acknowledged.


\begin{thebibliography}{10}

\bibitem{AcF97}
Azmy~S. Ackleh and Ben~G. Fitzpatrick.
\newblock Modeling aggregation and growth processes in an algal population
  model: analysis and computations.
\newblock {\em J. Math. Biol.}, 35(4):480--502, 1997.

\bibitem{CaM11}
Jos{\'e}~A. Ca{\~n}izo and St{\'e}phane Mischler.
\newblock Regularity, local behavior and partial uniqueness for self-similar
  profiles of {S}moluchowski's coagulation equation.
\newblock {\em Rev. Mat. Iberoam.}, 27(3):803--839, 2011.

\bibitem{Dra72}
R.~Drake.
\newblock A general mathematical survey of the coagulation equation.
\newblock In G.~M. Hidy, editor, {\em Topics in current aerosol research (part
  2)}, International Reviews in Aerosol Physics and Chemistry, pages 203--376.
  Pergamon Press, Oxford, 1972.

\bibitem{EMR05}
M.~Escobedo, S.~Mischler, and M.~Rodriguez~Ricard.
\newblock On self-similarity and stationary problem for fragmentation and
  coagulation models.
\newblock {\em Ann. Inst. H. Poincar\'e Anal. Non Lin\'eaire}, 22(1):99--125,
  2005.

\bibitem{FoL05}
Nicolas Fournier and Philippe Lauren{\c{c}}ot.
\newblock Existence of self-similar solutions to {S}moluchowski's coagulation
  equation.
\newblock {\em Comm. Math. Phys.}, 256(3):589--609, 2005.

\bibitem{FoL06a}
Nicolas Fournier and Philippe Lauren{\c{c}}ot.
\newblock Well-posedness of {S}moluchowski's coagulation equation for a class
  of homogeneous kernels.
\newblock {\em J. Funct. Anal.}, 233(2):351--379, 2006.

\bibitem{Fri00}
Sheldon~K. Friedlander.
\newblock {\em Smoke, Dust, and Haze: Fundamentals of Aerosol Dynamics}.
\newblock Topics in Chemical Engineering. Oxford University Press, 2000.

\bibitem{HNV17}
Michael Herrmann, Barbara Niethammer, and Juan J.~L. Vel\'azquez.
\newblock Instabilities and oscillations in coagulation equations with kernels
  of homogeneity one.
\newblock {\em Quart. Appl. Math.}, 75(1):105--130, 2017.

\bibitem{LaM04}
Philippe Lauren{\c{c}}ot and St{\'e}phane Mischler.
\newblock On coalescence equations and related models.
\newblock In {\em Modeling and computational methods for kinetic equations},
  Model. Simul. Sci. Eng. Technol., pages 321--356. Birkh\"auser Boston,
  Boston, MA, 2004.

\bibitem{LNV16}
Philippe Lauren{\c{c}}ot, Barbara Niethammer, and Juan J.~L. Vel{\'a}zquez.
\newblock Oscillatory dynamics in {S}moluchowski's coagulation equation with
  diagonal kernel.
\newblock {\em Preprint arXiv:1603.02929}, March 2016.

\bibitem{MeP04}
Govind Menon and Robert~L. Pego.
\newblock Approach to self-similarity in {S}moluchowski's coagulation
  equations.
\newblock {\em Comm. Pure Appl. Math.}, 57(9):1197--1232, 2004.

\bibitem{NTV15}
B.~Niethammer, S.~Throm, and J.~J.~L. Vel{\'a}zquez.
\newblock A revised proof of uniqueness of self-similar profiles to
  {S}moluchowski's coagulation equation for kernels close to constant.
\newblock {\em Preprint arXiv:1510.03361v2}, October 2015.

\bibitem{NTV16a}
B.~Niethammer, S.~Throm, and J.~J.~L. Vel\'azquez.
\newblock Self-similar solutions with fat tails for {S}moluchowski's
  coagulation equation with singular kernels.
\newblock {\em Ann. Inst. H. Poincar\'e Anal. Non Lin\'eaire},
  33(5):1223--1257, 2016.

\bibitem{NTV16}
B.~Niethammer, S.~Throm, and J.~J.~L. Vel{\'a}zquez.
\newblock A uniqueness result for self-similar profiles to {S}moluchowski’s
  coagulation equation revisited.
\newblock {\em J Stat Phys}, 164(2):399–409, Jun 2016.

\bibitem{NiV13a}
B.~Niethammer and J.~J.~L. Vel{\'a}zquez.
\newblock Self-similar solutions with fat tails for {S}moluchowski's
  coagulation equation with locally bounded kernels.
\newblock {\em Comm. Math. Phys.}, 318(2):505--532, 2013.

\bibitem{NiV14a}
B.~Niethammer and J.~J.~L. Vel{\'a}zquez.
\newblock Uniqueness of self-similar solutions to {S}moluchowski's coagulation
  equations for kernels that are close to constant.
\newblock {\em J. Stat. Phys.}, 157(1):158--181, 2014.

\bibitem{Nor99}
James~R. Norris.
\newblock Smoluchowski's coagulation equation: uniqueness, nonuniqueness and a
  hydrodynamic limit for the stochastic coalescent.
\newblock {\em Ann. Appl. Probab.}, 9(1):78--109, 1999.

\bibitem{PrK10}
H.R. Pruppacher and J.D. Klett.
\newblock {\em Microphysics of Clouds and Precipitation}.
\newblock Springer Netherlands, 2010.

\bibitem{Thr16}
Sebastian Throm.
\newblock {\em Self-Similar Solutions with fat tails for {S}moluchowski's
  coagulation equation}.
\newblock PhD thesis, University of Bonn, 2016.

\bibitem{Thr17}
Sebastian Throm.
\newblock Precise tail behaviour of self-similar profiles with infinite mass
  for {S}moluchowski's coagulation equation.
\newblock {\em Preprint arXiv:1703.09192}, March 2017.

\bibitem{Smo17}
Marian von Smoluchowski.
\newblock Versuch einer mathematischen {T}heorie der {K}oagulationskinetik
  kolloider {L}{\"o}sungen.
\newblock {\em Zeitschrift f{\"u}r physikalische Chemie}, 92:129 -- 168, 1917.

\bibitem{Yos78}
K{\^o}saku Yosida.
\newblock {\em Functional analysis}.
\newblock Springer-Verlag, Berlin-New York, fifth edition, 1978.
\newblock Grundlehren der Mathematischen Wissenschaften, Band 123.

\bibitem{Zif80}
Robert~M. Ziff.
\newblock Kinetics of polymerization.
\newblock {\em J Stat Phys}, 23(2):241–263, Aug 1980.

\end{thebibliography}
\end{document}